\documentclass[11pt]{amsart}

\usepackage{geometry}
\usepackage{amssymb}
\usepackage{amsmath}
\usepackage{mathabx}
\usepackage[latin1]{inputenc}
\usepackage{enumitem}
\usepackage[T1]{fontenc}
\usepackage{mathtools}
\usepackage{xcolor}
\usepackage{hyperref}

\usepackage[swedish,english]{babel}
\usepackage{textcomp}

\usepackage{mathtools, stmaryrd}

\usepackage{pstricks}
\usepackage{pst-plot}

\definecolor{mygray}{gray}{0.8}
\definecolor{mygray2}{gray}{0.8}

\renewcommand{\hat}{\widehat}
\newcommand{\ti}{\widetilde}
\renewcommand{\bar}{\overline}

\newcommand{\iii}[1]{{\left\vert\kern-0.25ex\left\vert\kern-0.25ex\left\vert #1 
    \right\vert\kern-0.25ex\right\vert\kern-0.25ex\right\vert}}

\def\Z{{\mathbb Z}}\def\T{{\mathbb T}}\def\R{{\mathbb R}}\def\C{{\mathbb C}}
\def\cG{\mathcal{G}}
\def\cT{\mathcal{T}}

\def\CC{{\mathcal C}}

\def\e{\eta}
\def\g{\gamma}
\def\be{\beta}\def\d{\delta}
\def\be{\beta}
\def\th{\theta}

\def\L{\Lambda}\def\D{\Delta}

\def\cU{\mathcal{U}}

\def\Leb{\mathrm{Leb}}

\def\toitself{\righttoleftarrow}
\def\dist{\mathrm{dist}}

\def\b#1{\lbrace#1\rbrace}
\def\a#1{\left|#1\right|}

\def\l#1{\langle #1\rangle}

\def\<{\langle}
\def\>{\rangle}

\theoremstyle{plain}
\newtheorem{Thm}{Theorem}[section]

\newtheorem{defin}[Thm]{Definition}

\newtheorem{notation}{Notation}[section]

\newtheorem*{MainMain*}{Main Theorem}
\newtheorem{MainMain}{Main Theorem}

\newtheorem{MainMainprime}{Main Theorem}

\newtheorem*{question*}{Question}
\newtheorem{question}{Question}

\newtheorem*{TheoOthers}{Theorem}
\newtheorem*{Principle*}{Principle}

\newtheorem{Main}{Theorem}

\newtheorem{Mainprime}{Theorem}

\def\iff{\Longleftrightarrow}

\def\cC{\mathcal C}

\newcommand{\ph}{\varphi}

\def\a{\alpha}

\def\d{\delta}
\def\ti{\tilde}
\def\e{\varepsilon}
\def\pa{\partial}
\def\s{\sigma}
\def\th{\theta}
\def\ud{\mathop{\underline{\textrm{d}}}}

\def\urad{\mathop{\underline{\textrm{rad}}}}

\def\ovrad{\mathop{\overline{\textrm{rad}}}}
\def\ua{\mathop{\underline{\textrm{a}}}}

\def\uc{\underline{\textrm{card}}}

\let\newpf\proof \let\proof\relax

\def\area{\operatorname {area}}

\def\cN{\mathcal{N}}
\def\cL{\mathcal{L}}
\def\cD{\mathcal{D}}

\newcommand{\ba}{\overline{A}}

\newcommand{\cF}{\mathcal{F}}

\newcommand{\cO}{\mathcal{O}}

\def\be{\begin{equation}}
\def\ee{\end{equation}}

\def\ba{{\begin{align}}}
\def\ea{{\end{align}}}

\def\bm{\begin{pmatrix}}
\def\em{\end{pmatrix}}

\def\a{{\alpha}}

\def\g{{\gamma}}

\def\fO{{\frak O}}

\def\fd{{\frak d}}
\def\cS{\mathcal{S}}

\def\bD{\mathbb{D}}

\def\0{{\mathbf 0}}

\newtheorem{theo}{Theorem}[section]
\newtheorem{thm}{Theorem}[section]
\newtheorem{cor}[thm]{Corollary}

\newtheorem{lemma}[thm]{Lemma}

\newtheorem{prop}[thm]{Proposition}

\theoremstyle{remark}
\newtheorem{rem}{Remark}[section]

\theoremstyle{definition}

\newenvironment{proof}{ \noindent{\it Proof.}\quad}{\ \hfill $\Box$\vskip .2cm}

\def\iff{\Longleftrightarrow}

\def\ssm{\smallsetminus}

\renewcommand{\setminus}{\ssm}

\newcommand{\De}{{\Delta}}

\newcommand{\N}{{\mathbb N}}
\newcommand{\Q}{{\mathbb Q}}
\newcommand{\bR}{{\mathbb R}}
\newcommand{\bT}{{\mathbb T}}
\newcommand{\bA}{{\mathbb A}}
\newcommand{\bB}{{\mathbb B}}
\newcommand{\bZ}{{\mathbb Z}}

\def\B0{{\bold{0}}}
\def\b{\beta}
\def\a{\alpha}

\def\l{\lambda}
\def\be{\begin{equation}}
\def\ee{\end{equation}}
\def\cA{\mathcal{A}}
\def\cA{\mathcal{A}}

\def\cC{\mathcal{C}}
\def\cE{\mathcal{E}}

\def\cF{\mathcal{F}}
\def\cM{\mathcal{M}}

\def\cQ{\mathcal{Q}}

\def\cU{\mathcal U}

\def\cX{\mathcal X}

\def\1{{\bf 1}}

\catcode`\@=12

\def\Empty{}
\newcommand\oplabel[1]{
  \def\OpArg{#1} \ifx \OpArg\Empty {} \else
  	\label{#1}
  \fi}
		
\newcommand{\comm}[1]{}
\newcommand{\comment}[1]{}

\renewcommand{\ti}{\widetilde}
\renewcommand{\check}{\widecheck}
\def\hs{\hskip 0.05cm}

\begin{document}

\title{On the divergence of Birkhoff Normal Forms}
\author{Raphaël Krikorian }

\address{
Department of Mathematics, CNRS UMR 8088, 
CY Cergy Paris Université (University of Cergy-Pontoise),  2, av. Adolphe Chauvin F-95302 Cergy-Pontoise, France.} 
\email{raphael.krikorian@cyu.fr.}
%


\thanks{This work was supported by a Chaire d'Excellence LABEX MME-DII, the project ANR BEKAM : ANR-15-CE40-0001 and an AAP project from CY Cergy Paris  Université.}

\maketitle
\begin{abstract}  It is well known that a  real  analytic  symplectic diffeomorphism of the $2d$-dimensional disk ($d\geq 1$) admitting the origin as a non-resonant elliptic fixed  can be {\it formally} conjugated to its Birkhoff Normal Form, a formal power series defining a {\it formal integrable} symplectic  diffeomorphism at the origin.   We prove in this paper that this Birkhoff Normal Form is in general divergent.  This solves, in any dimension,   the question of determining which of the two  alternatives of Perez-Marco's  theorem \cite{PM} is true and answers    a question by H. Eliasson.    Our  result is a consequence  of the fact that when $d=1$ the convergence of the formal object that is the BNF has strong dynamical consequences on the Lebesgue measure of the set of invariant circles in arbitrarily small neighborhoods of the origin. Our proof, as well as our results, extend to the case of real-analytic diffeomorphisms of the annulus  admitting a Diophantine invariant torus.

\end{abstract}

\tableofcontents 

\section{Introduction}
We consider in this paper {\it real analytic}  diffeomorphisms defined on an open set  of the $2d$-cartesian space $\R^{d}\times\R^d$ or respectively of the  $2d$-cylinder (or annulus) $(\R/2\pi\Z)^d\times \R^d$ ($d\geq 1$), which are {\it symplectic} with respect to  the canonical symplectic forms $\sum_{j=1}^d dx_{j}\wedge dy_{j}$, $(x,y)\in\R^d\times\R^d$, resp.  $\sum_{j=1}^dd\th_{j}\wedge dr_{j}$, $(\th,r)\in (\R/2\pi\Z)^d\times \R^d$, and leave invariant $\{(0,0)\}\in\R^d\times \R^d$, resp. the torus $\cT_{0}:=(\R/2\pi\Z)^d\times\{0\}\subset  (\R/2\pi\Z)^d\times \R^d$. We shall assume that the invariant sets $\{(0,0)\}\in\R^d\times\R^d$, resp.  $(\R/2\pi\Z)^d\times \{0\}$, are {\it elliptic equilibrium sets} in the following sense: there exists $\omega=(\omega_{1},\ldots\omega_{d})\in \R^d$, the {\it frequency vector}, such that 
\be\begin{cases} f:(\R^{d}\times\R^d,(0,0))\righttoleftarrow,\qquad  f=Df(0,0)\circ (id+O^2(x,y))\\
{\rm spec}(Df(0,0))=\{e^{\pm2\pi \sqrt{-1}\omega_{j}},\  1\leq j \leq d\}
\end{cases}\label{n1.1}\ee
and respectively
\be f:((\R/2\pi\Z)^d\times \R^{d},\cT_{0})\righttoleftarrow,\qquad  f(\th,r)=(\th+\omega,r)+(O(r),O(r^2)).\label{n1.1bis}\ee
We can assume without loss of generality  that in (\ref{n1.1}) the derivative $Df(0,0)$ of $f$ at the fixed point $(0,0)$ is a {\it symplectic rotation}:
 for any  $x=(x_{1},\ldots,x_{d})$, $y=(y_{1},\ldots,y_{d})$, $\ti x=(\ti x_{1},\ldots,\ti x_{d})$, $\ti y=(\ti y_{1},\ldots,\ti y_{d})$ one has ($i=\sqrt{-1}$)
$$Df(0,0)\cdot(x,y)=(\ti x,\ti y) \ \iff\  \begin{cases}& \ti x_{j}+i\ti y_{j}=e^{2\pi i \omega_{j}}(x_{j}+i y_{j}) \\
&\forall \ 1\leq j\leq d\end{cases}.$$

We shall refer to  situation (\ref{n1.1}) as the {\it Elliptic fixed point} or the {\it Cartesian Coordinates} ((CC) for short) case and to  situation (\ref{n1.1bis}) as the {\it Action-Angle} ((AA) for short)  case.

Important examples of such diffeomorphisms are provided by flows  $(\Phi_{H}^t)_{t\in\R}$, or  by  suitable  Poincaré sections  on some energy level, of Hamiltonian systems
$$\dot x=\frac{\pa H}{\pa y}(x,y),\quad \dot y=-\frac{\pa H}{\pa x}(x,y),\quad \textrm{resp.}\quad \dot \th=\frac{\pa H}{\pa r}(\th,r),\quad \dot r=-\frac{\pa H}{\pa \th}(\th,r)$$
where $H:(\R^{d'}\times\R^{d'},(0,0))\to\R$ resp. $H:((\R/2\pi\Z)^{d'}\times\R^{d'},\cT_{0})\to\R$ 
($d'=d$ or $d'=d+1$)
 is real analytic and satisfies 
 \begin{align}&\textrm{(CC)-case}\qquad H(x,y)=2\pi\sum_{j=1}^{d'}\omega_{j} \frac{x_{j}^2+y_{j}^2}{2}+O^3(x,y),\label{HamCaseCC}\\
&\textrm{(AA)-case}\qquad  H(\th,r)=2\pi\sum_{j=1}^{d'}\omega_{j} r_{j}+O(r^2).\label{HamCaseAA}
\end{align}

If we denote by $\Phi_{H}$ the time-1 map of a  hamiltonian $H$ and define the observable 
$r_{j}:(x,y) \mapsto (1/2)(x_{j}^2+y_{j}^2)$, resp. $r_{j}:r\mapsto r_{j}$ ($1\leq j\leq d$)
 we can write (\ref{n1.1}), resp. (\ref{n1.1bis}), as
\begin{align}&\textrm{(CC)-case}\quad f:(\R^{d}\times\R^d,(0,0))\righttoleftarrow,\qquad  f=\Phi_{2\pi\<\omega, r\>}+O^2(x,y)\label{newn1.1}\\
 &\textrm{(AA)-case}\quad f:((\R/2\pi\Z)^d\times \R^{d},\cT_{0})\righttoleftarrow,\qquad  f=\Phi_{2\pi\<\omega, r\>}+(O(r),O(r^2))\label{newn1.1bis}
 \end{align}
where $\<\omega,r\>=\sum_{j=1}^d\omega_{j}r_{j}$, $r=(r_{1},\ldots,r_{d})$. 

The representations (\ref{newn1.1}), resp. (\ref{newn1.1bis}), give a very rough understanding of the behavior of the {\it finite time} dynamics of the diffeomorphism $f$ in a neighborhood  of the elliptic equilibrium sets $\{(0,0)\}$, resp. $\cT_{0}$:
it is interpolated\footnote{For $(x,y)$ $\e$-close to  $(0,0)$ and $n\in\N$ not too large $n=O(\e^{-\a})$, $0<\a<1$ the iterates $f^k(x,y)$, $k\leq n$ ($f^k$ denotes the composition $f\circ\mathop{\cdots}\circ f$, $k$ times) stay $\e^{2-\a}$-close to those of the the symplectic rotation, $\Phi^k_{2\pi\<\omega,r\>}(x,y)$.}by the   dynamics of $\Phi_{2\pi\<\omega,r\>}$ which is  {\it quasi-periodic} in the sense that all its orbits are quasi-periodic with frequencies $\omega_{1},\ldots\omega_{d}$.   Improving this approximation is an old   and important problem 
(it was  a central  theme of  research of the astronomers of the XIXth century; see the  references  of the very instructive introduction by P\'erez-Marco in \cite{PM}) that has a solution  at least in the (CC)-case (\ref{newn1.1})   if  the frequency vector  $\omega$ is {\it nonresonant}:  
 any relation $k_{0}+k_{1}\omega_{1}+\cdots+k_{d}\omega_{d}=0$ with $k_{0},k_{1},\ldots,k_{d}\in\Z$ implies that $k_{0}=k_{1}=\ldots=k_{d}=0$. Indeed,  after using nice changes of coordinates (symplectic transformations) 
one can 
 interpolate, in {\it small neighborhoods of the origin},  the dynamics of $f$ by  quasi-periodic ones
    with  much better orders of approximation and  for much longer times.  There are two remarkable features of this interpolation:  the first, is that  the frequencies of the interpolating quasi-periodic motions now depend on the initial point and do not  necessarily coincide with the frequencies at the origin;  the second, is that if one pushes the order of approximation, these frequencies stabilize in some sense.  This is the content of the famous {\it Birkhoff Normal Form Theorem}, formalized by Birkhoff in the 1920's \cite{Bi1}, \cite{Bi2}, \cite{SiMo},  which paved the way to  the major achievements of  the KAM theory (named after Kolmogorov, Arnold and Moser) in the 1960's, on the existence of (infinite time) quasi-periodic motions for a wide class of  diffeomorphisms of the form (\ref{n1.1}), (\ref{n1.1bis}); see  \cite{Ko}, \cite{Ar},  \cite{Mo} (and \cite{Ne} for finite time approximations).  We now describe in more details the Birkhoff Normal Form Theorem.

\subsection{Birkhoff Normal Forms}\label{sec:1.1}
We begin with the Elliptic fixed point case  ((CC)-case). The first statement of the Birkhoff Normal Form Theorem is the following.  For any $N\in\N^*$, there exist a  polynomial $B_{N}\in \R[r_{1},\ldots,r_{d}]$, $B_{N}(r)=2\pi\<\omega,r\>+O(r^2)$, of total degree $N$ and a symplectic diffeomorphism $Z_{N}:(\R^{d}\times \R^d,(0,0))\righttoleftarrow$ (preserving the standard symplectic form $\sum_{k=1}^d dx_{k}\wedge dy_{k}$ and tangent to the identity $Z_{N}=id+O^2(x,y)$) such that 
\be Z_{N}\circ f\circ Z_{N}^{-1}(x,y)=\Phi_{B_{N}}(x,y)+\cO^{2N+1}(x,y).\label{eq:1.1}
\ee   
The diffeomorphism $\Phi_{B_{N}}:(\R^d\times\R^d,0)\toitself$ is a {\it generalized symplectic rotation}   
 \be \Phi_{B_{N}}(x,y)=(\ti x,\ti y) \ \iff\  \begin{cases}& \ti x_{j}+i\ti y_{j}=e^{ i \pa_{j} B_{N}(r)}(x_{j}+i y_{j}) \\
&\forall \ 1\leq j\leq d\end{cases}\label{def:symprot}\ee
(recall $r=((1/2)(x_{1}^2+y_{1}^2),\ldots,(1/2)(x_{d}^2+y_{d}^2))$) and defines an {\it integrable} dynamics in a strong sense: every orbit of $\Phi_{B_{N}}$ is {\it quasi-periodic} and, in addition, the origin is {\it  Lyapunov stable}.  Indeed, for each $c=(c_{1},\ldots,c_{d})\in (\R^*_{+})^d$, the $d$-dimensional torus 
$$\cT_{c}:=\{(x,y)\in\R^{2d},\ \forall\ 1\leq j\leq d,\ r_{j}:=(1/2)(x_{j}^2+y_{j}^2)=c_{j}\}$$
is globally invariant by $\Phi_{B_{N}}$ and the restricted dynamics  of $\Phi_{B_{N}}$ on the torus $\cT_{c}\simeq \T^{d}:= \R^d/(2\pi\Z)^d$ is conjugated to a translation $\T^d\ni \th\mapsto \th+2\pi\omega(c)\in\T^d$ with {\it frequency vector} $\omega(c)=(2\pi)^{-1}\nabla B_{N}(c)$.  The dynamics of $\Phi_{B_{N}}$ is thus completely understood on the whole phase space\footnote{When $c$ has some zero components, $\cT_{c}$ is a $d_{c}$-dimensional torus, $0\leq d_{c}\leq d$, and  the restricted dynamics of $\Phi_{B_{N}}$ on $\cT_{c}$ is again conjugate to a translation on a torus.} $\R^{d}\times \R^d$.

Here comes the second part of the statement. The polynomials $B_{N}$ and the components of $Z_{N}-id$ 
 converge as {\it formal} power series when $N$ goes to infinity: $B_{N}\to B_{\infty}\in \R[[r_{1},\ldots ,r_{d}]]$, $Z_{N}\to Z_{\infty}\in \R[[x,y]]$ and, in $\R[[x,y]]$, one has the following formal conjugacy relation
 \be Z_{\infty}\circ f\circ Z_{\infty}^{-1}(x,y)=\Phi_{B_{\infty}}(x,y).\label{eq:1.2}
\ee
 The formal power series $B_{\infty}$ is unique if $Z_{\infty}$ is tangent to the identity and is therefore {\it invariant by (smooth or formal) conjugation} tangent to the identity; it    is called the {\it Birkhoff Normal Form} (BNF for short) of $f$ and we shall denote it by $BNF(f)$:
$$BNF(f)=B_{\infty}(r_{1},\ldots,r_{d})\in \R[[r_{1},\ldots,r_{d}]].
$$
On the other hand the formal conjugacy $Z_{\infty}$, which is called the {\it normalization transformation},  is not unique (but if properly normalized is unique). 

The preceding results hold  in the Action-Angle case (\ref{newn1.1bis}) but under a {\it Diophantine assumption} on $\omega$ (this is stronger that mere nonresonance):
\be \forall \ k\in\Z^d\setminus\{0\},\ \min_{l\in\Z}|\<k,\omega\>-l|\geq \frac{\kappa}{|k|^\tau}\qquad (\tau\geq d).\label{firstdc}\ee
The exponent $\tau$ is called the {\it exponent} of the Diophantine condition\footnote{The set of vectors of $\R^d$ satisfying a Diophantine condition with fixed exponent $\tau$ has positive Lebesgue measure if $\tau>d$ and the union of these sets on all $\tau\geq d$ has full Lebesgue measure in $\R^d$.}.
One can then prove similarly   the existence:  (a) for any $N\in\N^*$, of a polynomial $B_{N}\in \R[r_{1},\ldots,r_{d}]$, $B_{N}(r)=2\pi\<\omega,r\>+O(r^2)$ and of a symplectic diffeomorphism $Z_{N}:((\R/2\pi\Z)^d\times\R^d,\cT_{0})\righttoleftarrow$ (preserving the standard symplectic form $\sum_{k=1}^d d\th_{k}\wedge dr_{k}$, $Z_{N}=id+(O(r),O(r^2))$) such that 
\begin{align} &Z_{N}\circ f\circ Z_{N}^{-1}(\th,r)=\Phi_{B_{N}}(\th,r)+(\cO^N(r),\cO^{N+1}(r)) \label{eq:1.1.t}\\
&\Phi_{B_{N}}(\th,r)=(\th+\nabla B_{N}(r),r)\label{eq:1.1.tbis}
\end{align}
($\Phi_{B_{N}}$ is called an {\it integrable twist});
and: (b) of a formal power series  $B_{\infty}\in \R[[r_{1},\ldots,r_{d}]]$, the Birkhoff Normal Form,  and of a formal symplectic transformation $Z_{\infty}=id+(O(r),O(r^2))$ in $C^\omega(\T^d)[[r_{1},\ldots,r_{d}]]$ (the set of formal power series with coefficients in the set of real analytic functions $\T^d\to\T$) such that one has in $C^\omega(\T^d)[[r_{1},\ldots,r_{d}]]$ the formal conjugation relation
\be Z_{\infty}\circ f\circ Z_{\infty}^{-1}(\th,r)=(\th+\nabla B_{\infty}(r),r).\label{eq:1.2.t}
\ee
Again we denote $BNF(f)=B_{\infty}(r_{1},\ldots,r_{d})\in \R[[r_{1},\ldots,r_{d}]]$.

All the preceding discussion on Birkhoff Normal Forms holds if we only assume $f$ to be smooth\footnote{In the $C^k$ category, one cannot define in general $BNF(f)$.}. 
We can summarize this:
 \begin{TheoOthers}[Birkhoff] Any  smooth symplectic diffeomorphism $f:(\R^{d}\times\R^d,(0,0))\toitself$ ($d\geq 1$) (resp. $f:((\R/2\pi\Z)^{d}\times\R^d,\cT_{0})\toitself$) admitting the origin as a non resonant elliptic fixed point (resp. of the form (\ref{newn1.1bis}) with $\omega$ Diophantine) is formally (strongly) integrable: it is conjugated  in $\R[[x,y]]$ (resp. $C^\infty((\R/2\pi\Z)^d)[[r]]$) to the formal generalized symplectic rotation (resp. the  formal integrable twist) $\Phi_{BNF(f)}$. The formal series $BNF(f)$ is an invariant of formal conjugation. 
\end{TheoOthers}

These  formal (and approximate) Birkhoff Normal Forms can be defined in the more classical setting of Hamiltonian flows $\dot x=\frac{\pa H}{\pa y}(x,y),\quad \dot y=-\frac{\pa H}{\pa x}(x,y)$ (or $\dot \th=\frac{\pa H}{\pa r}(\th,r),\quad \dot r=-\frac{\pa H}{\pa \th}(\th,r)$): $f$ and $\Phi_{B}$ ($B=BNF(f)$)  are then replaced by $(\Phi^t_{H})_{t\in\R}$ an $(\Phi^t_{B})_{t\in\R}$ in (\ref{eq:1.2}), (\ref{eq:1.2.t})  (we shall then write $B=BNF(H)$\footnote{A more classic  equivalent formulation is $H=B\circ Z$. }).
In the Hamiltonian case, there is a weaker notion of integrability, usually called {\it Poisson integrability}, which corresponds to the situation where  the considered Hamiltonian has a complete system of functionally independent integrals (observables constant under the motion) which commute for the Poisson bracket.

Poincaré  discovered \cite{Po} that,  in general, real analytic  Hamiltonian flows do not admit other analytic first integrals than the Hamiltonian itself and hence  that in general no relation like \ref{eq:1.2} can hold with  converging $Z_{\infty}$ and $B_{\infty}$. Siegel proved  \cite{Si} in 1954 (see also \cite{Si2}, \cite{SiMo}, \cite{V1}, \cite{PM})   that, whatever the {\it fixed}  non resonant frequency vector at the origin $\omega$ is,  the normalizing conjugation $Z_{\infty}$ cannot in general\footnote{Here it means $G_{\d}$-dense in some set of real analytic functions with fixed radius of convergence. This phenomenon is even  ``prevalent'' as shown by Pérez-Marco \cite{PM}.} define a convergent series.  Indeed,  the existence of a convergent normalizing transformation yields real analytic Poisson integrability\footnote{If $Z_{\infty}$ converges the observables $r_{j}\circ Z_{\infty}$, $j=1,\ldots, d'$  are a complete set of real analytic and functionally independent Poisson commuting integrals.} a fact (known to Birkhoff \cite{Bi1}) that is not compatible with the richness\footnote{By which we mean the coexistence of quasi-periodic motions and hyperbolic behavior in any neigborhood of the equilibrium; see for a global view on these topics and references  the book \cite{AKN}.} of a generic dynamics near a non resonant elliptic equilibrium.   Note that the converse statement is true: real analytic Poisson integrability implies the existence of a real analytic normalizing Birkhoff transformation  ({\it cf.} \cite{It}, \cite{Kap}, \cite{Zung}).

As for the Birkhoff Normal Form itself,   H. Eliasson formulated the following natural question \cite{E1}, \cite{E2} (see also the references in \cite{PM}):
\begin{question}[Eliasson]
Are there examples of real analytic symplectic diffeomorphisms or  Hamiltonians  admitting divergent ({\it i.e.} with a null radius of convergence)   Birkhoff Normal Form?
\end{question}
The preceding  question has an easy  positive answer in the smooth case  (the map $f$ is only assumed to be smooth): indeed, one can choose $f$ to be of the form $f=\Phi_{\Omega}$ where $\Omega:(\R^{d},0)\to\R$ is smooth with a divergent Taylor series at the origin; since equalities (\ref{eq:1.2}) (\ref{eq:1.2.t})  only depend on the infinite jet  $J(f)$ of $f$ at 0, the special integrable form of $f$ implies  $BNF(f)=J(f)$ thus $BNF(f)$ is diverging. The situation is not so clear if  $f$ is real analytic.  In contrast with the aforementioned generic divergence of the normalizing transformation,
  there  seems\footnote{We shall in fact see in this paper  that {\it there are} such dynamical obstructions.} to be  {\it a priori} no obvious dynamical obstruction\footnote{Like the accumulation at the origin of hyperbolic periodic points or normally hyperbolic tori.} to the  divergence of the Birkhoff Normal Form.    

The first breakthrough in connection with   Eliasson's question    came from R. Pérez-Marco \cite{PM} who proved,  in the setting of Hamiltonian systems having a non resonant elliptic fixed point, the following   {\it dichotomy}:

\begin{TheoOthers}[Pérez-Marco \cite{PM}] 
 For {\rm  any fixed  nonresonant }frequency vector $\omega\in\R^{d'}$, $d'\geq 2$,  one has the following {\it dichotomy}: either for all real analytic Hamiltonian  $H$  of the form (\ref{HamCaseCC})  $BNF(H)$ converges (defines a converging analytic series) or there is a ``prevalent'' set of such $H$ for which $BNF(H)$ diverges.
\end{TheoOthers}
We refer to Subsection \ref{sec:1.2.2} for a precise definition of ``prevalent''.
A similar dichotomy holds in the setting of real analytic symplectic diffeomorphisms  in the  (CC)-case, and, both in the Hamiltonian or diffeomorphism framework,  it can be extended to the (AA)-case (but under the stronger assumption that  $\omega$  is Diophantine); {\it cf.} Theorem \ref{theo:PMdichotomy} of our paper.

Pérez-Marco's argument is not based on an analysis of the {\it dynamics} of $f$  but rather focuses on the {\it coefficients} of  the BNF and exploit their polynomial dependence on the coefficients of the initial perturbation  by using  techniques from potential theory.

The following two Theorems are  an answer (in the symplectomorphism setting) to Eliasson's question and decide which of the two assertions of Pérez-Marco's alternative holds (see Theorem \ref{theo:main3} of Subsection \ref{sec:1.2.2} for a more precise statement).

\begin{MainMain}[(CC)-Case]\label{MainTheorem1}For any $d\geq 1$ and any nonresonant frequency vector $\omega\in\R^d$,  there exists a ``prevalent'' set of real analytic symplectic diffeomorphism  $f:(\R^d\times \R^d,(0,0))\toitself $ of the form (\ref{newn1.1})   the Birkhoff Normal Forms of which  are divergent.
\end{MainMain}
In the Action-Angle Case (\ref{newn1.1bis}) it takes the following form:
\begin{MainMainprime}[(AA)-Case] \label{MainTheorem1AA}For any $d\geq 1$ and any Diophantine frequency vector $\omega\in\R^d$, there exists a ``prevalent'' set of real analytic symplectic diffeomorphism   $f:((\R/2\pi\Z)^d\times \R^{d},\cT_{0})$ of the form (\ref{newn1.1bis})   the Birkhoff Normal Forms of which  are divergent.
\end{MainMainprime}
Main Theorems \ref{MainTheorem1}, \ref{MainTheorem1AA} also extend to the Hamiltonian case (\ref{HamCaseCC})-(\ref{HamCaseAA}) (with $d'=d+1$)\footnote{It is not clear whether one can, for general systems, deduce the case of  Hamiltonian flows from the case  of diffeomorphisms and {\it vice versa}. On the other hand the proofs of  Main Theorems \ref{MainTheorem1}, \ref{MainTheorem1AA} and in particular the {proofs} of Main Theorem \ref{MainTheorem2} and  of  Theorems  \ref{theo:main1}--\ref{theo:main2},  \ref{theo:mainprime1}--\ref{theo:main2prime} below extend to Hamiltonian flows with $1+1$ degrees of freedom.}. 
Note that from Pérez-Marco's Theorem (and its analogue in the symplectomorphism case), in order  to prove that the divergence of the Birkhoff Normal Form holds in a prevalent way, it is enough to provide, for each fixed frequency vector $\omega$, one example for which the  BNF is divergent. On the other hand, if one is able to construct one such example for some $d_{0}$, then it is easy to construct other such examples for any  
$d> d_{0}$ (see for example the proof of Theorem \ref{theo:main3}).  Proving  Main Theorem \ref{MainTheorem1} (resp. \ref{MainTheorem1AA})  thus boils down to construct when $d=1$, for each irrational (resp. Diophantine)  $\omega\in\R$,   one example of a real analytic symplectic  diffeomorphism with a diverging BNF.

Gong already provided in  \cite{Gong} (by a direct analysis of the coefficients of the BNF)   examples of   real analytic Hamiltonians $\<\omega, r\>+F:(\R^2\times\R^2,0)\to\R$, $F=O^3(x,y)$,  with Liouvillian frequency $\omega\in\R^2$ at the origin and a divergent BNF and Yin \cite{Yin} produced  analogue of  Gong's examples in the diffeomorphism case (area preserving map of $(\R^2,0)$ with a very Liouvillian elliptic fixed point). In these examples the divergence of the BNF is caused by the presence of very small denominators (due to the Liouvillian character of $\omega$) appearing in the  coefficients of the BNF. 
After our result was announced,  Fayad \cite{F}  constructed simple  examples of real analytic Hamiltonian systems in $(\R^8,0)$  ($d'=d+1=4$ degrees of freedom) with any fixed nonresonant frequency vector at the origin and     divergent BNF. 
The  argument again is based on an analysis of the { coefficients} of the BNF;  one consider Hamiltonians with two degrees of freedom where two extra action variables are added as formal parameters, one of them appearing later in the denominators of the BNF.
These type of examples can be constructed in the diffeomorphism case for $d\geq 3$

 In a different context, that of reversible systems, let us mention a result of divergence of normal forms  in  \cite{GongStolo1} and a result of divergence of  normalizing transformations in \cite{MoWeb}. 

\medskip
 Eliasson's question can now be formulated in a  stronger form:
\begin{question}\label{QB} Does  the convergence of a {\rm formal} conjugacy invariant like the  Birkhoff Normal Form of a  real analytic symplectic diffeomorphism (or Hamiltonian) have consequences  on the {\rm dynamics} of the diffeomorphism (or  Hamiltonian)?
\end{question}
Note that  the similar question about the dynamical consequences of the convergence of the {\it Normalizing transformation} is answered  by  Siegel's Theorem   \cite{Si}: the convergence of the normalizing transformation of a Hamiltonian $(\R^4,0)\to\R$ admitting a non resonant elliptic fixed point   implies the existence of a three-parameter family of periodic solutions having arbitrarily large period and lying in arbitrarily small neighborhoods of the origin.  

As for Question \ref {QB}, there are indeed   various results pointing to some  kind of rigidity phenomena if analyticity (and some arithmetic properties on $\omega$) is assumed.  To be more specific, let us mention a striking one: Bruno \cite{Br} and R\"ussmann \cite{Ru} proved  that if $f$ is real analytic  and if its BNF  is {\it trivial}, ${BNF(f)}=2\pi\<\omega,r\>$ (in particular $BNF(f)$  converges),   then $f$ is real analytically conjugated to $\Phi_{2\pi\<\omega,r\>}$,  {\it provided} the  frequency vector at the origin $\omega$ satisfies a {\it Diophantine condition}.
We refer to \cite{V2}, \cite{It}, \cite{E1}, \cite{EcVa}, \cite {Stolovitch}, \cite{GongStolo2}, \cite{EFK}, \cite{EFKbis} for generalizations of the Bruno-R\"ussmann Theorem  and  related results.

The Main Result of our  paper is in some sense one answer, amongst possibly others, to the previous question at least when $d=1$ and if $f$ is assumed to satisfy  some {\it twist condition}. 

Let us say that a diffeomorphism of the form  $(\R^2,0)\toitself$ (\ref{n1.1}) or  $(\R\times\T,\cT_{0})\toitself$   is {\it twist} (or satisfies a  {\it twist condition}) if  the second order term of its BNF is not zero\footnote{An easily checkable condition.}: $(2\pi)^{-1}BNF(f)(r)=\omega r+b_{2}r^2+O(r^3)$, $b_{2}\ne 0$.
\begin{MainMain}\label{MainTheorem2}
If the Birkhoff Normal Form of a real-analytic symplectic  twist  diffeomorphism $(\R^2,0)\toitself$ (\ref{n1.1}) or  $(\R\times\T,\cT_{0})\toitself$    (\ref{n1.1bis})  converges
 then the measure  of  the complement of the  union  of all  {\it invariant curves} accumulating the origin  is  {\it much smaller} than what it is   for a general such diffeomorphism.
\end{MainMain}
In other words, the convergence of a {\it formal} object like the BNF has consequences on the {\it dynamics} of the diffeomorphism.
Precise statements are given  in Theorems \ref{theo:main1}--\ref{theo:main2},  \ref{theo:mainprime1}--\ref{theo:main2prime}) of Subsection \ref{ssec:1.2}.
Combined with  (the extension to the diffeomorphism case of) Pérez-Marco's Theorem \cite{PM},   this gives that  in any number of degrees of freedom,  a general  real analytic  symplectic diffeomorphism   admitting the origin as an non resonant elliptic equilibrium     has a divergent Birkhoff Normal Form (see Theorem \ref{theo:main3}).

 Having in mind the aforementioned result by Bruno and R\"ussmann,  a natural stronger question is whether the following  {\it rigidity} result is true:  
\begin{question}\label{QC} Is it true that a   real analytic symplectic diffeomorphism or  Hamiltonian system having a Diophantine elliptic equilibrium and a non degenerate  and  convergent  BNF is (real analytically)  integrable (in some neighborhood of the origin)?
\end{question}
The examples by  Farré and Fayad   in \cite{FaFé} of real analytic Hamiltonians on $\T^{d+1}\times \R^{d+1}$ with convergent BNF and with an unstable diophantine elliptic fixed point  show that such a generalization is not true   for $d\geq 2$ if by non degenerate we mean that $BNF(f)$ is trivial. The question is still open for $d<2$. Note that though in Farré-Fayad's examples the BNF (which is explicit)  is not completely degenerate (like it is   in the  Bruno-R\"ussmann Theorem) it has some degeneracies (it is not completely non degenerate). If one drops in the question  the Diophantine assumption and assumes the BNF to be trivial  the question is open (this question is related to a question of Birkhoff on pseudo-rotations (area preserving maps with no periodic points except the origin) and to  the problem  of constructing real analytic Anosov-Katok examples; {\it cf.} for details and references \cite{FaKri}).

When $d=1$ the situation might be more favorable. To any twist area preserving diffeomorphism  $f:(\R^2,0)\toitself$ (\ref{n1.1}) or  $f:(\R\times\T,\cT_{0})\toitself$    (\ref{n1.1bis}) one can associate (we use the notations  and terminology of \cite{Sib}) its {\it minimal action} $\a:I\to\R$ ($I$ is an open interval containing $\omega$) that assigns to each $\ph\in I$ the average action of any minimal orbit with rotation number $\ph$. The function $\a$ is strictly convex (in fact differentiable at any irrational) and one can thus define  its Legendre conjugate function $\a^*:r\mapsto\sup_{\ph \in I}(\ph r-\a(\ph))$ \footnote{The functions $\a$ and $\a^*$ are called Mather's or also   $\b$ and $\a$ Mather's functions.} (see \cite{Ma}, \cite {MaFo}, \cite{Sib} for further details). The function $r\mapsto \a^*(r)$ (defined on a neighborhood of 0) can be seen as a {\it frequency map}  in the sense that if $\gamma$ is an invariant circle for $f$ with ``symplectic height'' (area with respect to the origin) $c$   then $\a^*(c)$ is the rotation number of $f$ restricted on $\g$. It has the following properties: the Taylor series of $\a^*$ at 0 coincides with the Birkhoff Normal Form of $f$; morerover,    if $\a^*$ (hence $\a$) is differentiable then $f$ is $C^0$-integrable (see \cite{Sib}). This  $C^0$-integrability often yields rigidity (we refer to \cite{AdSK}, \cite{KaSo}  for an illustration of this fact in the context of billiard maps). The techniques developed in our paper are probably enough to prove that  if the function $\a^*$ is real analytic then $f$ is in fact real analytically integrable. A more delicate  issue is   to establish real analyticity of $\a^*$ by only knowing that its Taylor series at 0 (the BNF) defines a converging series. Note that if $f$ is real analytic  one can construct  dynamically relevant   holomorphic  functions (frequency maps) defined on  complex domains  having   positive Lebesgue measure  intersections (Cantor sets)  with the real axis (see \cite{CMS}, \cite{Pos}) and which coincide on these intersections with $\a^*$.
The restrictions of these holomorphic functions on these Cantor sets  have some quasi-analyticity properties  but it seems that there are not strong enough  to deduce  that $\a^*$ behaves like a real quasi-analytic function (in particular that the convergence of the Taylor series at 0 implies analyticity); we refer to \cite{CMS} for references and for more details.

\medskip
We conclude this subsection by the following question.
\begin{question}\label{QD}Is a given real analytic symplectic diffeomorphism accumulated\footnote{This means that if the given diffeomorphism $f$  has a holomorphic extension to some domain $W$ there exists a slightly smaller  subdomain $\ti W\subset W$ and a sequence of real-symmetric holomorphic  diffeomorphisms $f_{n}$ defined on $\ti W$ such that $\lim_{n\to\infty}\sup_{\ti W} |f-f_{n}|=0$.} by real analytic symplectic diffeomorphisms having convergent BNF's? (We do not ask the radii of convergence of the BNF's to be bounded below).
\end{question}
Positive answers to Questions \ref{QC}, \ref{QD} would imply that any real analytic symplectic diffeomorphism admitting an elliptic equilibrium set is accumulated  in the strong real analytic topology  by diffeomorphisms of the same type that are in addition integrable in a neighborhood of the equilibrium set.

\subsection{Invariant circles}\label{ssec:1.2}

As  suggest (\ref{eq:1.1}), (\ref{eq:1.1.t})  the BNF  (more precisely its  approximate version $B_{N}$) is, as we have already mentioned,  a precious tool 
  to study the problem of  the existence of quasi-periodic motions in the neighborhood of an elliptic equilibrium, an old and  fundamental question  in Celestial and  Hamiltonian Dynamics.  A bright illustration of this fact is certainly the KAM Theorem (\cite{Ko}, \cite{Ar}, \cite{Mo}) that yields,  under   suitable  {\it non-degeneracy} conditions on the BNF (non-planarity),  the existence of many KAM tori \footnote{A KAM torus is an invariant Lagrangian torus on which the dynamics is conjugated to a linear translation  with a Diophantine frequency vector.}accumulating the origin (see \cite{EFK}, \cite{EFKbis}  for results under much weaker non-degeneracy assumptions).  
  
We shall be mainly concerned with the $2$-dimensional case ($d=1$) and we restrict to this case in this subsection. Recall our notation  $\T=\R/2\pi\Z$ for  the 1-dimensional torus.
  An {\it invariant circle} (or also {\it invariant curve})   for a real analytic (or smooth) diffeomorphism $f:(\R\times\R,(0,0))\toitself$ of the form (\ref{newn1.1}) is the image  $\gamma=g(\T)$ of  an injective  $C^1$ map $g:\T\to \R^2\setminus\{0\}$  with  index $\pm1$ at $0$ such that $f(\gamma)=\gamma$. Likewise, in the (AA) case, an  {\it invariant circle} or  {\it invariant curve})   for a real analytic (or smooth) diffeomorphism $f:(\T\times\R,\cT_{0})\toitself$ of the form (\ref{newn1.1bis})  is the image  $\gamma=g(\T)$ of  an injective  $C^1$ map $g:\T\to \T\times\R$ which is homotopic to the circle $\cT_{0}=\T\times \{0\}$ and  such that $f(\gamma)=\gamma$.
  Note that in this latter case, by a theorem of Birkhoff ({\it cf.} \cite{Bi2}, \cite{He}),  invariant circles close enough to $\cT_{0}$ are in fact {\it graphs}  if $f$ satisfies a {\it twist condition}:
\be b_{2}(f)\ne 0\qquad\textrm{if}\quad  (2\pi)^{-1}BNF(f)=\omega r+b_{2}(f)r^2+\cdots.\label{twistcond1.10}\ee
 In both  cases,  we  denote by $\bar \cG_{f}$ the set of $f$-invariant curves and for $t>0$,  by $\bar \cL_{f}(t)$ the set of points in 
$M_{\R}:=\R^2$ or $\T\times \R$
which belong to an  invariant curve $\gamma\in \bar \cG_{f}$ such that $\gamma\subset M_{\R}\cap \{|r|<t\}$ \footnote{ In the (AA) case $M_{\R}\cap \{|r|<t\}=\{(\th,r)\in\T\times \R,\ |r|<t\}$ and in the (CC) case  $M_{\R}\cap \{|r|<t\}=\{(x,y)\in\R\times \R,\ (1/2)(x^2+y^2)<t\}$.}
$$\bar\cL_{f}(t)=\bigcup_{\substack{\g\in\bar \cG_{f}\\ \g\subset M_{\R}\cap\{|r|<t\}}}\gamma.
$$
We then define ($M_{\R}=\R\times \R$ or $\T\times \R$)
$$m_{f}(t)={\Leb}_{M_{\R}}((M_{\R}\cap\{|r|<t\})\setminus \bar\cL_{f}(2t)).$$
The 2-dimensional version of the KAM Theorem is the celebrated Moser's twist Theorem \cite{Mo}:
\begin{TheoOthers}[Moser] If $f$ is a symplectic smooth  diffeomorphism like (\ref{newn1.1}) or (\ref{newn1.1bis}) satisfying the twist condition (\ref{twistcond1.10}) one has for some constant $a>0$
\be m_{f}(t)\lesssim t^{a}.\label{eMoser}\ee
\end{TheoOthers}
In the previous theorem $\omega$ need not be diophantine in the (AA)  or (CC) case (this is due to the twist condition), the twist condition can be considerably weakened (see for example   \cite{EFK}, \cite{EFKbis}) and when $d=1$ symplecticity (area preservation) can be replaced by the weaker {\it intersection property}.

On the other hand, if $f$ is real analytic and $\omega$ (both in the (CC) and (AA) cases)  is Diophantine one can get, by pushing to its limit the ``standard''  KAM method, a better estimate: for any $0<\b\ll 1$ and $t\ll_{\b}1$
\be m_{f}(t)\lesssim \exp(-(1/t)^{\frac{1}{1+\tau(\omega)}-\b})\label{e1.12}
\ee  
where we have defined for any irrational $\omega$
\be \tau(\omega)=\limsup_{k\to\infty}\frac{-\ln\min_{l\in\Z}|k\omega-l|}{\ln k} =\limsup_{n\to\infty} \frac{\ln q_{n+1}}{\ln q_{n}};
\label{eq:1.5}
\ee
in the preceding formula  $(p_{n}/q_{n})_{n\geq 0}$ is  the sequence of convergents\footnote{As usual, if $\omega=1/(a_{1}+1/(a_{2}+1/(\cdots)))$, $a_{i}\in\N^*$,    we define $p_{n}/q_{n}=1/(a_{1}+1/(a_{2}+1/(\cdots+1/a_{n}))$.}
 of $\omega$.
Note that  if $\tau(\omega)<\infty$,  $\omega$ is Diophantine with exponent $\tau'$ for any $\tau'>\tau(\omega)$. If $\tau(\omega)=\infty$ we say that  that $\omega$ is  {\it Liouvillian}.

The main results of our  paper are that: (a)  one can improve  the  exponent in (\ref{e1.12}) if $BNF(f)$ converges; (b) in the ``general case'' the exponent in (\ref{e1.12}) is almost optimal.
More precisely 
\begin{Main}\label{theo:main1} Let $f$ be a real  analytic   symplectic diffeomorphim $f:(\R\times\R,(0,0))\toitself$ like  (\ref{newn1.1}) or $f:(\T\times\R,\cT_{0})\toitself$ like (\ref{newn1.1bis})  satisfying the twist condition (\ref{twistcond1.10}) and assume that in both cases $\omega$ is Diophantine. 
Then,  if $BNF(f)$ defines a converging series one has for any $0<\b\ll 1$ and  $0<t\ll_{\b} 1$ 
\be m_{f}(t)\lesssim \exp\biggl(-\biggl(\frac{1}{t}\biggr)^{(\frac{2}{1+\tau(\omega)})-\b}\biggr).\label{eq:main}
\ee 
\end{Main}
On the other hand general  real analytic twist symplectic  diffeomorphisms like  (\ref{newn1.1}) ,  (\ref{newn1.1bis})   behave quite differently:

\begin{Main} \label{theo:main2}Let  $\omega\in \R$ be Diophantine. There  exist  real analytic twist  symplectic diffeomorphisms $f:(\R\times\R,(0,0))\toitself$ like   (\ref{newn1.1}) or $f:(\T\times \R,\cT_{0})\toitself)$ like  (\ref{newn1.1bis})  satisfying the twist condition (\ref{twistcond1.10}) and a sequence of positive numbers $(t_{k})$ converging to zero such that for any $0<\b \ll 1$, $0<t_{k}\ll_{\b} 1$
\be m_{f}(t_{k})\gtrsim  \exp\biggl(-\biggl(\frac{1}{t_{k}}\biggr)^{(\frac{1}{1+\tau(\omega)})+\b}  \biggr).
\ee  
\end{Main}
In the (CC)-case, one has the following results if  $\omega$ is Liouvillian (note that to define the BNF in the general  (AA) case one needs $\omega$ to be Diophantine).
 Let us define for $\omega\in\R\setminus\Q$
\be t_{n}(\omega)=\frac{5\min(|b_{2}(f)|,|b_{2}(f)|^{-1})}{q_{n+1}q_{n}}.\label{ee1.20}\ee
\begin{Mainprime}\label{theo:mainprime1} Let $\omega$ be Liouvillian and $f:(\R\times\R,(0,0))\toitself$ be a real-analytic symplectic diffeomorphism of the form  (\ref{newn1.1}) satisfying the twist condition (\ref{twistcond1.10}).
 Then, if $BNF(f)$ defines a converging series one has for every $n\in\N$ large enough such that $q_{n+1}\geq q_{n}^{10}$
\be m_{f}(t_{n}(\omega))\lesssim \exp(-q_{n+1}^{1/5}).
\ee

\end{Mainprime}
Note that if $\omega$ is Liouvillian, one has for infinitely many $n$, $q_{n+1}\geq q_{n}^{10}$.

On the other hand:

\begin{Mainprime} \label{theo:main2prime}For any $\omega\in \R\setminus\Q$, there exist  real-analytic  symplectic diffeomorphisms   $f:(\R\times\R,(0,0))\toitself$  of the form   (\ref{newn1.1}) satisfying the twist condition (\ref{twistcond1.10}) such that for every $\e>0$ and   infinitely many  $n\in\N$ 
\be m_{f}(t_n(\omega))\gtrsim  \exp(-q_{n+1}^\e).
\ee  
\end{Mainprime}
Theorems \ref{theo:main1}--\ref{theo:main2},  \ref{theo:mainprime1}--\ref{theo:main2prime} clearly imply  Main Theorem \ref{MainTheorem2} and when $d=1$  Main Theorems \ref{MainTheorem1}-\ref{MainTheorem1AA} in the elliptic fixed point case and the action-angle case.

Theorems \ref{theo:main1}  is a consequence of the following theorem (this is Theorem \ref{prop:12.7} of Section \ref{sec:8curves}):
\begin{Main}\label{MainE} Let $\omega$ be Diophantine and let $f:(\R\times\R,(0,0))\toitself $ of the form (\ref{newn1.1}) or $f:(\T\times\R,\cT_{0})\toitself$ of the form (\ref{newn1.1bis}) be a real analytic symplectic diffeomorphisms satisfying  the twist condition (\ref{twistcond1.10}). Then,  for any $0<\b\ll 1$, $0<t \ll_{\b} 1$, there exists a finite collection $\check \cD_{t}$ of pairwise disjoint disks  of the complex plane $\check D\in \check \cD_{t}$,  centered on the real axis such that: 
\begin{enumerate}
\item $\#\check\cD_{t}\lesssim (1/t)^{1-\b}$
\begin{align}
&\forall \ \check D\in \check\cD_{t}\qquad |\check D\cap \R|\lesssim  \exp(-(1/\rho)^{\frac{1}{1+\tau(\omega)}-\b} )\label{eq:1.18ante}\\
&m_{f}(t)\lesssim \exp(-(1/t)^{\frac{2}{1+\tau(\omega)}-\b} )+\sum_{D\in\check \cD_{t}}|\check D\cap \R|.\label{eq:1.18}
\end{align}
\item If $BNF(f)$ converges, then  for any  $t\ll_{\b}1$ one has for each $\check D\in \check\cD_{t}$
\be |\check D\cap \R|\lesssim \exp(-(1/\rho)^{\frac{2}{1+\tau(\omega)}-\b} ). \label{2ndpart}\ee
\end{enumerate}
\end{Main}
Estimate (\ref{eq:1.18ante}) explains why in general (without the assumption that $BNF(f)$ converges) one only gets the  estimate (\ref{e1.12}).

There is a corresponding theorem in the Liouvillian  (CC) case, see Theorem \ref{prop:12.11}.

We explain in  Subsection \ref{sec:1.3.1} where these disks $\check D$ come from.

\subsection{Prevalence of divergent  BNF's}\label{sec:1.2.2}
 Let us explain more precisely  the aforementioned  dichotomy of R. Pérez-Marco  of Subsection \ref{sec:1.1}. 

Any real analytic symplectic diffeomorphism  $f:(\R^{d}\times\R^d,(0,0))\toitself$  of the form (\ref{newn1.1}) or $f:(\T^d\times\R^d,\cT_{0})\toitself$ of the form (\ref{newn1.1bis}) can be parametrized in the following convenient form:
\be f=\Phi_{ 2\pi\<\omega,r\> }\circ f_{F},\label{ee1.22}\ee
where,  $F:(\R^{d}\times\R^d,(0,0))\to\R$, $F=O^3(x,y)$ or $F:(\T^d\times\R^d,\cT_{0})\to \R$, $F=O^2(r)$ is some real-analytic function and where we denote $f_{F}:(x,y)\mapsto (\ti x,\ti y)$ or $(\th,r)\mapsto (\ti\th,\ti r)$  the {\it exact-symplectic} map (see Subsection \ref{sec:4.5})  defined implicitly  by 
\be \begin{cases}
&\ti x=x+\pa_{\ti y}F(x,\ti y),\quad \ y=\ti y+\pa_{x}F(x,\ti y)\quad (\textrm{CC\  case} )\\
&\textrm{or}\\
&\ti \th=\th+\pa_{\ti r}F(\th,\ti r),\quad r=\ti r+\pa_{\th}F(\th,\ti r)\quad (\textrm{AA \ case}).
\end{cases}\label{ee1.16}
\ee
For $d\geq 1$, $\omega\in\R^d$ non resonant, we define $\cS_{\omega}(\R^{d}\times\R^d)$ (resp. $\cS_{\omega}(\T^d\times \R^{d})$) the set of real analytic symplectic diffeomorphisms $f:(\R^{d}\times\R^d,(0,0))\righttoleftarrow$ (resp. $f:(\T^d\times\R^{d},\cT_{0})\righttoleftarrow$) of the form $f=\Phi_{\Omega_{\omega}}\circ f_{F}$ with  $F:(\R^d\times\R^d,(0,0))\to\R$, $F=O^3(x,y)$ (resp. $F:(\T^d\times\R^d,\cT_{0})\to \R$, $F=O^2(r)$) real-analytic. Let say that a subset $\cA$ of a real vector space $\cE$ is {\it $(PM)$-prevalent}\footnote{See   \cite{HSY} for the concept of prevalence.} if there exists $F_{0}\in\cA$ such that for any $F\in\cE$  the set $\{t\in\R,\ t F_{0}+ (1-t)F\notin\cA\}$ has 0 Lebesgue measure\footnote{We can replace zero Lebesgue measure by  zero (logarithmic) capacity like in Pérez-Marco's paper.}. We then say that  a subset of $S_{\omega}(\R^{d}\times\R^d)$ (resp. $S_{\omega}(\T^d\times \R^{d})$)   is {\it $(PM)$-prevalent} if it is of the form $\{\Phi_{2\pi\<\omega,r\>}\circ f_{F},\ F\in\cA\}$ for some $(PM)$-prevalent subset $\cA$ of  $C^\omega(\R^{d}\times\R^d,\R)\cap O^3(x,y)$ (resp. $C^\omega(\T^d\times \R^{d},\R)\cap O^2(r)$). 

Here is the version of Pérez-Marco's Dichotomy Theorem \cite{PM} for real analytic symplectic diffeomorphisms of the $2d$-disk or the $2d$-cylinder.
\begin{theo}\label{theo:PMdichotomy}Let $d\geq 1$ and  $\omega\in\R^d$ be a non-resonant frequency vector.  Then,  either for any $f\in \cS_{\omega}(\R^{d}\times \R^d)$, the formal series $BNF(f)$ converges ({\it i.e.} the series it defines has a positive radius of convergence), or there exists a $(PM)$-prevalent subset of  $ \cS_{\omega}(\R^{d}\times \R^d)$ such that for any $f$ in this subset $BNF(f)$ diverges.  

The same dichotomy holds in  $\cS_{\omega}(\T^d\times \R^d)$ provided $\omega$ is Diophantine. 
\end{theo}

As we mentioned earlier Pérez-Marco's Dichotomy Theorem was proved in the setting of real analytic Hamiltonians having an elliptic fixed point. The extension to the diffeomorphism setting follows essentially Pérez-Marco's arguments. We refer to Subsection \ref{sec:PMdichotomy} for further details in particular in  the Action-Angle case ({\it cf.} Lemma \ref{BNF:lemma6.3}). 

As a Corollary of Theorems  \ref{theo:main1} and \ref{theo:main2},  Theorems \ref{theo:mainprime1} and \ref{theo:main2prime} we thus have:
\begin{Main}  \label{theo:main3}For any $d\geq 1$ and any non-resonant $\omega\in \R^d$, the set of $f\in \cS_{\omega}(\R^{d}\times\R^d)$ with a divergent BNF is $(PM)$-prevalent. If $\omega$ is Diophantine the same result holds with  $ \cS_{\omega}(\T^{d}\times \R^d)$ in place of $ \cS_{\omega}(\R^{d}\times\R^d)$.
\end{Main}
\begin{proof}We give the proof in the case of  real analytic symplectic diffeomorphisms of the $2d$-disk. 

Let $\omega=(\omega_{1},\ldots,\omega_{d})\in\R^d$ be non resonant.  According to Pérez-Marco's result it is enough to provide one example of a real analytic symplectic diffeomorphism of the $2d$-disk with diverging BNF and frequency vector $\omega$ at the origin to get the conclusion. Since $\omega$ is non resonant, there exists $1\leq j\leq d$ such that $\omega_{j}$ is irrational. According to whether $\omega_{j}$ is Diophantine or Liouvillian we use Theorems \ref{theo:main1} and \ref{theo:main2} or Theorems \ref{theo:mainprime1} and \ref{theo:main2prime} to produce a real analytic symplectic diffeomorphism $f_{j}:(\R^2,0)\righttoleftarrow$ with frequency $\omega_{j}$ at the origin and with a divergent BNF.
We now define $f:(\R^{d}\times\R^d,(0,0))\righttoleftarrow$ by $f(x_{1},\ldots,x_{d},y_{1},\ldots,y_{d})=(\ti x_{1},\ldots \ti x_{d},\ti y_{1},\ldots,\ti y_{d})$,
$$\begin{cases}\textrm{for}\ k\ne j,\ (\ti x_{k}+\sqrt{-1}\ti y_{k})=e^{2\pi \sqrt{-1} \omega_{k}}(x_{k}+\sqrt{-1}  y_{k})\\
(\ti x_{j},\ti y_{j})=f_{j}(x_{j},y_{j}).\end{cases}
$$
This diffeomorphism is real analytic, symplectic and 
$$BNF(f)(r_{1},\ldots,r_{d})=BNF(f_{j})(r_{j})+\sum_{k\in\{1,\ldots,d\}\setminus j}2\pi \omega_{k}r_{k}$$
which is diverging since $BNF(f_{j})$ is.
\end{proof}

When $d=1$ we can be a little bit more precise. Let $\cX$ be the set $([-1,1]^2)^{\N^*}=\{(\zeta_{1,k},\zeta_{2,k})\in [-1,1]^2,\ k\in\N^*\}$ endowed with the product measure $\mu_{\infty}=  ({\rm Leb}_{[-1,1]^2})^{\otimes \N^*}$.  For $\zeta\in\cX$ and $h>0$ we define
 $G_{\zeta}\in C^\omega(\T\times\R)$  ($h>0$ fixed)
$$G_{\zeta}(\th,r)= r^{\ \bar a}\sum_{k\in\N^*}e^{-|k|h}(\zeta_{1,k}\cos(k\th)+\zeta_{2,k}\sin(k\th))$$
where $\bar a$ is some universal integer. 
\begin{Main}\label{theo:MainD}For any  $\omega$  Diophantine and any $f\in \cS_{\omega}(\T\times\R)$  the BNF of $f\circ f_{G_{\zeta}}$ is divergent for $\mu_{\infty}$-almost $\zeta\in\cX$. 
\end{Main}
There is a similar result for the case $f:(\R^{2},0)\righttoleftarrow$, but  with an extra assumption on  the sign of the twist coefficient $b_{2}$ depending on the sequence of convergents of    $\omega$; see Subsection \ref{sec:16.2}.

\subsection{Some words on the proofs}
The starting point of the proofs of Theorems  \ref{theo:main1}, 
 \ref{theo:mainprime1},
  and \ref{MainE} is a KAM scheme that we implement on a holomorphic  extension  of the real analytic diffeomorphism $f$. This allows to work with holomorphic functions defined on complex domains ``with holes'' (i.e. disks).  If these domains are ``nice'' we can use some quantitative form of the analytic continuation principle to propagate informations in the neighborhood of the origin, like the convergence of the BNF,  to the  neighborhoods of each hole.   We illustrate this with the proof of Theorem \ref{MainE}.
  
\subsubsection{Sketch of the proof of Theorem \ref{MainE}}\label{sec:1.3.1} We describe it in the (AA) case. 
Let   $f:(\T\times\R,\cT_{0}) \toitself$, $\cT_{0}=\T\times \{0\}$, be a real analytic symplectic  diffeomorphism of the form (\ref{ee1.22})  with $b_{2}(f)\ne 0$ and $\omega$ Diophantine:
\be f=\Phi_{\Omega}\circ f_{F}, \quad \Omega(r)=2\pi \omega r,\quad F=O(r^2)\label{e1.13}\ee
where $f_{F}$ is the exact symplectic map ({\it cf.} (\ref{ee1.16})) associated to some  real-symmetric\footnote{This  means that it takes real values when $\th$ and $r$ are real} holomorphic function $F:\T_{h}\times \bD(0,\bar \rho)\to\C$ ($h,\bar \rho>0$);  the notations    $\T_{h}$,  $\bD(0,\bar \rho)$ are for $\T_{h}:=((\R+i]-h,h[)/(2\pi\Z))$,  $\bD(0,\bar\rho)=\{r\in\C,\ |r|<\bar\rho\}$.

\paragraph{\it Adapted KAM Normal Form}
 Theorem \ref{MainE} can be seen as an improved version of the classic KAM Theorem on the  positive Lebesgue measure of the set  of points lying on   invariant curves ({\it cf. } Moser's Theorem of Subsection \ref{ssec:1.2}). There are several ways to prove this standard KAM Theorem. A direct approach (which goes back to Kolmogorov) is to find a sequence of (real-symmetric) holomorphic  symplectic diffeomorphisms $g_{i}$ close to the identity, defined on smaller and smaller  complex domains $\T_{h_{i}}\times U_{i}$ ($h_{i-1}\geq h_{i}\geq h/2$, $U_{i}\subset U_{i-1}\subset  \bD(0,\bar \rho)$) and such that   $g_{i}^{-1}\circ f\circ g_{i}$ gets closer and closer to some  integrable\footnote{This means that $\Omega_{i}$ depends only on the $r$ variable.} models   $\Phi_{\Omega_{i}}$:
 \be[\T_{h_{i}}\times U_{i}]\qquad  g_{i}^{-1}\circ f\circ g_{i}=\Phi_{\Omega_{i}}\circ f_{F_{i}},\qquad \|F_{i}\|\ll 1 \label{ee1.19}
 \ee
(in the preceding formula, the set written  on the left is a domain where the conjugation relation holds). 
  One then proves  that $g_{i}$ and $\Omega_{i}$ converge (in some sense) on $\T\times (U_{\infty}\cap\R)$ ($U_{\infty}:=\bigcap_{i}U_{i}$) to some limits $g_{\infty},\  \Omega_{\infty}$ and that $U_{\infty}\cap \R$ (in general a Cantor set)  has positive Lebesgue  measure.  The searched for set of  $f$-invariant curves  is then $\bigcup_{c\in U_{\infty\cap\R}}g_{\infty}(\{r=c\})$ and one has for some constant $a>0$ and any  $\rho<\bar\rho$
 \be m_{f}(\rho)\lesssim \|F\|^{a}.\label{standKAM}\ee
 We refer to Theorem \ref{measureestimates} for more details. 
  The domains $U_{i}$ can be chosen to be {\it holed dmains} {\it i.e.} disks  $\bD(0, \rho_{i})$ $(\rho_{i}\approx \bar \rho$)  from which a finite number of small complex disks centered on the real axis (the ``holes'' of $U_{i}$) have been removed. Removing these small disks is due to the necessity of avoiding {\it resonances} when one inductively construct $g_{i},\Omega_{i},F_{i}$ from $g_{i-1},\Omega_{i-1},F_{i-1}$. More precisely, $U_{i}$ is essentially obtained from $U_{i-1}$ by removing ``resonant disks'' {\it i.e.} disks where the ``frequency map'' $(2\pi)^{-1}\pa\Omega_{i-1}$ is close to a rational of the form $l/k$,  $(l,k)\in\Z\times \N^*$, $\max (|l|,k)\lesssim N_{i-1}$ ($N_{i}$ is  an exponentially  increasing (in $i$) sequence which is defined at the beginning of the inductive procedure). The sizes of the holes  of $U_{i}$ created by removing a finite number of  disks from  $U_{i-1}$ decay very fast with $i$. We shall call a conjugation relation like (\ref{ee1.19}) a(n) (approximate)   {\it KAM Normal Form} for $f$. Its construction is  presented  in Section \ref{sec:5}.
  
A useful observation  ({\it cf.} Section \ref{sec:adaptedkam}) is that, depending on $\rho<\bar\rho$,   one can choose  indices $i_{-}(\rho)<i_{+}(\rho)$ such that all the holes $D$ of the domain $U_{i_{+}(\rho)}$ that  intersect $\bD(0,\rho)$, are disjoint and are created at some step $i-1=i_{D}\in [i_{-}(\rho),i_{+}(\rho)]$ (hence $D\subset U_{i_{D}}$); moreover,   $i_{-}(\rho)$ is large enough to ensure  that the size of $D$ is small.  Writing  (\ref{ee1.19})  with $i=i_{+}(\rho)$ we get  (note the change of notations) 
\be[\T_{h/2}\times U^{KAM}_{i_{+}(\rho)}]\qquad  g_{i_{+}(\rho)}^{-1}\circ f\circ g_{i_{+}(\rho)}=\Phi_{\Omega^{KAM}_{i_{+}(\rho)}}\circ f_{F^{KAM}_{i_{+}(\rho)}},\qquad \|F^{KAM}_{i_{+}(\rho)}\|\ll 1. \label{e1.19}
 \ee
This is what we call our  {\it adapted KAM Normal Form} (adapted to $\bD(0,\rho)$). With the choice we make for $i_{+}(\rho)$ we have
 \be \|F^{KAM}_{i_{+}(\rho)}\|\lesssim \exp(-(1/\rho)^{\frac{2}{1+\tau}-} ).\label{e1.22}
 \ee
\paragraph{\it Hamilton-Jacobi Normal Forms}{\it cf.} Section \ref{sec:HJBNF}.
A hole $D\subset U_{i_{D}}$  of the domain $U_{i_{+}(\rho)}$  that is  created at step $i_{D}$ corresponds as we have mentioned  to a resonance $(2\pi)^{-1}\pa\Omega^{KAM}_{i_{D}}\approx l/k$,   $(l,k)\in\Z\times \N^*$, $\max (|l|,k)\lesssim N_{i_{D}}$  that appears when one constructs  the KAM Normal Form  (\ref{ee1.19}) from step $i_{D}$ to step $i_{D}+1$.
In this resonant  situation we are able to associate to $D$ a  {\it Hamilton-Jacobi  Normal Form}, {\it cf.} Section \ref{sec:HJBNF}, Proposition \ref{proppropHJ}: there exists an  annulus 
 $\hat D\setminus\check D$ ($\hat D, \check D$ are disks),  $\hat D\subset U_{i_{D}}$, $\hat D\supset \check D$,  $\hat D\supset D$ ($\hat D$ is much bigger than $D$) on which  one has 
\begin{align}&[ \T_{h/9}\times \hat D\setminus \check D] \qquad  (g_{D}^{HJ})^{-1}\circ \Phi_{\Omega_{i_{D}}}\circ f_{F_{i_{D}}}\circ (g_{D}^{HJ})=\Phi_{\Omega_{D}^{HJ}}\circ f_{F_{D}^{HJ}},\label{e1.20}\\
& \|F_{D}^{HJ}\|\lesssim \|F^{KAM}_{i_{+}(\rho)}\|. \label{e1.21}
\end{align}
This HJ Normal Form also satisfies the important {\it Extension Property} which in some situation allows to bound above the size of  $\check D$ (note that in general the sizes of $\check D$ and $D$ are comparable).
 
 \medskip 
 \paragraph{\it We can now explain  (\ref{eq:1.18}) of Theorem \ref{MainE}}
Applying the aforementioned standard KAM estimate (\ref{standKAM}) on the holed  domain $U_{i_{+}(\rho)}$ to $\Phi_{\Omega_{i_{+}(\rho)}}\circ f_{F^{KAM}_{i_{+}(\rho)}}$ ({\it cf.} (\ref{e1.19})) and on each annulus $\hat D\setminus\check D$ to $\Phi_{\Omega_{D}^{HJ}}\circ f_{F_{D}^{HJ}}$, ({\it cf. } (\ref{e1.20})) together with  the estimate (\ref{e1.21})  we get that outside a set of measure $\sum_{D\in\cD_{\rho}}|\check D\cap \R|$ the invariant curves of $f$ cover a set  the complement of which in $\bD(0,\rho)$ has a measure $\lesssim\|F^{KAM}_{i_{+}(\rho)}\|^{a}$ for some $a>0$ hence the conclusion by (\ref{e1.22}).

\medskip
\paragraph{\it To prove the second part of Theorem \ref{MainE}, (\ref{2ndpart})} We need to introduce one further approximate Normal Form, namely the approximate Birkhoff Normal Form ({\it cf.} Section  \ref{sec:BNF}) valid on $\T_{h/2}\times \bD(0,\rho^{b_{\tau}})$ ($b_{\tau}=\tau+2$), $\bD(0,\rho^{b_{\tau}})\subset U^{KAM}_{i_{+}(\rho)}$
\begin{align}& [\T_{h/2}\times \bD(0,\rho^{b_{\tau}})]\qquad  (g_{\rho}^{BNF})^{-1}\circ f\circ (g_{\rho}^{BNF})=\Phi_{\Omega_{\rho}^{BNF}}\circ f_{F_{\rho}^{BNF}},\label{e1.20BNF}\\
& \|F_{\rho}^{BNF}\|\lesssim \|F^{KAM}_{i_{+}(\rho)}\|. \label{e1.21BNF}
\end{align}

\medskip
Having the three  Normal Forms (\ref{e1.19}), (\ref{e1.20BNF}), (\ref{e1.20}) in hands (see Figure \ref{Figure2})  the proof of the second part of Theorem \ref{MainE} relies on the following  three principles.

\medskip 
\noindent {\it -- Comparison  Principle} {\it cf.} Section \ref{sec:comparingnf}: since $F_{i_{+}(\rho)}^{KAM}$, $F^{HJ}_{D}$ $F_{\rho}^{BNF}$ are equally very small, all the previous Normal Forms almost coincide on the intersections of their respective domains of defintions (this is done  in Proposition \ref{comparingBNF}), more precisely their frequency maps almost coincide
\be \Omega_{\rho}^{BNF}\mathop{\approx}_{\bD(0,\rho^{b_{\tau}})\cap U_{i_{+}(\rho)}^{KAM}}\Omega_{i_{+}(\rho)}^{KAM}\mathop{\approx}_{U_{i_{+}(\rho)}^{KAM}\cap (\hat D\setminus \check D)} \Omega^{HJ}_{D}\label{plan1}
\ee
where the symbol $a\approx b$ means here  $|a-b|\lesssim  \exp(-(1/\rho)^{\frac{2}{1+\tau}-})$.
Moreover, if the formal BNF converges and equals a holomorphic function $\Xi$ defined on, say, $\bD(0,1)$  one has also ({\it cf.} Corollary \ref{theo:compBNF})
$$\Xi\mathop{\approx}_{\bD(0,1)\cap \bD(0,\rho^{b_{\tau}})}\Omega_{\rho}^{BNF}$$
and in particular from (\ref{plan1}) (we have $\bD(0,\rho^{b_{\tau}})\subset U_{i_{+}(\rho)}^{KAM}$)
\be \Xi\mathop{\approx}_{\bD(0,1)\cap \bD(0,\rho^{b_{\tau}})}\Omega_{i_{+}(\rho)}^{KAM}.\label{e1.15}
\ee

\medskip 
\noindent {\it -- No-Screening Principle, {\it cf.} Section \ref{sec:noscreening}:} Since $\Xi$ and $\Omega_{i_{+}(\rho)}^{KAM}$ are both defined on $U_{i_{+}(\rho)}^{KAM}\supset \bD(0,\rho^{b_{\tau}})$  we use Proposition \ref{lem:fewholes} to extend   the approximate equality (\ref{e1.15}) valid on  $\bD(0,\rho^{b_{\tau}})$ to the bigger domain $\bD(0,1)\cap U_{i_{+}(\rho)}^{KAM}$:
\be  \Xi\mathop{\approx}_{\bD(0,1)\cap U_{i_{+}(\rho)}^{KAM}}\Omega_{i_{+}(\rho)}^{KAM}.\label{e1.16}
\ee
We have to choose here $i_{+}(\rho)$ not too large to avoid a ``screening phenomenon'' that  a too large number of holes could cause. This is studied in Section \ref{sec:adaptedkam}, Proposition \ref{prop:10.4noscreening}.

\medskip 
\noindent{\it -- Residue or Extension Principle {\it cf. Subsection \ref{sec:10.4}}:} From (\ref{plan1}), (\ref{e1.16}) one has
\be\Xi\mathop{\approx}_{U_{i_{+}(\rho)}^{KAM}\cap (\hat D\setminus \check D)} \Omega^{HJ}_{D}\label{e1.17}
\ee
or, in other words, $\Omega_{D}^{HJ}$ which is defined on the  {\it annulus} $\hat D\setminus \check D$ coincides with a very good approximation with a {\it holomorphic function} defined on the {\it whole disk} $\hat D$.
From the Extension Principle of Proposition \ref{proppropHJ} we then get that the radius of  $\check D$ is $\approx 0$ i.e. finally, $|\check D\cap \R|\lesssim \exp(-(1/\rho)^{\frac{2}{1+\tau}-})$. This is (\ref{2ndpart}).

\subsubsection{The Elliptic fixed point case}
The proof in the  non-resonant  elliptic fixed point case, $f:(\R^{2},0)\righttoleftarrow$, follows the same strategy especially if the frequency $\omega$ is Diophantine. A technical point is that to be able to implement  the No-Screening Principle of  Section \ref{sec:noscreening} we need to work with domains $U_{i_{+}(\rho)}^{KAM}\supset U_{\rho}^{BNF}$ where $U_{\rho}^{BNF}$ is a disk around 0 (the estimate on the capacity of this disk is then favorable). This is the reason why we cannot in this situation  use Action-Angle variables since this would force us to work on  angular sector domains and not  disks\footnote{To say it shortly, in Jensen's formula on subharmonic functions,  the ``weight'' of a  small disk $\bD(0,\rho)\subset\bD(0,1)$  is $1/|\ln \rho|$ while the  ``weight'' of $\bD(0,\rho)\cap\D\subset\D$, $\D$ being an angular sector at 0 is only $\rho^{a}$, $a>0$. }. Instead, we define our approximate BNF and KAM Normal Forms directly in Cartesian Coordinates.  The formalism turns out to be the same  as in the Action-Angle case (see Section \ref{sec:conj}), so we  treat these two cases  simultaneously.  The case where $\omega$ is Liouvillian is done in a similar (and even simpler) way. 

\subsubsection{On the proofs of Theorems \ref{theo:main2},  \ref{theo:main2prime} and \ref{theo:MainD} }
The proofs are based on the fact   that, in the general case, resonances are associated to the existence of hyperbolic periodic points in the neighborhood of which no (``horizontal'') invariant circle can exist. We then just estimate the strength of the hyperbolicity and the size of the corresponding local  stable and unstable manifolds (see Section \ref{sec:openingtheeyes}). A key ingredient in this computation are the Resonant Normal Forms; see  the Appendix \ref{appendixE:RNF}.

\subsection{Organization of the paper} 
Section \ref{sec:2} is essentially dedicated to fixing some notations and introducing the notion of domains with holes that plays  a central role in the KAM approach (à la Kolmogorov). We discuss Cauchy's estimates and Whitney's extension Theorem in this framework. 

In Section \ref{sec:noscreening} we give  a No Screening criterium. This is just Jensen's formula on subharmonic functions applied in a domain with not too many holes.

In Section \ref{sec:4} our main purpose is to check that estimates on compositions of generating functions  hold in the case of domains with holes. We treat in a unified way the CC and AA cases. We also discuss invariant curves.

In Section \ref{sec:conj} we study the (co)homological equations and give the proposition on the basic KAM step (Proposition \ref{prop:n5.5}).

Birkhoff Normal Forms (approximate and formal) are presented  in Section \ref{sec:BNF} and  Appendix \ref{formalD}.   We explain in Subsection \ref{sec:PMdichotomy} how Pérez-Marco's dichotomy extends to the diffeomorphism case. 

Section \ref{sec:5} is dedicated to the KAM scheme which is central in our paper; we pay particular attention to the location of the holes of the KAM-domains.

In Section \ref{sec:HJBNF} we present the  Hamilton-Jacobi Normal Form associated to each resonance appearing during the KAM scheme.  Their construction is based on a Resonant Normal Form and an argument of approximation by vector fields the proofs of which are left in the Appendix, Sections \ref{appendixE:RNF} and \ref{sec:appendixG}. The most important property of these Hamilton-Jacobi Normal Forms is the {\it Extension Property} that states that if the corresponding frequency map defined on a annulus  is very close to a holomorphic function defined on a bigger disk, the domain of validity of this Normal Form is essentially this disk.

The Matching or Comparison  Principle is presented in Section \ref{sec:comparingnf}. It quantifies the fact that (exact) symplectic maps have essentially one frequency map.

We construct in Section \ref{sec:adaptedkam} and \ref{sec:10.2} our coexisting adapted KAM, BNF and HJ Normal Forms in the respective cases $\omega$ Diophantine or Liouvillian, the latter being easier to treat.

In Section \ref{sec:8curves} we first state a generalization of the classical KAM estimate on the measure of the set of invariant curves that holds on domains with holes (Theorem \ref{measureestimates}) and we apply it to our adapted KAM and HJ Normal forms to  get measure estimates on the set of invariant curves lying in the union of the domains of definitions of these Normal Forms. These provides the important Theorems   \ref{prop:12.7} and \ref{prop:12.11}. In particular this proves the first part of Theorem \ref{MainE}.

In Section \ref{sec:smallholes13} we use the Extension Principle of Section  \ref{sec:HJBNF}  to prove the second part of Theorem \ref{MainE} and in Section  \ref{sec:conclAA'} we conclude the proofs of Theorems  \ref{theo:main1} and   \ref{theo:mainprime1}. 

The mechanism for the creation  of zones  of the phase space that do not intersect the set of invariant circles (``Hyperbolic eyes'')  is presented in  Section \ref{sec:openingtheeyes}  (Proposition \ref{prop:16.1}). This allows us to construct (prevalent) examples  that satisfy Theorems  \ref{theo:main2},  \ref{theo:main2prime} and \ref{theo:MainD} in Section \ref{sec:16}.

Finally, an Appendix completes the text by giving more details on the proofs of some  statements or by presenting more or less classical methods that had to be adapted to our more specific case.

\subsection*{Acknowledgments} The author wishes to thank     Alain Chenciner and H\char229 kan Eliasson for their  continuous encouragements,  Abed Bounemoura,  Bassam Fayad, Jacques Féjoz, Jean-Pierre Marco,  Stefano Marmi, Laurent Niederman,  Ricardo Pérez-Marco, Laurent Stolovitch for interesting discussions and  all the participants of the Groupe de travail de Jussieu for their patient  listening of this work and their constructive comments.

\section{Notations, preliminaries}\label{sec:2}
Let  $\bT$ be  the 1-dimensional torus  $\T:=\R/(2\pi\Z)=\{x+2\pi \Z,x\in\R\}$ and for $0<h\leq\infty$
$$\T_{h}=\{x+iy+(2\pi\Z),\ x,y\in\R,\ |y|<h\}\qquad (i^2=-1)$$
the complex cylinder of width $2h$. If $\th_{1}=(x_{1}+iy_{1})+(2\pi\Z),\th_{2}=x_{2}+iy_{2}+(2\pi\Z)\in\T_{\infty}$ we set $|\th_{1}-\th_{2}|_{\T_{\infty}}:=\min_{l\in\Z}|(x_{1}-x_{2}-2\pi l)+i(y_{1}-y_{2})|$.

If  $\rho>0$ we denote by $\bD(z,\rho)\subset \C$ the open  disk of center $z$ and radius $\rho$ and by $\bar \bD(z,\rho)$ its closure \footnote{With this notation $\bar \bD(z,0)=\emptyset$.}; sometimes for short we shall write $\bD_{\rho}$ for $\bD(0,\rho)$ (and by $\bar\bD_{\rho}$ its closure).

If $z=x+i y\in\C$, ($i=\sqrt{-1}$) $x,y\in\R$,  (resp. $\th =x+i y+(2\pi\Z)\in\T_{\infty}$),   we denote by $\s_{0}(z)=\bar z=x-iy $  (resp. $\s_{0}(\th)=\bar \th=x-i y+(2\pi\Z)$) its complex conjugate. 

We define the {\it involutions} $\sigma_{1},\sigma_{2}:\C^2\to\C^2$ and $\sigma_{3}:\T_{\infty}\times \C\to \T_{\infty}\times \C$ by
\be \sigma_{1}(x,y)=(\bar x,\bar y),\qquad \sigma_{2}(z,w)=(i\bar w,i\bar z),\qquad \sigma_{3}(\th,r)=(\bar\th,\bar r).\label{invol}\ee

For $w=(w_{1},w_{2}),w'=(w_{1}',w_{2}')\in \C\times\C$ (resp. $\in\T_{\infty}\times\C$) we define   the distance $d(w,w')=\max (|w_{1}-w'_{1}|,|w_{2}-w_{2}'|)$ (resp. $d(w,w')=\max (|w_{1}-w'_{1}|_{\T_{\infty}},|w_{2}-w_{2}'|)$).

If  $W$  is an open subset of $\C\times \C$    or of  $\T_{\infty}\times \C$ and if  $F:W\to \C$ we set
$$\|F\|_{W}=\sup_{W}|F|.
$$
If a  function $W\ni (w_{1},w_{2})\mapsto F(w_{1},w_{2})$  is differentiable enough  (for the standard real differentiable structure on $M$) we can as usual define  its partial derivatives\footnote{Here we use the standard notation: if $w=t+is$, $(t,s)\in\R^2$,  $\pa_{w}=(1/2)(\pa_{t}-i\pa_{s})$ and  $\pa_{\bar w}=\bar \pa_{w}=(1/2)(\pa_{t}+i\pa_{s})$.} $\pa^{k_{1}}_{w_{1}}\pa_{\bar w_{1}}^{l_{1}}\pa^{k_{2}}_{w_{2}}\pa_{\bar w_{2}}^{l_{2}}F$ ($k_{1},k_{2},l_{1},l_{2}\in\N$) and its (total) $j$-th derivative  $D^jF=(\pa^{k_{1}}_{w_{1}}\pa_{\bar w_{1}}^{l_{1}}\pa^{k_{2}}_{w_{2}}\pa_{\bar w_{2}}^{l_{2}}F)_{k_{1}+k_{2}+l_{1}+l_{2}=j}$    ($j\in\N$). We then define
$$\|D^jF\|_{W}=\mathop{\max}_{\substack{(k_{1},l_{1},k_{2},l_{2})\in\N^4\\ k_{1}+l_{1}+k_{2}+l_{2}=j}}\|\pa^{k_{1}}_{w_{1}}\pa_{\bar w_{1}}^{l_{1}}\pa^{k_{2}}_{w_{2}}\pa_{\bar w_{2}}^{l_{2}}F\|_{W},\qquad \|F\|_{C^n(W)}=\max_{0\leq j\leq n}\|D^j F\|_{W}.
$$
We denote by $C^k(W)$ the set of functions $F:W\to\C$ such that $\|F\|_{C^n(W)}<\infty$ and by $\cO(W)$ the set of holomorphic functions $F:W\to\C$ (all the preceding partial derivatives of the form $\bar \pa_{w}=\pa_{\bar w}$ then  vanish).

 We say that an  open set  $W$ of $M:=\C^2$    or of  $M:=\T_{\infty}\times \C$ is  {\it $\s_{i}$-symmetric} ($i=1,2,3$)  if it is invariant by $\sigma_{i}$ ($\s_{i}(W)=W$);    if $W$ is $\sigma_{i}$-symmetric we  say that a function $F:W\to \C$ is $\s_{i}$-symmetric  if $F\circ \sigma_{i}=\s_{0}\circ F=\bar F$ (here $\s_{0}$ denotes the complex conjugate) and we denote by $C^n_{\s_{i}}(W)$, resp. $\cO_{\sigma_{i}}(W)$, the set of $C^n$ resp.  holomorphic functions $F:W\to\C$ that are $\sigma_{i}$-symmetric.   When no confusion is possible on the nature of the relevant $\sigma_{i}$ involved,  we shall often say $\sigma$-symmetric  or even real-symmetric instead of $\s_{i}$-symmetric. If $W$ is $\s$-symmetric we use the notation 
 $W_{\R}=\{w\in W,\ \s(w)=w\}$; if   $W_{\R}\ne\emptyset$ then $F\in\cO_{\s}(W)$ defines by restriction a map (still denoted by $F$)  $F:W_{\R}\to \R$.    Note that a  function $F:(\R^2,0)\to\R$  which is real analytic is in $\cO_{\sigma_{1}}(\bD(0,\rho)\times\bD(0,\rho)) $ for some $\rho>0$.

Let $W$ be a open set of $M:=\C^2$ or $\T_{\infty}\times \C$. We denote by ${\rm Diff}^n(W)$, resp.  ${\rm Diff}^{\cO}(W)$, the set of  $C^n$, resp. holomorphic, diffeomorphism $f:\ti W\to f(\ti W)\subset M$ defined on an open  neighborhood $\ti W$ of $W$ containing the closure of $W$.

Note that there exists a constant $C$ depending only on $M$ such that for any  $C^1$-diffeomorphisms $f_{1},f_{2}:M\to M$
satisfying  $\|f_{1}-id\|_{C^1}+\|f_{2}-id\|\leq 1$ then 
\be \|(f_{2}\circ f_{1})-id\|_{C^1}\leq C(\|f_{1}-id\|_{C^1}+\|f_{2}-id\|_{C^1}).\label{2.40} \ee

If now $W$ is a $\s$-symmetric open set of $(M,\s)$ we  denote by ${\rm Diff}^{n}_{\s}(W)$ resp. ${\rm Diff}^{\cO}_{\s}(W)$ the set of $f\in{\rm Diff}^{n}_{}(W)$,  resp. $f\in{\rm Diff}^{\cO}_{}(W)$, such that $f\circ \s=\s\circ f$. It then defines by restriction a $C^n$, resp. real analytic,  diffeomorphism (that we still denote $f$)  $f:W_{\R}\to f(W_{\R})\subset M_{\R}$. 

When $f,g$ are two $\s$ symmetric holomorphic diffeomorphisms we  write
\be [W]\qquad  f=g\label{ee2.40}\ee
to say that  $f,g\in{\rm Diff}_{\s}^\cO(W)$  coincide on an open neighborhood of $W$ containing the closure of $W$.

\subsection{Domains $W_{h,U}$}\label{sec:new2.3}
Let $h>0$ and 
$U$ an  open  connected set of $\C$;  we shall  define domains $W^{AA}_{h,U}$ of $M=M^{AA}=\T_{\infty}\times \C$  (AA stands for ``Action-Angle'') and $W^{CC}_{h,U}$, $W^{CC*}_{h}$ of  $M=M^{CC}=M^{CC*}=\C^2$ (CC for ``Cartesian Coordinates'') the following way:
\begin{itemize}
\item {\it Cartesian Coordinates ($CC*$)}: if $\rho_{U}:=\sup\{|r|,\ r\in U\}$, the set  $W^{CC*}_{h,U}\subset \C\times\C$ is
\be W^{CC*}_{h,U}=\{(x,y)\in \C^2,\ |x\pm iy|\leq \sqrt{2}e^h\rho_{U}^{1/2},\ \frac{x^2+y^2}{2}\in U\}; \label{holedomaincc*}\ee 
\item {\it Cartesian Coordinates (CC)}: if $\rho_{U}:=\sup\{|r|,\ r\in U\}$, the set  $W^{CC}_{h,U}\subset \C\times\C$ is
\be W^{CC}_{h,U}=\{(z,w)\in \bD(0,e^h\rho_{U}^{1/2})\times \bD(0,e^h\rho_{U}^{1/2}),\ -izw\in U\}; \label{holedomaincc}\ee 
\item {\it Action Angle coordinates (AA)}: the set  $W^{AA}_{h,U}$  of $\bT_{\infty}\times\C$ is
\be W^{AA}_{h,U}=\T_{h}\times U\label{holedomaintc}.\ee
\end{itemize}
In all these three cases we denote by $r$ the observable $(x,y)\mapsto (1/2)(x^2+y^2)$, $(z,w)\mapsto -izw$, $(\th,r)\mapsto r$.

\subsection{Cauchy estimates}

If $\d>0$ we denote by $\cU_{\d}(W)=\{w\in W,\ \bB(w,\d)\subset W\}$ (here $\bB(w,\d)$ is  the ball $\{z\in M,\ d(z,w)<\delta\}$). Assume that $F\in\cO(W)$.
By differentiating  $(k_{1}+k_{2})$- times  Cauchy complex integration formula 
$$ F(w_{1},w_{2})=\frac{1}{(2\pi i)^2}\int_{|w_{1}-\zeta_{1}|=\d}\int_{|w_{2}-\zeta_{2}|=\d}\frac{F(\zeta_{1},\zeta_{2})}{(w_{1}-\zeta_{1})(w_{2}-\zeta_{2})}d\zeta_{1}d\zeta_{2}
$$
one sees that if $\cU_{\d}(W)$ is  not empty
\be\|\pa_{w_{1}}^{k_{1}}\pa_{w_{2}}^{k_{2}} F\|_{\cU_{\d}(W)}\leq C_{k_{1},k_{2}}\delta^{-(k_{1}+k_{2})}\|F\|_{W}.
\label{eq:cauchy:derivatives}
\ee

\subsection{Holed domains}
\subsubsection{ Holed domain of $\C$} \label{se:2.5.1}
A holed domain of $\C$ is an open set of $\C$ of the form 
\be U=\bD(c,\rho )\setminus \bigcup_{i\in I} \bar \bD(c_{i},\rho_{i}),\label{reprholeddomain}
\ee
 for some $c\in\C$, 
 $\rho>0$ , $c_{i}\in\C$, $\rho_{i}>0$
and where $I$ is a finite set which is either empty or such that for any $i\in I$, 
$\bD(c_{i},\rho_{i})\cap \bD(c,\rho)\ne \emptyset$. 
 It is not difficult to see that there exists a unique minimal $J_{U}\subset I$ (for the inclusion) such that $\bigcup_{i\in J_{U}} \bar \bD(c_{i},\rho_{i})=\bigcup_{i\in I} \bar \bD(c_{i},\rho_{i})$ and that the representation (\ref{reprholeddomain}) with $I$ replaced by $J_{U}$ is then unique:
\be U=\bD(c,\rho)\setminus \bigcup_{i\in J_{U}} \bar \bD(c_{i},\rho_{i}).\label{reprholeddomain-bis}
\ee
We then denote by
\be \cD(U)=\{\bD(c_{i},\rho_{i}),\ i\in J_{U}\}.\label{reprholeddomain-ter}
\ee
  We shall call $\bD(c,\rho)$ 
  the {\it external disk} of $U$. We then set 
\be\begin{cases}&\rho_{U}:=\ovrad{U}:=\rho,\quad\urad(U)=\min_{i\in J_{U}}\rho_{i}\\
&\ua(U)=(\sum_{i\in J_{U}}\rho_{i}^2)^{1/2} \\
&\uc(U)=\# J_{U}\\
&\ud(U)=\ovrad(U)\  \textrm{if}\ J_{U}=\emptyset,\quad \ud(U)=\min(\ovrad(U),\urad({U})) \ \textrm{if}\ J_{U}\ne \emptyset.
\end{cases}\label{uaud}\ee
If $J_{U}$ is empty or if all the disks $\bD(c_{i},\rho_{i})$, $i\in J_{U}$,  are pairwise disjoint and included in $\bD(c,\rho)$ we say that the holed domain $U$ has disjoint holes and we call $\bD(c_{i},\rho_{i})$ the holes of $U$ (the bounded connected components of $\C\setminus U$). We denote by $\cD(U)$ the set of all these disks.

\bigskip
\noindent
{\bf Note:} We shall only consider in this paper holed domains (\ref{reprholeddomain}) where the $c_{i}$ are on the real axis.

\subsubsection{Holed domains of $\C\times\C$ or $\T_{\infty}\times \C$} These are by definition sets of the form $W_{h,U}$ where $h>0$ and  $U$ is a holed domain; see (\ref{holedomaincc}) or (\ref{holedomaintc}). We then define 
$$\ud(W_{h,U})=\min(h,\ud(U)).$$

\subsubsection{Deflation of a holed domain}
If $\d\in\R$ we use the notation $e^{-\d}\bD(c,\rho)$ for 
$$e^{-\d}\bD(c,\rho)=\bD(c,e^{-\d}\rho).
$$
If $U\subset \C$ is a holed domain of the form (\ref{reprholeddomain-bis}) and if $\delta>0$
 we denote by $e^{-\delta}U\subset U$ the (possibly empty) open set
$$e^{-\d}U=\bD(c,e^{-\delta}\rho)\setminus\bigcup_{i\in J_{U}} \bar \bD(c_{i},e^{\delta}\rho_{i}).$$
Similarly if $0<\delta<h$
$$e^{-\d}W_{h,U}=W_{h-\delta/2, e^{-\d}U}.$$

We make the following simple observations (the first two items are proved by area considerations):
\begin{lemma}\label{lemma2.1}For $1>\d>0$ one has:
\begin{enumerate}
\item\label{i1.L.2.1} For any $z\in \bD(c,\rho)$, $\dist(z,U)\leq 2\ua(U)$.
\item\label{i2.L.2.1} If $\rho^2>2e^{4\d}\sum_{i\in J_{U}}\rho_{i}^2$ then $e^{-\d}U$ is not empty.
\item\label{i3.L.2.1} If $e^{-\d}U$ is not empty, for any  $z\in e^{-\d}U$ one has 
$$\bD(z,(1/2)\d\ud(U))\subset U.
$$
\end{enumerate}
\end{lemma}
\subsubsection{Reformulation of Cauchy's Inequalities}\label{subsec:2.6.4}
Using item \ref{i3.L.2.1} of Lemma \ref{lemma2.1} we can in particular reformulate inequalities (\ref{eq:cauchy:derivatives}) when $W$ is of the form $W_{h,U}$ and $F\in\cO(W_{h,U})$:
\be \|D^mF\|_{e^{-\d}W_{h,U}}\leq C_{m} \d^{-m}\ud(W_{h,U})^{-m}\|F\|_{W_{h,U}}.
 \label{eq:cauchy:derivativesbis}
\ee
One can sometimes obtain better estimates.
\begin{itemize} 
\item In the (AA)-case, if $0<\d<h$, one has 
\be \|\pa^k_{\th}F\|_{e^{-\d}W_{h,U}}\lesssim \d^{-k}\|F\|_{W_{h,U}} \label{eq:cauchy:derivativester}
\ee
\item In the (CC)-case, if $U=\bD(0,\rho)$ and  $\d<1/2$ one has  $e^{-\d}W_{h,\bD(0,\rho)}\subset \cU_{\ti\d}(W_{h,\bD(0,\rho)})$ with $\ti \d=\rho^{1/2}e^{-h}\d/4$ and thus
\be \|\nabla F\|_{e^{-\d}W_{h,U}}\lesssim e^{h}\d^{-1}\rho^{-1/2}\|F\|_{W_{h,U}}. \label{eq:cauchy:derivativesquarto}
\ee
\end{itemize}

\subsection{Whitney type extensions on domains with holes}\label{sec:2.4}The discussion that follows  will be useful in the construction of the KAM Normal Form of Section \ref{sec:5}.

Let $U$ be a real symmetric  holed domain
\be U=\bD(0,\rho)\setminus \bigcup_{i\in J_{U}} \bar \bD(c_{i},\rho_{i}),\qquad c_{i}\in\R,\label{reprholeddomain-bisbis}
\ee
$h>0$, $W_{h,U}$ one of the domains defined in Subsection \ref{sec:new2.3}
and $F:W_{h,U}\to\C$ be a $C^k$   \footnote{Differentiability here is related to the real differentiable structure of $W_{h,\C}$.}$\s$-symmetric function i.e. $F\circ \s=\s_{0}\circ F$ ($\s_{0}$ is the complex conjugation). We say that a $C^k$, $\s$-symmetric  function \footnote{The exponent $Wh$ stands for ``Whitney''.}  $F^{Wh}:W_{h,\C}\to\C$ is a {\it Whitney extension}\footnote{See \cite{Whi}, \cite{Stein} }  for $(F,W_{h,U})$  if   
$$\forall \ m\in W_{h,U},\quad F^{Wh}(m)=F(m).$$
Note that since $U$ is open this implies that  for all $0\leq j\leq k$, $D^jF$ and $D^{j}F^{Wh}$ coincide on $W_{h,U}$.

We shall construct such Whitney's extensions in two situations. 

\begin{lemma}\label{lemma:2.3ee}Let  $F\in\cO_{\s}(W_{h,U})$. For any $\d\in ]0,1[$, there exists a $C^k$, $\s$-symmetric function $F^{Wh}:W_{h,\C}\to \C$ such that 
\begin{align}&\forall \ m\in e^{-\d}W_{h,U},\quad F^{Wh}(m)=F(m)\\
&\sup_{0\leq j\leq k}\|D^j F^{Wh} \|_{W_{h,\C}}\leq  C(1+\#J_{U})^k(\d\ud(U))^{-2k}\max_{0\leq j\leq k}\|D^jF\|_{W_{h,e^{-\d/10}U}}.
\end{align}
\end{lemma}
\begin{proof} See Subsection \ref{sec:appendixBbis} of the Appendix.
 \end{proof}

\begin{notation}\label{notation:2.1}We denote by $\ti\cO_{\s}(W_{h,U})$ the set of $C^3$,  $\s$-symmetric maps $F:W_{h,\C}\to\C$ such that the restriction of $F$ on $W_{h,U}$ is holomorphic.
\end{notation}

\medskip
Let $A\geq 1$, $B\geq 0$, $U\subset \C$ a symmetric  holed domain. We say that a $\s$-symmetric $C^3$ function $\Omega:U\to\C$ satisfies an {\it $(A,B)$-twist condition} on $U$  if 
\be  \forall \ r\in U\cap \R,\  A^{-1}\leq \frac{1}{2\pi}\pa^2\Omega(r)\leq A,\quad \textrm{and}\ \|\frac{1}{2\pi}D^3\Omega\|_{U}\leq B. \label{twist2.57} \ee
 
If $U$ is a disk $\bD(0, \rho_{0})$ one can construct  for some  $0<\bar \rho< \rho_{0}$ a $C^3$, $\s$-symmetric Whitney extension for $\Omega$ on $\bD(0,\bar \rho)$ that satisfies an $(A,B)$-twist condition on $\bD(0,\bar \rho)$.
 \begin{lemma}\label{Whextension} Let $\Omega\in\cO_{\s}(\bD(0,\rho_{0}))$ 
 $$(2\pi)^{-1}\Omega(z)=\omega_{0}z+b_{2}z^2+O(z^3),\qquad \|\Omega\|_{\bD(0, \rho_{0})}\leq 1.$$
 There exists $0<\bar \rho<\rho_{0}$, $B\geq 0$ and a  $C^3$, real symmetric extension  $\Omega^{Wh}\in\ti\cO_{\s}(\C)$ of $(\Omega,\bD(0,\bar \rho))$ that satisfies an $(A,B)$-twist condition on $\C$ with $A=3\max(b_{2}, b_{2}^{-1})$.
\end{lemma}
\begin{proof}
See the Appendix \ref{sec:appendixBbisbis}.
\end{proof}
\begin{notation} \label{notation:2.2} We denote by $\cT\cC(A,B)$ the set of $C^3$, real symmetric maps $\Omega:\C\to\C$ satisfying an $(A,B)$-twist  condition (\ref{twist2.57}) with $U=\C$.
\end{notation}

\medskip
Let $U\subset \C$ be a symmetric  connected holed domain.
 \begin{prop}\label{proppramexclusion}
 There exists $\bar C_{0}$ such that if  $\Omega\in\ti\cO_{\s}(U)\cap \cT\cC(A,B)$ with 
  \be \bar C_{0}\times \max(\rho,\ua(U))\times A\times B<1\label{condUnew}\ee
then  the following holds. For  any $\nu\ll_{A,B} 1$ and any $\b\in\R$, either for any $z\in U$ 
 $$|\omega(z)-\beta|\geq \nu\qquad (\omega=(2\pi)^{-1}\pa\Omega)
 $$
 or there exists a unique $c_{\b}\in ]-\rho-10A^2\nu,\rho+10A^2\nu[$ such that $\omega^{}(c_{\b})=\b$ and  for any $z\in U\setminus \bar\bD(c_{\b},10A^2\nu)$ one has
  $$|\omega(z)-\beta|\geq  5A\nu.
 $$
 \end{prop}
 \begin{proof}See Appendix \ref{sec:appendixB}.
 \end{proof}

\subsection{Notation $\fO_{p}$}
Let $h>0$, $U$ be a holed domain, functions $F_{1},\ldots,F_{n}\in\cO(W_{h,U})$ and $l\in\N^*$. We define the  relation 
$$G=\fO_{l}(F_{1},\ldots,F_{n})$$
as the follows:
there exist $a\in\N^*$, $C>0$ and   $Q(X_{1},\ldots,X_{n})$ a homogeneous polynomial  (independent of $U$) of degree $l$  in the variables $(X_{1},\ldots,X_{n})$ such that for any  $0<\d<h/2$ satisfying 
\be C \ud({W_{h,U}})^{-a}\d^{-a}\max_{1\leq i\leq n}\|F_{i}\|_{W_{h,U}}\leq 1\label{2.35}
\ee
one has $G\in\cO(e^{-\d}W_{h,U})$ and  
\be 
\|G\|_{e^{-\d}W_{h,U}}\leq \ud({W_{h,U}})^{-a}\delta^{-a} Q(\|F_{1}\|_{W_{h,U}},\ldots,\|F_{n}\|_{W_{h,U}}).\label{2.36}
\ee
When $\d$ satisfies (\ref{2.35}) we write  $$\d=\fd(F_{1},\ldots,F_{n};W_{h,U}).$$
\begin{rem}\label{rem:2.3}
Note that if $U=\bD(0,\rho_{0})$ is a disk containing 0 and $F\in \cO(W_{h,U}^{*})$, $*=CC,AA$,  one has 
\begin{align*}& F(z,w)=O^p(z,w)\iff \forall\ 0\leq\rho\leq\rho_{0},\ \|F\|_{W_{h,\bD(0,\rho)}}\lesssim \rho^{p/2}\\
& F(\th,r)=O^p(r)\iff \forall\ 0\leq \rho\leq\rho_{0},\ \|F\|_{W_{h,\bD(0,\rho)}}\lesssim \rho^{p}
\end{align*}
hence if $F_{1},\ldots, F_{n}\in \cO(W_{h,U}^{CC})$ (resp. $\in \cO(W_{h,U}^{AA})$) satisfy $F_{i}(z,w)=O^{p}(z,w)$ (resp. $F_{i}=O^{p}(r)$), $1\leq i\leq n$,  one has 
$$\fO_{m}(F_{1},\ldots,F_{n})=O^{lp-2a}(z,w)\qquad (\textrm{resp.}\  O^{mp-a}(r)).$$
\end{rem}

 We shall use the notation $\dot \fO_{p}(F_{1},\ldots,F_{n})$ if  the polynomial $Q$ is null when $X_{1}=0$:   $Q(0,X_{2},\ldots,X_{n})=0$; for example if $l=n=2$, $Q(X_{1},X_{2})=X_{1}X_{2}+X_{1}^2$.

When we want to keep track of the exponent $a$ appearing in (\ref{2.35}), (\ref{2.36}) we shall use the symbol $\fO_{p}^{(a)}$.
\section[Poisson-Jensen formula]{A no-screening criterium on domains with holes}\label{sec:noscreening}

Let $D$ be a disk  and  $U$ be an open subset of $D$ of the form $U=D\setminus(\bigcup_{1\leq j\leq N}\overline{D}_{j})$ where $(D_{j})_{j1\leq j\leq N}$ is a collection of open sub-disks such that ${\overline D}_{j}\subset D$. We can define the {\it Green function} of $U$, $g_{U}:U\times U\to \bR$ as follows: for any $z\in U$, $-g(z,\cdot)$ is the  function equal to 0 on the boundary $\pa U$ of $U$, which is subharmonic on $U$,  harmonic on $U\setminus\{w\}$ and which behaves like $\log|z-w|$ when $z\to w$. The Green function $g_{U}$ is thus nonnegative. We denote by $\omega_{U}:U\times Bor(\pa U)\to [0,1]$ the {\it harmonic measure} of $U$ ($Bor(\pa U)$ is the set of borelian subsets of $\pa U$) defined as follows:     if $z\in U$ and $I\in Bor(\pa U)$ (one can assume $I$ is an arc for example) then the function $\omega_{U}(\cdot, I)$ is  the unique harmonic function defined on $U$, having a continuous extension to $\bar U$ and such that $\omega_{U}(z,I)=1$ if $z\in I$ and 0 if $z\in \pa U\setminus I$
\footnote{It can also be defined by using Brownian motions: $\omega_{U}(z,I)=\mathbb{E}({\bf 1}_{I}(W_{z}(T_{z,I})))$
where $W_{z}(t)$ is the value at time $t$ of a Brownian motion issued from the point $z$ (at time 0) and $T_{z,I}$ is the stopping time adapted to the filtration $\cF_{z}$ of hitting $I$ before $\pa U\setminus I$.}.
{\it Poisson-Jensen formula} ({\it cf.} \cite{Ra}) asserts that for any subharmonic function $u:U\to \C$
$$ u(z)=\int_{\pa U}u(w)d\omega_{U}(z,w)-\int_{U}g_{U}(z,w)\Delta u(w)
$$
where $\Delta u$ is the usual laplacian of $u$.
In particular, if $f$ is a holomorphic function on $U$, the application of this formula to $u(z)=\ln |f(z)|$ gives
$$ \ln|f(z)|=\int_{\pa U}\ln|f(w)|d\omega_{U}(z,w)-\sum_{w:f(w)=0}g_{U}(z,w)
$$
and thus since $g_{U}$ is nonnegative
$$ \ln|f(z)|\leq \int_{\pa U}\ln|f(w)|d\omega_{U}(z,w).
$$

\begin{prop}\label{lem:fewholes} Let $U$ be a domain $U=\bD(0,\rho)\setminus(\bigcup_{1\leq j\leq N}\overline{\bD(z_{j},\e_{j})})$ ,   $\overline{\bD(z_{j},\e_{j})}\subset \bD(0,\rho)$ ($\rho\in ]0,1[$) and let $B\subset U$, $B=\bD(0,\sigma)$. Assume that   $f\in \cO(U)$ satisfies
$$  \|f\|_{U}\leq 1$$
and 
$$ \|f\|_{\pa B}\leq m.
$$
Then
 for any point $z\in \hat U:=\bD(0,\rho)\setminus(\bigcup_{1\leq j\leq N}\bD(z_{j},d_{j}))$, $2\e_{j}<d_{j}<1$
\be \ln |f(z)|\leq \biggl(\frac{\ln(|z|/\rho)}{\ln (\sigma/\rho)} -\sum_{j=1}^N \frac{\ln (d_{j}/2\rho)}{\ln (\e_{j}/\rho)}\biggr)\ln m. \label{eq:0center}
\ee
\end{prop}
\begin{proof} 
Replacing $z/\rho$ by $z$, $z_{j}/\rho$ by $z_{j}$, $\sigma/\rho$ by $\sigma$, $\e_{j}/\rho$ by $\e_{j}$ and $d_{j}/\rho$ by $d_{j}$,  we can reduce to the case $\rho=1$. We then denote $D=\bD(0,1)$, $D_{j}=\bD(z_{j},\e_{j})$, $B=\bD(0,\s)$.

By Poisson-Jensen formula
\begin{align} \ln|f(z)|&\leq \int_{\pa (U\setminus B)}\ln |f(w)|d\omega_{U\setminus{\overline{B}}}(z,w)\notag\\
&\leq\omega_{U\setminus \overline{B}}(z,\pa B) \ln m. \label{eq:8.33'}
\end{align}
We now compare $\omega_{U\setminus \bar B}(z,\pa B)$ with $\omega_{D\setminus \bar B}(z,\pa B)$. We observe that the  function $z\mapsto \omega_{U\setminus \bar B}(z,\pa B)$ is the unique harmonic function defined on $U\setminus\bar B$ which is 1 on $\pa B$ and 0 on  $\pa D\cup \pa(D\setminus U)$; since   $\pa(U\setminus \bar B)=\pa B\cup \pa D\cup \pa(D\setminus U)$ we deduce  by the Maximum Principle that  it takes its values in $[0,1]$. Similarly, the function $z\mapsto \omega_{D\setminus \bar B}(z,\pa B)$ is the unique harmonic function  defined on $D\setminus \bar B$ which is 1 on $\pa B$ and 0 on $\pa D $, hence it takes also its values in $[0,1]$. So 
\be v(\cdot):=\omega_{U\setminus \bar B}(\cdot,\pa B)-\omega_{D\setminus \bar B}(\cdot,\pa B)\label{3.46}\ee is a harmonic function defined on $U\setminus \bar B$, $-1\leq v\leq 1$,  which is 0 on $\pa B\cup\pa D$. For $1\leq j\leq N$, let $v_{j}$ be the harmonic function defined on  $D\setminus (\bar B\cup \bar  D_{j})$ which is 0 on $\pa(D\setminus \bar B)=\pa D\cup \pa B$ and $-1$ on $\pa D_{j}$; by the Maximum Principle  $-1\leq v_{j}\leq 0$.
\begin{lemma} The function $\sum_{j=1}^Nv_{j}$ is harmonic on $U\setminus \bar B$ and on this set
$$\sum_{j=1}^Nv_{j}\leq v.
$$
\end{lemma}
\begin{proof}We notice that the function $\sum_{j=1}^Nv_{j}$ is defined and harmonic on  $D\setminus (\bar B\cup\bigcup_{j=1}^N \bar D_{j})=U\setminus \bar B$. We want to compare $v$ and $\sum_{j=1}^Nv_{j}$ on the boundary  $\pa(U\setminus \bar B)=\pa D\cup \pa B\cup \pa(D\setminus U)$. On $\pa D\cup\pa B$ the two functions $v$ and $\sum_{j=1}^Nv_{j}$ are equal (they are both equal to 0). To compare them on $\pa(D\setminus U)$ we notice that 
$\pa(D\setminus U)\subset \bigcup_{j=1}^N\pa D_{j}$ and since $v_{j}\ _{| \ \pa D_{j}}=-1$ and for $i\ne j$, $v_{i}\leq 0$ we have 
at each point $z\in \pa(D\setminus U)$ which is in $\pa D_{j}$, $\sum_{i=1}^Nv_{i} (z )\leq -1$ hence $\sum_{i=1}^Nv_{i}\ _{|\ \pa(D\setminus U)}\leq -1$. But we have seen that $-1\leq v\leq 1$ on $U\setminus B$. We have thus proven that on $\pa(U\setminus B)$ one has $\sum_{j=1}^Nv_{j}\leq v.
$ and we conclude the proof by the Maximum Principle.
\end{proof}
Since by the Maximum Principle  on $D\setminus (\bar B\cup \bar D_{j})$
$$  -\frac{\ln|z-z_{j}|-\ln 2}{\ln \e_{j}}\leq v_{j}(z)
$$
one has for $z\in \hat U$
$$ v(z)\geq - \sum_{j=1}^N \frac{\ln (d_{j}/2)}{\ln \e_{j}}.
$$
On the other hand 
$$ \omega_{D\setminus B}(z,B)=\frac{\ln|z|}{\ln \sigma},
$$
so that from (\ref{3.46}) one has for $z\in \hat U$
$$\omega_{U\setminus\bar B}(z,B)\geq \frac{\ln|z|}{\ln \sigma}- \sum_{j=1}^N \frac{\ln (d_{j}/2)}{\ln \e_{j}}.
$$
Finally since $\ln m\leq 0$,  (\ref{eq:8.33'}) gives that  for any $z\in \hat U$
$$ \ln |f(z)|\leq (\frac{\ln|z|}{\ln \sigma} -\sum_{j=1}^N \frac{\ln (d_{j}/2)}{\ln \e_{j}})\ln m.\label{eq:0center}
$$

\end{proof}

\begin{defin}\label{defin:3.1}Let $U,U_{1},U_{2}$ be three  nonempty open sets of $\C$ such that , 
$$U_{1}\subset U,\qquad U_{2}\subset U.$$ 
We say that the triple $(U,U_{1},U_{2})$ is $A$-good ($A>0$)  if  for any $f\in\cO(U)$, $\sup_{U}|f|\leq 1$, one has
$$ \ln \|f\|_{U_{1}}\leq A\ln \|f\|_{U_{2}}.
$$
\end{defin}
\begin{rem}\label{rem:3.1}Notice that if there exists an open set $U'\subset U$, $U_{1}\subset U'$, $U_{2}\subset U'$ such that   
$(U',U_{1},U_{2})$ is $A$-good, then $(U,U_{1},U_{2})$ is also $A$-good. 
\end{rem}

We recall that we denote by $\bA(z;\l_{1},\l_{2})$, $0<\l_{1}<\l_{2}$, the annulus $\bD(z,\l_{2})\setminus\bar\bD(z,\l_{1})$.

Here is an immediate Corollary of Lemma (\ref{lem:fewholes}):
\begin{cor}\label{cor:3.2}Assume that the assumptions of Lemma (\ref{lem:fewholes}) hold with $\s=\rho^b/2$ ($b>1$). Then for all $1\leq i\leq N$ such that  $\bD(z_{i},d_{i})\subset \bD(0,e^{-\d}\rho)$ ($\d>0$) the triple $$\biggl(U,\bA(z_{i}; (d_{i}/10),d_{i}), \bD(0,\rho^b/2)\biggr)$$ is $A$-good with
$$A=\frac{\d}{b|\ln \rho|} -\sum_{j=1}^N \frac{\ln (d_{j}/20\rho)}{\ln (\e_{j}/\rho)}.$$
\end{cor}

\section{Symplectic diffeomorphisms on holed domains}\label{sec:4}

\subsection{Cartesian Coordinates (CC) and Action-Angle variables (AA)}\label{sec:2.2} 
We define on $\R^2:=\{(x,y),\ x, y \in \R\}$ (resp.  $\T\times \R:=\{(\th,r),\ \th\in\T,\ r\in \R\}$) the canonical  symplectic structure (area) $\beta_{\R}^{CC*}:=dx\wedge dy$ (resp. $\beta_{\R}^{AA}:=d\th\wedge dr$).
This space as well as its  symplectic structure can be complexified: the space $\C^2:=\{(x,y),\ x, \ y \in \C\}$ (resp. $\T_{\infty}\times \C:=\{(\th,r),\ \th\in\T_{\infty},\ r\in \C\}$) carries the symplectic structure $\beta_{\C}^{CC*}:=dx\wedge dy$ (resp. $\beta_{\C}^{AA}:=d\th\wedge dr$)   and the involution $\s_{1}$  (resp. $\s_{3}$)  defined  in (\ref{invol}) preserves $(\C^2,\beta^{CC*}_{\C})$ (resp. $(\T_{\infty}\times \C,\beta^{AA}_{\C})$) and fixes  $(\R^2,\beta_{\R}^{CC*})$ (resp.  $(\T\times \R,\beta^{AA}_{\R})$).

When working in the elliptic fixed point case, it will be more convenient to use other cartesian coordinates. 
Let's introduce the (holomorphic) complex change of coordinates
 $\ph:\C^2\to\C^2$,  $\ph:(x,y)\mapsto (z,w)$, 
 \be\begin{cases}&z=\frac{1}{\sqrt{2}}{(x+iy)}\\ &w=\frac{i}{\sqrt{2}}{(x-iy)}\end{cases}\iff \begin{cases}&x=\frac{1}{\sqrt{2}}(z-iw)\\ &y=\frac{-i}{\sqrt{2}}(z+iw).\end{cases}\label{changecoordxyzw}
\ee
We see that ($\s_{2}$ is as in  (\ref{invol})) with the notations of Subsection \ref{sec:new2.3}
$$ dx\wedge dy= \ph^*(dz\wedge dw),\qquad  \ph\circ\s_{1}\circ \ph^{-1}=\s_{2},\qquad \ph(W_{h,U}^{CC*})=W_{h,U}^{CC}.
$$

We shall denote 
$(M^{CC},\b^{CC},\s_{2})$, {\it resp.} $(M^{CC*},\b^{CC^*},\s_{1})$, (CC stands for Cartesian Coordinates) the space $\C^2$ endowed with the symplectic structure $\beta^{CC}:=dz\wedge dw$, {\it resp.} $\beta^{CC*}=dx\wedge dy$, and the involution $\s_{2}$, {\it resp.} $\s_{1}$.  Similarly,  $(M^{AA},\b^{AA},\s_{3})$ (AA for Action-Angle coordinates) is the space $\T_{\infty}\times \C$ endowed with the symplectic structure $\beta^{AA}:=d\th\wedge dr$ and the involution $\s_{3}$. 
We shall use for short the generic notation $(M,\b,\s)$ to denote either of  the preceding sets endowed with their symplectic structure and involution. We also use the notation $M_{\R}$ or $(M)_{\R}$ for $M\cap \s(M)$. The 2-form $\b$ restricted to $M_{\R}$ is still a symplectic form.
We shall call the  {\it origin} $O$ in  $M_{\R}$, the set  $O=\{(0,0)\}$ if $M=\C^2=M^{CC}$ or  $M^{CC^*}$ and $O=\T\times\{0\}$ if $M=M^{AA}=\T_{\infty}\times \C$.

If $W$ is a nonempty open set of $M$ (resp. $M_{\R}$) we say that $f\in{\rm Diff}^\cO(W)$ (resp.  $f\in{\rm Diff}^{C^1}(W)$) is {\it symplectic} if it preserves the canonical symplectic form $\beta$: $f^*\beta=\beta$. We denote by ${\rm Symp}^{\cO}(W)$ (resp. ${\rm Symp}^{C^1}(W)$) the set of such symplectic holomorphic (resp. $C^1$) diffeomorphisms. If furthermore $f\circ \s=\s\circ f$ we write $f\in {\rm Symp}_{\s}^{\cO}(W)$. We shall say that a  symplectic diffeomorphism $f$ is {\it exact symplectic} if there exists a 1-form $\l$, the {\it Liouville form},  such that $d\l=\beta$ and $f^*\l-\l$ is exact: there exists a function $S$ such that $f^*\l-\l=dS$; $S$ is called the {\it generating function} of $f$ (w.r.t. $\l$). We then denote $f\in{\rm Symp}^\cO_{ex.,\s}(W)$ (resp.  $f\in{\rm Symp}_{ex.}^{C^1}(W)$). In our case the relevant Liouville forms will be 
\be \textrm{(AA)}\  \l=rd\th,\quad \textrm{(CC)} \ \l=(1/2)(wdz-zdw),\quad \textrm{(CC*)}\ \l=(1/2)(xdy-ydx).\label{eq:Liouville} \ee
Let $W\subset M$ be $\s$-symmetric ($\s(W)=W$) and such that $(W)_{\R}:=W\cap \s(W)=W\cap M_{\R}$ is a nonempty open set of $M_{\R}$. Then, if $f\in{\rm Symp}^\cO_{ex.,\s}(W)$,  its   restriction $f_{\hs|  (W)_{\R} }:M_{\R}\supset (W)_{\R}\to f((W)_{\R})\subset M_{\R}$ defines a real analytic (exact) symplectic diffeomorphism.  If  $S\subset W$ is $f$-invariant ($f(S)=S$) the set   $(S)_{\R}:=S\cap M_{\R}$ is also left invariant by $f_{\hs|  (W)_{\R} }$.  Notice that If $U\subset \C$ is a real-symmetric  open set such that $U\cap \R\ne\emptyset$ we have
$$\begin{cases}&(W_{h,U}^{AA})_{\R}=(\T_{h}\times U)_{\R}=\T\times (U\cap \R)=W^{AA}_{0,U\cap \R}\\
&(W^{CC}_{h,U})_{\R}=\{(z,w)\in W^{CC}_{0,U\cap \R_{+}},\ w=i\bar z\}= (W^{CC}_{0,U\cap \R_{+}})_{\R}\\
&(W^{CC*}_{h,U})_{\R}=\{(x,y)\in\R^2,\ \frac{x^2+y^2}{2}\in U\cap\R_{+} \}=W^{CC*}_{0,U\cap\R_{+}}.
\end{cases}
$$
Notice that in any case $(W_{h,U})_{\R}=\{r\in U\}\cap M_{\R}=\{r\in U\cap\R\}\cap M_{\R}$.

There are  symplectic  changes of coordinates  $\psi_{\pm}$ that allow to pass from   the $(z,w)$-coordinates ((CC)- coordinates) to the $(\th,r)$ coordinates ((AA)-coordinates). 
They are  defined as follows. 
The maps $r\mapsto r^{1/2}$,  $te^{is}\mapsto t^{1/2}e^{is/2}$ for $t>0$ and $-\pi<s<\pi$   ({\it resp.} for $t>0$ and $0<s<2\pi$) define  holomorphic functions  on $\C\setminus\R_{-}$ ({\it resp.} on $\C\setminus\R_{+}$). We can thus define 
the biholomorphic diffeomorphisms  
\begin{align}&\T_{\infty}\times (\C\setminus\R_{\pm} )\ni (\th,r)\ \mathop{\longrightarrow}^{\psi_{\pm}}\ (z,w)\in  \{(z,w)\in\C^2,\ -izw\notin \R_{\pm} \} \notag\\
&\begin{cases}&z=e^{i\pi/4}r^{1/2}e^{-i \th}\\
&w=e^{i\pi/4}r^{1/2}e^{i \th}
\end{cases}
\iff
\begin{cases}&r=-izw\\
&e^{i\th}=e^{-i\pi/4}\frac{w}{(-izw)^{1/2}}=e^{i\pi/4}\frac{(-izw)^{1/2}}{z}
\end{cases}\label{defPsi}
\end{align}
which satisfy 
$$dz\wedge dw=d\th\wedge dr\ \textrm{and}\ \psi_{\pm}\circ\s_{2}\circ \psi_{\pm}^{-1}=\s_{3}.
$$
Notice that if $h>0$
\be  \T_{h}\times (\bD(0,\rho)\setminus\R_{\pm} ) \ \mathop{\longrightarrow}^{\psi_{\pm}} \  \biggl\{(z,w)\in\C^2,\ \begin{cases}&-izw\in\bD(0,\rho)\setminus \R_{\pm}\\ & e^{-2h}<|z/w|<e^{2h} \end{cases} \biggr\}\label{e2.14bis}
\ee
hence with the notations of Subsection \ref{sec:new2.3}
\be W^{CC}_{h,U\setminus\R_{\pm}}\supset\psi_{\pm}(W^{AA}_{h,U\setminus\R_{\pm}}).\label{e2.14}
\ee

%
%
%
%
%
%
%
%
%
%
%
%
%
%
%

\subsection{Symplectic vector fields}\label{sec:4.2}
If $(M,\beta)=(M^{CC},\beta)$ or $(M^{AA},\beta)$ and $F\in \cO_{\s}(M)$ we define   the holomorphic {\it symplectic} vector field $X$ by 
$i_{X}\beta=dF$. If  $J$ is the matrix $\bm0& 1\\ -1&0\em$  one has 
$$X_{F}=J\nabla F.$$

We denote by $\phi_{J\nabla F}^t$ the flow at time $t\in\R$ of the vector field $J\nabla F$ and  $\Phi_{F}=\phi^1_{J\nabla F}$ its time 1-map. It is a symplectic diffeomorphism.

If $G:M\to\R\ \textrm{or}\ \C$ is another  smooth observable we define the {\it Poisson bracket} of $F$ and $G$ by the formula $\{F,G\}=\beta(X_{F},X_{G})$ or equivalently 
$$ \{F,G\}:=\<\nabla F,J\nabla G\>.
$$
One then has 
$$ \frac{d}{dt}(G\circ\Phi_{F}^t)_{|t=0}=L_{J\nabla F}G=\{F,G\},\qquad [L_{X_{F}},L_{X_{G}}]=L_{X_{\{F,G\}}}.$$

If $f$ is a symplectic diffeomorphism one has
\be f\circ \Phi_{F}\circ f^{-1}=\Phi_{f_{*}F},\qquad \textrm{where}\  f_{*}F=(f^{-1})^*F=F\circ f^{-1}.\label{eq:4.4}
\ee

\subsection{Integrable models}\label{sec:4.3}
We assume that $(M,\beta,\s)$ is $(\C^2,dx\wedge dy,\s_{1})$, $(\C^2, dz\wedge dw,\s_{2})$ or  $(\T_{\infty}\times \C,d\th\wedge dr,\s_{3})$. In all these examples there exists a natural (Lagrangian) foliation given by the level lines of the  observable $r:M\to \C$
\be r(x,y)= \frac{x^2+y^2}{2},\quad r(z,w)= -i zw,\quad r(\th,r)= r,\label{observabler}\ee
 which has the property that  for every $m\in M$, such that $r(m)\in\R$, the map $\R\ni t\mapsto \phi^t_{J\nabla r}(m)$ is $2\pi$-periodic. In particular, for $c\in\R$, the set $\{r=c\}\subset M$ is itself foliated by the $2\pi$-periodic orbits of the flow $\phi^t_{J\nabla r}$; they are either points or homeomorphic to $\mathbb{S}^1$.
We shall say that a symplectic diffeomorphism of $M$ is {\it integrable} if it is symplectically conjugated to a diffeomorphism that leaves globally invariant each level line of the preceding function $r$. It is not difficult to see that a diffeomorphism satisfying the previous  condition is of the form $\Phi_{H}$ where $H=\Omega\circ r$.

Let $U$ be a $\s$-symmetric  holed domain  of $\C$  and $\Omega\in\cO_{\s}(U)$. Then,
\be\begin{cases}
&(CC): \ \Phi_{\Omega}(z,w)=(e^{-i\pa\Omega(r)}z,e^{i\pa\Omega(r)}w),\\
&(AA):\ \Phi_{\Omega}(\th,r)=(\th+\pa\Omega(r),r),\\
&(CC*):\  \Phi_{\Omega}(x,y)=(\Re(e^{-i\pa\Omega(r)}(x+iy)),\Im(e^{-i\pa\Omega(r)}(x+iy)))
\end{cases}\label{Phiflow}
\ee
and in any case 
$$
\Phi_{\Omega}(W_{\ti h,U})\subset W_{h,U},\qquad \ti h=h-\|\Im(\pa\Omega)\|_{U}.$$

On the other hand, since $\Omega$ is $\s$-symmetric, one has whenever $U$ is $\s$-symmetric,  $$\Phi_{\Omega}((W_{h,U})_{\R})=((W_{h,U})_{\R}).$$

Notice that in all cases $\Phi_{\Omega}$ is an integrable diffeomorphism of $M$.

\subsection{KAM circles}\label{sec:kamcircles}

A {\it circle} of $M_{\R}$ ($M_{\R}$ equals $M^{CC*}_{\R}=\R^2$, $M^{CC}_{\R}$, $M^{AA}_{\R}=\T\times\R$) is any set of the form $(\{r=c\})_{\R}=(\{r=c)\}\cap M_{\R}$, $c\in\R$, of cardinal $>1$ ($r$ is the observable of (\ref{observabler})). Notice that in the (AA) {\it resp.} (CC*) cases this set coincides with the usual circle $\T\times \{r=c\}$ {\it resp.} $\{(x,y)\in\R^2,\ (1/2)(x^2+y^2)=r\}$; in the (CC) or (CC*) cases $(\{r=c\})_{\R}$ is a circle if and only if $c>0$ (it is  empty if $c<0$ and reduced to $\{(0,0\}$ if $c=0$). 

Let $W$ be an open subset of  $M_{\R}$ and $f\in  {\rm Symp}^{C^1}_{ex.}(W)$  a $C^1$ symplectic diffeomorphism $W\mapsto f(W)$. For example $f$ could be the restriction on $W=(W_{h,U})_{\R}$ of  $f\in  {\rm Symp}^\cO_{ex.,\s}(W_{h,U})$, $W_{h,U}\subset M$.
 A {\it KAM-circle} (or {\it KAM-curve}) for $f$ is the image $g((\{r=c\})_{\R})\subset W$  of a circle $(\{r=c\})_{\R}$, $c\in\R$,  by a $C^1$  symplectic diffeomorphism  $g:M_{\R}\to M_{\R}$ fixing the origin ($g(\{r=0\}_{\R})=\{r=0\}_{\R}$) and such that 
$$ g^{-1}\circ f\circ g=\Phi_{2\pi\omega r}+O(r-c),\qquad \omega\in\R\setminus\Q.$$
The set  $g((\{r=c\})_{\R})\subset W$ is then $f$-invariant, homeomorphic to $\mathbb{S}^1$ and non homotopically trivial in the following sense: in the (AA)-case it is homotopic to $\{r=0\}_{\R}=\T\times\{0\}$   and in the  (CC) or (CC*) case  it has degree $\pm 1$ w.r.t. to the origin $\{r=0\}_{\R}=\{(0,0)\}$. Moreover, the restriction of $f$ on  $g((\{r=c\})_{\R})\subset W$ is conjugated to a rotation on a circle with frequency $\omega\in\R$.

We denote by $\cG(f,W)$ the set of $f$-invariant  KAM-circles   $\g\subset  (W)_{\R}$ and by  $\cL(f,W)\subset (W)_{\R}$ their union: $\cL(f,W)=\bigcup_{\g\in\cG^{}(f,W)} \g$.
\begin{rem}Let  $g,f, f_{1},f_{2}:M_{\R}\to M_{\R}$ be  $C^1$ symplectic diffeomorphisms  where $g(\{r=0\}_{\R})=\{r=0\}_{\R}$. Then,
\begin{enumerate}
\item  If $A\subset B\subset (M)_{\R}$, then $\cL(f,A)\subset \cL(f,B)$.
\item If $f_{1},f_{2}$ coincide on a set $A$, $\cL(f_{1},A)=\cL(f_{2},A)$
\item  For any set $A\subset M_{\R}$
\be g(\cL(f,A))=\cL(g\circ f\circ g^{-1}, g(A)).\label{4.47}\ee
\item If $g^{-1}\circ f_{1}\circ g$ and $f_{2}$ coincide on a set $A$ one has
\be \cL(f_{1},g(A))=g(\cL(f_{2}, A)).\label{4.47ter}\ee
\end{enumerate}

\end{rem}

\begin{notation}If $A\subset \C$ we define 
$W_{A}=\{r\in A\}\cap M_{\R}=\{r\in A\cap \R\}\cap M_{\R}$.
\end{notation}

\medskip Let us now state  a criterium that ensures the existence of KAM-circles.
 Assume that there exist 
 $$\emptyset\ne L\subset A=I\setminus\bigcup_{j\in J}I_{j}\subset \ti A\subset\R,$$ where $L$ is compact  and $A$ is of the form $I\setminus\bigcup_{j\in J}I_{j}$ where $I\subset\R $ is an interval and the $I_{j}$ are pairwise disjoint intervals.
\begin{prop}\label{prop:prop4.1}Let $f\in{\rm Symp}^{C^1}(W_{\ti A})$ and suppose that there exist $\Omega\in C^1(\R)$ and a $C^1$ symplectic diffeomorphism  $g:M_{\R}\to M_{\R}$ fixing the origin,   $\|g-id\|_{C^1}\leq C^{-1}$ ($C$ depends only on $M$), such that 
$$\textrm{on}\  W_{L}\qquad  g^{-1}\circ f\circ g=\Phi_{\Omega(r)}\qquad \textrm{and}\qquad g(W_{L})\subset W_{\ti A}.$$
Then, if $ \sum_{j\in J}|I_{j}|^{1/2}\leq 1$, one has 
$${\rm Leb}_{M_{\R}}(W_{A}\setminus\cL(f,W_{\ti A} ))\leq C\times ( {\rm Leb}_{\R}(A\setminus L)+\|g-id\|_{C^0}^{1/2}).$$
\end{prop}
\begin{proof}
Since $W_{L}=\cL(\Phi_{\Omega(r)},W_{L})$  one has from (\ref{4.47ter}) $g(W_{L})=g(\cL(\Phi_{\Omega(r)},W_{L})=\cL(f,g(W_{L}))$  and since $g(W_{L})\subset W_{\ti A}$ one has $g(W_{L})\subset \cL(f,W_{\ti A})$.
On the other hand if we define $E$ by $W_{A}=W_{L}\cup E$, one has
$g(W_{A})=g(W_{L})\cup g(E)$ and thus $g(W_{A})\subset \cL(f,W_{\ti A})\cup g(E)$. We therefore have
$${\rm Leb}_{M_{\R}}(W_{A}\setminus\cL(f,W_{\ti A} ))\lesssim {\rm Leb}_{M_{\R}}(g(E))\\ +{\rm Leb}_{M_{\R}}(W_{A} \ \triangle\  g(W_{A} )).
$$
Since $A=I\setminus\bigcup_{j\in J}I_{j}$ and $\sum_{j\in J}|I_{j}|^{1/2}\leq 1$,  Lemma \ref{lemma42appendix} from the Appendix  yields  ${\rm Leb}_{M_{\R}}(W_{A} \ \triangle\  g(W_{A} ))\leq \|g-id\|_{C^0}^{1/2}$ and since ${\rm Leb}_{M_{\R}}(g(E))={\rm Leb}_{M_{\R}}(E)$ we get the conclusion.

\end{proof}

\subsection{Generating functions}\label{sec:4.5}
Let   $h>0$, $U\subset \C$ be a real-symmetric  holed domain and $W_{h,U}^{AA}$ and $W_{h,U}^{CC}$ the domains defined in (\ref{holedomaintc}) and (\ref{holedomaincc}) 
$$\begin{cases}&W^{AA}_{h,U}=\T_{h}\times U\\
&W^{CC}_{h,U}=\{(z,w)\in \bD(0,e^h\rho_{U}^{1/2})\times \bD(0,e^h\rho_{U}^{1/2}),\ r:=-izw\in U\}.
\end{cases}$$
We shall associate to each   $F\in \cO_{\s}(W_{h,U})$ small enough a  real symmetric holomorphic symplectic diffeomorphism $f_{F}$ of $W_{h,U}$ which  is  {\it exact} with respect to the respective Liouville forms as defined in (\ref{eq:Liouville})). It is defined as follows:
in the (AA)-case
\be f_{F}(\th,r)=(\ph,R) \ \iff\  \begin{cases}&\ph=\th+\pa_{R}F(\th,R)\\ &r=R+\pa_{\th}F(\th,R)\label{4.43}\end{cases}
\ee

and in the (CC)-case
\be f_{F}(z,w)=(\ti z,\ti w) \ \iff\  \begin{cases}&\ti z=z+\pa_{\ti w}F(z,\ti w)\\ &w=\ti w+\pa_{\ti w}F(z,\ti w).\label{4.43cc}\end{cases}
\ee

\begin{lemma}\label{domaindef} There exists a constant $\bar C$ such that if $F\in\cO_{\s}(W_{h,U})$ and $0<\d<h$ satisfy  
\be \bar C(\d\ud(W_{h,U}))^{-2}\|F\|_{W_{h,U}}<1,\label{e:4.45bis}
\ee
the map $f_{F}$ defined by (\ref{4.43}), (\ref{4.43cc}) is a real-symmetric holomorphic exact symplectic  diffeomorphism from  $e^{-\d}W_{h,U}$ onto its image
 and 
\be e^{-2\d}W_{h,U}\subset  f_{F}^{}(e^{-\d}W_{h,U})\subset W_{h,U}.\label{f-dom}\ee
We shall call $f_{F}$   the {\it generating map} of  $F$.
Moreover
\be f_{F}^{-1}=f_{-F+O(\|D^2F\|\|DF\|)}.\label{EE427}\ee
\end{lemma}
\begin{proof}
See Appendix \ref{sec:A2eante}.
\end{proof}
\begin{rem}\label{sec:rem4.1}The symplectic change of coordinates $\psi_{\pm}$ introduced in  Subsection \ref{sec:2.2} preserves  exact symplecticity: if $f^{CC}$ is exact symplectic the same is true for $f^{AA}=\psi_{\pm}^{-1}\circ f^{CC}\circ \psi_{\pm}$. Indeed, if $\psi_{\pm}(\th,r)=(z,w)$, $z=e^{i\pi/4}r^{1/2}e^{-i \th}$, $w=e^{i\pi/4}r^{1/2}e^{i \th}$, one computes   the Liouville form 
$(1/2)(wdz-zdw)=rd\th$.

\end{rem}

Conversely, if a diffeomorphism $(\th,r)\mapsto (\ph,R)$ is exact symplectic and close enough to the identity, it admits this type of parametrization.

More precisely:
\begin{lemma}\label{lemma:6.2}Let $f\in {\rm Symp}^\cO_{ex.\s}(W_{h,U})$ be an exact symplectic diffeomorphism  close enough to the identity. Then, there exist $\d=\fd(f-id,W_{h,U})$ and  $F\in\cO_{\s}(e^{-\d}W_{h,U})$ such that on $e^{-\d}W_{h,U}$ one has 
$$ f=f_{F},\qquad  F=O(\|Df-id\|)=\fO_{1}(f-id).
$$
This $F$ is unique up to the addition of a constant.

Conversely, given $F\in \cO(W_{h,U})$ one has 
\be f_{F}=\Phi_{F}\circ f_{\fO_{2}(F)}=id+J\nabla F+O(\|D^2F\|\|DF\|).\label{l465}
\ee
\end{lemma}
\begin{proof}
See the Appendix \ref{sec:A2e}.
\end{proof}

The composition of two exact symplectic maps is again exact symplectic and more precisely
\begin{lemma}\label{lemma:6.2.bis} Let $F,G\in \cO(W_{h,U})$
then on $e^{-\d}W_{h,U}$, $\d=\fd(F,G;W_{h,U})$
\begin{align} &f_{F}\circ f_{G}=f_{F+G+O(\|DF\|_{h,U}\|D G\|_{h,U})}\label{eq:4.53}\\
&f_{F+G}=f_{F+\|DF\|_{{h,U}}\fO_{1} (G)}\circ f_{G}=f_{F}\circ f_{G+ \|DG\|_{{h,U}}\fO_{1} (F)  }.\label{n4.64}
\end{align}
In the Action-Angle case, if  $\Omega$ depends only  on the variable $r$ then $\Phi_{\Omega}=f_{\Omega}$ and 
\be \Phi_{\Omega}\circ f_{F}=f_{\Omega+F}\label{eq:4.55}
\ee
\end{lemma}
\begin{proof}
See the Appendix, Section \ref{sec:A3}.
\end{proof}
\subsection{Parametrization}
We shall parametrize  {\it perturbations} of integrable symplectic diffeomorphisms defined on a domain $W_{h,U}$ by
$$ f=\Phi_{\Omega(r)}\circ f_{F}$$
where $\Omega\in\cO_{\s}(U)$ and $F\in \cO_{\s}(W_{h,U})$. 
Note that if $f_{F}=id+O^2(z,w)$ or $f(\th,r)=id+(O(r),O(r^2))$ then:
$$\textrm{Case\ (CC)} \quad F(z,w)=O^3(z,w),\qquad \textrm{Case\ (AA)}\quad  F(\th,r)=O(r^2).$$

\subsection{Transformation by conjugation}
We now define
\be[\Omega]\cdot Y=Y\circ\Phi_{\Omega}-Y.\label{defcrochetOmega}
\ee
Note that
\be\begin{cases}&  \textrm{(AA)-case\  if}\  Y=Y(\th,r),\quad ([\Phi]\cdot Y)(\th,r)=Y(\th+\pa\Omega(r),r)-Y(\th,r);\\
& \textrm{(CC)-case\  if}\  Y=Y(z,w),\quad ([\Phi]\cdot Y)(z,w)=Y(e^{-i\pa\Omega(r)}z, e^{i\pa\Omega(r)}w)-Y(z,w).
\end{cases}\label{ee4.80}
\ee
If $W=W_{h,U}$ is a holed domain and $\d>0$ we introduce the notation 
$$W_{h,U}^{\Omega}=W_{h,U}^{\Phi_{\Omega}}:=W_{h,U}\cup \Phi_{\Omega}(W_{h,U}).
$$
The main result of this section is the following:
\begin{prop}\label{lemma:6.4} Let $\Omega\in\cO_{\s}(U)$, $F\in \cO_{\s}(W_{h,U})$, $Y\in \cO_{\s}(W_{h,U}^\Omega)$.
Then, for any $\d>0$, $\d=\fd(F,W_{h,U})\cap \fd(Y,W^\Omega_{h,U})$
there exists $\ti F\in \cO_{\s}(e^{-\d}W_{h,U})$ such that 
$$ [e^{-\d}W_{h,U}]\qquad  f_{Y}\circ (\Phi_{\Omega}\circ f_{F})\circ f_{Y}^{-1}=\Phi_{\Omega}\circ f_{\ti F}$$
(see the notation (\ref{ee2.40})) and 
\begin{align*} \ti F&=F+[\Omega]\cdot Y+\|DF\|_{W}\fO_{1}(Y)\\
&=F+[\Omega]\cdot Y+\dot \fO_{2}(Y,F).
\end{align*}
\end{prop}
\begin{proof}See the Appendix, section \ref{sec:proofprop4.4}.
\end{proof}
\begin{rem}\label{rem:4.3} A direct computation shows that  if $\Omega(r)=2\pi\omega_{0}r+O(r^2)$ and 
\begin{align*}&\textrm{Case\ (CC)} \quad F(z,w)=O^{k}(z,w) ,\ Y(z,w)=O^k(z,w),\\
& \textrm{Case\ (AA)}\quad  F(\th,r)=O(r^k), \ Y(\th,r)=O(r^k)
\end{align*}
then 
$$\textrm{Case\ (CC)} \quad \ti F(z,w)=O^{2k-2}(z,w),\qquad \textrm{Case\ (AA)}\quad  \ti F(\th,r)=O^{2k-1}(r).$$
\end{rem}

\subsection{Symplectic Whitney extensions }
Let $U\subset \C$ be a real symmetric holed domain $W_{h,U}\subset M$, $F\in\cO_{\s}(W_{h,U})$ and $F^{Wh}:M\to\C$ be a $\s$-symmetric $C^2$ Whitney extension of $(F,W_{h,U})$ ({\it cf.} Subsection \ref{sec:2.4}). There exists a constant $C>0$ (depending only on $M$) such that if $\| F^{Wh}\|_{C^2(M)}<C^{-1}$,   the equations (\ref{4.43}), (\ref{4.43cc}) define a $C^1$-diffeomorphism $f_{F^{Wh}}:M\to M$ such that
\be \max(\|f_{F^{Wh}}-id\|_{C^1(M)}, \|f_{F^{Wh}}^{-1}-id\|_{C^1(M)})\leq C^{-1}\|F^{Wh}\|_{C^2(M)}.\label{ee4.85}
\ee
Note that $f_{F^{Wh}}$  and $f_{F^{Wh}}^{-1}$ are  $C^1$ $\s$-symmetric extensions of $(f_{F},e^{-\d}W_{h,U})$ and $(f_{F}^{-1},e^{-\d}W_{h,U})$ for any $\d$ satisfying (\ref{e:4.45bis}), {\it cf.} Lemma \ref{domaindef}.
  
In general, the  diffeomorphism $f_{F^{Wh}}$  is {\it not symplectic} on $M$ but since $F^{Wh}$ takes real values on $M_{\R}$, $f_{F^{Wh}}:M_{\R}\to M_{\R}$ is an exact symplectic diffeomorphism of $M_{\R}$. 

\begin{notation}\label{notation:4.2}We shall denote by $\ti{\rm Symp}_{\s}(W_{h,U})$, resp. $\ti{\rm Symp}_{ex.,\s}(W_{h,U})$,  the set of $C^1$ $\s$-symmetric difffeomorphisms  $M\to M$ that are in ${\rm Symp}^\cO_{\s}(W_{h,U})$, resp.  ${\rm Symp}^\cO_{ex,s}(W_{h,U})$, (hence holomorphic on $W_{h,U}$) and symplectic, resp. exact symplectic, when restricted  $M_{\R}\to M_{\R}$.
\end{notation}

\section{Cohomological equations and conjugations}\label{sec:conj}
Our aim in this section is to provide a unified treatment, both in the (AA) and (CC) cases, of the resolution of the (co)homological equations (Proposition \ref{prop:n5.2}) involved in the  Fundamental conjugation step (Proposition \ref{prop:n5.5}) 
that we shall use  to construct  all our  different Normal Forms  (for instance  the approximate Birkhoff Normal Form of Section \ref{sec:BNF}, the  KAM Normal Forms of Section \ref{sec:5} and the resonant Normal Form of Appendix \ref{appendixE:RNF}).
\subsection{Fourier coefficients and their generalization}

In this section we assume that either:
\begin{itemize}
\item {\it Case (CC)}: $(M,\omega)=(\C\times \C,dz\wedge dw)$ and we denote by $r(z,w)=-izw$
\item  or,  {\it Case (AA)}:  $(M,\omega)=(\T_{\infty}\times \C,d\th\wedge dr)$ and we denote by $r:(\th,r)\mapsto r$.
\end{itemize}
In both cases the flow $t\mapsto \phi^t_{J\nabla r}$ is $2\pi$-periodic w.r.t.  $t\in\R$ ({\it cf.} (\ref{Phiflow}).

Let $U$ be a connected open set of $\C$ and $F\in\cO(W_{h,U})$. For any  $m\in W_{h,U}$ and any $t\in \R$, $\phi^t_{J\nabla r}(m)\in W_{h,U}$:
$$\begin{cases}
&(CC): \ \phi^t_{J\nabla r}(z,w)=(e^{-i t}z,e^{it}w),\\
&(AA):\ \phi^t_{J\nabla r}(\th,r)=(\th+t,r).
\end{cases}
$$
 We can hence define $t\mapsto F(\phi^t_{J\nabla r}(m))$ which is a $2\pi$-periodic function $\R\to\C$ and    for $n\in\Z$ we introduce its $n$-th Fourier coefficient $\cM_{n}(F)(m)$:
\begin{align}& \cM_{n}(F)(m)=\frac{1}{2\pi}\int_{0}^{2\pi}e^{-int}F\circ \phi_{J\nabla r}^t(m)dt\label{defDnF}\\
&F(\phi^t_{J\nabla r}(m))=\sum_{n\in\Z}\cM_{n}(F)(m)e^{int}.
\end{align}
The dependence of  $\cM_{n}(F)(m)$  is holomorphic in $m$ and we have thus defined  $\cM_{n}(F)\in\cO(W_{h,U})$.
We observe that 
\be \cM_{n}(F)\circ \phi_{J\nabla r}^{2\pi/n}=\cM_{n}(F)\label{eq:5.69}
\ee
and 
$$ \forall\ t\in\R,\ \cM_{0}(F)\circ \phi_{J\nabla r}^{t}=\cM_{0}(F).
$$
\subsubsection{Case (CC)}\label{sec:5.1.1} One has 
\be \phi^t_{J\nabla r}(z,w)=(e^{-it}z,e^{it}w)\label{ee5.83}\ee
and if $F=F(z,w)$, (\ref{defDnF}) becomes
$$\cM_{n}(F)(z,w)=\frac{1}{2\pi}\int_{0}^{2\pi}e^{-int}F(e^{-it}z,e^{it}w)dt.$$
If furthermore $F(z,w)=\sum_{(k,l)\in \N^2}F_{k,l}z^kw^l$ is converging on some polidisk $\bD(0,{\mu})\times \bD(0,{\mu})$ one has 
\be \cM_{n}(F)(z,w)=\sum_{\substack{(k,l)\in\N^2\\l-k=n}}F_{k,l}z^kw^l\label{formDn}
\ee
hence, if for some $p\in\N^*$, $F(z,w)=O^p(z,w)$, then for any $n\in\N$,  $\cM_{n}(F)(z,w)=O^p(z,w)$. 
\subsubsection{(AA) Case }\label{sec:5.1.2} In that case
$$\phi^t_{J\nabla r}(\th,r)=(\th+t,r)$$ and if $F=F(\th,r)$ we define 
\begin{align*}\cM_{n}(F)(\th,r)&=(2\pi)^{-1}\int_{0}^{2\pi}e^{-int} F(\th+t,r)dt\\
&= \hat F(n,r)e^{in\th}
\end{align*}
where $$\hat F(n,r)=(2\pi)^{-1}\int_{0}^{2\pi}e^{-in\th}F(\th,r)d\th$$ is the $n$-th Fourier coefficient of $F(\cdot,r)$. Notice that though $F$ is only defined on $\T_{h}\times U$,   $\cM_{n}(F)$ is defined in $\T_{\infty}\times U$.

\begin{rem}\label{rem:5.1}We see from (\ref{defDnF}) that if for some $p>0$, $F=O^p(r)$ (which means that for any $m\in W_{h,U}$ one has  $|F(m)|\leq C|r(m)|^p$ for some $C>0$) then $\cM_{n}F=O^p(r)$ for any $n\in\N$. 
\end{rem}
\begin{rem}Using the fact that  $\cM_{n}(F)\circ \phi^{2\pi/n}_{J\nabla r}=\cM_{n}(F)$ one can show that $f_{F}\circ \phi_{J\nabla r}^{2\pi/n}=\phi_{J\nabla r}^{2\pi/n}\circ f_{F}$ both in the (AA) and (CC) Case.
\footnote{For example in the (CC)-case, since $\phi^{2\pi/n}_{J\nabla r}(z,w)=(e^{-2\pi i/n}z,e^{2\pi i/n}w)$, the condition on $F$ implies 
$F(e^{-2\pi i/n}z,e^{2\pi i/n}\ti w)=F(z,\ti w)$ and the conclusion follows from (\ref{4.43cc}). }
\end{rem}

\subsubsection{Form of $\cM_{0}(F)$}

\begin{lemma}\label{ncor:5.4}If $F\in\cO_{\s}(W_{h,U})$ there exists $M(F)\in\cO_{\s}(U)$ such that 
\be \cM_{0}(F)=M(F)\circ r,\qquad \|M(F)\|_{U}\leq \|F\|_{h,U}.\label{eq:cor5.5ante}\ee

Moreover
\be f_{M(F)}=\Phi_{M(F)}\circ f_{\fO_{2}(F)}.\label{eq:cor5.5}
\ee
\end{lemma}
\begin{proof}
By definition of $\cM_{0}(F)$ we see that  for every $t\in\R$
$$\cM_{0}(F)\circ \phi^t_{J\nabla r}=\cM_{0}(F).$$
Lemma \ref{lem:6.3} of the Appendix gives us $M(F)\in\cO_{\s}(U)$ such that $\cM_{0}(F)=M(F)\circ r$.
We just have to prove (\ref{eq:cor5.5}) in the (CC) case. If $(\ti z,\ti w)=f_{M(F)}(z,w)$ one has 
$$\ti z=(1+\pa (M(F))(z\ti w))z,\qquad \ti w=(1+\ti w \pa (M(F))(z\ti w))^{-1}w
$$
and since $\ti w(z,w)=w+\fO(F)$ we get
$$(\ti z,\ti w)=(e^{-\pa (M(F))(zw)}z, e^{\pa (M(F))(zw)} w)+\fO_{2}(F).
$$
\end{proof}

\subsubsection{Decay of the $\cM_{n}(F)$. }
We observe that in {\it Case (AA)},  for $m=(\th,r)$ fixed in $W_{h,U}$,  the function 
$$\begin{cases}&\T_{h-|\Im\th|}\to\C\\
&t\mapsto F(\phi^t_{J\nabla r}(m))=F(\th+t,r)\end{cases}$$
 is  well defined and holomorphic and  that in {\it Case (CC)}, for $(z,w)\in W_{h,U}$ fixed (recall (\ref{ee5.83}) and   the definition (\ref{holedomaincc}) of $W_{h,U}^{CC}$),  the function 
 \be\begin{cases}&\R+i ]-\ln(e^h\rho^{1/2}/|w|), \ln(e^h\rho^{1/2}/|z|)[\to\C\\
 &t\mapsto F(\phi^t_{J\nabla r}(m))=F(e^{-it}z,e^{it}w)\end{cases}\label{e5.94}\ee
 (with $\rho=\sup\{|r|,\ r\in U\}$)
  is also a well defined  $2\pi\Z$-periodic holomorphic function.
 The Fourier coefficients $\cM_{n}(F)(m)$ of the function  $t\mapsto F\circ \phi_{J\nabla r}^t(m)$
\be F\circ \phi_{J\nabla r}^t=\sum_{n\in\Z}e^{int}\cM_{n}(F)\label{expressionF}
\ee
thus satisfy
\be  |\cM_{n}(F)(m)|\lesssim \begin{cases}& e^{-|n|(h -|\Im\th|)}\|F\|_{W_{h,U}}\qquad \textrm{in\ Case\ (AA)} \\
& e^{-|n| h_{z,w}}\|F\|_{W_{h,U}}\qquad\textrm{in\ Case\ (CC)} 
\end{cases} \label{ee5.87}
\ee
where $h_{z,w}=h+\min (\ln(\rho^{1/2}/|w|), \ln(\rho^{1/2}/|z|))$.

\subsubsection{Truncations operators}
Let us define for $N\in \N\cup\{\infty\}$,
$$T_{N}F=\sum_{|n|<N}\cM_{n}(F), \qquad R_{N}F=F-T_{N}F.
$$
\begin{lemma}\label{lemma:n5.1}If $F\in \cO(W_{h,U})$ one has
\begin{align}
&\textrm{on}\ \quad W_{h,U},\quad  F=\sum_{n\in\N}\cM_{n}(F)\label{n5.82}\\
 & \|\cM_{n}(F)\|_{W_{h -\d,U}}\lesssim e^{-|n|\d}\|F\|_{W_{h,U}},\label{n5.81}\\
& \|R_{N}F\|_{W_{h-\d,U}}\lesssim \d^{-1}e^{-N\d}\|F\|_{W_{h,U}}, \label{n5.83}\\
&\|T_{N}F\|_{W_{h-\d,U}}\lesssim \|F\|_{W_{h,U}}\qquad (\textrm{if}\ \d^{-1}e^{-N\d}\leq 1). \label{n5.83bis}
\end{align}
Furthermore, if for some $p>0$, $F=O^p(r)$ then 
\be R_{N}F=O^p(r) \label{n5.81ccbisbis}\ee in the (CC Case), if $F\in \cO(W_{h,U})\cap O^3(z,w)$, one has
\be (R_{N}F)(z,w)=O^{N}(z,w).\label{n5.81ccbis}
\ee
\end{lemma}
\begin{proof}We use (\ref{ee5.87}).
In the (AA Case) if  $m=(\th,r)\in  \T_{h-\d}\times U$ one has $h-|\Im\th|\geq \d $ and in the (CC Case), if $m=(z,w)\in  W^{CC}_{h-\delta,U}$ one has  $\max(|z|,|w|)\leq e^{h-\delta}\rho^{1/2}$ thus  $h_{z,w}\geq \d$. We thus have (\ref{n5.81}) in all cases. Equality (\ref{n5.82}) comes from taking $t=0$, (\ref{n5.83}) is a consequence of (\ref{n5.81}) and (\ref{n5.83bis}) is clear from (\ref{n5.83}).
 Inequality (\ref{n5.81ccbisbis}) is a consequence of Remark \ref{rem:5.1}.
 Inequality (\ref{n5.81ccbis}) is a consequence of (\ref{formDn}).
\end{proof}

\subsection{Solution of the truncated cohomological equation}
We assume that $0<\rho\leq 1$ and that $U$ is a $\s$-symmetric  open connected set of $\bD$. 

We recall that we have defined in (\ref{defcrochetOmega}) ({\it cf.} Proposition \ref{lemma:6.4}) for any $\Omega\in\cO(U)$ and $Y\in\cO(W_{h,U}^\Omega)$ 
$$[\Omega]\cdot Y=Y\circ\Phi_{\Omega(r)}-Y.$$
The main Proposition is the following
\begin{prop}\label{prop:n5.2}Let $\tau\geq 0$, $\Omega\in \cO_{\s}(U)$, $K>0$, $N\in\N^*\cup\{\infty\}$ be such that  one has on $U$
\be \forall \ (k,l)\in\N^*\times \Z,\ 1\leq k<N\ \implies\   |k\frac{1}{2\pi}\pa \Omega(\cdot)-l|\geq K^{-1}|k|^{-\tau}.\label{ncohomcond}
\ee
Then, for any  $F\in\cO_{\s}(W_{h,U})$, there exists   $Y\in \cO_{\s}(W_{h,U}^\Omega)$  such that, on $W_{h,U}$, one has  $\cM_{0}(Y)=0$, $\cM_{k}(Y)=0$ for $|k|\geq N$ and 
\be T_{N}F-\cM_{0}(F)=[\Omega]\cdot Y.\label{cohoeq}
\ee
This $Y$ satisfies for any $0<\d<h$ 
\be \|Y\|_{W_{h-\d,U}^\Omega}\lesssim K\min(\d^{-(1+\tau)},N^{\tau+1})\|F\|_{h,U}.\label{ne5.86}
\ee
\end{prop}
\begin{proof}
We observe that both in {\it Case (AA)} or {\it Case (CC)} one has on $W_{h,U}\cap \{r\in\R\}$ ({\it cf.} (\ref{Phiflow}))
$$\Phi_{\Omega(r)}=\phi^{ \pa \Omega(r)}_{J\nabla r}.$$
Hence, if $G$ is a function in $\cO(W_{h,U})$ one has on $W_{h,U}\cap \{r\in\R\}$
\begin{align*} \cM_{n}(G)\circ \Phi_{\Omega(r)}&=\frac{1}{2\pi}\int_{0}^{2\pi}e^{-int}G\circ \phi_{J\nabla r}^{t+\pa\Omega(r)}dt\\
&=e^{ in\pa\Omega(r)}\frac{1}{2\pi}\int_{0}^{2\pi}e^{-int}G\circ \phi_{J\nabla r}^{t}dt\\
&=e^{ in\pa\Omega(r)}\cM_{n}(G)
\end{align*}
and since $\cM_{n}(G)\in \cO(W_{h,U})$, the left hand side of the preceding equations can be holomorphically  extended to a function in $\cO(W_{h,U})$. 
We then have in $\cO(W_{h,U})$
$$[\Omega]\cdot \cM_{n}(G)=(e^{in \pa\Omega(r)}-1)\cM_{n}(G).$$
Note that from Lemma \ref{lem:D1} one has for $r\in U$,  $|e^{in\pa\Omega(r)}-1|\geq K^{-1}|n|^\tau$.
If we define $Y$ by 
$$Y=\sum_{0<|n|< N} \frac{1}{e^{in\pa\Omega(r)}-1} \cM_{n}(F)$$
we have
from Lemma \ref{lemma:n5.1}
\begin{align*}\|Y\|_{W_{h-\d,U}}&\lesssim K\sum_{1\leq |n|< N} |n|^\tau e^{-|n|\d}\|F\|_{h,U}\\
&\lesssim \min(K{\d}^{-(1+\tau)},KN^{\tau+1})\|F\|_{h,U}
\end{align*}
and 
$$[\Omega]\cdot Y=Y\circ \Phi_{\Omega}-Y=T_{N}F.$$
This last formula shows that 
if we define $\ti Y$ on $\Phi_{\Omega(r)}(W_{h-\d,U})$ by  $\ti Y\circ \Phi_{\Omega(r)}=T_{N}F+Y$ the functions $\ti Y$ and $Y$ coincide on $\Phi_{\Omega(r)}(W_{h-\d,U})\cap W_{h-\d,U}$ and thus $Y$ can be holomorphically  extended to $\Phi_{\Omega(r)}(W_{h-\d,U})\cup W_{h-\d,U}=:W_{h-\d,U}^\Omega$ and 
$$\|Y\|_{W_{h-\d,U}^\Omega}\lesssim \min(K{\d}^{-(1+\tau)},KN^{\tau+1})\|F\|_{h,U}.
$$
The fact that $\cM_{0}(Y)=0$ and its uniqueness (under the condition $\cM_{0}(Y)=0$) comes again from Lemma \ref{lemma:n5.1}. Finally, the $\s$-symmetry of $Y$ on $W_{h,U}$ is clear.
\end{proof}
\begin{rem}\label{rem:5.1}If in Proposition \ref{prop:n5.2} $U=\bD(0,\rho)$ is a disk centered at 0  and 
$$\begin{cases}&\textrm{(AA)-case}\ F(\th,r)=\displaystyle\sum_{k\in\N}\sum_{l\in\Z}\hat F_{k}(l)e^{il\th}r^k\\
&\textrm{(CC)-case}\ F(z,w)=\displaystyle\sum_{(k,l)\in\N }F_{k,l}z^kw^l
\end{cases}$$
one has the more explicit expressions
\be \begin{cases}&\textrm{(AA)-case}\ Y(\th,r)=\displaystyle\sum_{k\in\N}\sum_{l\in\Z^*}\displaystyle\frac{\hat F_{k}(l)}{e^{il\pa\Omega(r) }-1}e^{il\th}r^k\\
&\textrm{(CC)-case}\ Y(z,w)=\displaystyle\sum_{\substack{ (k,l)\in\N \\ {l\ne k} }}\displaystyle\frac{F_{k,l}}{e^{i(l-k)\pa\Omega(r)} -1}z^kw^l.
\end{cases}\label{ee5.94}\ee
In particular, if 
$$\Omega(r)=2\pi\omega_{0}r\qquad  \textrm{and}\ \qquad \begin{cases}& \textrm{(CC)-case}\qquad F(z,w)=O^{m}(z,w)\\
&\textrm{(AA)-case} \qquad F(\th,r)=O^m(r)
\end{cases}
$$
then $Y$ satisfies also (see the remarks at the end of Subsections \ref{sec:5.1.1} and \ref{sec:5.1.2})
$$\begin{cases}& \textrm{(CC)-case}\qquad Y(z,w)=O^{m}(z,w)\\
&\textrm{(AA)-case} \qquad Y(\th,r)=O^m(r).
\end{cases}
$$

\end{rem}

\subsection{Fundamental conjugation step}
We begin by the following  consequence of Proposition \ref{lemma:6.4}. Let $U$ be a holed domain, $h>0$.
\begin{lemma}\label{lemma:5.4}There exists $\bar a\geq 2$ such that if $\Omega\in\cO_{\s}(U)$, $F\in\cO_{\s}(W_{h,U})$, $Y\in\cO_{\s}(W_{h,U}^\Omega)$ one has on $e^{-\d}W_{h,U}$, $\d=\fd(F,Y; W_{h,U})$  ({\it cf.} Lemma \ref{ncor:5.4} for the definition of $M(F)$)
\be f_{Y}\circ \Phi_{\Omega}\circ  f_{F}\circ f_{Y}^{-1}=\\ \Phi_{\Omega+M(F)}\circ f_{F-\cM_{0}(F)+[\Omega+M(F)]\cdot Y+\dot\fO_{2}^{(\bar a)}(Y,F)}.\label{5.117b}
\ee
\end{lemma}
\begin{proof}
We first observe that since $F=F-\cM_{0}(F)+\cM_{0}(F)$ we have by (\ref{n4.64}) and Lemma \ref{ncor:5.4}
\begin{align*}f_{F}&=f_{\cM_{0}(F)}\circ f_{F-\cM_{0}(F)+\fO_{2}(F)}\\
&=\Phi_{M(F)}\circ f_{\fO_{2}(F)} \circ  f_{F-\cM_{0}(F)+\fO_{2}(F)}\\
&=\Phi_{M(F)}\circ f_{F-\cM_{0}(F)+\fO_{2}(F)}
\end{align*}
and thus
$$ \Phi_{\Omega}\circ f_{F}=\Phi_{\Omega+M(F)}\circ f_{F-\cM_{0}(F)+\fO_{2}(F)}.$$
Now we use Proposition \ref{lemma:6.4}: for some $\bar a\geq 2$
\begin{multline} f_{Y}\circ \Phi_{\Omega+M(F)}\circ  f_{F-\cM_{0}(F)+\fO_{2}(F)}\circ f_{Y}^{-1}=\\ \Phi_{\Omega+M(F)}\circ f_{F-\cM_{0}(F)+[\Omega+M(F)]\cdot Y+\dot\fO_{2}^{(\bar a)}(Y,F)}.\label{5.117}
\end{multline}
\end{proof}

\begin{prop}\label{prop:n5.5}Let $\bar a_{0}=\bar a+4$. There exists $\bar C$ such that the following holds. Let $U$ be a holed domain, $\Omega\in \cO_{\s}(U)$,  and $F\in\cO_{\s}(W_{h,U})$. Assume that there exists a holed domain $V\subset U$,   $N\in\N^*\cup\{\infty\}$ and $K>0$ such that on $V$ the following non-resonance  condition ({\it cf.} (\ref{ncohomcond})) is satisfied: 
\be  \forall \ (k,l)\in\N^*\times \Z,\ 1\leq k<N\ \implies\   |k\frac{1}{2\pi}\pa \Omega(\cdot)-l|\geq K^{-1}|k|^{-\tau}\label{diophcondOmegabis}\ee
and assume that $\bar CN^{-1}< \d<\min(h,\bar C^{-1})$,  is such that 
\be (\d \ud(W_{h,V}))^{-(\bar a_{0}+\tau)}K\|F\|_{h,U}<\bar C^{-1}.\label{ne5.878}
\ee
 Then there exists  $Y\in \cO(W_{h,V}^\Omega)$ solution on $W_{h,V}^\Omega$ of the cohomological equation  ({\it cf. }(\ref{cohoeq}), (\ref{ne5.86})):
\be T_{N}F-\cM_{0}(F)=-[\Omega]\cdot Y,\qquad \|Y\|_{e^{-\d/2}W_{h,V}^\Omega} \lesssim K\d^{-(1+\tau)}\|F\|_{W_{h,U}}\label{ee5.100}\ee
and $\ti\Omega\in\cO(e^{-\d}W_{h,V})$, $\ti F\in \cO_{\s}(W_{h,U})$
  such that one has on $e^{-\d}W_{h,V}$
\begin{align}& f_{Y}\circ \Phi_{\Omega(r)}\circ f_{F}\circ f_{Y}^{-1}=\Phi_{\ti \Omega(r) }\circ f_{\ti F},\qquad \ti\Omega=\Omega+M(F)\notag\\
& \|\ti F\|_{C^3(e^{-\d}W_{h,V})}\leq K(\d^{-1}\ud(W_{h,V})^{-1})^{\bar a_{0}+\tau}\biggl( \|F\|_{h,U}^2+e^{-N\d/2}\|F\|_{h,U}\biggr). \label{ne5.89bis}
\end{align}
\end{prop}
\begin{proof}
We apply the preceding Lemma \ref{lemma:5.4} with $Y$ solution of  (\ref{ee5.100}).
Since ({\it cf.} (\ref{ee4.80})) $[\Omega+M(F)]\cdot Y=[\Omega]\cdot Y+O(|\nabla Y||\nabla (M(F))|)=[\Omega]\cdot Y+\fO_{2}^{(2)}(Y,F)$, we get using $[\Omega]\cdot Y+F-\cM_{0}(F)=R_{N}F$ ({\it cf.} (\ref{cohoeq})),
$$e^{-\d}W_{h,U},\qquad f_{Y}\circ \Phi_{\Omega}\circ  f_{F}\circ f_{Y}^{-1}=:\Phi_{\ti\Omega}\circ f_{\ti F}
$$
with
\begin{align}&\ti\Omega=\Omega+M(F)\\
&\ti F=R_{N}F+\fO_{2}^{(\bar a)}(Y,F).
\end{align}
 The definition of the symbol $\fO_{2}^{(\bar a)}$, (\ref{n5.83}) and  (\ref{ee5.100}) 
 show that there exist a universal positive constant $\bar C$ such that if (\ref{ne5.878}) is satisfied one has
\be  \|\ti F\|_{e^{-\d/2}W_{h,V}}\lesssim K\d^{-(1+\tau)}(\d^{-1}\ud(W_{h,V})^{-1})^{\bar a}\|F\|_{h,U}^2+\d^{-1}e^{-N\d/2}\|F\|_{h,U}.\label{ee5.104}\ee
 Inequalities (\ref{ne5.89bis}), 
 comes 
 from  (\ref{ee5.104})
  and Cauchy inequalities, see Subsection \ref{subsec:2.6.4}.
\end{proof}

\section{Birkhoff  Normal Forms}\label{sec:BNF}
\subsection{Formal Normal Forms}
We recall in this subsection the classical results on  (formal) Birkhoff Normal Forms. For more details on the related   formal aspects we refer to  Appendix \ref{formalD}. We also explain how Pérez-Marco's dichotomy extdends to the diffeomorphism case (in particular in the (AA)-case).

\subsubsection{BNF near a non resonant elliptic fixed point ((CC) case)}\label{sec:6.1}
Let $\ti f:(\R^2,0)\to (\R^2,0)$ be a real  analytic symplectic diffeomorphism of the form $\ti f(x,y)=D\ti f(0,0)\cdot(x,y)+O^2(x,y)$ where 
$$D\ti f(0,0)=\Phi_{2\pi \omega_{0}r}=\bm\cos(2\pi\omega_{0}) &-\sin(2\pi\omega_{0})\\
\sin(2\pi\omega_{0})& \cos(2\pi\omega_{0})\em$$ with $\omega_{0}\in \R\setminus \Q$.

 If $\ph:\C^2\to\C^2$ is the change of coordinates  $\ph(x,y)=(z,w)$ defined in (\ref{changecoordxyzw}) the diffeomorphism $f:=\ph\circ \ti f\circ \ph^{-1}$ is exact symplectic and of the form $ f(z,w)=\Phi_{2\pi\omega_{0}r}(z,w)+O^2(z,w)$ where $r(z,w)=-izw$
$$\Phi_{2\pi\omega_{0}r}(z,w)=(e^{-2\pi i \omega_{0}}z,e^{2\pi i \omega_{0}}w).
$$ 
From 
 Lemma \ref{lemma:6.2}   we have the representation
$$ f=\Phi_{2\pi\omega_{0}r}\circ f_{F},\qquad F=O^{3}(z,w)$$
for some $F\in \cO_{\s}(\bD(0,\mu)^2)$, $\mu>0$.
We then have the following classical proposition that establishes the existence   of Birkhoff Normal Forms to arbitrarily high order.
 \begin{prop}\label{statementBNFante}Let $\omega_{0}\in \R\setminus\Q$. Then, for any $N\geq 3$ there exist $\s$-symmetric holomorphic maps $\Omega_{N}:(\C,0)\to \C$, $Z_{N},F_{N}:(\C^2,0)\to \C$  such that on a neighborhood of $0\in\C^2$ one has   ($r=-izw$)
 \be \begin{cases}&f_{Z_{N}}\circ (\Phi_{2\pi\omega_{0}r}\circ f_{F})\circ f_{Z_{N}}^{-1}=\Phi_{\Omega_{N}}\circ f_{F_{N}}\\ 
 &F_{N}(z,w)=O^{2(N+1)}(z,w),\quad Z_{N}(z,w)=O^3(z,w),\quad \Omega_{N}(r)=2\pi\omega_{0}r+O^2(r).\end{cases}
 \label{serformBNF}
 \ee
  \end{prop}
\begin{rem}The sequences $(Z_{N})_{N}$ and $(\Omega_{N})_{N}$  converge respectively  in $\C[[z,w]]$ and in $\R[[r]]$. If $Z_{\infty}\in \C[[z,w]]$ and $\Omega_{\infty}\in \R[[r]]$ are there respective limits one has in $\C[[z,w]]$ the formal identity
\be \begin{cases}&f_{Z_{\infty}}\circ (\Phi_{2\pi\omega_{0}r}\circ f_{F})\circ f_{Z_{\infty}}^{-1}=\Phi_{\Omega_{\infty}}\\
&Z_{\infty}(z,w)=O^3(z,w),\qquad \Omega_{\infty}(r)=2\pi\omega_{0}r+O^2(r).\end{cases}\label{eee6.95}\ee
Conversely, (\ref{eee6.95}) defines $\Omega_{\infty}$ uniquely\footnote{The normalizing map $Z_{\infty}$ is unique up to composition on the left by a formal generalized symplectic rotation  $\Phi_{A}$, $A\in\R[[r]]$. }; $\Omega_{\infty}$ is the Birkhoff Normal Form $BNF(f)$ of $f$ (and $BNF(\ti f)$ of $\ti f$). In particular, $BNF(f)$ is invariant by (formal) symplectic conjugacies which are tangent to the identity.
\end{rem}
\begin{rem} \label{rem:6.2}If $f=\Phi_{\Omega}\circ f_{F}$ with $\Omega=\Omega(r)=2\pi\omega_{0}r+O(r^2)$ and $F(z,w)=O^{2(N+1)}(z,w)$ then 
\be BNF(f)(r)=\Omega(r)+O^{N+1}(r). \label{ee6.118}
\ee
\end{rem}

\subsubsection{BNF near a   KAM circle (Action-Angle case)}
Let $f:(\T\times \R,\T\times\{0\})\to (\T\times \R,\T\times\{0\})$  be a real  analytic symplectic diffeomorphism of the form $f(\th,r)=(\th+2\pi\omega_{0},r)+(O(r),O(r^2))$. We notice that $\Phi_{2\pi\omega_{0}r}:(\th,r)\mapsto (\th+2\pi\omega_{0},r)$. We can thus  write $f$ under the form ($h,\rho>0$)
$$f=\Phi_{2\pi\omega_{0}r}\circ f_{F},\qquad F\in \cO_{\s}(e^{2h}(\T_{h}\times \bD(0,\rho))),\qquad F=O^2(r).
$$

 \begin{prop}\label{statementBNF}
 Let $\omega_{0}\in \R$ be Diophantine. Then, for any $N\geq 3$ there exist real analytic maps  $\Omega_{N}:(\R,0)\to \R$, $Z_{N},F_{N}:(\T\times \R,\T\times\{0\})\to \R$  such that  
 \be \begin{cases}&f_{Z_{N}}\circ (\Phi_{2\pi\omega_{0}r}\circ f_{F})\circ f_{Z_{N}}^{-1}=\Phi_{\Omega_{N}}\circ f_{F_{N}}\\ 
 &F_{N}(\th,r)=O^{N+1}(r),\quad Z_{N}(\th,r)=O^2(r),\quad \Omega_{N}(r)=2\pi\omega_{0}r+O^2(r).\end{cases}
 \label{serformBNF}
 \ee
\end{prop}
\begin{rem}Let  $C^\omega(\T)[[r]]$ (where $C^\omega(\T)=\bigcup_{h>0}C^{\omega}_{h}(\T)$)  be the set of   formal power series  
\be F(\th,r)=\sum_{n\in\N}F_{n}(\th)r^n,\qquad F_{n}\in C^\omega(\T)\ \textrm{for\ all}\ n\in\N.
\ee 
The sequence $(Z_{N})_{N}$  converges in $C^\omega(\T)[[r]]$ and  the sequence $(\Omega_{N})_{N}$  converges in $\R[[r]]$. If $Z_{\infty}\in C^\omega(\T)[[r]]$ and $\Omega_{\infty}\in \R[[r]]$ are there respective limits one has in $C^\omega(\T)[[r]]$ the formal identity
\be \begin{cases}&f_{Z_{\infty}}\circ (\Phi_{2\pi\omega_{0}r}\circ f_{F})\circ f_{Z_{\infty}}^{-1}=\Phi_{\Omega_{\infty}}\\
&Z_{\infty}(\th,r)=O^2(r),\qquad \Omega_{\infty}(r)=2\pi\omega_{0}r+O^2(r).\end{cases}\label{eee6.98}\ee
Conversely, (\ref{eee6.98}) defines  $\Omega_{\infty}$ uniquely\footnote{The normalizing map $Z_{\infty}$ is unique up to composition on the left by a formal integrable twist of the form $\Phi_{A}$, $A\in\R[[r]]$. }; $\Omega_{\infty}$ is the Birkhoff Normal Form $BNF(f)$ of $f$. In particular, $BNF(f)$ is invariant by (formal) symplectic conjugacies which are of the form $id+(O(r),O(r^2))$. 

\end{rem}
\begin{rem}\label{rem:6.4}If $f=\Phi_{\Omega}\circ f_{F}$ with $\Omega=\Omega(r)=2\pi\omega_{0}r+O(r^2)$ and $F(\th,r)=O^{N+1}(r)$ then 
\be BNF(f)(r)=\Omega(r)+O^{N+1}(r). \label{ee6.122}
\ee

\end{rem}

\subsection{Pérez-Marco's Dichotomy}\label{sec:PMdichotomy}
We now discuss the extension of Pérez-Marco's  Dichotomy, Theorem \ref{theo:PMdichotomy}, to the difeomorphism setting. 

The first part of Pérez-Marco's argument in  \cite{PM}, translated in our (CC)-setting,  is based on the fact that the  coefficients of the Birkhoff Normal Form  $B(r)=\sum_{n\in\N^d} b_{n}r^n=\sum_{n\in\N^d} b_{n}(-izw)^n$  of $\Phi_{2\pi\<\omega,r\>}\circ f_{F}$ depend polynomially on the coefficients of $F(z,w)=\sum_{(k,l)\in\N^d\times\N^d}F_{k,l}z^kw^l$. More precisely,  if we denote by $[F]_{j}$, $j\geq 3$, the homogeneous part of $F$ of degree $j$, $[F]_{j}=\sum_{|k|+|l|=j}F_{k,l}z^kw^l$,  the coefficients of the homogeneous part of degree $2j$,  $[B\circ r]_{2j}=\sum_{|k|=j} b_{k}(-izw)^k$ of $B\circ r$,  are polynomials of degree $2 j-2$ in the coefficients of $[F]_{3},\ldots [F]_{j}$. As a consequence, if $(z,w)\mapsto F(z,w)$, $(z,w)\mapsto G(z,w)$ are two $\s$-symmetric holomorphic functions such that $F(z,w)=O^3(z,w)$, $G(z,w)=O^3(z,w)$, for any $n\geq 3$,  the maps $t\mapsto b_{n}(tF+(1-t)G)$ are polynomials of degree $\leq 2|n|-2$.  The second argument in \cite{PM} is then to use results from potential theory (in particular the Bernstein-Walsh Lemma \footnote{On the control of the size of a polynomial of known degree which is bounded above  on a  not pluripolar set.}) applied  to the family of polynomials $t\mapsto b_{n}(tF+(1-t)G)$ that  have a degree which behaves linearly in  $n$. 

To check that the arguments of \cite{PM} adapts to the diffeomorphism case it is hence  enough to check  that $t\mapsto b_{n}(tF+(1-t)G)$ are polynomials of degree $\leq 2( |n|-1)$. 
\begin{lemma} \label{BNF:lemma6.3}If $F,G$ are $\s$-symmetric holomorphic maps $F,G=O^3(z,w)$ in the  (CC)-case (resp. $F,G=O^2(r)$ in the (AA)-case) for every $n\in\N^d$, $|n|\geq 2$, $t\mapsto b_{n}(tF+(1-t)G)$ is a polynomial of degree $\leq 2( |n|-1)$ (resp. $\leq |n|-1$).
\end{lemma}
\begin{proof} We refer to the Appendix \ref{formalD} where we discuss  formal aspects of the BNF (mainly in the (AA)-case) and give a proof of the Lemma in  Subsection  \ref{appendix:D4}).
\end{proof}

\subsection{Approximate BNF}
\subsubsection{Elliptic fixed point case ((CC)-Case)}\label{sec:6.2.1}
Our aim is to give a more quantitative version of Proposition \ref{statementBNFante}.

Recall that $W_{h,\bD(0,\rho)}=\{(z,w)\in \bD(0,e^h\rho^{1/2})^2,\ -izw\in \bD(0,\rho)\}$ and we denote sometimes by $W_{h,\rho}$ the set $W_{h,\bD(0,\rho)}$.

Let $m\geq 4$ be an integer.  Applying    Proposition \ref{statementBNFante} with $m=N-1$ we can assume that the diffeomorphism $f$ is of the form 
\be\begin{cases}
&f=\Phi_{\Omega_{0}}\circ f_{F_{0}}\\
& \Omega_{0}(r)=2\pi\omega_{0}r+O^2(r), \quad\textrm{and}\ F_{0}(z,w)=O^{2m}(z,w).\end{cases}\label{defm}\ee
 In particular ({\it cf.} Remark \ref{rem:2.3}) for some $h>0$ and any $\rho>0$ small enough we can assume that 
\be \|F\|_{e^{h}W_{h,\bD(0,\rho)}}\lesssim \rho^{m},\qquad m=\bar a_{1}:=\max(2\bar a+1,30)\label{ee6.96}
\ee
$\bar a$ being the exponent that appears in Proposition \ref{prop:n5.5} that we can assume to be larger than 10.

Denote by $(p_{n}/q_{n})_{n\geq 1}$ the sequence of best rational approximations of $\omega_{0}$ which has the following properties ({\it cf.} \cite{He79}, Chap. 5, formulae (7.3.1)-(7.3.2) and  Prop. 7.4): for all $n\in\N^*$
 \be \frac{1}{q_{n}+q_{n+1}}< (-1)^n(q_{n}\omega_{0}-p_{n})< \frac{1}{q_{n+1}},\label{7.88}\ee
and 
 \be  \forall \ 0<k<q_{n},\ \forall \ l\in\Z,\   |k\omega_{0}-l|\geq |q_{n-1}\a-p_{n-1}|>\frac{1}{2q_{n}}\label{BNF7.122}.\ee

We refer to Notations \ref{notation:2.1}, \ref{notation:2.2} and \ref{notation:4.2} before stating the following proposition.

\begin{prop}\label{nprop:9.2}Assume that  (\ref{ee6.96}) holds. Then for any $\b>0$
and any $n\gg_{\b} 1$ there exist $g_{q_{n}^{-1}}^{BNF}\in\ti {\rm Symp}_{ex.,\s}(W_{h,q_{n}^{-6}})$, and functions $F^{BNF}_{q_{n}^{-1}}\in\cO_{\s}(W_{h,q_{n}^{-6}})\cap O^{q_{n}^{1-\b}}(z,w)$, $\Omega^{BNF}_{q_{n}^{-1}}\in\ti \cO_{\s}(\bD(0,q_{n}^{-6}))$ such that 
\begin{align}& [W_{h,q_{n}^{-6}}]\qquad  (g_{q_{n}^{-1}}^{BNF})^{-1}\circ \Phi_{\Omega_{0}}\circ f_{F_{0}}\circ g_{q_{n}^{-1}}^{BNF}=\Phi_{\Omega^{BNF}_{q_{n}^{-1}}}\circ f_{F^{BNF}_{q_{n}^{-1}}} \label{ind}\\
& \Omega_{q_{n}^{-1}}^{BNF}(r)-BNF(f)(r)=O^{q_{n}^{1-\b}}(r),\ \textrm{in} \ \R[[r]]  \label{ee6.128}\\
&\|\Omega_{q_{n}^{-1}}^{BNF}\|_{C^3}\lesssim 1\\
&\|g_{q_{n}^{-1}}^{BNF}-id\|_{C^1}\leq q_{n}^{-(m-27)}\\
&\|F^{BNF}_{q_{n}^{-1}}\|_{W_{h,q_{n}^{-6}}}\leq \exp({-q_{n}^{1-\b}}).\label{estFBNqn-1}
\end{align}
If $\Omega\in \cT\CC(A,B)$ one can choose $\Omega_{q_{n}^{-1}}^{BNF}\in\cT\cC(2A,2B)$. 
\end{prop}
\begin{proof}
See the Appendix Section \ref{appendixnprop:9.2}.

\end{proof}

\subsubsection{(AA) or (CC) case when $\omega_{0}$ is Diophantine}
We  formulate here  a more quantitative version of the classical Birkhoff Normal Form Theorem (Propositions \ref{statementBNFante}, \ref{statementBNF}) which holds both in the (AA) or (CC) cases, provided $\omega_{0}$ is Diophantine:
\be \forall \ k\in\Z\setminus\{0\},\ \min_{l\in\Z}| k\omega_{0}-l|\geq \frac{\kappa}{|k|^\tau}\qquad (\tau\geq 1).\label{omega0dioph}\ee

Let as usual $W_{h,\bD(0,\rho)}$ be equal to either $W^{CC}_{h,\bD(0,\rho)}$ or $W^{AA}_{h,\bD(0,\rho)}$ and 
$\Omega\in \cO_{\s}(\bD(0,1))$, $\Omega(r)=2\pi\omega_{0}r+O(r^2)$, where $\omega_{0}$ is assumed to be diophantine with exponent $\tau$.

We assume that $F\in\cO_{\s}(e^hW_{h,\bD(0,1/2)})$, $h>0$ satisfies 
\be \begin{cases}
& (CC)-Case: \quad F(z,w)=O^{2m}(z,w)\\
&(AA)-Case: \quad F(\th,r)=O(r^{m}), 
\end{cases}\quad \textrm{with}\quad m=\bar a_{1,\tau}:=2(\tau+\bar a+2)+1\label{defm}
\ee
(as usual $\bar a$ is the constant introduced in Proposition  \ref{prop:n5.5}) and we define
\be\begin{cases}& (CC)-Case: \quad b_{\tau}=2(\tau+2)\\
& (AA)-Case: \quad b_{\tau}=\tau+2.
\end{cases}\label{e6.89}
\ee

\begin{prop} \label{BNFprop}Assume (\ref{defm}). For any $\b>0$ 
and  any $0<\rho\ll_{\b} 1$, 
 there exist  $\Omega_{\rho}^{BNF}\in\ti \cO_{\s}(\bD(0,\rho^{b_{\tau} }))$, $F_{\rho}^{BNF}\in\cO_{\s}(W_{h,\bD(0,\rho^{b_{\tau}})})\cap O^{(1/\rho)^{1-\b}}(r)$ and $g_{\rho}^{BNF}\in\ti {\rm Symp}_{ex.,\s}(W_{h,\bD(0,\rho^{b_{\tau}})})$ such that on $W_{h,\bD(0,\rho^{b_{\tau}})}$ one has 
\begin{align}&(g_{\rho}^{BNF})^{-1}\circ \Phi_{\Omega}\circ f_{F}\circ g_{\rho}^{BNF}=\Phi_{\Omega_{\rho}^{BNF}}\circ f_{F_{\rho}^{BNF}}\label{6.145}\\
& \Omega_{\rho}^{BNF}(r)-BNF(f)(r)=O^{(1/\rho)^{1-\b}}(r),\ \textrm{in} \ \R[[r]]\label{ee6.128bis}\\
&\|\Omega_{\rho}^{BNF}\|_{C^3}\lesssim 1\notag\\
&\|g_{\rho}^{BNF}-id\|_{C^1}\leq \rho^{m-10}\notag\\
& \|F_{\rho}^{BNF}\|_{W_{h,\bD(0,\rho^{b_{\tau}})}}\lesssim \exp(-(1/\rho)^{1-\b}  ).
\end{align}
If $\Omega\in\cT\cC(A,B)$ then $\Omega_{\rho}^{BNF}\in\cT\cC(2A,2B)$.

\end{prop}
\begin{proof} See the Appendix, Section \ref{appendixBNFprop}.
\end{proof}
\begin{rem}
In the (CC)-case and  when $\omega_{0}$ is Diophantine, one can prove the previous proposition (maybe not with the same value for the exponent $b$)  by using Proposition \ref{nprop:9.2} and the fact that    $n$ large enough
$q_{n}\leq q_{n+1}\leq q_{n}^{\tau(\omega_{0})+}.$
\end{rem}

\subsection{Consequence of the convergence of the BNF}
\begin{lemma}Assume that $BNF(f)$ coincides as a formal power series with a holomorphic function $\Xi\in \cO(\bD(0,\bar \rho))$ and, for $0<\rho\leq \bar \rho$,  let $\Omega\in \cO(\bD(0, \rho))$ be such that 
\be\begin{cases}&\Omega(r)-{BNF}(f)(r)=O^{N+1}(r)\qquad \textrm{in}\ \R[[r]]\\
&\|\Omega\|_{\bD(0,\rho)}\leq 1.
\end{cases}\label{OmegaBNF}
\ee
 Then
$$\|\Omega-\Xi\|_{\bD(0,e^{-1}\rho)}\lesssim \exp(-N).$$ 
\end{lemma}
\begin{proof}Let $\Xi(z)=\sum_{k=0}^\infty \xi_{k} z^k$, $\Omega(z)=\sum_{k=0}^\infty b_{k}z^k$, $\Xi_{N}=\sum_{k=0}^N \xi_{k} z^k$ and $\Omega_{N}=\sum_{k=0}^N b_{k}z^k$. We have  from (\ref{OmegaBNF})  and the fact that $\Xi=BNF(f)$ in $\R[[r]]$
\be\Xi_{N}=\Omega_{N}.\label{6.105}
\ee
On the other hand, we observe that if $g:z\mapsto \sum_{k\in\N}g_{k}z^k$ is in $\cO(\bD(0,\rho))$ one has by Fourier estimates  $|g_{k}|\rho^k\leq \|g\|_{\bD(0,\rho)}$ hence 
for $|z|<e^{-1}\rho$
\begin{align*}|\sum_{k\geq N+1}g_{k}z^k|&\leq \sum_{k\geq N+1}\|g\|_{\bD(0,\rho)}(z/\rho)^{k}\\
&\leq 2e^{-N}\|g\|_{\bD(0,\rho)}.
\end{align*}
As a consequence,
$$\|\Xi-\Xi_{N}\|_{\bD(0,e^{-1}\rho)}\lesssim e^{-N}\|\Xi\|_{\bD(0,\bar\rho)},\qquad \|\Omega-\Omega_{N}\|_{\bD(0,e^{-1}\rho)}\lesssim e^{-N}\|\Omega\|_{\bD(0,t)} .$$
We conclude using (\ref{6.105}).
\end{proof}

To summarize,
\begin{cor}\label{theo:compBNF}If $BNF(\Phi_{\Omega}\circ f_{F})$ converges and coincide on $\bD(0,\bar \rho)$ with $\Xi\in\cO(\bD(0,\bar\rho))$ then  for any $\b>0$ and $\rho\ll_{\b}1$
one has: 
\begin{itemize}
\item If $\omega_{0}$  is $\tau$-diophantine ((AA) or (CC)-case)
$$\|\Omega^{BNF}_{\rho}-\Xi\|_{\bD(0,\rho^{b_{\tau}})}\lesssim \exp(-(1/\rho)^{1-\b}).
$$
\item In the (CC) case for any $\omega_{0}$ irrational 
$$\|\Omega^{BNF}_{q_{n+1}^{-1}}-\Xi\|_{\bD(0,q_{n+1}^{-6})}\lesssim \exp(-q_{n+1}^{1-\b}).
$$
\end{itemize}
\end{cor}

\bigskip

\section{KAM Normal Forms }\label{sec:5}
We present now, in the unified (AA)-(CC) framework, the KAM scheme that is central in all this paper. This will be used in Sections \ref{sec:adaptedkam} and \ref{sec:10.2} to construct the adapted Normal Forms and in Section \ref{sec:8curves} to get estimates on the Lebesgue measure of the set of KAM circles. For sake of clarity we decompose our main result into three propositions: Propositions \ref{prop:1.enonce}, \ref{prop:1.enoncebis}, \ref{prop:1.enoncebisbis}.

As usual we denote in the (AA)-case $M=\T_{\infty}\times \C$, $M_{\R}=\T\times \R$, $O=\T\times\{0\}$ and in the (CC)-case $M=\C\times\C$ and $M_{\R}=M\cap\{r\in\R\}$, $O=\{(0,0)\}$.
\subsection{The KAM statement}
Let  $0<\bar\rho<h/2<1/2$, $A>1$, $B>0$ and  $\Omega\in\ti \cO_{\s}(e^{h}\bD(0,\bar \rho))$ satisfying the following twist condition:
\be \forall \ r \in \R,\  A^{-1}\leq (1/2\pi) \pa^2\Omega(r)\leq A,\quad \textrm{and}\ \ \|(1/2\pi) D^3\Omega\|_{\C}\leq B.\label{7.119}
\ee
 
Let $\omega(r)=(2\pi)^{-1}\pa\Omega(r)$. The image of $\bD(0,e^h\bar \rho)$ by $\omega$ is contained in a disk $\bD(\omega(0),3A\bar\rho)$. We can assume without loss of generality that $\omega(0)\in [-1/2,1/2]$ and consequently, if $\bar \rho$ is small enough we can assume 
\be \omega(\bD(0,e^h\rho)\cap\R)\subset [-3/4,3/4].\label{e7.133}
\ee

Let $\bar C,\bar a_{0}$ be the constants of  Proposition \ref{prop:n5.5}. 
We introduce
\be \bar a_{2}=2(\bar a_{0}+2)+10
\ee
and  assume that $F\in\cO_{\s}(e^{h}W_{h,\bD(0,\bar \rho)})$ satisfies 
\be \|F\|_{e^{h}W_{h,\bD(0,\bar \rho)}}\leq   \bar \rho^{\hs \bar a_{2}}.\label{enew7.135}  \ee
By Cauchy's inequality (\ref{eq:cauchy:derivativesbis}) one has
\be \bar \e:= \max_{0\leq j\leq 3}\|D^j F\|_{W_{h,\bD(0,\bar\rho)}}  \leq  \bar \rho^{\hs 2(\bar a_{0}+2)+1}.\label{eee7.125}\ee 

Associated to this   $\bar\e>0$ there exists a unique $N>0$ such that 
$$-\ln \bar\e=N/(\ln N)^{2}.$$ We then define for $n\geq 1$ the following sequences that depend on   $\bar\e=\bar\e_{1}$, $h$  and $\bar \rho>0$:
\be \begin{cases}
& N_{n}=(4/3)^{n-1}N\\
&\bar \e_{n}=e^{-hN_{n}/(\ln N_{n})^2}\\
&K_{n}^{-1}=\bar\e_{n}^{\hs \frac{1}{2(\bar a_{0}+2)}},\qquad (\bar a_{0}\geq 5)\\
& \delta_{n}=2(\ln N_{n})^{-2}h\\
& \rho_{n}=\bar \rho\exp(-\sum_{j=1}^{n-1}\delta_{j}), \qquad     h_{n}=h-(1/2)\sum_{j=1}^{n-1}\delta_{j}>h/2.
\end{cases}\label{requirements6.76}\ee
If $\bar \rho $ is small enough, for all $n\geq 1$
$$\rho_{n}\geq e^{-1/20}\bar \rho,\qquad h_{n}\geq e^{-1/20}h$$
and ({\it cf.} (\ref{eee7.125})), 
\begin{align}& \rho_{n}/2 >2K_{n}^{-1}\label{ee7.127ante}\\
&  (\d_n(2K_{n}^{-1}))^{-\bar a_{0}}K_{n}\bar\e_{n}<\bar C^{-1}\label{ee7.127}
\end{align}
($\bar C$ is the constant of  Proposition \ref{prop:n5.5}).

\begin{prop}\label{prop:1.enonce}Assume that $\Omega$ and $F$ are as above and that $\bar\rho\ll_{A,B} 1$.
Then, with the notations (\ref{requirements6.76}) the following holds: for $n\geq 1$
 there exist a decreasing (for the inclusion) sequence of  holed domains $(U_{n})_{n\geq 1}$,  functions  $\Omega_{n}\in \ti\cO_{\s}(U_{n})$, $F_{n}\in\cO_{\s}(W_{h_{n},U_{n}})$  with $U_{1}=\bD(0,\bar \rho)$, $\Omega_{1}=\Omega$, $F_{1}=F$ and, for   $n\geq 2$, $1\leq m< n$,  diffeomorphisms
$g_{m,n}\in\ti{\rm Symp}_{ex., \s}^{}(W_{h_{n},U_{n}})$, such that:
\begin{align}&\Omega_{n}\ \textrm{satisfies\ a}\ (2A,2B)-\textrm{twist\ condition} \label{7.139twist}\\
&  g_{m,n}(W_{h_{n},U_{n}})\subset W_{h_{m},U_{m}} \label{ee7.124}\\
& \textrm{on}\  W_{h_{n},U_{n}},\quad  g_{m,n}^{-1}\circ \Phi_{\Omega_{m}}\circ f_{F_{m}}\circ g_{m,n}=  \Phi_{\Omega_{n}}\circ f_{F_{n}}\label{8.112}\\
& \|g_{m,n}-id\|_{C^1}\leq \bar\e_{m}^{1/2},\label{8.112bis}\\
& \max_{0\leq j\leq 3}\|D^jF_{n}\|_{W_{h_{n},U_{n}}}\leq \bar\e_{n}.\label{8.114}
\end{align}
\end{prop}

\begin{proof}
We  construct inductively for $n\geq 2$  sequences $U_{n},F_{n},\Omega_{n},g_{m,n}$ satisfying the conclusion of the proposition with the additional requirements

\smallskip\noindent{\bf Requirement 1:} For $n\geq 2$, $U_{n}$ is of the form 
\begin{align}&U_{n}=\bD(0,\rho_{n})\setminus\bigcup_{i\in I_{n}}\bD(c_{i},\kappa_{i}),\qquad c_{i}\in\R,\qquad \# I_{n}\leq 2N_{n-1}^2\label{e7.145}\\
&K_{n-1}^{-1}\leq \kappa_{i}\leq K_{1}^{-1}e^{\sum_{l=1}^{n-1}\d_{l}},\qquad (\sum_{i\in I_{n}}\kappa_{i}^2)^{1/2}\leq \sqrt{2}e^{\sum_{l=1}^{n-1}\d_{l}}\sum_{l=1}^{n-1}N_{l}K_{l}^{-1}. \label{e7.146}
\end{align}

\smallskip\noindent{\bf Requirement 2:} For $n\geq 2$, $\Omega_{n}\in\ti\cO_{\s}(U_{n})$  satisfies an $(A_{n},B_{n})$-twist condition with
\begin{align}& 1\leq A_{n}\leq 2A-K_{n}^{-1},\qquad 0\leq B_{n}\leq 2B-K_{n}^{-1}\label{eqAnBn}\\
&  \bar C_{0}\max(\bar \rho,{\ua}(U_{n}))\times A_{n}\times B_{n}<1.\label{condUk}\\
&\|\Omega_{n}^{}-\Omega\|_{C^3(\bD(0,\bar\rho))}\leq \sum_{l=1}^{n-1}\bar \e_{l}^{1/2}\leq 2\bar \e_{1}^{\hs 1/2}\label{ee7.150}\\
&\textrm{and}\quad \forall\ m<n,\ \|g_{m,n}-id\|_{C^1}\leq C\sum_{l=m}^{n-1}\e_{l}\lesssim \bar \e_{m}^{1/2}\qquad (C\  \textrm{from}\ (\ref{2.40})).\label{7.150}
\end{align}

For some $n\geq 1$, assume the existence of $U_{n},F_{n},\Omega_{n}$ and the validity of  conditions  (\ref{e7.145}), (\ref{e7.146}),
   (\ref{eqAnBn}), 
   (\ref{condUk}), (\ref{ee7.150}) (if $n\geq 2$) and define $\omega_{n}=(1/2\pi)\Omega_{n}$,   $\omega_{n}^{}=(1/2\pi)\Omega_{n}^{}$. Since (\ref{condUk}) is satisfied we can apply Proposition \ref{proppramexclusion} (with $A=A_{n}$, $B=B_{n}$, $10A^2\nu=K_{n}^{-1}$):  for each $(k,l)\in\Z^2$, $0<k<N_{n}$, such that $\bD(l/k,(10A^2K_{n})^{-1})\cap \omega_{n}(U_{n})\ne \emptyset$,    there exists $c^{(n)}_{l/k}\in\R$ such that 
\be \begin{cases}&\omega^{Wh}_{n}(c_{l/k}^{(n)})=l/k\\
&\forall\ r\in \C\setminus\bD(c_{l/k}^{(n)},K_{n}^{-1}),\ |\omega^{Wh}_{n}(r)-(l/k) |\geq (2A_{n}K_{n})^{-1}.\end{cases}\label{a7.151}\ee
We  denote
$$E_{n}=\{(k,l)\in\Z^2,\ 0<k<N_{n},\ 0\leq | l |\leq N_{n},\  \bD(l/k,(10A^2K_{n})^{-1})\cap \omega_{n}(U_{n})\ne \emptyset \}$$
and we see that 
\be \#E_{n}\leq 2N_{n}^2.\label{cardEn}
\ee
Note that from  (\ref{ee7.150}) and (\ref{e7.133}) we have $|l/k|\leq 1$. Hence, if we define 
\be V_{n}=U_{n}\setminus\bigcup_{(k,l)\in E_{n}} \bD(c_{l/k}^{(n)},K_{n}^{-1})\label{defVn}
\ee
we have for any $r\in V_{n}$ ({\it cf.} (\ref{eqAnBn}))
$$ \forall \ (k,l)\in\N_{n}^*\times \Z,\ 1\leq k<N_{n}\ \implies\   |k\frac{1}{2\pi}\pa \Omega_{n}(r)-l|\geq (4AK_{n})^{-1}$$
hence the non-resonance condition (\ref{diophcondOmegabis}) (with $\tau=0$, $K=4AK_{n}$, $N=N_{n}$)  is satisfied. On the other hand (\ref{e7.145})-(\ref{e7.146})
  ($n\geq 2$) and (\ref{ee7.127ante}) ($n=1$) show using (\ref{defVn}) that (recall $\rho_{n}<\bar\rho<h/2$)
\be \ud(W_{h_{n},V_{n}})=\ud(V_{n})=\min(\ud(U_{n}),K_{n}^{-1})=K_{n}^{-1}\label{aa7.151}
\ee
and  (\ref{ee7.127}) and (\ref{8.114}) show that 
\be  (\d_{n}\ud(W_{h_{n},V_{n}}))^{-\bar a_{0}}K_{n} \|F_{n}\|_{h_{n},U_{n}}<\bar C^{-1}.\label{condKdeltaepsilonn}
\ee 
We can thus apply Proposition \ref{prop:n5.5} (with $\tau=0$, $K=4AK_{n}$, $\d=\d_{n}$, $N=N_{n}$) on $V_{n}$: if one defines 
\be U_{n+1}=e^{-\d_{n}}V_{n}=e^{-\d_{n}}U_{n}\setminus \bigcup_{(k,l)\in E_{n}} \bD(c_{l/k}^{(n)},e^{\d_{n}}K_{n}^{-1}) \label{relUn+1Vn}
\ee
 there exist $Y_{n}\in \cO(e^{-\d_{n}/2}W_{h_{n},V_{n}}^{\Omega_{n}})$, $F_{n+1}\in\cO_{\s}(W_{h_{n+1},U_{n+1}})$, $\Omega_{n+1}\in\cO_{\s}(U_{n+1})$ such that ($\bar\rho$ small enough)
\be \|Y_{n}\|_{e^{-\d_{n}/2}W_{h_{n},V_{n}}}\lesssim K_{n}\d_{n}^{-1}\|F_{n}\|_{W_{h_{n},U_{n}}}\label{7.143}\ee 
\be W_{h_{n+1},U_{n+1}},\qquad   f_{Y_{n}}\circ  \Phi_{\Omega_{n}}\circ f_{F_{n}}\circ  f_{Y_{n}}^{-1}=\Phi_{\ti\Omega_{n+1}}\circ f_{F_{n+1}}\label{eqYnante}\ee

\be \ti \Omega_{n+1}=\Omega_{n}+M(F_{n})\label{eqOmegan}\ee
\be \max_{0\leq j\leq 3}\|D^jF_{n+1}\|_{W_{h_{n+1}, U_{n+1}}}\leq K_{n}(\d_{n}K_{n}^{-1})^{-\hs \bar a_{0}}(\|F_{n}\|_{W_{h_{n},U_{n}}}^2+e^{-\d_{n} N_{n}/2 }\|F_{n}\|_{W_{h_{n},U_{n}}}).\label{eqestKAM}\ee

\medskip 
Let us show that the Requirements 1  (\ref{e7.145})-(\ref{e7.146}) 
 are satisfied for $n+1$. From (\ref{relUn+1Vn}) and (\ref{e7.145}) we see that 
$$U_{n+1}=\bD(0,\rho_{n+1})\setminus\bigcup_{i\in I_{n+1}}\bD(c_{i},\kappa_{i})$$
where $I_{n+1}\leq I_{n}+2N_{n}^2$ ({\it cf.} (\ref{cardEn})) and for all $i\in I_{n+1}$, $\min(e^{\d_{n}}K_{n-1}^{-1},e^{\d_{n}} K_{n}^{-1})\leq \kappa_{i}\leq K_{1}^{-1}e^{\sum_{l=1}^n\d_{l}}$. Similarly, $(\sum_{i\in I_{n+1}}\kappa_{i}^2)^{1/2}\leq e^{\d_{n}}((\sum_{i\in I_{n}}\kappa_{i}^2)^{1/2}+ (2N_{n}^2K_{n}^{-2})^{1/2})$. In other words, (\ref{e7.145})-(\ref{e7.146}) are satisfied for $n+1$.

\medskip Let us now prove that the Requirements 2, 
(\ref{eqAnBn}), 
(\ref{condUk})  (\ref{ee7.150}) are satisfied for $n+1$ and in particular that $\ti \Omega_{n+1}$ has a nice Whitney extension $\Omega_{n+1}:=\ti\Omega_{n+1}^{Wh}$. We first apply Lemma \ref{lemma:2.3ee}  to get a $C^3$,  $\s$-symmetric extension $M(F_{n})^{Wh}:\C\to\C$ for $(M(F_{n}),U_{n})$ such that 
$$\sup_{0\leq j\leq 3}\|D^j M(F_{n})^{Wh} \|_{\C}\lesssim (1+\#J_{U_{n}})^3(\d_{n}\ud(U_{n}))^{-6}\max_{0\leq j\leq 3}\|D^jM(F_{n})\|_{e^{-\d_{n}/10}U_{n}}.
$$
In particular, using Cauchy inequalities, (\ref{e7.145}), (\ref{e7.146}), (\ref{requirements6.76}), (\ref{eq:cor5.5ante}) one gets 
\begin{align}\sup_{0\leq j\leq 3}\|D^j M(F_{n})^{Wh} \|_{\C}&\lesssim N_{n-1}^{6}(\d_{n}K_{n-1}^{-1})^{-6}\d_{n}^{-3}\|M(F_{n})\|_{U_{n}}\notag \\
&\lesssim  K_{n-1}^7\bar \e_{n}\leq \bar \e_{n}^{1/2}.\label{eee7.159}\end{align}
From (\ref{eqOmegan}) we see that if we define the $\s$-symmetric function 
\be \Omega_{n+1}:=\Omega_{n}+ M(F_{n})^{Wh}\label{eee7.160}\ee
one has 
$$\Omega_{n+1}{}_{\ {\big |} U_{n+1} \ }=\ti\Omega_{n+1}
$$
and $(\ref{eqAnBn})_{n+1}$, 
are satisfied (since $-K_{n}^{-1}+\bar \e_{n}^{1/3}<-K_{n+1}^{-1}$).  To see that  $(\ref{condUk})_{n+1}$ holds we use the fact that since the second inequality in (\ref{e7.146})  is true for $n+1$ (as already checked) one has $\ua(U_{n+1}) \leq (\sum_{i\in I_{n+1}}\kappa_{i}^2)^{1/2}\leq 2\sum_{l=1}^{n}N_{l}K_{l}^{-1}\leq K_{1}^{-1/2}$.  If $\bar \rho$ is small enough we see that (\ref{eee7.125}), (\ref{requirements6.76}) and $(\ref{eqAnBn})_{n+1}$ ensures the validity of $(\ref{condUk})_{n+1}$.

Finally let us check $(\ref{7.150})_{n+1}$.  From Lemma \ref{lemma:2.3ee} we see that $(Y_{n}, e^{-(3/4)\d_{n}}W_{h_{n},V_{n}}^{\Omega_{n}})$ has a   $C^3$  $\s$-symmetric Whitney extension $Y_{n}^{Wh}$  such that 
\be  \|Y_{n}^{Wh}\|_{C^3}\lesssim (1+\#J_{V_{n}})^3(\d_{n}\ud(U_{n}))^{-6}\max_{0\leq j\leq 2}\|D^jY_{n}\|_{e^{-(2/3)\d_{n}}V_{n}}.
\ee
From  (\ref{e7.145}), (\ref{defVn}),  (\ref{e7.146})  we see that 
$\#J_{V_{n}}\leq 2N_{n}^2$, $\ud(V_{n})\geq  K_{n}^{-1}$
hence using Cauchy's inequalities, (\ref{7.143}),
(\ref{requirements6.76}) and the fact that $\d_{n},N_{n}=K_{n}^{0+}$ and 
$K_{n}^7\bar\e_{n}\leq \bar\e_{n}^{(1/2)+}$, 
we get
\be \|Y_{n}^{Wh}\|_{C^3}\leq \bar \e_{n}^{1/2}.\label{7.167}
\ee

If we define 
$g_{n,n+1}=f_{Y_{n}^{Wh}}^{-1}\in \ti{\rm Symp}_{ex,\s}(W_{h_{n+1},U_{n+1}})$ and for $m\leq n$, $g_{m,n=1}=g_{m,n}\circ g_{n,n+1}$ we have from (\ref{ee4.85}) and (\ref{2.40})
$$\|g_{m,n+1}-id\|_{C^1}\leq C( \|g_{m,n}-id\|_{C^1}+\|g_{n,n+1}-id\|_{C^1})\leq C \sum_{l=1}^n\bar\e_{l}\lesssim \bar\e_{m}$$
which is  $(\ref{7.150})_{n+1}$ and implies $(\ref{8.112bis})_{n+1}$.

Note that $f_{Y_{n}^{Wh}}^{-1}=f_{Y_{n}}^{-1}$ on $W_{h_{n+1},U_{n+1}}$ and  (\ref{eqYnante}) shows that $(\ref{8.112})_{n+1}$ and $(\ref{ee7.124})_{n+1}$ are satisfied.

\medskip 
We now check  that (\ref{8.114}) holds for $n+1$; from (\ref{eqestKAM}) it is enough to verify   
\be K_{n}^{\bar a_{0}+2}(\bar\e_{n}^{\hs 2}+e^{-\d_{n}N_{n}/2}\bar\e_{n})<\bar\e_{n+1}\label{eqbarepsilonk+1}\ee
or equivalently since $e^{-\d_{n}N_{n}/2}=\bar \e_{n}$,  $K_{n}={\bar \e_{n}}^{-\frac{1}{2(\bar a_{0}+2)}}$,
$$2\bar\e_{n}^{2-(1/2)}\leq \bar \e_{n+1}
$$ which is clearly satisfied since $3/2>4/3$,  {\it cf.} (\ref{requirements6.76}).

\medskip

\end{proof}

\subsection{Localization of the holes}
We can  localize the holes of the domains $U_{n}$:
\begin{prop}[Localization of the holes]\label{prop:1.enoncebis}
For each $1\leq m<n$, 
one has 
\be  \|\pa\Omega_{n}^{}-\pa\Omega_{m}^{}\|_{C^2}\lesssim \bar \e_{n}^{1/2}\label{7.114bis}
\ee
and for some sets $E_{i}\subset \{(k,l)\in\Z^2,\ 0<k<N_{i},\ 0\leq |l|\leq N_{i}\}$ ($1\leq i\leq n-1$) one can write $U_{n}$ as
\be \bD(0,\rho_{n})\setminus\bigcup_{i=1}^{n-1}\bigcup_{(k,l)\in E_{i} }\bD(c_{l/k}^{(i)},s_{i,n-1}K_{i}^{-1}),\quad s_{i,n-1}=e^{\sum_{t=i}^{n-1}{\d_{t}}}\in [1,2] \label{formUn}\ee
 where  $\rho_{n}\geq e^{-1/5}\bar \rho$ and  $c^{(i)}_{l/k}$ is  on the real axis and is the unique solution of the equation $\omega_{i}^{}(c^{(i)}_{l/k}):=(2\pi)^{-1}\pa\Omega_{i}^{}(c^{(i)}_{l/k})=l/k$.
\end{prop} 
\begin{proof} 
Inequality  (\ref{7.114bis}) is  consequence of   (\ref{eee7.160}), (\ref{eee7.159}).
The expression (\ref{formUn}) comes from    (\ref{relUn+1Vn}). 

\end{proof}

We now give  a more detailed description of the structure  of $\cD(U_{n})$, the set of  holes of the domains $U_{n}$ appearing in  Proposition \ref{prop:1.enoncebis}, {\it cf.} (\ref{formUn}).

 \begin{lemma} \label{cor:6.5}With the notations of Proposition \ref{prop:1.enonce}-\ref{prop:1.enoncebis}:
 \begin{enumerate}
  \item\label{ni1}  For any $n_{1}\leq n_{2}$, $(k_{j},l_{j})\in E_{n_{j}}$, $j=1,2$ , 
 \be\begin{cases}&\textrm{if }\ l_{1}/k_{1}=l_{2}/k_{2}\ \textrm{then}\   |c^{(n_{1})}_{l_{1}/k_{1}}-c^{(n_{2})}_{l_{2}/k_{2}}|\lesssim \bar \e_{n_{1}}^{\hs 1/2}\\
 &\textrm{if }\ l_{1}/k_{1}\ne l_{2}/k_{2}\ \textrm{then}\  |c^{(n_{1})}_{l_{1}/k_{1}}-c^{(n_{2})}_{l_{2}/k_{2}}|\gtrsim N_{n_{2}}^{-2}.
 \end{cases}\label{8.131}
 \ee
 \item \label{ni2} Let    $n_{1}, n_{2}\in\N$, $0<\kappa_{{2}}<\kappa_{{1}}$ 
$$\kappa_{{1}}+\kappa_{{2}}\ll N_{\max(n_{1},n_{2})}^{-2},\qquad \bar \e_{\min(n_{1},n_{2})}^{1/2}\ll \kappa_{{1}}-\kappa_{{2}}.
$$
Then, two   disks $\bD(c^{({j})}_{l_{j}/k_{j}},\kappa_{{j}})$,  $(k_{j},l_{j})\in E_{n_{j}}$, $j=1,2$, are either    disjoint or $l_{1}/k_{1}=l_{2}/k_{2}$ and  $\bD(c^{(n_{2})}_{l_{2}/k_{2}},\kappa_{{2}})\subset \bD(c^{(n_{1})}_{l_{1}/k_{1}},\kappa_{{1}})$.
 \end{enumerate}
 \end{lemma}
\begin{proof}Item (\ref{ni1}) is due to  
(\ref{7.114bis}) and  the fact that if $l_{1}/k_{1}\ne l_{2}/k_{2}$
$$|(l_{1}/k_{1})-(l_{2}/k_{2})|\geq 1/(k_{1}k_{2})\geq N_{n_{2}}^{-2}.$$
Item (\ref{ni2}) is a consequence of Item (\ref{ni1}). Indeed, if $l_{1}/k_{1}\ne l_{2}/k_{2}$ then since $\kappa_{{1}}+\kappa_{{2}}\ll N_{n_{2}}^{-2}$  and  $|c^{(n_{1})}_{l_{1}/k_{1}}-c^{(n_{2})}_{l_{2}/k_{2}}|\gtrsim N_{n_{2}}^{-2}$ (we assume $n_{1}\leq n_{2}$), the disks $\bD(c_{l_{1}/k_{1}}^{(n_{1})},\kappa_{{1}})$ and $\bD(c_{l_{2}/k_{2}}^{(n_{2})},\kappa_{{2}})$ must have an empty intersection. On the other hand, if  $l_{1}/k_{1}= l_{2}/k_{2}$ then due to the fact that  $|c^{(n_{1})}_{l_{1}/k_{1}}-c^{(n_{2})}_{l_{2}/k_{2}}|\lesssim \bar \e_{n_{1}}^{1/2}$ the disk $\bD(c_{l_{1}/k_{1}}^{(n_{1})},\kappa_{{1}})$ contains $\bD(c_{l_{2}/k_{2}}^{(n_{2})},\kappa_{{2}})$ since $\bar \e_{n_{1}}^{1/2}+\kappa_{{2}}\ll \kappa_{{1}}$.
\end{proof}

\subsection{Whitney conjugation to an integrable model}
By applying Lemma \ref{lemma:2.3ee} one  sees that   $(F_{n},e^{-\d_{n}}W_{h_{n},U_{n}})$ and $(f_{F_{n}}, e^{-\d_{n}}W_{h_{n},U_{n}})$ have  $C^3$ real symmetric Whitney extensions $F_{n}^{Wh}\in \ti\cO_{\s}(e^{-\d_{n}} W_{h_{n},U_{n}})$, $f_{F_{n}}\in\ti {\rm Symp}_{ex,\s}^{}(W_{h_{n},U_{n}})$ such that (see the discussion leading to (\ref{eee7.159}) and inequality (\ref{ee4.85}))
$$\|F_{n}^{Wh}\|_{C^3}\lesssim \bar \e_{n}^{1/2},\qquad \|f_{F_{n}}-id\|_{C^1}\lesssim \bar e_{n}^{1/3}.
$$
We hence have
\be  \textrm{on}\  e^{-\d_{n}}W_{h_{n},U_{n}},\quad  g_{m,n}^{-1}\circ \Phi_{\Omega^{}_{m}}\circ f_{F^{Wh}_{m}}\circ g_{m,n}=  \Phi_{\Omega^{}_{n}}\circ f_{F^{Wh}_{n}}.\label{7.172}
\ee
We show in the next Proposition that shrinking a little bit the domain of validity of the preceding formula one can impose that $g_{m,n}$ leaves invariant the origin $O=\{r=0\}\cap M_{\R}$.
\begin{lemma}\label{prop:1.enoncebisbis}There exists $\ti g_{m,n}\in \ti {\rm Symp}_{ex,\s}^{}(W_{h_{n}/2,U_{n} \setminus \bD(0,K_{m}^{-1}}))$ that coincide with $g_{m,n}$ on $W_{h_{n}/2,\C \setminus \bD(0,K_{m}^{-1})}$ and 
\be \ti g_{m,n}(\{r=0\})=\{r=0\},\qquad \|\ti g_{m,n}-id\|_{C^1}\leq \bar \e_{m}^{1/4}.\ee
\end{lemma}
\begin{proof}Recall that $g_{m,n}=f_{Y_{m}^{Wh}}^{-1}\circ\cdots\circ f_{Y_{n-1}^{Wh}}^{-1}$ with $Y_{k}^{Wh}\in C^3\cap \cO_{\s}(e^{-(1/2)\d_{n}}W_{h_{k},V_{k}})$ satisfying (\ref{7.167}).  Let $\chi:\R\to [0,1]$ be a smooth function with support in $[-1,1]$ and equal to 1 on $[-1/2,1/2]$  and define the $C^3$ $\s$-symmetric function $\ti Y_{k}=(1-\chi((K_{m}r/2)^2))Y_{k}^{Wh}$. One has $\ti Y_{k}=Y_{k}^{Wh}$ on $W_{h,\C\setminus\bD(0,K_{m}/2)}$ and $\|\ti Y_{k}\|_{C^3}\lesssim K_{m}^3\|Y^{Wh}\|_{C^2}\leq \bar \e_{m}^{1/4}$ hence $f_{\ti Y_{k}}^{-1}$ coincide with $f_{Y_{k}^{Wh}}^{-1}$ on $W_{h,\C\setminus \bD(0,K_{m}^{-1})}$ and $\|f^{-1}_{\ti Y_{k}}-id\|_{C^1}\lesssim \bar \e_{k}^{1/4}$. Since $\ti Y_{k}$ is null on a neighborhood of $\{r=0\}$ the diffeomorphisms $f_{\ti Y_{k}}^{-1}$ fix $\{r=0\}$.
\end{proof}
Note that the sequence of diffeomorphisms $n\mapsto \ti g_{m,n}$ converges in $C^1$ to a $\s$-symmetric diffeomorphism $\ti g_{m,\infty}:\C\to\C$ fixing the origin and  that satisfies $\|\ti g_{m,\infty}-id\|_{C^1}\lesssim  \bar \e_{m}$. On the other hand, the sequence of diffeomorphisms $(f_{F_{n}^{Wh}})_{n}$ converges in $C^1$ to 0 and  from  (\ref{7.114bis}) the sequence of functions $(\Omega_{n}^{})_{n}$, $\Omega_{n}\in\ti\cO_{\s}(U_{n})$ converges in $C^2$ to some $\s$-symmetric  limit $\Omega_{\infty}^{}\in C_{\s}^2(\C)$, ;  hence from (\ref{7.172})
$$ \textrm{on}\  \bigcap_{n\geq m}e^{-\d_{n}}W_{h_{n}/2,U_{n}\setminus\bD(0,K_{m}^{-1})},\quad  \ti g_{m,\infty}^{-1}\circ \Phi_{\Omega^{}_{m}}\circ f_{F^{Wh}_{m}}\circ \ti g_{m,\infty}=  \Phi_{\Omega^{}_{\infty}}.
$$
Recall the notations of Section \ref{sec:kamcircles} and let 
$$L_{m}=\R\cap \bigcap_{n\geq m} e^{-\d_{n}}(U_{n}\setminus \bD(0,K_{m}^{-1})), \qquad  W_{L_{m}}=M_{\R}\cap \bigcap_{n\geq m}e^{-\d_{n}}W_{h_{n}/2,U_{n}\setminus\bD(0,K_{m}^{-1})}.$$
\begin{prop}\label{prop:7.5}For any $m\geq 1$ one  has
\begin{align}&  \textrm{on}\  W_{L},\quad  \ti g_{m,\infty}^{-1}\circ \Phi_{\Omega^{}_{m}}\circ f_{F^{}_{m}}\circ \ti g_{m,\infty}=  \Phi_{\Omega^{}_{\infty}}\label{7.172}\\
& \ti g_{m,\infty}(W_{L_{m}})\subset W_{\R\cap \bar U_{m}}, \label{7.173}\\
& \ti g_{m,n}(\{r=0\})=\{r=0\},\qquad \|\ti g_{m,n}-id\|_{C^1}\leq \bar \e_{m}^{1/4}\\
&{\rm Leb}_{{\R}}({(\R\cap e^{-2\d_{m}} U_{m})}\setminus {L_{m} }))\lesssim \bar\e_{m}^{\hs \frac{1}{2(\bar a_{0}+3)}}.\label{7.174}
\end{align}
\end{prop}
\begin{proof}
Let us prove (\ref{7.173})). Note that since $g_{m,n}$ and $\ti g_{m,n}$ coincide on $W_{h_{n}/2,\C\setminus\bD(0,K_{m}^{-1})}$ one has  from (\ref{ee7.124}) $\ti g_{m,n}(e^{-\d_{n}}W_{h_{n},U_{n}\setminus\bD(0,K_{m}^{-1}) } )\subset W_{h_{m},U_{m}}$ hence  since $\ti g_{m,n}$ is $\s$-symmetric,  $\ti g_{m,n}(W_{L_{m}})\subset W_{\R\cap U_{m}}$ and $\ti g_{m,\infty}(L_{m})\subset \overline{ W_{\R\cap U_{m}}}=W_{\R\cap \bar U_{m}}$. 

The conjugation relation (\ref{7.172}) comes from the fact that    $\Phi_{\Omega_{m}^{Wh}}\circ f_{F_{m}^{Wh}}$ coincide on $W_{\R\cap \bar U_{m}}$ with $\Phi_{\Omega_{m}^{}}\circ f_{F_{m}^{}}$.

For the proof of (\ref{7.174}) we first  observe that 
 from the expression  (\ref{formUn}), for each $n>m$ the set 
$e^{-\sum_{l=m}^{n} \d_{l}}U_{m}\setminus e^{-\d_{n}}U_{n}$ is a union of at most $2N_{n}^2$ disks of radii $\leq 2K_{n}^{-1}$
hence the Lebesgue measure of its intersection with $M_{\R}$ is $\leq 4N_{n}^2K_{n}^{-1}$. In consequence, the Lebesgue measure of $\R\cap e^{-\sum_{l=m}^{\infty}\d_{l}}U_{m}\setminus \bigcap_{n\geq m+1}e^{-\d_{n}}U_{n}$ is $\lesssim \sum_{n=m+1}^\infty N_{n}^2K_{n}^{-1}\leq \bar\e_{m}^{\hs \frac{1}{2(\bar a_{0}+3)}}$ hence
$${\rm Leb}_{M_{\R}}(e^{-2\d_{m}}U_{m}\setminus \bigcap_{n\geq m} e^{-\d_{n}} U_{n})\lesssim \bar\e_{m}^{\hs \frac{1}{2(\bar a_{0}+3)}}$$
and since $L\supset (\bigcap_{n\geq m} e^{-\d_{n}} U_{n})\setminus e^{\d_{m}}\bD(0,K_{m}^{-1})$ we get that ${\rm Leb}_{M_{\R}}((\R\cap e^{-2\d_{m}} U_{m})\setminus L_{m} ))\leq \bar\e_{m}^{\hs \frac{1}{2(\bar a_{0}+3)}}+e^{\d_{m}}K_{m}^{-1}$; (\ref{7.174}) follows from this inequality.
\end{proof}

\bigskip
\begin{rem}\label{rem:7.1}If $U$ is a holed domain, Propositions \ref{prop:1.enonce}, \ref{prop:1.enoncebis}, \ref{prop:7.5} as well as their proofs, extend without any change to the situation where  $F\in\cO_{\s}(e^{h}W_{h,U})$ and  $\Omega\in\ti \cO_{\s}(e^{h}U)$ satisfies   the twist condition (\ref{7.119})-(\ref{condUnew}) and 
if the following  smallness assumption  on $F$ holds
\be  \|F\|_{e^{h}W_{h,U}} \leq \ud(W_{h,U})^{\bar a_{2}}.\label{smallcond}
\ee
\end{rem}

\section{Hamilton-Jacobi Normal Form and the Extension Property}\label{sec:HJBNF}

Our aim in this section is to provide a useful approximate Normal Form (that we call the Hamilton-Jacobi Normal Form)  in a neighborhood of a {\it $q$-resonant circle} $\{r=c\}$:
for some  $(p,q)\in\Z\times \N^*$, $p\wedge q=1$
$\omega(c)=\frac{p}{q}$.

\medskip Let  $0<\hat \rho<h/2<1/20$, $c\in\R$, $(p,q)\in\Z\times \N^*$, $p\wedge q=1$,  $\Omega\in\ti \cO_{\s}(\bD(c,6\hat\rho))$, $F\in \cO_{\s}(W_{h,\bD(c,6\hat \rho)})$ such that 
\begin{align}& \forall \ r\in\R,\  A^{-1}\leq  (2\pi)^{-1}\pa^2\Omega(r)\leq A,\quad \textrm{and}\ \ \|(2\pi)^{-1}D^3\Omega\|_{\C}\leq B.\label{eq:11.257bis}\\
&\bar \e:=  \|F\|_{W_{h,\bD(c,6\hat  \rho)}}\leq  (6\hat \rho)^{\hs \bar a_{3}},\label{8.177}\\
&\omega(c):=(2\pi)^{-1}\pa^2\Omega(c)=\frac{p}{q}\label{eq:11.257,5bis}\\
&(6\hat \rho)^{1/8}<(Aq)^{-1}<h/10,\qquad 6\hat  \rho<|c|/4.\label{8.179}
\end{align}

The purpose of this section is to prove the following result:
\begin{prop}[Hamilton-Jacobi Normal Form]\label{proppropHJ} There exist  a disk $\check D:=\bD(\check c,\check \rho)\subset\hat D$  with 
 $ \check\rho\leq \bar \e^{1/33}$
 and 
$$\Omega^{HJ}_{\hat D}\in \ti \cO_{\s}(\hat D\setminus \check D)),\quad  F^{HJ}_{\hat D}\in \cO_{\s}(W_{h/9,(\hat D\setminus \check D)}),\quad g^{HJ}_{\hat D}\in\ti {\rm Symp}_{\s}((W_{h/9,(\hat D\setminus \check D)})$$  such that  
 \begin{align}&\Omega^{HJ}_{\hat D}\ \textrm{satisfies \ a}\  (2A,2A)-\textrm{twist\ condition}\\
 & W_{h/9, (\hat D\setminus \check D)},\qquad   (g^{HJ}_{\hat D})^{-1}\circ \Phi_{\Omega}\circ f_{F}\circ g^{HJ}_{\hat D}=\Phi_{\Omega^{HJ}_{\hat D}}\circ f_{F^{HJ}_{\hat D}}\\
 &\|g^{HJ}_{\hat D}-id\|_{C^1} \lesssim q \hat \e^{\hs 1/8}\label{estgHJ}\\
&\|F^{HJ}_{\hat D}\|_{W_{h/9, (\hat D\setminus \check D)}}\lesssim \exp(-1/(6q\hat \rho)^{1/4})\bar \e.
\end{align}
Moreover, one has the following:

\smallskip\noindent {\bf Extension property:} $(\Omega^{HJ}_{D},\hat D,\check D)$ satisfes the following Extension Principle: If there exists a holomorphic function $\Xi\in\cO(\hat D)$ such that 
$$\|\Omega^{HJ}_{D}-\Xi\|_{(4/5)\hat D\setminus (1/5)\hat D}\lesssim \nu
$$
then
$\check \rho\lesssim \nu^{1/200}.$
\end{prop}

\begin{rem}
From Lemma \ref{lemma:4.5} and Remark \ref{rem:4.2} we just have to prove the Proposition in the (AA)-setting. This is the setting in which we shall work in all this Section.
\end{rem}

The proof of the first part of Proposition \ref{proppropHJ} is done in Subsection \ref{sec:8.4}  and that of the second part (Extension Principle), based on Proposition \ref{prop:cosofholom},  in Subsection  \ref{subsec10.5}.

From now on we define
$$\bar \rho=6\hat \rho.$$

\subsection{Putting the system into Resonant Normal Form}

From Proposition \ref{prop:11.5} on the existence of approximate $q$-Resonant Normal Form, we know that there exist
 $\bar\Omega\in\ti \cO_{\s}(\bD(c,e^{-1/q}\bar\rho))$, $g_{RNF}\in \ti {\rm Symp}_{ex,\s}(e^{-1/q}W_{h,\bD(c,\bar \rho)})$,  $\bar F^{res},F^{cor}\in\cO_{\s}(e^{-1/q}W_{h,\bD(c,\bar \rho)})$ such that 
$\bar F^{res}$ is $2\pi/q$-periodic, $\cM_{0}(\bar F^{res})=0,$ and 
\be\begin{cases}&e^{-1/q}W_{h,\bD(c,\bar \rho)},\quad  g_{RNF}^{-1}\circ \Phi_{\Omega}\circ  f_{F}\circ g_{RNF}= \Phi_{2\pi(p/q)r}\circ \Phi_{\bar\Omega}\circ  f_{\bar F^{res}}\circ f_{ F^{cor}}\\
&\bar F^{res}\ \textrm{ is}\  2\pi/q-\textrm{periodic},\qquad  \cM_{0}(\bar F^{res})=0,
\end{cases}
\label{10.162bis}
\ee
with
\be\begin{cases}&\|\bar \Omega-(\Omega-2\pi(p/q)r)\|_{\bD(c,e^{-1/q}\bar\rho)}\lesssim \|F\|_{W_{h,\bD(c,\bar\rho)}}\\
&\|\bar F^{res}\|_{e^{-1/q}W_{h,\bD(c,\bar\rho)}}\lesssim \|F\|_{W_{h,\bD(c,\bar\rho)}}\\
& \|F^{cor}\|_{e^{-1/q}W_{h,\bD(c,\bar \rho)}}\lesssim  \exp(-\bar \rho^{\hs-1/4}) \|F\|_{W_{h,\bD(c,\bar\rho)}}\\
&\|g_{RNF}-id\|_{C^1}\leq (q\bar \rho^{\hs-1})^5\|F\|_{h,\bD(c,\bar\rho)}
\end{cases}\label{barFresbis}
\ee
Inequalities (\ref{barFresbis}) and the fact that $\Omega$ satisfies an $(A,B)$-twist condition on $\bD(0,\bar \rho)$  show that there exists a unique $\bar c\in\R$ such that 
$$ \pa\bar\Omega(\bar c)=0,\qquad |\bar c-c|\lesssim \bar\e.
$$
\subsection{Coverings}\label{sec:8.2}
We denote $\R_{h}=\R+\sqrt{-1}[-h,h]$ and by $j_{q}$ the $q$-covering 
\begin{align}j_{q}\:\ (\C/(2\pi \Z))\times \C&\to (\C/(2\pi/q)\Z)\times\C\notag\\
(\th+2\pi\Z,r)&\mapsto (\th+(2\pi/q)\Z,r).\label{defjq}\end{align} Since the function $\bar F^{res}:(\th,r):(\R_{h-2/q}/(2\pi)\bZ)\times \bD(\bar c,e^{-2/q}\bar \rho)\to\C$ is invariant by $(\th,r)\mapsto (\th+2\pi/q,r)$ one can push it down to a function 
$$\bar F^{res}_{j_{q}}:(\R_{h-2/q}/(2\pi/q)\Z)\times \bD(\bar c,e^{-2/q}\bar \rho)\to\C,\qquad \bar F^{res}_{j_{q}}\circ j_{q}= \bar F^{res}.$$  Let 
\begin{align}\L_{q}\ :\ (\C/(2\pi/q)\Z)\times \C&\to \C/(2\pi)\Z\times \C\label{defLambda}\\
(\th,r)&\mapsto (q\th,q(r-\bar c))\notag\end{align} 
and define  $\ti F^{res}:(\R_{qh-2}/(2\pi)\bZ)\times \bD(0,e^{-2/q}q\bar \rho)\to \C$ by 
$$\ti F^{res}=q^2\bar F^{res}_{j_{q}}\circ \L_{q}^{-1};$$
 for all $(\ti\th,\ti r)\in \T_{qh-2}\times \bD(0,qe^{-2/q}\bar \rho)$ and $(\th,r)\in \T_{h-2/q}\times \bD(\bar c, e^{-2/q}\bar \rho)$ such that  $\ti \th=q\th,\ti r=q(r-\bar c)$ one has
\be \ti F^{res}(\ti \th,\ti r)=q^{2}\bar F^{res}(\th,\bar c+r).\label{10.163}\ee
Let  $f_{\ti F^{res}}$ be the (exact) symplectic mapping (for the symplectic form $d\ti\th\wedge d\ti r$) defined by (\ref{4.43}): if $(\ti \ph,\ti R)=\Lambda (\ph,R)$, $(\ti \th,\ti r)=\Lambda(\th,r)$
$$(\ti \ph,\ti R)=f_{\ti F^{res}}(\ti \th,\ti r)\ \iff\ (\ph,R)=f_{\bar F^{res}_{j_{q}}}( \th, r).
$$
 If we set 
 \begin{align} \ti \Omega(r)&:=q^2\biggl(\bar \Omega(\bar c+(r/q))-2\pi(p/q)(r/q)\biggr)\notag\\
 &=(1/2)\pa^2 \bar \Omega(\bar c) r^2+O(r^3)\notag\\
 &=\varpi r^2+r^3b(r)
 \label{tiOmega}\end{align}
  we have
 \be \Lambda_{q}\circ \Phi_{\bar \Omega}\circ f_{\bar F_{j_{q}}^{res}}\circ \Lambda_{q}^{-1}=\Phi_{\ti\Omega}\circ f_{\ti F^{res}}.\label{tiFres}
 \ee
 Note that since $\Omega$ satisfies an $(A,B)$-twist condition, one has from the first equation of (\ref{barFresbis}), the estimate
 \be \forall \ r\in \bD(0,e^{-1/10}\bar \rho),\quad \pa^2\ti \Omega(r)\asymp 1, \qquad \|\ti\Omega\|_{C^3(\bD(0,e^{-1/10}q\bar \rho))}\lesssim 1.\label{ee8.183}
 \ee

\subsection{Approximation by a Hamiltonian flow}
The following proposition says that up to some very  good approximation $\Phi_{\ti\Omega}\circ f_{\ti F^{res}}$ can be seen as the time-1 map of a hamiltonian vector field in the plane.
\begin{prop}\label{prop:10.6}There exists $\ti F^{vf}, \ti F^{per}\in\cO_{\s}(\T_{e^{-2/q}qh/2}\times\bD(0,e^{-2/q}q\bar \rho/2))$, such that on $\T_{e^{-2/q}qh/2}\times\bD(0,e^{-2/q}q\bar \rho/2)$ one has 
\begin{align}&\Phi_{\ti\Omega}\circ  f_{\ti F^{res}}=\Phi_{\ti\Omega+\ti F^{per}}\circ f_{\ti F^{vf}}\label{10.200}\\
& \ti F^{per}=\ti F^{res}+O(\bar\rho^{\hs 1/4} \|\ti F^{res}\|_{\T_{e^{-2/q}qh}\times \bD(0,e^{-2/q}q\bar\rho)})=O(q^2\|F\|_{h,\bD(0,\bar\rho)})\label{11.193}\\
& \|\ti F^{vf}\|_{e^{-2/q}qh/2,\bD(0,e^{-2/q}q\bar \rho/2)}\lesssim \exp(-1/(q\bar \rho)^{1/4})\|F\|_{h,\bD(0,\bar\rho)}.\label{estFvf}
\end{align}
\end{prop}
\begin{proof} This is a consequence of 
(\ref{barFresbis}),  (\ref{10.163}) and  Proposition \ref{lemma:vf2} applied to $\Phi_{\ti\Omega}\circ  f_{\ti F^{res}}$ (since by (\ref{barFresbis}), (\ref{10.163}), condition (\ref{condSG}) is satisfied).
\end{proof}

\bigskip Let $F(\th,r)=\sum_{i=0}^2f_{i}(\th)r^{i}+r^3\ti f(\th,r)$ and  define ({\it cf.} (\ref{tiOmega}))
\begin{align}\ti \Pi(\th,r)&=\ti\Omega(r)+\ti F^{per}(\th,r)\label{e11.212}\\
&=:\varpi r^2+f_{0}(\th)+ f_{1}(\th)r+ f_{2}(\th)r^2+r^3(b(r)+ \ti f(\th,r))\label{e11.212bis}\\
&\begin{multlined}=(\varpi+ f_{2}(\th))\biggl(r+\frac{1}{2}\frac{f_{1}(\th)}{\varpi+ f_{2}(\th)}\biggr)^2-\frac{1}{4}\frac{f_{1}(\th)^2}{\varpi+ f_{2}(\th)} \\
+f_{0}(\th)+r^3(b(r)+\ti f(\th,r))\end{multlined}\label{10.203}
\end{align}
where
 \be \max_{\T_{e^{-2/q}qh/2}\times \bD(0,e^{-2/q}q\bar \rho/2)}(| f_{0}|, | f_{1}|,|f_{2}| , |\ti f|)\lesssim (q\bar \rho)^{-3}q^2\bar\e. \label{eq:another}
\ee

\comm{
\begin{rem}\label{rem:7.1}We notice that 
\be \ti F_{per}=T_{\infty}^{res}F+\cO_{2}(F)
\ee
\end{rem}
}

\subsection{From $\ti \Pi$ to $\bar \Pi$}

We assume in the rest of this section that $\varpi>0$ and we set
\be \rho_{q}=q\bar \rho/3.\label{eq:rhoq}
\ee

The next lemma provides a more convenient expression for the function, viewed as a hamiltonian,  $\ti\Pi=\ti\Omega+\ti F^{per}$ defined in (\ref{e11.212}).
\begin{lemma}\label{lem:10.7}There exists a (not exact) symplectic change of coordinates $G\in {\rm Symp}^\cO_{\s}(\T_{qh/3}\times \bD(0,\rho_{q}))$ of the form $G(\th,r)=(\th,r-e_{0}(\th))$  and  $\bar \Pi\in\cO(\T_{qh/3}\times \bD(0,e^{-1/10}\rho_{q}))$ such that 
\be\bar \Pi(\th,r):=\ti \Pi\circ G^{-1}(\th,r)=\varpi(\th)(r^2-e_{1}(\th)+r^3f(\th,r))\label{10.208}\ee
with  $\varpi, e_{0},e_{1}\in \cO_{\s}(\T_{qh/3})$, $f\in \cO_{\s}(\T_{qh/3}\times \bD(0,\rho_{q}))$,
\be \|\varpi(\cdot)-\varpi\|_{qh/3}\lesssim q\bar \rho^{\hs -2}\bar \e,\qquad \max(\|e_{0}\|_{qh/3},\|e_{1}\|_{qh/3}),\lesssim q\bar \rho^{\hs -1} \bar \e,\qquad \|f\|_{qh/3,\rho_{q}}\lesssim 1.\label{deftiPi}
\ee
\end{lemma}

\begin{proof} See the Appendix \ref{lemma:tiPibarPi}.
\end{proof}

\begin{rem} The previous lemma and (\ref{10.203}) show that 
$$\varpi(\th)=\varpi +f_{2}(\th)+O(\rho_{q}^3\bar\e)$$
and 
$$e_{0}(\th)=-\frac{1}{2}\frac{f_{1}(\th)}{\varpi+ f_{2}(\th)}+O(\rho_{q}^3\bar\e),\quad e_{1}(\th)=-\frac{1}{4}\frac{f_{1}(\th)^2}{(\varpi+ f_{2}(\th))^2}+\frac{f_{0}(\th)}{\varpi+f_{2}(\th)}+O(\rho_{q}^3\bar \e).$$
\end{rem}
\begin{rem}\label{rem:8.3}The change of coordinates $G:(\th,r)\mapsto (\th,r+e_{0}(\th))$ is a symplectic diffeomorphism (but not exact symplectic).
\end{rem}
\begin{rem}Since $\ti \Pi$ is defined up to an additive constant (this will not change the value of $e_{0}$), we can assume that 
$$\int_{\T}\frac{\ti \Pi(\th,e_{0}(\th))}{\varpi(\th)^{1/2}}\frac{d\th}{2\pi}=0$$
which is equivalent to the following condition that we will assume to hold from now on
\be \int_{\T}\varpi(\th)^{1/2}e_{1}(\th)\frac{d\th}{2\pi}=0.\label{condone1}
\ee
\end{rem}

\subsection{Hamilton-Jacobi Normal Form for $\bar \Pi$}\label{subsec:7.3}
The symplectic diffeomorphism $\Phi_{\bar \Pi}$ is the time-1 map of a hamiltonian defined in the cylinder, and as such, it is integrable in the Hamilton-Jacobi sense: the level lines of the hamiltonian foliate the cylinder  and naturally provide invariant curves for the hamiltonian flow. On some open sets\footnote{These are cylindrical domains outside the ``eyes'' defined by separatrices (think of a pendulum).} it is possible to conjugate $\Phi_{\bar \Pi}$ to a hamiltonian depending only on the action variable: this is   the Hamilton-Jacobi Normal Form; see Proposition \ref{prop:8.10}. The purpose  of this  Subsection is to quantify this fact.

\medskip Recall the expression for $\bar \Pi$
$$\bar \Pi(\th,r)=\varpi(\th)(r^2-e_{1}(\th)+r^3f(\th,r)).$$
Let $0\leq s\leq h/3$. We denote 
\be \e_{1}:=\|e_{1}\|_{C^0(\T)}\lesssim q\bar \rho^{\hs -1} \bar \e,\qquad  \e_{1,s}=\e_{1}(s)=\|e_{1}\|_{qsh/3}, \label{def:epsilon1}
\ee
and  for $L\gg 1$ we introduce
\be  \l_{0,L}:=L\e_{1}^{1/2},\qquad  \l_{s,L}=\l(s,L)= L\e_{1,s}^{1/2}, \label{lambda-L}\ee
with the requirement
\be \l_{s,L}<q\bar\rho/6=\rho_{q}/2 \quad \textrm{or \ equivalently }\    1\ll L \lesssim q \bar\rho \e_{1,s}^{-1/2}. \label{lambda-L'}
\ee
We notice that $0<\l_{s,L} \lesssim \rho_{q}$ and that from the Three Circles Theorem
\be \e_{1}(0)\leq  \e_{1}(s)\leq \e_{1}(0)^{1-s}\e_{1}(1)^{s} \label{10.215}
\ee
hence
\be L\e_{1}^{1/2}\leq \l_{s,L}\leq L\e_{1}^{(1-s)/2}.\label{e8.199}
\ee

\bigskip\noindent{\bf Notation:} For $0<a_{1}<a_{2}$ and $z\in\C$ we denote by $\bA(z;a_{1},a_{2})$ the annulus centered at $z$ with inner and outer radii  of sizes respectively $a_{1}$ and $a_{2}$. When $z=0$ we simply denote this annulus by $\bA(a_{1},a_{2})$

Before giving the Hamilton-Jacobi Normal Form of $\bar \Pi$ we need two lemmas.

\begin{lemma}\label{lemma:10.2}There exists a holomorphic function $g$ defined  on $${\rm Dom}(g):=\bigcup_{0\leq s\leq 1}(\T_{qsh/3}\times\bA(\l_{s,L},\rho_{q}))$$  such that for every $(\th,z)\in{\rm Dom}(g)$ one has
\be \bar \Pi(\th,g(\th,z))=z^2.\label{7.192}
\ee Moreover, there exists $\mathring{g}\in \cO({\rm Dom}(g))$ such that on ${\rm Dom}(g)$ one has 
\be g(\th,z)=\varpi(\th)^{-1/2}z(1+\mathring{g}(\th,z)),\qquad \|\mathring{g}\|_{{\rm Dom(g)}}\lesssim L^{-2}.\label{12.214}
\ee
\end{lemma}
\begin{proof}
See the Appendix Section \ref{sec:G2}.
\end{proof}

Since $\T\times \bA(\l_{0},\rho_{q})\subset {\rm Dom}(g)$ we can  define the function $\Gamma\in \cO(\bA(\l_{0},\rho_{q}))$ by $\Gamma:\bA(\l_{0,L},\rho_{q})\to\C$
\be \Gamma(u)=(2\pi)^{-1}\int_{0}^{2\pi}g(\ph,u)d\ph.\label{defGamma}
\ee
Using (\ref{12.214}) we see that $\Gamma$ can be written 
$$\Gamma(u)=\gamma u(1+\mathring{\Gamma}(u)),\qquad \gamma:=(2\pi)^{-1}\int_{0}^{2\pi}\varpi(\th)^{-1/2}d\th,\qquad \|\mathring{\Gamma}\|_{\bA(\l_{s,L},\rho_{q})}\lesssim L^{-2}. 
$$

\begin{lemma}\label{lemma:7.6}There exists a solution $H\in\cO(\bA(2\l_{s,L},\rho_{q}/2))$ of the equation
\be \Gamma(H(z))=z\label{eq:8.197}
\ee
Moreover it can be written
\be H(z)=\gamma^{-1}z(1+\mathring{H}(z)),\qquad \|\mathring{H}\|_{\bA(2\l_{s,L},(1/2)\rho_{q})}\leq L^{-2}. \label{10.223}
\ee
\end{lemma}
\begin{proof}
See the Appendix Section \ref{sec:G3}.
\end{proof}

We now apply the preceding results with 
$$s=1/12.$$

\begin{prop}\label{prop:8.10}[Hamilton-Jacobi]
There exists an exact  symplectic change of coordinates $\bar W\in \ti {\rm Symp}_{ex,\s}(\T_{qh/25}\times \bA(2\e_{1}^{1/32},\rho_{q}/2))$ such that 
\be \bar W^{-1}\circ \Phi_{\bar\Pi}\circ \bar W=\Phi_{H^2}\label{10.188}
\ee
\be  \|\bar W-id\|_{C^1}\lesssim q\e_{1}^{1/4}\label{estbarW}
\ee
\end{prop}
\begin{proof} 
Let $H$ be the function defined by the previous lemma (with $s=1/16$) and define for $z\in\bA(2\l_{1/8,L},\rho_{q}/2)$ and  $\th\in J_{qh/48}:= [-4\pi,4\pi]+i[-qh/48,qh/48]$
\be S(\th,z)=\int_{[0,\th]} g(\ph,H(z))d\ph.\label{eq:8.199}
\ee
We notice that by Cauchy Formula, (\ref{defGamma}) and (\ref{eq:8.197}) 
\begin{align*} S(\th+2\pi,z)-S(\th,z)&=\int_{[\th,\th+2\pi]}g(\ph,H(z))d\ph\\
&=\int_{0}^{2\pi}g(\ph,H(z))d\ph\\
&=2\pi\Gamma(H(z))\\
&=2\pi z
\end{align*}
hence
$$\Sigma:(\th,z)\mapsto S(\th,z)-\th z$$ defines a holomorphic function on $\T_{qh/6}\times \bA(2\l_{1/2,L},\rho_{q}/2)$.
Moreover, from (\ref{eq:8.199}), (\ref{12.214}) and (\ref{10.223}) one can write 
\begin{align*} S(\th,z)&=\int_{0}^\th g(\ph,H(z))d\ph\\
&=\int_{0}^\th \varpi(\ph)^{-1/2}H(z)(1+\mathring{g}(\ph,H(z)))d\ph\\
&=\gamma \th H(z)+\int_{0}^\th \varpi(\ph)^{-1/2}H(z)\mathring{g}(\ph,H(z))d\ph\\
&=\th z(1+\mathring{H}(z))+\int_{0}^\th \varpi(\ph)^{-1/2}H(z)\mathring{g}(\ph,H(z))d\ph
\end{align*}
and we see that 
$$\|\Sigma\|_{\T_{qh/6}\times \bA(2\l_{1/2,L},\rho_{q}/2)}\lesssim L^{-2}(1+qh/6).\label{}
$$

Define $U_{L,\d}=\T_{qh/6-\d}\times \bA(2L\e_{1}^{1/4}+\d,\rho_{q}/2-\d)$ and note that by  (\ref{e8.199}) one has  $\l_{1/2,L}\leq L\e_{1}^{1/4}$ so that $U_{L,0}\subset \T_{qh/6}\times \bA(2\l_{1/2,L},\rho_{q}/2)$ and 
$$\|\Sigma\|_{U_{L,0}}\lesssim qL^{-2}.
$$
By Cauchy estimates 
\be \|\Sigma\|_{C^2(U_{L,\d})}\lesssim q(\d L)^{-2}.\label{10.233bis}
\ee

Let us choose, $\d=\e_{1}^{1/16}$, $L=\e_{1}^{-7/32}$. We then have  $L\e_{1}^{1/4}=\e_{1}^{(-7/32)+(8/32)}=\e_{1}^{1/32}$,
$L^{-2}\d^{-2}=\e_{1}^{(7/16)-(2/16)}=\e_{1}^{5/16}$, $L^{-2}\d^{-3}=\e_{1}^{(7/16)-(3/16)}=\e_{1}^{1/4}$ hence 
\be \|\Sigma\|_{C^2(\T_{qh/6}\times \bA(2\e_{1}^{1/32},\rho_{q}/2 ))}\lesssim q\e_{1}^{1/4}.\label{10.233bis}
\ee

Using  Lemma \ref{lemma:2.3ee} and  Lemma \ref{domaindef}, we see that $(\Sigma, \T_{qh/6}\times \bA(2\e_{1}^{1/32},\rho_{q}/2 ))$ has a $C^2$, $\s$-symmetric Whitney extension $\Sigma^{Wh}$ such that    
\be \bar W=f_{\Sigma^{Wh}}^{-1}\quad  \bar W^{-1}=f_{\Sigma^{Wh}} \in \ti  {\rm Symp}(\T_{qh/7}\times \bA(3\e_{1}^{1/32},\rho_{q}/3))\label{10.233bisbis}
\ee
 and ($\|\bar W-id\|_{C^1}\lesssim \e_{1}^{-4/32}\e_{1}^{1/4}$) 
\be \|\bar W-id\|_{C^1}\lesssim q\e_{1}^{1/8}.\label{8.214}
\ee

On the other hand taking the derivative of (\ref{eq:8.199}) we have 
\be \pa_{\th}S(\th,z)=g(\th,H(z))\label{eq:8.201}
\ee
and so $S$ is a solution of the Hamilton-Jacobi equation
\begin{align}\bar \Pi(\th,\frac{\pa S}{\pa \th}(\th,z))&=\bar\Pi(\th,g(\th,H(z)))\label{10.224}\\
&=H^2(z)\qquad (\textrm{by}\ (\ref{7.192})).
\end{align}

Hence,  the exact symplectic  change of variable $\bar W=f_{\Sigma}^{-1}$
\be
\bar W^{\ -1}=f_{\Sigma}:(\th,w)\mapsto (\ph,z)\iff 
\begin{cases}w&=\frac{\pa S}{\pa \th}=w+\pa_{\th}\Sigma(\th,z)\\
\ph&=\frac{\pa S}{\pa z}=\th+\pa_{z}\Sigma(\th,z)
\end{cases}\label{9.256}
\ee
conjugates $\Phi_{\bar \Pi(\th,w)}$ to $\Phi_{H(z)^2}$ since from (\ref{eq:4.4})
$$\bar\Pi\circ \bar W=H^2\qquad \iff \qquad \bar W^{\ -1}\circ \Phi_{\bar\Pi}\circ \bar W=\Phi_{H^2}.
$$
This concludes the proof.
\end{proof}

\subsection{Consequences on $\Phi_{\ti \Omega}\circ f_{\ti F}$}
Let $G:(\th,r)\mapsto (\th,r+e_{0}(\th))$ be the diffeomorphism introduced in Lemma \ref{lem:10.7} and 
\be \ti W=G\circ \bar W.\label{whoistiW}\ee
We notice that  $\ti W\in \ti  {\rm Symp}_{\s}(\T_{qh/7}\times \bA(3\e_{1}^{1/32},\rho_{q}/3))$  and that its image contains  $G(\T_{qh/7}\times \bA(3\e_{1}^{1/32},\rho_{q}/3))$ (see (\ref{10.233bisbis}));  from  (\ref{8.214}) and (\ref{deftiPi}) we have
\be \|\ti  W-id\|_{C^1}\lesssim q\e_{1}^{1/8}. \label{10.233}
\ee
\begin{cor} \label{cor:10.11}One has 
\be \ti W^{-1}\circ \Phi_{\ti \Omega}\circ f_{\ti F^{res}}\circ  \ti  W=\Phi_{H^2}\circ f_{\hat F^{vf}} \label{10.257}
\ee
with 
\be \|\hat F^{vf}\|_{\T_{qh/8}\times \bA(4\e_{1}^{1/32},\rho_{q}/4))}\lesssim \exp(-1/(q\bar \rho)^{1/4})\|F\|_{h,\bD(0,\bar\rho)}
\label{e10258}
\ee
\end{cor}
\begin{proof}
Recall that from (\ref{10.200}) and the definition of $\ti \Pi$ (\ref{e11.212})) 
\begin{align*}\Phi_{\ti\Omega}\circ  f_{\ti F^{res}}&=\Phi_{\ti\Omega+\ti F^{per}}\circ f_{\ti F^{vf}}\\
&=\Phi_{\ti\Pi}\circ f_{\ti F^{vf}}.
\end{align*}
By Lemma \ref{lem:10.7} and Proposition \ref{prop:8.10} 
$$G^{-1}\circ \Phi_{\ti\Pi}\circ G=\Phi_{\bar \Pi},\qquad \bar W^{-1}\circ \Phi_{\bar\Pi}\circ \bar W=\Phi_{H^2}
$$
hence
$$\ti W^{-1}\circ \Phi_{\ti \Pi}\circ \ti W=\Phi_{H^2}
$$
and so 
$$\ti W^{-1}\circ \Phi_{\ti \Pi}\circ f_{\ti F^{vf}}\circ \ti W=\Phi_{H^2}\circ f_{\hat F^{vf}},\qquad f_{\hat F^{vf}}= \ti W^{-1}\circ f_{\ti F^{vf}}\circ \ti W$$
which is (\ref{10.257}).

The estimate on $\hat F^{vf}$ comes from (\ref{estFvf}) and (\ref{10.233}).
\end{proof}
\subsection{Proof of Proposition \ref{proppropHJ}: Existence of  Hamilton-Jacobi Normal Form}\label{sec:8.4} 
Let $\ti W$ be the  diffeomorphism constructed in  Corollary \ref{cor:10.11}.  The map $\Lambda_{q}$  (defined in (\ref{defLambda})) sends  $((\R+i]-h/8,h/8[)/(2\pi/q)\Z)\times \bA(\bar c;4q^{-1}\e_{1}^{1/32},\bar \rho/4)$ to  $\T_{qh/8}\times \bA(4\e_{1}^{1/32},q\bar \rho/4)$. 
From (\ref{tiFres}), (\ref{10.257}) one has
$$ \ti W^{-1}\circ \L_{q}\circ \Phi_{\bar\Omega}\circ f_{\bar F^{res}_{j_{q}}}\circ \L_{q}^{-1}\circ  \ti W=\Phi_{H^2}\circ f_{\hat F^{vf}}
$$
hence
\begin{multline*} (\L_{q}^{-1}\circ \ti W^{-1}\circ \L_{q})\circ \Phi_{\bar\Omega}\circ f_{\bar F^{res}_{j_{q}}}\circ (\L_{q}^{-1}\circ  \ti W\circ \L_{q})=\\(\L_{q}^{-1}\Phi_{H^2}\circ \L_{q})\circ(\L_{q}^{-1}\circ  f_{\hat F^{vf}}\circ \L_{q}).
\end{multline*}
Let $W$, $\Phi_{\mathring{\Omega}^{HJ}}$ and $f_{\mathring{F}^{vf}}$
be lifts by $j_{q}$ of $\Lambda_{q}^{-1}\circ \ti W\circ \Lambda_{q}$, $\Lambda_{q}^{-1}\circ \Phi_{H^2}\circ \Lambda_{q}$, $\Lambda_{q}^{-1}\circ f_{\hat F^{vf}}\circ \Lambda_{q}$.
 Since  $\Phi_{\bar\Omega}\circ f_{\bar F^{res}}$ is a lift by $j_{q}$ of  $\Phi_{\bar\Omega}\circ f_{\bar F^{res}_{j_{q}}}$ one has for some $m\in\Z$ ($0\leq m\leq q-1$)
$$W^{-1}\circ \Phi_{\bar\Omega}\circ f_{\bar F^{res}}\circ W=\Phi_{2\pi (m/q)r}\circ \Phi_{\mathring{\Omega}^{HJ}}\circ f_{\mathring{F}^{vf}}
$$
where 
\be \mathring{\Omega}^{HJ}(r)=q^{-2}H^2(q(r-\bar c)),\qquad  \mathring{F}^{vf}=O(\hat F^{vf}).\label{a10262}
\ee
If we define
\begin{align}& f_{{\hat F}^{cor}}=W^{-1}\circ f_{ F^{cor}}\circ W,\qquad \hat F^{cor}=O(F^{cor})\label{estFcorhat}\\
& g^{HJ}=f_{Z}\circ W\qquad(f_{Z}\ \textrm{from} \ (\ref{10.162bis})) \label{8.223}\end{align}
one has  from (\ref{10.162bis})  (note that $W$ commutes with $\Phi_{2\pi(p/q)r}$)
\begin{align*} (g^{HJ})^{-1}\circ \Phi_{\Omega}\circ f_{F}\circ g^{HJ}&=\Phi_{2\pi(p/q)r}\circ W^{-1}\circ\Phi_{\bar\Omega}\circ   f_{\bar F^{per}}\circ W\circ  f_{ {\hat F}^{cor}}\\
&\begin{multlined}= \Phi_{2\pi(p/q)r}\circ \Phi_{2\pi (m/q)r}\circ \Phi_{\mathring{\Omega}^{HJ}}\circ f_{\mathring{F}^{vf}}\circ f_{ {\hat F}^{cor}}\end{multlined}\\
&=: \Phi_{{\Omega^{HJ}}}\circ f_{F^{HJ}}
\end{align*}
with (see  (\ref{a10262}), (\ref{e10258}), \ref{barFresbis})
\begin{align}&\Omega^{HJ}\in \cO_{\s}(\bA(\bar c;5q^{-1}\e_{1}^{1/32},\bar \rho/5)),\\
&\Omega^{HJ}(r)=2\pi((p+m)/q)r+q^{-2}H^2(q(r-\bar c))\label{10.239}\\
&F^{HJ}=\mathring{F}^{vf}+{\hat F}^{cor}+\fO_{2}(\mathring{F}^{vf},{\hat F}^{cor})\in \cO_{\s}(\T_{h/9}\times \bA(\bar c;5q^{-1}\e_{1}^{1/32},\bar \rho/5))\\
&\|F^{HJ}\|\lesssim  \exp(-1/(q\bar\rho)^{1/4})\bar \e.
\end{align}

The disk $\check D$  of Proposition \ref{proppropHJ} can be taken to be ({\it cf.} (\ref{def:epsilon1}), (\ref{deftiPi}))
\be \begin{cases}&\check D=\bD(\check c, \check \rho)\subset \bD(\bar c,\bar\e^{\ 1/33}) \\
&\check c=\bar c,\qquad \check \rho=\e_{1}^{1/32}=\|e_{1}\|_{C^0(\T)}^{1/32}\leq \bar \e^{\hs 1/33}\qquad (\e_{1}\lesssim q\bar \rho^{\ -1}\bar \e)
\end{cases}\label{eq:checkD}
\ee
and the disk $\hat D$ can be taken to be (recall $|\bar c-c|\lesssim \bar \e$)
\be \hat D=\bD( c, \bar  \rho/6)=\bD(c,\hat\rho).
\label{eq:checkhatD}
\ee

With a slight abuse of notation, we can write $W= \Lambda_{q}^{-1}\circ \ti W\circ \Lambda_{q}$ and using (\ref{8.223}), (\ref{whoistiW}) and  the definition of $\bar W$ ({\it cf.} Proposition \ref{prop:8.10})  we can write 
$$g^{HJ}=g_{RNF}\circ \L_{q}^{-1}\circ (G\circ \bar W)\circ \L_{q}\in \ti{\rm Symp}_{\s}(\T_{h/9}\times \bA(\bar c;\e_{1}^{1/32},\bar \rho/5)).
$$
 The last inequality of (\ref{barFresbis}) and (\ref{10.233})  show that (remember (\ref{def:epsilon1}))
$$\|g^{HJ}-id\|_{C^1}\lesssim q\e_{1}^{1/8}+q\bar\e^{1-}\lesssim q\bar\e^{\hs1/8}.
$$
which is (\ref{estgHJ}).

Let us check  that one can choose $\Omega^{HJ}$ in $\ti \cO_{\d}(\hat D\setminus \check D)$ which satisfies a $(2A,2B)$-twist condition. Indeed,  from    (\ref{10.223}) (recall $L=\e_{1}^{-7/32}$) we see that 
$$\|q^{-2}H^2(q(r-\bar c))-\gamma^{-2}(r-\bar c)^2\|_{C^{3}(\T_{h/9}\times \bA(\bar c;6\e_{1}^{1/32},\bar \rho/6))}\lesssim q  \e_{1}^{-1/32}\e_{1}^{7/16}\leq q\e_{1}^{13/32}.
$$
We now apply Lemma \ref{lemma:2.3ee}: since $\e_{1}^{-6/32}\times \e_{1}^{13/32}\lesssim \e_{1}^{7/32} $, there exists a $C^3$ $\s$-symmetric Whitney extension  with $C^3$-norm less that $\e_{1}^{1/5}$ for $(q^{-2}H^2(q(r-\bar c))-\gamma^{-2}(r-\bar c)^2,\T_{h/9}\times \bA(\bar c;6\e_{1}^{1/32},\bar \rho/6))$. We then conclude by (\ref{10.239}).

\ \hfill $\Box$

\subsection{Extending the linearizing map inside the hole}\label{sec:10.4}
In general the  previously defined maps $g,\Gamma,H$ {\it are not} holomorphically defined  on a {\it whole} disk  but rather on an annulus with inner disk of radius $\e_{1}^{1/32}$ where $\e_{1}=\|e_{1}\|_{C^0(\T)}$.  In this sub-section we quantify to which extent the domains of holomorphy of these maps can be extended if one knows that the frequency map $\Omega^{HJ}$ coincides on this annulus with a holomorphic function defined on a disk (containing the annulus).
\begin{prop}\label{prop:cosofholom} If there exists a holomorphic function $\ti \Xi$ defined on $\bD(0,\rho_{q})$ such that 
\be \|\ti \Xi-H^2\|_{C(0,\rho_{q}/2)}\leq \nu\label{eqprop8.8}
\ee
then 
$$ \e_{1}=\|e_{1}\|_{C^0(\T_{})}\lesssim  \nu^{(1/6)-}.
$$
\end{prop}
We prove this proposition in Subsection \ref{sec:8.8.2}.

We now take 
$$s=0.\qquad (\textit{cf.} \ (\ref{lambda-L}))$$

\medskip
By (\ref{12.214}) for $z\in A(\l/2,\l/4) $, $|g(\th,z)|$ compares to $\l$ and thus from (\ref{7.192})and  (\ref{10.208})
$$ z^2=\varpi(\th)\biggl(g(\th,z)^2-e_{1}(\th)\biggr)+O( g^3)$$
so that 
\begin{align}g(\th,z)&=\biggl(z^2/\varpi(\th)+e_{1}(\cdot)+O( g^3)\biggr)^{1/2}\\
&=\biggl((z^2/\varpi(\th))+e_{1}(\th) \biggr)^{1/2}+O(\l^2)\label{8.216}.
\end{align}
Let's introduce 
\be \ti g(\th,z)=\biggl(\frac{z^2}{\varpi(\th)}+e_{1}(\th)\biggr)^{1/2}\label{8.217}
\ee
\be \ti \Gamma(\cdot)=(2\pi)^{-1}\int_{0}^{2\pi}\ti g(\th,\cdot)d\th,\qquad \ti H=\ti\Gamma^{-1}\label{8.17kov}
\ee
where  the   inverse is with respect to   composition. The functions $\ti \Gamma$ and $\ti H$ are defined on $\{z\in\C, L\e_{1}^{1/2}<|z|\}$ for some fixed $L\gg 1$, independent of $\e_{1}$,  satisfying  
\be L\leq \bar \rho\e_{1}^{-1/2}\label{10.246}\ee
(we take here $s=0$, {\it cf.} (\ref{lambda-L})). 

\bigskip\noindent {\bf Notation:} In the following we denote by $C(0,t)$ the circle of center 0 and radius $t>0$.

We have the following proposition:

\subsubsection{Computation of a residue}

\begin{lemma}\label{lemma:9.10}For any  circle $C(0,t)$ centered at 0 with $t> L\e_{1}^{1/2}$ one has 
$$ \frac{1}{2\pi i}\int_{C(0,t)} z\ti H(z)^2dz=(\gamma/4)\int_{\T}\varpi(\th)^{3/2}e_{1}(\th)^2\frac{d\th}{2\pi}
$$
where $\gamma=(2\pi)^{-1}\int_{0}^{2\pi}\varpi(\th)^{-1/2}d\th$.
\end{lemma}
\begin{proof}
We compute the   expansion of $\ti g(\th,\cdot)$ ({\it cf.} (\ref{8.217}))  into Laurent series: on $\C\setminus \bD(0,L \e_{1}^{1/2})$:
\begin{align*}\ti g(\th,z)&=(z/\varpi(\th)^{1/2})\biggl(1+\varpi(\th)e_{1}(\th)z^{-2}\biggr)^{1/2}\\
&=(z/\varpi(\th)^{1/2})\biggl(1+\frac{1}{2}\varpi(\th)e_{1}(\th)z^{-2}-\frac{1}{8}(\varpi(\th)e_{1}(\th))^2z^{-4}+O(z^{-6})\biggr)\\
&=\frac{z}{\varpi(\th)^{1/2}}+\frac{1}{2}\varpi(\th)^{1/2}e_{1}(\th)z^{-1}-\frac{1}{8}\varpi(\th)^{3/2}e_{1}(\th)^2z^{-3}+O(z^{-5}).
\end{align*}
As a consequence 
since $\ti \Gamma(z)=(2\pi)^{-1}\int_{0}^{2\pi}\ti g(\th,z)d\th$ we have with the notation $\gamma=(2\pi)^{-1}\int_{0}^{2\pi}\varpi(\th)^{-1/2}d\th$ the identity 
$$ \ti \Gamma(z)=\gamma (z+a_{-1}z^{-1}+a_{-3}z^{-3})+O(z^{-5})
$$
where 
$$ a_{-1}=\gamma^{-1}(1/2)(2\pi)^{-1}\int_{0}^{2\pi}\varpi(\th)^{1/2}e_{1}(\th)d\th$$
\be a_{-3}=\gamma^{-1}(-1/8)(2\pi)^{-1}\int_{0}^{2\pi}\varpi(\th)^{3/2}e_{1}(\th)^2d\th.\label{a-3}\ee
By our choice (\ref{condone1}) we have $a_{-1}=0$ and 
we can thus write
\be\ti\Gamma=\Lambda_{\gamma}\circ (id+u)\label{8.229}
\ee
where $\Lambda_{\gamma}z=\gamma z$ and 
$$ u(z)=a_{-3}z^{-3}+O(z^{-5}).
$$

If  $v$ is defined by
$$(id+u)\circ (id+v)=id$$ 
we have
$$v(z)=-a_{-3}z^{-3}+O(z^{-4})$$
and therefore
\begin{align}(z+v(z))^2&=(z-a_{-3}z^{-3}+O(z^{-4}))^2\notag\\ 
&=z^2-2a_{-3}z^{-2}+O(z^{-3}).\label{8.231}
\end{align}

Now since $\ti H$ is the inverse  for the composition of $\ti \Gamma$  ({\it cf.} \ref{8.17kov})), $z=(\ti \Gamma\circ \ti H)(z)$, we have by (\ref{8.229}) $\ti H=(id+u)^{-1}\circ \Lambda_{\gamma}^{-1}=(id+v)\circ \Lambda_{\gamma}^{-1}$ and we get by (\ref{8.231})
$$ \ti H(z)^2=\gamma^{-2}z^2-2a_{-3}\gamma^{2}z^{-2}+O(z^{-3})
$$
and thus
$$z\ti H(z)^2=\gamma^{-2}z^3-2a_{-3}\gamma^{2}z^{-1}+O(z^{-2}).
$$
Hence by Cauchy formula and (\ref{a-3}), for any circle $C(0,t)$, $t> L\e_{1}^{1/2}$:
\begin{align*} \frac{1}{2\pi i}\int_{C(0,t)} z\ti H(z)^2dz&=-2a_{-3}\gamma^{2}\\
&=(\gamma/4)\int_{\T}\varpi(\th)^{3/2}e_{1}(\th)^2\frac{d\th}{2\pi}.
\end{align*}
\end{proof}

\subsubsection{Proof of Proposition \ref{prop:cosofholom} }\label{sec:8.8.2}
\begin{lemma}\label{lemma:9.11} Let $ L\e_{1}^{1/2}\leq\l<\rho_{q}/2$, $L\gg 1$ (independent of $\e_{1}$). One has for $z\in A(\l/4,\l/2)$
\be| H(z)^2-\ti H(z)^2|\lesssim \l^3.
\ee
\end{lemma}
\begin{proof} 
For  $z\in A(\l/4,\l/2)$, $\th\in\T$ one has by (\ref{8.216}), (\ref{8.217})
$$ |g(\th,z)-\ti g(\th,z)|\lesssim \l^2
$$
so ({\it cf.} (\ref{defGamma}), (\ref{8.17kov}))
\be |\Gamma(z)-\ti\Gamma(z)|\lesssim \l^2.\label{8.234}
\ee
On the other hand, from Lemma \ref{lemma:11.10}
$$ e^{-3/L^2}\leq \biggl|\frac{ \ti g(\th,z)-\ti g(\th,z')}{z-z'}\biggr|\leq e^{2/L^2}
$$
hence
\be e^{-3/L^2}\leq \biggl|\frac{ \ti \Gamma(z)-\ti \Gamma(z')}{z-z'}\biggr|\leq e^{2/L^2}.\label{8.236}
\ee
Since $z=\Gamma(H(z))=\ti\Gamma(\ti H(z))$  and $H(z), \ti H(z)\asymp z$ ({\it cf.} (\ref{10.223})), one has from (\ref{8.234})
$$ |\ti \Gamma(H(z))-\ti \Gamma(\ti H(z))|\lesssim \l^2
$$
and so from (\ref{8.236})
$$ |\ti H(z)-H(z)|\lesssim \l^2.
$$
Since  from (\ref{10.223})  $|\ti H(z)+H(z)|\lesssim \l$ we  thus have
$$ |\ti H(z)^2-H(z)^2|\lesssim \l^3.
$$

\end{proof}
We recall that $\e_{1}=\|e_{1}\|_{C^0(\T)}$.
The function $\ti\Xi-H^2$ satisfies ({\it cf.} (\ref{eqprop8.8}), (\ref{10.223}))
$$ \|\ti \Xi- H^2\|_{C(0,\rho_{q}/2)}\lesssim \nu,\qquad \|\ti \Xi- H^2\|_{C(0,L\e_{1}^{1/2})}\lesssim 1.
$$
Let $M>5$ and 
\be \l_{M}:=(\rho_{q}/2)^{1/M}(L \e_{1}^{1/2})^{1-1/M}\leq (L\e_{1}^{1/2})^{1-1/M}\label{10.263}\ee
 (we can assume $\rho_{q}\leq 1$).
By the Three Circles Theorem,
$$  \|\ti \Xi- H^2\|_{C(0,\l_{M})}\lesssim \nu^{1/M}.
$$
Lemma \ref{lemma:9.11} tells us that
$$ \|\ti\Xi-\ti H^2\|_{C(0,\l_{M})}\lesssim \nu^{1/M}+\l_{M}^3
$$
hence for any $z$ in the circle $C(0,\l_{M})$
$$ |{z}\ti\Xi(z)-{z}\ti H^2(z)|\lesssim {\l_{M}}(\nu^{1/M}+\l_{M}^3)
$$
and 
$$ \biggl| \frac{1}{2\pi i}\int_{C(0,\l_{M})} ({z}\ti\Xi(z)-{z}\ti H^2(z))dz\biggr| \lesssim \l_{M}^{2}(\nu^{1/M}+\l_{M}^3).
$$
Since $z\mapsto {z}\ti\Xi(z)$ is holomorphic on $D(0,2\l_{M})$, $\int_{C(0,\l_{M})} {z}\ti\Xi(z)dz=0$ and by Lemma \ref{lemma:9.10} we get
$$ \int_{\T}\varpi^{3/2}(\th)e_{1}(\th)^2d\th\lesssim {\l_{M}^2}(\nu^{1/M}+\l_{M}^3).
$$
Since $\varpi(\th)\gtrsim 1$ this gives
$$ \int_{\T}e_{1}(\th)^2d\th\lesssim {\l_{M}^2}(\nu^{1/M}+\l_{M}^3)
$$
hence remembering (\ref{10.263})
\begin{align*} \|e_{1}\|_{L^2(\T)}&\lesssim {\l_{M}}\nu^{1/(2M)}+\l_{M}^{5/2}\\
&\lesssim L^{(1-1/M)}\e_{1}^{(1/2)(1-1/M)}\nu^{1/(2M)}+L^{(5/2)(1-1/M)}\e_{1}^{(5/4)(1-1/M)}.
\end{align*}
If we define
$$\d_{M}=L^{(5/2)(1-1/M)}\e_{1}^{(5/4)(1-1/M)-1},\qquad \mu_{M}= L^{(1-1/M)}\e_{1}^{(1/2)(1-1/M)}\nu^{1/(2M)}$$
this can be written (recall that $\|e_{1}\|_{C^0(\T)}=\e_{1}\lesssim q\bar \rho^{\ -1}\bar\e$) for some $C>0$
$$ \|e_{1}\|_{L^2(\T)}\leq C\d_{M} \|e_{1}\|_{C^0(\T)}+C\mu_{M}
$$
and we are in position to apply Lemma \ref{L2C0} (our choice  $M>5$ implies that for some $\b>0$,  $\d_{M}\leq  \e_{1}^{2\b}\ll 1 $):
\begin{align*}\e_{1}=\|e_{1}\|_{C^0(\T)}&\leq (\mu_{M}/\d_{M})+Ch^{-1}\exp(-h/(C\d_{M}^2))q\bar \rho^{\ -1}\bar\e\\
&\lesssim (\mu_{M}/\d_{M})+\exp(-(1/\e_{1})^{\b})\qquad (\b>0)\\
&\lesssim (\mu_{M}/\d_{M})+(1/2)\e_{1}
\end{align*}
which gives 
$$ \e_{1}\lesssim L^{ -(3/2)(1-1/M) }\e_{1}^{1- ((3/4)(1-1/M))   }\nu^{1/(2M)} 
$$
or equivalently
$$\e_{1}^{(3/4)(1-1/M)}\lesssim L^{ -(3/2)(1-1/M)} \nu^{1/(2M)} 
$$
and taking $M=5+$, one finally gets:
\begin{align*}   \e_{1}^{  }&\lesssim L^{ -(2-) }\nu^{(1/6)-}\\
&\leq \nu^{(1/6)-}.
\end{align*}
This completes the proof of Proposition \ref{prop:cosofholom}.
\ \hfill$\Box$

\subsection{Proof of Proposition \ref{proppropHJ}: the Extension Property }\label{subsec10.5}
From (\ref{10.239}) we see that if there exists a holomorphic function $\Xi$ defined on $\hat D$ such that 
$$\|\Xi-\Omega^{HJ}_{D}\|_{(4/5)\hat D\setminus (1/5)\hat D}\lesssim \nu 
$$
there exists a holomorphic function $\ti\Xi$ defined on $\bD(0,\rho_{q})$ (recall that $\rho_{q}=q\bar \rho/3$ {\it cf.} (\ref{eq:rhoq})) such that 
$$\|\ti \Xi-H^2\|_{C(0,\rho_{q}/2)}\lesssim \nu
$$
and thus by Proposition \ref{prop:cosofholom}
$$\e_{1}=\|e_{1}\|_{C^0(\T_{})}\lesssim \nu^{(1/6)-}.
$$
Now (\ref{eq:checkD}) shows that  that the conclusion of 
Proposition \ref{prop:cosofholom} holds with $\check D=D(c,\nu^{1/200})$.
\hfill $\Box$

\section{Comparison Principle for  Normal Forms}\label{sec:comparingnf}
In this section, if  $0\leq \rho_{1}<\rho_{2}$, we denote by   $\bA(c;\rho_{1},\rho_{2})$ the annulus $\{z\in\C,\  \rho_{1}\leq |z-c|<\rho_{2}\}$ (it is thus the  disk $\bD(c,\rho_{2})$ if $\rho_{1}=0$).

\begin{prop}\label{comparingBNF}[(AA) Case] There exist positive  constants $\bar C,\bar a_{5},\bar a_{6}$ for which the following holds. Let $0<\rho_{1}<\rho_{2}$ (resp. $\rho_{1}=0<\rho_{2}$), $\e,\nu>0$ and for $j=1,2$, $\Omega_{j}\in\ti \cO_{\s}(\bA(c;\rho_{1},\rho_{2}))$, $F_{j}\in \cO_{\s}(W_{h,\bA(c;\rho_{1},\rho_{2})})$, $g_{j}\in\ti {\rm Symp}_{\s}(W_{h,\bA(c;\rho_{1},\rho_{2})})$ such that: $\Omega_{1},\Omega_{2}$ satisfy an $(A,B)$-twist condition and
\be \|g_{j}-id\|_{C^1}\leq \e< \bar C^{-1}h\label{estgj}
\ee
$$\|F_{j}\|_{W_{h,\bA(c;\rho_{1},\rho_{2})}}\leq \nu
$$
and  on $g_{1}(\bA(c;\rho_{1},\rho_{2}))\cap g_{2}(\bA(c;\rho_{1},\rho_{2}))$ one has 
$$g_{1}\circ \Phi_{\Omega_{1}}\circ f_{F_{1}}\circ g_{1}^{-1}=g_{2}\circ \Phi_{\Omega_{2}}\circ f_{F_{2}}\circ g_{2}^{-1}.
$$
Then, if  $\d>0$ satisfies
\be \bar C\e\leq \d/4<(\rho_{2}-\rho_{1})\quad  \textrm{and}\quad \bar C\d^{-\hs\bar a_{5}}\nu<1,\label{11308}\ee there exists $\g\in\R$, $|\g|\leq \bar C \e$ such that one has 
$$\|\pa\Omega_{1}(\cdot+\g)-\pa\Omega_{2}\|_{\bA(c;\rho_{1}+\d,\rho_{2}-\d)}\leq \bar C\d^{-\hs\bar a_{6}} \nu.
$$
$$(\textrm{resp.}\quad \|\pa\Omega_{1}(\cdot+\g)-\pa\Omega_{2}\|_{\bD(c,\rho_{2}-\d)}\leq \bar C\d^{-\bar a_{6}} \nu.)
$$
Furthermore, if $g_{1}$ and $g_{2}$ are exact symplectic on $M_{\R}$, one can choose $\g=0$.
\end{prop}
\begin{proof}
We only treat the case $\rho_{1}>0$ (the case $\rho_{1}=0$ is done similarly).

From  (\ref{estgj}) we see that there exists $C>0$ such that   one has on $W_{1}:=\T_{h-C\e}\times \bA(c;\rho_{1}+C\e,\rho_{2}-C\e)$
\be g\circ \Phi_{\Omega_{1}}=\Phi_{\Omega_{2}}\circ g\circ f_{F}\label{e12.276}
\ee
where
\begin{align*}&g:=g_{2}^{-1}\circ g_{1}\in \ti{\rm Symp}_{\s}(W_{1})\\
&F\in \cO(W_{1}),\qquad f_{F}:=g^{-1}\circ f_{F_{2}}\circ g\circ f_{F_{1}}^{-1},\qquad \|F\|_{W_{1}}\lesssim \nu.
\end{align*} We write
\be g(\th,r)=(\th+u(\th,r),r+v(\th,r))\label{eq:formZ}\ee
and we introduce the notations $\omega_{i}=\pa\Omega_{i}$, $i=1,2$ (we drop the usual factor $(2\pi)^{-1}$).
We have 
$$ g\circ \Phi_{\Omega_{1}}(\th,r)=(\th+\omega_{1}(r)+u(\th+\omega_{1}(r),r),r+v(\th+\omega_{1}(r),r))
$$
and
$$ \Phi_{\Omega_{2}}\circ g=(\th+u(\th,r)+\omega_{2}(r+v(\th,r)),r+v(\th,r))
$$
We thus have on $W_{2}:=\T_{h-C\e-B\rho_{2}-\d}\times \bA(c;\rho_{1}+C\e+\d,\rho_{2}-C\e-\d)$
\be 
\begin{cases}\omega_{2}(r+v(\th,r))-\omega_{1}(r)=I+u(\th+\omega_{1}(r),r)-u(\th,r)\\
v(\th+\omega_{1}(r),r)-v(\th,r)=II
\end{cases}\label{eq:6.64}
\ee
with $\max(\|I\|_{W_{2}},\|II\|_{W_{2}})=O(\delta^{-b}\nu)$.
We  observe that from the twist assumption  on $\Omega_{1}$
 there exists a set $R\subset \bA(c;\rho_{1}+C\e+\d,\rho_{2}-C\e-\d)$ of Lebesgue measure 
$\lesssim \d^{2}$, which is a countable union of disks centered    on the real axis,  such that one has for any $r\in \bA(c;\rho_{1}+C\e+\d,\rho_{2}-C\e-\d)\setminus R$ and any $k\in\Z^*$
\be \min_{l\in\Z}|\omega_{1}(r)-2\pi\frac{l}{k}|\geq  \frac{\d^{2}}{k^3}\label{e12.281}
\ee
so that the second identity in (\ref{eq:6.64}) gives for any $r\in \bA(c;\rho_{1}+C\e+\d,\rho_{2}-C\e-\d)\setminus R$ the following inequality on $\T_{h_{1}-2\d}$ (where $h_{1}=h-C\e-B\rho_{2}$)
\be \|v(\cdot,r)-\int_{\T}v(\th,r)d\th\|_{h_{1}-2\d}\lesssim \d^{-3}\d^{-b}\nu.\label{eq:6.66}
\ee
We now notice that  there exists $0\leq t\leq \d^2$ such that $R\cap \pa \bA(c;\rho_{1}+C\e+\d+t,\rho_{2}-C\e-\d-t)=\emptyset$. The maximum principle applied,  for any $\ph\in\T_{h_{1}-2\d}$,  to the holomorphic function $v(\ph,\cdot)-\int_{\T}v(\th,\cdot)d\th$  defined on $\bA(c;\rho_{1}+C\e+\d+t,\rho_{2}-C\e-\d-t)$ shows that (\ref{eq:6.66}) holds for any $r\in \bA_{2\d}:=\bA(c;\rho_{1}+C\e+2\d,\rho_{2}-C\e-2\d)$. We thus have 
 \be\|\pa_{\th}v\|_{h_{1}-3\d, \bA_{3\d}}=O(\d^{-(4+b)}\nu).\label{eq:6.66'}
 \ee
Taking the $\pa_{\th}$ derivative of the first line of  (\ref{eq:6.64})  and using  the previous inequality show that  (from now on the value of $b$ may change from line to line)
$$\pa_{\th}u(\th+\omega_{1}(r),r)-\pa_{\th}u(\th,r)=O(\d^{-b}\nu).
$$
By the same argument used to establish (\ref{eq:6.66'}) we get
\be\|\pa_{\th}u\|_{h_{1}-4\d, \bA_{4\d}}=O(\d^{-b}\nu)\label{eq:6.66'u}
 \ee
 (we have used the fact that $\int_{\T}\pa_{\th}u(\th,r)d\th=0$).
 Since $g$ is symplectic on $W_{1}$, $\det Dg(\th,r)\equiv 1$ hence
 $$(1+\pa_{\th}u(\th,r))(1+\pa_{r}v(\th,r))-\pa_{r}u(\th,r)\pa_{\th}v(\th,r)=1
 $$
 and in view of (\ref{eq:6.66'}), (\ref{eq:6.66'u})
$$ \|\pa_{r}v\|_{h_{1}-4\d,\bA_{4\d}}=O(\d^{-b}\nu)
$$
 which combined with (\ref{eq:6.66'})  implies,
 \be \|v-\g\|_{h_{1}-4\d,\bA_{4\d}}=O(\d^{-b}\nu),\qquad \g=v(0,0)\in\R.\label{12.281}
\ee
The first equation of (\ref{eq:6.64}) implies that
$$ \|\omega_{2}(\cdot+\g)-\omega_{1}(\cdot)\|_{\bA_{4\d}}=O(\d^{-b}\nu).
$$

If $g_{1}$ and $g_{2}$ are exact symplectic, $g$ is also exact symplectic  and one can write $g=f_{Z}$ for some $Z=\fO_{1}(g-id)$ which means  $g(\th,r)=(\ph,R)$ if and only if $r=R+\pa_{\th}Z(\th,R)$, $\ph=\th+\pa_{R}Z(\th,R)$. In particular
$$ r=r+v(\th,r)+\pa_{\th}Z(\th,r+v(\th,r))
$$
and  since 
$$\frac{d}{d\th}Z(\th,r+v(\th,r))=\pa_{\th}Z(\th,r+v(\th,r))+\pa_{R}Z(\th,r+v(\th,r)\pa_{\th}v(\th,r)
$$
we get from (\ref{eq:6.66'})
$$v(\th,r)=-\frac{d}{d\th}Z(\th,r+v(\th,r))+O(\d^{-b}\nu)
$$
which after integration in $\th$ yields 
$$\int_{\T}v(\th,r)d\th=O(\d^{-b}\nu).$$
We can now conclude from (\ref{12.281}) that 
$$\g=O(\d^{-b}\nu).$$

\end{proof}
\begin{prop}\label{comparingBNFCC}[(CC)- Case] Under the assumptions of the previous Proposition \ref{comparingBNF}:
\begin{enumerate}
\item\label{i11} If  $c=0$, $\rho=\rho_{2}$, $\rho_{1}=0$ and $g_{1},g_{2}$ are exact symplectic then 
$$\|\pa\Omega_{1}(\cdot)-\pa\Omega_{2}(\cdot)\|_{\bD(0,\rho_{2}-\d)}\leq C\d^{-\bar a_{6}} \nu.
$$
\item\label{i12} If $\rho_{2}<|c|/4$ then all the conclusions of the previous Proposition \ref{comparingBNF} are valid. 
\end{enumerate}
\end{prop}
\begin{proof}
The proof of Item  (\ref{i12}) follows from Item (\ref{item2:4.80}) of Lemma \ref{lemma:4.5} applied to Proposition \ref{comparingBNF}.

So we concentrate on the proof of  Item (\ref{i11}), $c=0$, $\rho=\rho_{2}$, $\rho_{1}=0$. We use the symplectic change of coordinates of Section \ref{sec:4.6} $(\th,r)=\psi_{\pm}^{-1}(z,w)$, 
$$\psi_{\pm}:\T_{h}\times \Delta_{\a}^\pm(0,\rho)\to W_{h,\Delta_{\a}^\pm(0,\rho)}\cap \{e^{-2h}<|z|/|w|<e^{2h}\}$$
where $\a<\pi/10$.
Setting $g_{j,\pm}=\psi_{\pm}^{-1}\circ g_{j}\circ \psi_{\pm}$, $g_{\pm}=g_{2,\pm}^{-1}\circ g_{1,\pm}$, $g_{\pm}(\th,r)=(\th+u_{\pm}(\th,r),r+v_{\pm}(\th,r))$
 we  are then reduced to the preceding situation where $g$ is replaced by $g_{\pm}$,  the annulus $\bA(c_{2};\rho_{1},\rho_{2})$ is  replaced by the angular sector $\Delta_{\a+4\d}^\pm(\rho-4\d)$ and $h$ by $h-4\d$, so that (\ref{eq:6.64}) holds on  $\T_{h-4\d}\times \Delta_{\a-4\d}^\pm(\rho-4\d)$. Like in the previous case, one can find $0\leq t\leq \d^2$ such that the diophantine condition   (\ref{e12.281}) holds for any $r\in \Delta^\pm_{\a+4\d}(0;t,\rho-4\d-t):= \Delta_{\a+4\d}^\pm(\rho-4\d)\cap \bA(0;t,\rho-4\d-t)$. Still by the Maximum Principle (\ref{eq:6.66'}) holds on $\T_{h-5\d}\times  \Delta^\pm_{\a+5\d}(0;t,\rho-5\d-t)$ with $v$ replaced by $v_{\pm}$ and one can conclude as we've done before that (\ref{eq:6.66'u}) holds with $u$ replaced by $u_{\pm}$ as well.
 Finally this gives the existence of $\g_{\pm}=v_{\pm}(0,0)\in\R$ such that on $\Delta^\pm_{\a+5\d}(0;t,\rho-5\d-t)$
\be \|\omega_{2}(\cdot+\g_{\pm})-\omega_{1}(\cdot)\|_{\bA_{4\d}}=O(\d^{-b}\nu).\label{9.255}
\ee 
 Now, if $g_{1}$ and $g_{2}$ are exact symplectic the same is true for $g_{1,\pm},g_{2,\pm}$ ({\it cf.} Remark \ref{sec:rem4.1}) and hence $g_{\pm}$ is also exact symplectic; we can thus   prove, like in the proof of Proposition \ref{comparingBNF}, that $\pm\gamma=O(\d^{-b}\nu)$. We can hence assume that  $\g_{\pm}=0$ in equation (\ref{9.255}). Since $\a<\pi/10$ we deduce that on $\bA(0;t,\rho-5\d-t)=\Delta^+_{\a+5\d}(0;t,\rho-5\d-t)\cup \Delta^-_{\a+5\d}(0;t,\rho-5\d-t)$ one has 
$$ \|\omega_{2}(\cdot)-\omega_{1}(\cdot)\|_{\bA(0;t,\rho-5\d-t)}=O(\d^{-b}\nu).
$$ 
 But $\omega_{1},\omega_{2}\in \cO(\bD(0,\rho))$, hence by the Maximum Principle
 $$ \|\omega_{2}(\cdot)-\omega_{1}(\cdot)\|_{\bD(0,\rho-5\d-t)}=O(\d^{-b}\nu).
$$ 
\end{proof}

\section{Adapted Normal Forms: $\omega_{0}$ diophantine} \label{sec:adaptedkam}

Recall that 
$$DC(\kappa,\tau)=\{\omega_{0}\in\R,\ \forall\ k\in\Z^*,\ \min_{l\in\Z}|\omega_{0}-\frac{l}{k}|\geq \frac{\kappa}{|k|^{1+\tau}}\} \qquad DC(\tau)=\bigcup_{\kappa>0}DC(\kappa,\tau).$$

Let $h> 0$, $0<\bar \rho<1$,  $\Omega\in\ti \cO_{\s}(e^{10h}\bD(0,\bar\rho))$, $F\in \cO_{\s}(e^{10h}W_{h,\bD(0,\bar \rho)})$ such that 
\begin{align}& \forall \ r\in\R,\  A^{-1}\leq  (2\pi)^{-1}\pa^2\Omega(r)\leq A,\quad \textrm{and}\ \ \|(2\pi)^{-1}D^3\Omega\|_{\C}\leq B.\label{eq:11.257}\\
&\omega_{0}:=(2\pi)^{-1}\pa^2\Omega(0)\in DC(\tau)\label{eq:11.257,5}\\
&\forall\ 0<\rho\leq \bar\rho,\quad   \|F\|_{e^{10h}W_{h,\bD(0,\rho)}}\leq \rho^{m},\qquad m=\max(\bar a_{1,\tau},\bar a_{2}+4,\bar a_{3},\bar a_{5})\label{eq:11.258}
\end{align}
($\bar a_{1},\bar a_{2},\bar a_{3},\bar a_{5}$ are the constants appearing in Propositions \ref{nprop:9.2}, \ref{prop:1.enonce}, \ref{proppropHJ}, \ref{comparingBNF}).

We as usual denote  $\omega=(1/2\pi)\pa\Omega$ ($\omega(0)=\omega_{0}$).

\subsection{Adapted  KAM domains}

We use in this section the notations of section \ref{sec:5}, in particular we denote
\be \bar \e:= \max_{0\leq j\leq 3}\|D^j F\|_{h,\bD_{\bar \rho}}\leq \bar \rho^{\hs \bar a_{2}}.\label{8.124}\ee
Assumption (\ref{eq:11.258}) allows us to   apply Proposition \ref{prop:1.enonce} on the existence of a KAM Normal Form on the domain $W_{2h,\bD(0,\bar\rho)}$. We can thus define holed  domains $U_{n}$ and maps $F_{n}$, $\Omega_{n}$, $g_{m,n}$ satisfying the conclusions of Proposition  \ref{prop:1.enonce}.

\subsubsection{Definition of the domains $U_{i}^{(\rho)}$}

Let $0<\b\ll 1$ and $\mu\in ]1,2[$
 $$\mu=2(1-(\b/10))\in ]1,2[.$$
 We define for $\rho<\bar \rho/4$ two indices $i_{-}(\rho),\ i_{+}(\rho)\in \N$ as follows:
\be i_{-}(\rho)=\max\{i\geq 1,\ \bD(0,2\rho)\cap U_{i}=\bD(0,2\rho)\}.\label{eq:9.296}
\ee
and $i_{+}(\rho)$ is the unique index such that 
\be (N_{i_{-}(\rho)})^\mu\leq N_{i_{+}(\rho)}< (4/3)^{\mu}(N_{i_{-}(\rho)})^\mu\leq N_{i_{-}(\rho)}^2 .\label{e9.155}
\ee
We also define $\iota(\rho)\in\R^*_{+}$ by
\be  \rho=(N_{i_{-}(\rho)})^{-\iota(\rho)},\qquad N_{i_{-}(\rho)}=\rho^{-1/\iota(\rho)}. \label{eq:5.299ante}\ee
The next lemma shows how $N_{i_{-}(\rho)}$ and $N_{i_{+}(\rho)}$ compare with $\rho$.

\begin{lemma}\label{aalemma:10.1}One has 
\be 2+O(|\ln\rho|^{-1})\leq \iota(\rho)\leq (1+\tau)+O(|\ln\rho|^{-1}). \label{eq:5.299}\ee
In particular,
\be N_{i_{+}(\rho)}\asymp \rho^{-\mu/\iota(\rho)}.\label{eq:5.101}
\ee
where for $\rho\ll_{\beta} 1$
\be  \frac{2}{1+\tau} -(\b/2)\leq  \frac{\mu}{\iota(\rho)}\leq  1-(\b/2).\label{502}
\ee
\end{lemma}
\begin{proof}
To prove (\ref{eq:5.299}) we just have to check that
\be (N_{i_{-}(\rho)})^{-(1+\tau)}\lesssim \rho\lesssim  (N_{i_{-}(\rho)})^{-2}.\label{e10.204}
\ee
See the details in  Appendix \ref{sec:appendixholes1}.

\end{proof}
We shall say that the domains $U_{i}$, $i_{-}(\rho)\leq i\leq i_{+}(\rho)$, are $\rho$-adapted KAM domains.

For $t>0$ and $i_{-}(\rho)\leq i\leq i_{+}(\rho)$ we define
$$U_{i}^{(t)}=U_{i}\cap \bD(0,t),\qquad \cD_{t}(U_{i})=\cD(U_{i}^{(t)})=\{D\in\cD(U_{i}),\ D\cap \bD(0,t)\ne \emptyset\}$$  $U_{i}$ being  the domains of Proposition  \ref{prop:1.enonce} and where as usual   $\cD(U)$ denotes the holes of the holed domain $U$  (see Subsection \ref{se:2.5.1}).
By (\ref{formUn})  
\be\begin{cases} &U_{i}^{(t)}:=U_{i}\cap\bD(0,t)=\bD(0,t)\setminus\bigcup_{j=1}^{i-1}\bigcup_{(k,l)\in E_{j} }\bD(c_{l/k}^{(j)},s_{j,i-1}K_{j}^{-1}),\\
&s_{j,i-1}=e^{\sum_{m=j}^{i-1}{\d_{m}}}\in [1,2]\end{cases}\label{abc10.207}\ee
where
$$E_{j}\subset \{(k,l)\in\Z^2,\ 0<k<N_{j},\ 0\leq |l|\leq N_{j}\},\qquad \omega_{j}^{}(c_{l/k}^{(j)}) =l/k.$$
One can in fact in formula (\ref{abc10.207}) restrict the union indexed by $j$ to the set $j\in [i_{-}(\rho),i-1]\cap \N$; {\it cf.}  Lemma \ref{ji-rho} of the Appendix \ref{sec:appendixholes}.

One can also describe $U_{i}^{(t)}$ by means of its holes:\be U_{i}^{(t)}:=U_{i}\cap\bD(0,t)= \bD(0,t)\setminus \bigcup_{D\in\cD_{t}(U_{i})}D\label{10.207ante}\ee
this decomposition being minimal. In particular, if $D,D'\in\cD_{t}(U_{i})$ the inclusions $D\subset D'$, $D'\subset D$ do not occur.
\begin{prop}\label{prop:a10.4}Let $i_{-}(\rho)\leq i'<i\leq i_{+}(\rho)$.
\begin{enumerate}
\item\label{ab1} The holes  $D\in\cD_{(3/2)\rho}(U_{i})$ are pairwise disjoint.
\item\label{ab2} If $D\in \cD_{(3/2)\rho}(U_{i})$, $D'\in \cD_{(3/2)\rho}(U_{i' })$ one has either $D\cap D'=\emptyset$ or $D'\subset D$.
\item\label{ab3}  The number of holes of $U_{i}$ intersecting 
$\bD(0,\rho)$  satisfies 
\be \# \{D\in\cD(U_{i}),\ D\cap \bD(0,\rho)\ne \emptyset\}\lesssim \rho N_{i}^2 .\label{numberofholes}
\ee
\item \label{ab4}Let $D\in \cD_{\rho}(U_{i_{+}(\rho)})$ and define $$ i_{D}=-1+\min\{i:\ i_{-}(\rho)< i\leq i_{+}(\rho),\ \exists D'\in \cD_{\rho}(U_{i_{+}(\rho)}),\ D'\subset D\}.$$
Then, $D$ is of the form  $D= \bD(c_{D},s_{D}K_{i_{D}}^{-1}),\quad  s_{D}\in [1,2],\ c_{D}\in\R,\ \omega_{i_{D}}(c_{D})\in \{l/k,\ (k,l)\in E_{i_{D}}\}$ and one has $D\subset U_{i_{D}}$.
\item\label{ab5} With $\tau=\tau(\omega_{0})$ and $b_{\tau}$ as defined in  (\ref {e6.89}) one has
\be\bD(0,\rho^{b_{\tau}})\subset  U_{i_{+}(\rho)}.\label{eq:disk}
\ee

\end{enumerate}
\end{prop}
\begin{proof}We refer to the Appendix \ref{sec:appendixholes2} for the proofs of Items \ref{ab1}, \ref{ab2} and \ref{ab4}.

\medskip\noindent{\it Proof of Item \ref{ab3} on the number of holes.}
From   (\ref{abc10.207}) we  just have to check that for $N\in \N$
$$\#\{(k,l)\in\Z^2,\ l/k\in ]\omega_{0}-s,\omega_{0}+s[,\ 0<k<N, 0\leq |l|\leq N \}\lesssim s N^2.$$
If $(k,l)$ belongs to the preceding set one has  $|l-k\omega_{0}|<s N$ and thus $(k,l)$ belongs to $[-N,N]^2\cap \{(x,y)\in\R^2,\ |x-\omega_{0} y|\leq sN\}$ a set which has Lebesgue measure $\lesssim sN^2$. We thus have for large $N$,  $\#(\Z^2\cap[-N,N]^2\cap \{(x,y)\in\R^2,\ |x-\omega_{0} y|\leq sN\}\lesssim sN^2$.

\medskip\noindent{\it Proof of Item \ref{ab5}, inclusion(\ref{eq:disk}).}
Recall that $b_{\tau}\geq \tau+2$. Since $\omega_{0}$ is diophantine, for $(k,l)\in E_{j}$, $j\leq i_{+}(\rho)-1$ one has $|l/k-\omega_{0}| \gtrsim N_{i_{+}(\rho)}^{-(1+\tau)}$. Since $\Omega_{j}$ satisfies a $(2A,2B)$-twist condition $(2A)^{-1}\leq \pa\omega^{}_{j}\leq 2A$ one has $|c_{l/k}^{(j)}|\gtrsim N_{i_{+}(\rho)}^{-(1+\tau)}$. Now  (\ref{formUn}) shows that that $U_{i_{+}(\rho)}$ contains a disk $\bD(0,N_{i_{+}(\rho)}^{-(\tau+2)})$ and  we observe that from (\ref{502}), $ (\tau+2)(\mu/\iota(\rho))<\tau+2\leq b_{\tau}$ hence \be\bD(0,\rho^{b_{\tau}})\subset  \bD(0,\rho^{ (\tau+2)\mu/\iota(\rho)})\subset U_{i_{+}(\rho)}. \label{eq:diskbis}
\ee
\end{proof}

\subsubsection{Covering the holes with bigger disks} 
 Let us define (compare with (\ref{requirements6.76}))
\be \hat K_{i}=N_{i}^{\ln N_{i}}\ll K_{i}\ll e^{N_{i}/(\ln N_{i})^3} \label{12.330}\ee
and for any 
$D\in\cD_{\rho}:=\cD_{\rho}(U_{i_{+}(\rho)})$ set 
$$\hat D=\bD(c_{D},\hat K_{i_{D}}^{-1}),\qquad \hat \cD_{\rho}=\{\hat D,\ D\in \cD_{\rho}\}.
$$
Notice that for any $a>0$, $\rho\ll_{a} 1$ and $i_{-}(\rho)\leq i_{D}\leq i_{+}(\rho)$ one has 
\be \bar \e_{i_{D}}^{1/a}\ll \hat K_{i_{D}}^{-1}\ll |c_{D}|/4.\label{comphatKi}\ee
Indeed, the inequality of the RHS is due to the fact that  $|c_{D}|>\rho^{b_{\tau}}$ ({\it cf.} Proposition \ref{prop:a10.4}, Item \ref{ab5}) combined with the fact that  $N_{i_{-}(\rho)}^{-1}\lesssim \rho^{1/(1+\tau)}$ ({\it cf.}(\ref{e10.204})).   The inequality of the LHS  is a consequence of (\ref{requirements6.76}).

Let us mention that these disks $\hat D$ are the ones on which we shall later perform a Hamilton-Jacobi Normal Form as described in Proposition \ref{proppropHJ}.

\begin{lemma}\label{lemma:12.4}The elements of $\hat \cD_{\rho}$ are pairwise disjoint and for any $D\in\cD_{\rho}$ one has 
$$ D\subset (1/10)\hat D\subset 6 \hat D\subset U_{i_{D}}, \qquad\hat D\setminus(1/10)\hat D\subset U_{i_{+}(\rho)}.
$$
\end{lemma}
\begin{proof}
Let $D$ and $D'$ be two distinct elements of $\cD_{\rho}$. By Proposition \ref{prop:a10.4},  Item \ref{ab1}, $D\cap D'= \emptyset$  hence from Lemma \ref{cor:6.5}, Item \ref{ni1} $|c_{D}-c_{D'}|\gtrsim N_{i_{+}(\rho)}^{-2}$.  Since $\hat K_{i_{D}}^{-1}+\hat K_{i_{D'}}^{-1}\ll N_{i_{+}(\rho)}^{-2}$ we get that $\bD(c_{D},\hat K_{i_{D}}^{-1})\cap \bD(c_{D'},\hat K_{i_{D'}}^{-1})=\emptyset$.

 Let us now prove $6\hat D\subset U_{i_{D}}$. If $6\hat D$ is not a subset of $U_{i_{D}}$ one has for some $D'\in \cD(U_{i_{D}})$,   $(6\hat D)\cap D'\ne\emptyset$  hence  $|c_{D}-c_{D'}|\leq 6\hat K_{i_{D}}^{-1}+ K_{i_{D'}}^{-1}\ll N_{i_{+}(\rho)}^{-2}$. We can apply Lemma \ref{cor:6.5}, Item \ref{ni1} to deduce $|c_{D}-c_{D'}|\lesssim \bar \e_{i_{-}(\rho)}^{1/2}$; but this implies that $D\cap D'\ne \emptyset$, hence  $D=D'$ (we can apply Proposition \ref{prop:a10.4}, Item \ref{ab1},  since $D,D'\in\cD_{(3/2)\rho}$) and by Proposition \ref{prop:a10.4}, Item \ref{ab4} we obtain $D'\subset U_{i_{D}}$: a contradiction.
 
 Let us prove the second inclusion $\hat D\setminus(1/10)\hat D\subset U_{i_{+}(\rho)}$. If this is not the case then for some $D'\in \cD(U_{i_{+}(\rho)})$ one has $D'\cap (\hat D\setminus(1/10)\hat D)\ne\emptyset$ hence $|c_{D}-c_{D'}|\lesssim K_{i_{D}}^{-1}\ll N_{i_{+}(\rho)}^{-2}$ which implies as before using Lemma \ref{cor:6.5} that $D=D'$. But since $D\subset (1/10)\hat D$ this is leads to a contradiction (otherwise $D'\cap (\hat D\setminus(1/10)\hat D)=\emptyset$). 
 \end{proof}
\begin{rem}\label{rem:lemma3.1}Let us mention (this will be useful in the proof of Theorem \ref{prop:12.7}) that 
$$\sum_{\hat D\in\hat \cD_{\rho}}|\hat D\cap \R|^{1/2}\leq 1.
$$
\end{rem}

\subsubsection{No-Screening Property}
Our key proposition is the following.

\begin{prop} \label{prop:10.4noscreening} For any $D\in \cD(U_{i_{+}(\rho)})$ such that $D\cap \bD(0,\rho)\ne \emptyset$  the triple $(U_{i_{+}(\rho)},\hat D\setminus (1/10)\hat D,\bD(0,\rho^{b_{\tau}} /2) )$ is $(10b_{\tau})^{-1}|\ln \rho|^{-1}$-good (in the sense of Definition \ref{defin:3.1}).
\end{prop}
\begin{proof}
From  Remark (\ref{rem:3.1}) it is enough to prove that for some  $U'\subset U_{i_{+}(\rho)}$ containing both $\bD(0,\rho^{b_{\tau}})$ and $\hat D\setminus(1/10)\hat D$,  the triple $(U',\hat D\setminus(1/10)\hat D,\bD(0,\rho^{b_{\tau}}/2 ))$ is  $(10b_{\tau})^{-1}|\ln \rho|^{-1}$-good.

\begin{lemma}\label{lemma:12.5}There exists a constant $C>0$ such that for any $1\leq s\leq 4/3$, there exists $\rho'\in [s\rho,s\rho+ 10C\rho^2]$ such that 
$$\bD(0,\rho')\cap U_{i_{+}(\rho)}=\bD(0,\rho')\setminus\bigcup_{\substack{D\in \cD(U_{i_{+}(\rho)})\\ D\subset \bD(0,\rho')  }}D.$$
\end{lemma}
\begin{proof}From Lemma \ref{cor:6.5}  the holes of $\cD(U_{i_{+}(\rho)})$ are    $C^{-1}N_{i_{+}(\rho)}^{-2}$-separated (some $C>0$), hence $C^{-2}N_{i_{-}(\rho)}^{-4}$-separated and from  (\ref{e10.204})  they are $C^{-1}\rho^2$-separated (some $C>0$).  On the other hand each of these disks has a radius $\leq 2K_{i_{-}(\rho)}^{-1}\ll \rho^{4}$. Since they are centered on the real line  the conclusion follows.
\end{proof}
From the previous lemma we deduce   the existence of a $\rho'\in [(5/4)\rho,(4/3)\rho]$ such that all the holes $D\in \cD(U_{i_{+}(\rho)})$ of $U_{i_{+}(\rho)}$ intersecting $\bD(0,\rho')$ are indeed included in $\bD(0,\rho')$. We then set
$$U'=U_{i_{+}(\rho)}\cap \bD(0,\rho')=\bD(0,\rho')\setminus\bigcup_{\substack{D\in \cD(U_{i_{+}(\rho)})\\ D\subset \bD(0,\rho')  }}D$$
From (\ref{eq:disk}) we have $\bD(0,\rho^{b_{\tau}})\subset U'$ and for any $D=\bD(c_{D},K_{i_{D}}^{-1})\in \cD(U_{i_{+}(\rho)})$ such that $D\cap \bD(0,\rho)\ne \emptyset$  one has  $\hat D\subset \bD(0,(5/6)\rho')$: indeed, since $D\cap\bD(0,\rho)\ne \emptyset$, $|c_{D}|<\rho+K_{i_{D}}^{-1}<\rho+\rho^{4}$ hence $|c_{D}|+\hat K_{i_{D}}^{-1}<\rho+2\rho^4<(5/6)\rho'$. On the other hand, from   Lemma \ref{lemma:12.4} $\hat D\setminus(1/10)\hat D\subset U'$ ($\hat D\subset \bD(0,\rho')$). 
In this situation we can apply
 Corollary \ref{cor:3.2} with $U=U'$, $B=\bD(0,\rho^{b_{\tau}}/2  )$, $d_{i}=\hat K_{i_{D}}^{-1}$, $\e_{i}=2K_{i_{D}}^{-1}$:  the triple $(U',\hat D\setminus (1/10)\hat D,\bD(0,\rho^{b_{\tau}}/2)  )$ is $A$-good with 
\be A=\frac{\ln (6/5)}{b_{\tau}|\ln \rho|}-(I)\label{510}\ee
where
$$(I):=\sum_{i=i_{-}(\rho)}^{i_{+}(\rho)-1}\#\cC_{i}(\rho) \frac{\ln(\hat K_{i}^{-1}/(20\rho'))}{\ln (2K_{i}^{-1}/(\rho'))}$$
with
$$\cC_{i}(\rho)=\#\{D\in\cD(U_{i_{+}(\rho)}),\ D\cap \bD(0,\rho)\ne \emptyset, \ i_{D}=i\}.
$$

From  (\ref{numberofholes}) of Proposition \ref{prop:a10.4}, (\ref{eq:5.299ante}),
 (\ref{eq:5.299}), (\ref{12.330}), (\ref{requirements6.76}) one has
\begin{align*} (I)&\leq \rho \sum_{i=i_{-}(\rho)}^{i_{+}(\rho)-1} N_{i}^2\frac{\ln(\hat K_{i}^{-1}N_{i_{-}(\rho)}^{\iota(\rho)}/30)}{\ln (2K_{i}^{-1}N_{i_{-}(\rho)}^{\iota(\rho)})}\\
&\leq \rho\sum_{i=i_{-}(\rho)}^{i_{+}(\rho)-1} N_{i}^2\frac{-(\ln N_{i})^2+\ln (N_{i_{-}(\rho)}^{\iota(\rho)}/30)}{-(1/(2(\bar a_{0}+2)))hN_{i}/(\ln N_{i})^2+\ln (N_{i_{-}(\rho)}^{\iota(\rho)}/2)}\\
&\lesssim \rho \sum_{i=i_{-}(\rho)}^{i_{+}(\rho)-1}( N_{i})^{1+ \b/2}\qquad (\rho\ll_{\b}1)
\end{align*}
and since $N_{i}$ is exponentially growing with $i$,
$$(I)\lesssim \rho\times  (N_{i_{+}(\rho)})^{1+\b/2}.$$
From (\ref{eq:5.101}) we thus get
\be (I)\lesssim \rho^{1-(1+\b/2)\mu/\iota(\rho)} \leq \rho^{\b^2/4}
\ee
and from (\ref{510}), 
if $\rho\ll_{\b}1$
$$ \frac{1}{10 b_{\tau}}\frac{1}{|\ln \rho|}\leq A\qquad \textrm{(some\ } C>0).$$
\end{proof}

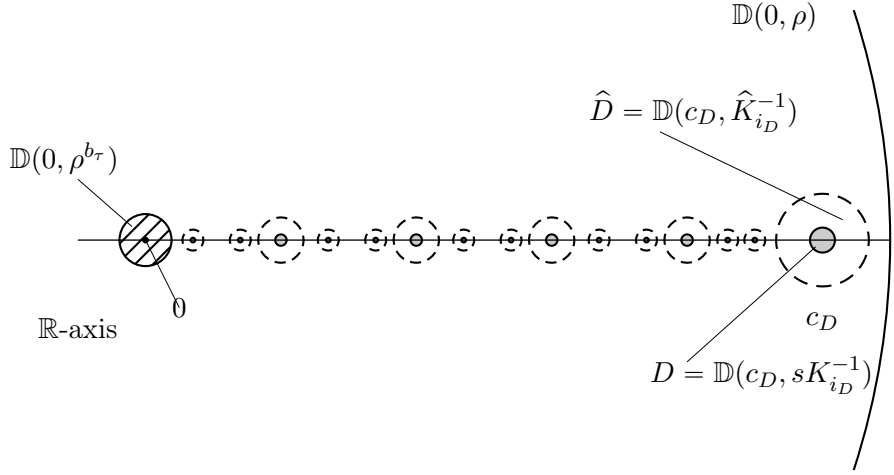
\begin{figure}
\begin{center}
\definecolor{mygray}{gray}{0.8}
\psset{unit=.9cm}
\begin{pspicture}(-6,-4)(6,4)

%
\psarc(-5,0){11}{-18}{18}
\pscircle[fillstyle=hlines](-5,0){0.4}

\pscircle[linestyle=dashed](5,0){0.7}
\pscircle[fillstyle=solid,fillcolor=mygray](5,0){0.2}

\pscircle[linestyle=dashed](3,0){0.35}
\pscircle[fillstyle=solid,fillcolor=mygray](3,0){0.1}

\pscircle[linestyle=dashed](1,0){0.35}
\pscircle[fillstyle=solid,fillcolor=mygray](1,0){0.1}

\pscircle[linestyle=dashed](-1,0){0.35}
\pscircle[fillstyle=solid,fillcolor=mygray](-1,0){0.1}

\pscircle[linestyle=dashed](-3,0){0.35}
\pscircle[fillstyle=solid,fillcolor=mygray](-3,0){0.1}

\pscircle[linestyle=dashed](3.6,0){0.17}
\pscircle[fillstyle=solid,fillcolor=mygray](3.6,0){0.05}

\pscircle[linestyle=dashed](4,0){0.17}
\pscircle[fillstyle=solid,fillcolor=mygray](4,0){0.05}

\pscircle[linestyle=dashed](1.7,0){0.17}
\pscircle[fillstyle=solid,fillcolor=mygray](1.7,0){0.05}

\pscircle[linestyle=dashed](2.4,0){0.17}
\pscircle[fillstyle=solid,fillcolor=mygray](2.4,0){0.05}

\pscircle[linestyle=dashed](-0.3,0){0.17}
\pscircle[fillstyle=solid,fillcolor=mygray](-0.3,0){0.05}

\pscircle[linestyle=dashed](0.4,0){0.17}
\pscircle[fillstyle=solid,fillcolor=mygray](0.4,0){0.05}

\pscircle[linestyle=dashed](-2.3,0){0.17}
\pscircle[fillstyle=solid,fillcolor=mygray](-2.3,0){0.05}

\pscircle[linestyle=dashed](-1.6,0){0.17}
\pscircle[fillstyle=solid,fillcolor=mygray](-1.6,0){0.05}

\pscircle[linestyle=dashed](-3.6,0){0.17}
\pscircle[fillstyle=solid,fillcolor=mygray](-3.6,0){0.05}

\pscircle[linestyle=dashed](-4.3,0){0.17}
\pscircle[fillstyle=solid,fillcolor=mygray](-4.3,0){0.05}

\psdot[dotsize=0.1](-5,0)
\rput(-4.5,-1){$0$}
\rput(3.1,1.9){$\hat D=\bD(c_{D},\hat K_{i_{D}}^{-1})$}
\rput(4.1,-2){$D=\bD(c_{D},sK_{i_{D}}^{-1})$}
\rput(4.3,3.3){$\bD(0,\rho)$}
\rput(-6.2,1.2){$\bD(0,\rho^{b_{\tau}})$}
\rput(-6,-1.3){$\bR$-axis}
\rput(5,-1.2){$c_{D}$}
\psline[linewidth=0.5pt](-6,0)(6,0)
\psline[linewidth=0.1pt](-4.5,-1)(-5,0)
\psline[linewidth=0.1pt](2.6,1.6)(5.3,0.3)
\psline[linewidth=0.1pt](-6,0.9)(-5.2,0.2)
\psline[linewidth=0.1pt](3,-1.8)(4.9,-0.1)


%
\end{pspicture}
\end{center}
\caption{Adapted KAM Normal Forms ($\omega_{0}$ diophantine) in the complex $r$-plane.  The triple $(U^{(\rho)},\hat D\setminus (1/10)\hat D,\bD(0,\rho^{b_{\tau})}))$ is $C_{b}|\ln\rho|^{-1}$-good}\label{Figure2}
\end{figure}

\subsection{Coexistence of KAM, BNF and HJ Normal Forms on the adapted KAM domain}
\begin{notation}If $W_{h,U}$ is a $\s$-symmetric holed domain, we denote by $\cN\cF_{\s}(W_{h,U})$ (resp.  $\cN\cF_{ex,\s}(W_{h,U})$) the set of triples $(\Omega,F,g)$ with $\Omega\in\ti\cO_{\s}(U)$, $F\in \cO_{\s}(W_{h,U})$, $g\in\ti{\rm Symp}_{\s}(W_{h,U})$ (resp. $g\in\ti{\rm Symp}_{ex,\s}(W_{h,U})$).
\end{notation}
\begin{prop}[Adapted Normal Forms]\label{adaptedNFAA}Let $\Omega\in\ti\cO_{\s}(U)$ and $F\in\cO_{\s}(W_{h,U})$ satisfy (\ref{eq:11.257}), (\ref{eq:11.257,5}), (\ref{eq:11.258}).  For any $\b\ll 1$ and $\rho\ll_{\b}1$ the following holds:

\smallskip\noindent{\bf(KAM):}
Adapted KAM Normal Form (Proposition \ref{prop:1.enonce}). Let $D\in\cD_{\rho}(U_{i_{+}(\rho)})$.
\begin{align}& [W_{h, U_{i_{\pm}(\rho)}}]\quad g_{1,i_{\pm}(\rho)}^{-1}\circ \Phi_{\Omega}\circ f_{F}\circ g_{1,i_{\pm}(\rho)}=\Phi_{\Omega_{i_{\pm}(\rho)}}\circ f_{F_{i_{\pm}(\rho)}}\label{12348}\\
 & [W_{h,U_{i_{+}(\rho)}}]\quad g_{i_{D},i_{+}(\rho)}^{-1}\circ \Phi_{\Omega_{i_{D}}}\circ f_{F_{i_{D}}}\circ g_{i_{D},i_{+}(\rho)}=\Phi_{\Omega_{i_{+}(\rho)}}\circ f_{F_{i_{+}(\rho)}}\label{12352}\\
 & [W_{h,U_{i_{D}}}]\quad g_{i_{-}(\rho),i_{D}}^{-1}\circ \Phi_{\Omega_{i_{-}(\rho)}}\circ f_{F_{i_{-}(\rho)}}\circ g_{i_{-}(\rho),i_{D}}=\Phi_{\Omega_{i_{D}}}\circ f_{F_{i_{D}}}.\label{12352bis}\\
&\|g_{1,i_{+}(\rho)}-id\|_{C^1}\lesssim \bar\e^{1/2}\leq \rho^{m/2}\label{12.353}\\
&\|g_{i_{D},i_{+}(\rho)}-id\|_{C^1}\leq \bar \e_{i_{D}}^{1/2}\label{12354}\\
& \|F_{i_{+}(\rho)}\|_{W_{h,U^{(\rho)}}}\lesssim  \exp(-(1/\rho)^{\frac{2}{1+\tau} -\b}).\label{12.355}
\end{align}
Note that  $(\Omega_i, F_i, g_i)\in\cN\cF_{ex,\s}(W_{h,U_{i}})$ and $\Omega_{i}\in\cT\cC(2A,2B)$ .

\smallskip\noindent{\bf(HJ):} Hamilton-Jacobi Normal Form. (Proposition \ref{proppropHJ}).  For any  $D\in\cD_{\rho}(U_{i_{+}(\rho)})$  there exists $\check D=\subset \hat D$ and $(\Omega_{\hat D}^{HJ},F_{\hat D}^{HJ},g_{\hat D}^{HJ})\in \cN\cF_{\s}(W_{h/9,\hat D\setminus\check D})$ such that 
\begin{align}&(g^{HJ}_{\hat D})^{-1}\circ \Phi_{\Omega_{i_{D}}}\circ f_{F_{i_{D}}}\circ g^{HJ}_{\hat D}=\Phi_{\Omega^{HJ}_{\hat D}}\circ f_{F^{HJ}_{\hat D}} \qquad [W_{h/9,\hat D\setminus \check D}]\label{12356}\\ 
& \|g_{\hat D}^{HJ}-id\|_{C^1}\lesssim \bar\e_{i_{D}}^{1/9}\label{12357}\\
&\Omega^{HJ}_{\hat D}\in\cT\cC(2A,2A)\\
&\|F^{HJ}_{\hat D}\|_{W_{{h/9}, (\hat D\setminus \check D)}}\lesssim \exp(-(1/\rho)).\label{12358}
\end{align}
The triple $(\Omega^{HJ}_{D},\hat D,\check D)$ satisfies the Extension Principle of Proposition \ref{proppropHJ}.

\smallskip\noindent{\bf(BNF):} Birkhoff Normal Form (Proposition \ref{BNFprop}): 

There exists $(\Omega_{\rho}^{BNF},F_{\rho}^{BNF},g_{\rho}^{BNF})\in\cN\cF_{ex,\s}(W_{h,\bD(0,\rho^{b_{\tau}})})$ such that  
\begin{align} &  (g_{\rho}^{BNF})^{-1}\circ \Phi_{\Omega}\circ f_{F}\circ g_{\rho}^{BNF}= \Phi_{\Omega_{\rho}^{BNF}}\circ f_{F_{\rho}^{BNF}},\qquad (W_{h,\bD(0,\rho^{b_{\tau}})})   \label{12347}\\
& \|g_{\rho}^{BNF}-id\|_{ C^1 }\lesssim \rho^{m-1}.\label{12350}\\
&\Omega_{\rho}^{BNF}\in\cT\cC(2A,2B)\\
&\|F_{\rho}^{BNF}\|_{W_{h,\bD(0,\rho^{b_{\tau}})}}\lesssim \exp(-(1/\rho)^{1-\b})\label{12349}
\end{align}
\end{prop}
\begin{proof}

\smallskip\noindent {\it KAM:} This is just the content of Proposition \ref{prop:1.enonce}. For inequality (\ref{12.355}) we note that 
 from (\ref{8.114}), (\ref{requirements6.76}), (\ref{eq:5.101})
\begin{align*} \|F_{i_{+}(\rho)}\|_{h,U_{i_{+}(\rho)}}&\lesssim  \exp(-N_{i_{+}(\rho)}/(\ln (N_{i_{+}(\rho)}))^2)\\
&\lesssim  \exp(-\rho^{-(\mu/\iota(\rho))^{-}})\\
&\lesssim \exp(-(1/\rho)^{\frac{2}{1+\tau} -\b}). 
\end{align*}

\smallskip\noindent {\it HJ:} Let $D\in\cD_{\rho}(U_{i_{+}(\rho)})$ where  $D=\bD(c_{D},s_{D} K_{i_{D}}^{-1})$, $\omega_{i_{D}}(c_{D})=p/q$, $q\leq N_{i_{D}}$, $p\wedge q=1$,  be one of the disk obtained in Proposition \ref{prop:a10.4}, Item \ref{ab4}. By Lemma \ref{lemma:12.4} the disk $6\hat D=\bD(c_{D},6\hat K_{i_{D}}^{-1})$ is included in $U_{i_{D}}$. We observe that  $6\hat K_{i_{D}}^{-1}<|c_{D}|/4$ ({\it cf.} \ref{comphatKi})).  Since
$$\min(6\hat K_{i_{D}}^{-1},|c_{D}|/4)=6\hat K_{i_{D}}^{-1}<(Aq)^{-8}\quad \textrm{and}\quad \|F_{i_{D}}\|_{h,6\hat D}\lesssim \bar \e_{i_{D}}<(6\hat K_{i_{D}}^{-1})^{\bar a_{3}}
 $$
 (the last inequality comes also from (\ref{comphatKi})) 
condition (\ref{8.179}), (\ref{8.177}) are satisfied and we can apply  Proposition \ref{proppropHJ}  on Hamilton-Jacobi Normal Forms to $\Phi_{\Omega_{i_{D}}}\circ f_{F_{i_{D}}}$ on the domain $W_{h, \hat D}\subset W_{h, U_{i_{D}}}$ with 
$\hat \rho=\hat K_{i_{D}}^{-1}$:
there exists a disk $\check D\subset \hat D$ 
\be\check D:=\bD(c_{\check D},\rho_{\check D})\subset (1/10) \hat D:=\bD(c_{D},(1/10)\hat K_{i_{D}}^{-1})\subset U_{i_{D}}\label{disks}\ee
and $(\Omega_{\hat D}^{HJ},F_{\hat D}^{HJ},g_{\hat D}^{HJ})\in \cN\cF_{\s}(W_{h/9,\hat D\setminus\check D})$ satisfying (\ref{12356})
 \begin{align} 
&\|g_{\hat D}^{HJ}-id\|_{C^1}   \lesssim q \bar\e_{i_{D}}^{1/8}\leq \bar\e_{i_{D}}^{1/9}\label{14.257b}\\
&\|F^{HJ}_{\hat D}\|_{W_{h/9, (\hat D\setminus \check D)}}
\lesssim \exp(-(\hat K_{i_{D}}/N_{i_{D} }  )^{1/4}).\label{14.270}
\end{align}
To obtain inequality (\ref{12358}) we observe that  since $\hat K_{i_{D}}=N_{i_{D}}^{\ln N_{i_{D}}}$ with $i_{+}(\rho)\geq i_{D}\geq i_{-}(\rho)$ we get 
\begin{align*}-(\hat K_{i_{D}}/N_{i_{D}} )&\lesssim  -N_{i_{-}(\rho)}^{\ln N_{i_{-}(\rho)}}N_{i_{-}(\rho)}^{-2} \qquad (\ref{e9.155}) \\
&\lesssim -(1/\rho)^{\iota(\rho)^{-1}(-2+|\ln \rho|/\iota(\rho))}\qquad (\ref{eq:5.299ante})\\
&\lesssim -(1/\rho)^{|\ln \rho|/(2(1+\tau)^2)}\qquad (\ref{eq:5.299}), \rho \ll 1\\
&\lesssim -(1/\rho),\qquad \rho\ll_{\tau} 1.
\end{align*}

\smallskip\noindent {\it BNF:} We observe that    $\bD(0,\rho^{b_{\tau}})\subset \bD(0,\rho)$ and apply Proposition \ref{BNFprop} to $(\Omega,F)$ on $e^{h}W_{h,\bD(0,\rho)}$  (we use the smallness condition 
(\ref{eq:11.258})).
\end{proof}

\subsection{Comparision Principle}
We now use the result of Section \ref{sec:comparingnf} to show that these various Normal Forms match to some very good order of approximation.
\begin{lemma}[Comparing Adapted Normal Forms] For any  $\b\ll 1$, and $\rho\ll_{\b}1$ 
\be \|\Omega_{i_{+}(\rho)}-\Omega^{BNF}_{\rho}\|_{(1/2)\bD(0,\rho^{b_{\tau}})}\leq \exp(-(1/\rho)^{\frac{2}{1+\tau}-\b}).\label{14.257}
\ee
and for any $D\in\cD_{\rho}$ there exists $\g_{D}\leq \hat K_{i_{D}}^{-2}$
\be\|\Omega_{i_{+}(\rho)}-\Omega_{\hat D}^{HJ}(\cdot+\g_{D})\|_{(4/5)\hat D\setminus (1/5)\hat D}\leq \exp(-(1/\rho)^{\frac{2}{1+\tau}-\b}). \label{14.258}
\ee
\end{lemma}
\begin{proof}

\smallskip\noindent 1) {\it Proof of  (\ref{14.257}).} From (\ref{12347}),  (\ref{12348}) and the fact that 
$$W_{h,\bD(0,\rho^{b_{\tau}})}\subset W_{h,\bD(0,\rho^{b_{\tau}})}\cap W_{h, U_{i_{+}(\rho)} }
$$
one has on $g_{1,i_{+}(\rho)}(  W_{h,\bD(0,\rho^{b_{\tau}})})\cap g_{\rho}^{BNF}(  W_{h,\bD(0,\rho^{b_{\tau}})})$
$$g_{1,i_{+}(\rho)}\circ \Phi_{\Omega_{i_{+}(\rho)}}\circ f_{F_{i_{+}(\rho)}}\circ (g_{1,i_{+}(\rho)})^{-1} = g_{\rho}^{BNF}\circ \Phi_{\Omega_{\rho}^{BNF}}\circ f_{F_{\rho}^{BNF}}\circ (g_{\rho}^{BNF})^{-1}.
$$
We can then apply Propositions \ref{comparingBNF}-\ref{comparingBNFCC} with $\rho_{2}=\rho^{b_{\tau}}$, $\rho_{1}=0$, $\d=\rho^{b_{\tau}}/2$, $\e=\rho^{(m-1)/2}$, $\nu=\exp(-(1/\rho)^{2/(1+\tau)-\beta})$ since from (\ref{12349}), (\ref{12350}), (\ref{12.353}), (\ref{12.355}) one sees that condition (\ref{11308}) reads
$$\bar C\rho^{(m-1)/2}\leq \rho^{b_{\tau}}/4<\rho^{b_{\tau}}\quad  \textrm{and}\quad \bar C(\rho^{\hs b_{\tau}} /2 )^{-\bar a_{5}}\exp(-(1/\rho)^{2/(1+\tau)-\beta})<1
$$
and is satisfied for $\rho\ll 1$.  Since $g_{1,i_{+}(\rho)}$ and $g_{\rho}^{BNF}$ are exact  symplectic we then get $\|\Omega_{i_{+}(\rho)}-\Omega^{BNF}_{\rho}\|_{\bD(0,(1/2)\rho^{b_{\tau}})}\leq \bar C \rho^{-(b_{\tau}+1)\bar a_{6}}\exp(-(1/\rho)^{\frac{2}{1+\tau}-\b})$ which is $\leq \exp(-(1/\rho)^{\frac{2}{1+\tau}-2\b})$ if $\rho$ is small enough.

\smallskip\noindent 2) {\it Proof of  (\ref{14.258}).}
Similarly, from (\ref{12352}), (\ref{12356}) 
one has on the set $$g^{HJ}_{\hat D}(W_{h/9,\hat D\setminus(1/5)\hat D})\cap g_{i_{D},i_{+}(\rho)}( W_{h/9,\hat D\setminus (1/5)\hat D})$$
$$
g^{HJ}_{\hat D}\circ \Phi_{\Omega^{HJ}_{\hat D}}\circ f_{F^{HJ}_{\hat D}}\circ (g^{HJ}_{\hat D})^{-1}=g_{i_{D},i_{+}(\rho)} \circ \Phi_{\Omega_{i_{+}(\rho)}}\circ f_{F_{i_{+}(\rho)}}\circ (g_{i_{D},i_{+}(\rho)})^{-1}$$
and from
(\ref{12354}) (\ref{12.355}), (\ref{12357}), (\ref{12358}),  we see that  Propositions \ref{comparingBNF}-\ref{comparingBNFCC}   apply with $c=c_{i_{D}}$, $\e=\bar\e_{i_{D}}^{1/9}$, $\d=K_{i_{D}}^{-1}/20$, $\rho_{1}=(1/10)\hat  K_{i_{D}}^{-1}$, $\rho_{2}=K_{i_{D}}^{-1}<|c_{i_{D}}|/4$ since condition (\ref{11308}) is implied by 
$$\bar C\bar \e_{i_{D}}^{1/9}\leq \hat K_{i_{D}}^{-1}/80<\hat K_{i_{D}}^{-1}/10\quad  \textrm{and}\quad \bar C(20\hat K_{i_{D}})^{\hs\bar a_{5}}\exp(-(1/\rho)^{2/(1+\tau)-\beta})<1
$$
which is  satisfied ({\it cf.} (\ref{requirements6.76}), (\ref{12.330})) if $\rho$ is small enough. We then get for some $\gamma_{D}\in\R$, $|\gamma_{D}|\lesssim  \bar C\bar\e_{i_{D}}^{1/9} \leq \hat K_{i_{D}}^{-2}$ ({\it cf.} (\ref{comphatKi}))  that on the annulus $(4/5)\hat D\setminus (1/5)\hat D$ one has $|\Omega_{i_{+}(\rho)}-\Omega_{\hat D}^{HJ}(\cdot+\g_{D})|\leq \exp(-(1/\rho)^{\frac{2}{1+\tau}-2\b})$.
\end{proof}

\section{Adapted Normal Forms: $\omega_{0}$ Liouvillian (CC case)}\label{sec:10.2}
Let $h> 0$, $0<\bar \rho<1$,  $\Omega\in\ti \cO_{\s}(e^{10h}\bD(0,\bar\rho))$, $F\in \cO_{\s}(e^{10h}W_{h,\bD(0,\bar \rho)})$ such that 
\begin{align}& \forall \ r\in\R,\  A^{-1}\leq  (2\pi)^{-1}\pa^2\Omega(r)\leq A,\quad \textrm{and}\ \ \|(2\pi)^{-1}D^3\Omega\|_{\C}\leq B.\label{eq:11.257bis}\\
&\omega_{0}:=(2\pi)^{-1}\pa^2\Omega(0)\in \R\setminus\Q\label{eq:11.257,5bis}\\
&\forall\ 0<\rho\leq \bar\rho,\quad   \|F\|_{e^{10h}W_{h,\bD(0,\rho)}}\leq \rho^{m},\label{eq:11.258bis}
\end{align}
where  \be m=4+\max(\bar a_{1},2000A\bar a_{2},\bar a_{3},\bar a_{5})\label{defmfinal}\ee
($\bar a_{1},\bar a_{2},\bar a_{3},\bar a_{5}$ are the constants appearing in Propositions \ref{nprop:9.2}, \ref{prop:1.enonce}, \ref{proppropHJ}, \ref{comparingBNF}).

Using  the notations of Subsection \ref{sec:6.2.1}, let $(p_{n}/q_{n})_{n}$ be the sequence of convergents of $\omega_{0}$: 
\be \frac{1}{2q_{n+1}q_{n}}\leq |\omega_{0}-(p_{n}/q_{n})|\leq \frac{1}{q_{n+1}q_{n}}.\label{e11.228}\ee
 \be  \forall \ 0<k<q_{n},\ \forall \ l\in\Z,\   |\omega_{0}-(l/k)|>\frac{1}{2kq_{n}}\label{aBNF7.122}.\ee

We assume that $n$ is large enough and we set 
 \be \rho_{n}=\frac{10A}{q_{n+1}q_{n}} \leq \bar \rho/10.\label{defrhonbis}\ee

 We introduce
 \be \bar \e:= \max_{0\leq j\leq 3}\|D^jF\|_{W_{2 h,\bD(0,10\rho_{n})}}\lesssim (10 \rho_{n})^{\hs m-3} \leq \rho_{n}^{2000A \bar a_{2}}.\label{e11.234}\ee
\subsection{Adapted KAM domains} 
Since 
 Condition   (\ref{enew7.135}) is  satisfied 
 we can  apply Proposition \ref{prop:1.enonce} (with $\bar \rho_{\textrm{Prop.\ \ref{prop:1.enonce} }}=10\rho_{n}$) and define holed domains $U_{i}$, functions $\Omega_{i}$, $F_{i}$, $\omega^{}_{i}$ {\it etc.} In particular for $0<t$
\be\begin{cases} &U_{i}\cap\bD(0,t)=\bD(0,t)\setminus\bigcup_{j=1}^{i-1}\bigcup_{(k,l)\in E_{j} }\bD(c_{l/k}^{(j)},s_{j,i-1}K_{j}^{-1}),\\
&s_{j,i-1}=e^{\sum_{m=j}^{i-1}{\d_{m}}}\in [1,2]\end{cases}\label{abcd10.207}\ee
where
$$E_{j}\subset\{(k,l)\in\Z^2,\ 0<k<N_{j},\ 0\leq |l|\leq N_{j}\},\qquad \omega_{j}^{}(c_{l/k}^{(j)}) =l/k.$$
Note that from (\ref{e11.234}) and the definition (\ref{requirements6.76}) of $K_{j}$
\be K_{j}^{-1}\leq  \bar\e^{\hs \frac{1}{2(\bar a_{0}+2)}}\leq   \rho_{n}^{1000 A}.\label{e11.236}
\ee
\begin{lemma} Let $j$ be such that $N_{j}<q_{n+1}/(10A)^2$ and $(k,l)\in E_{j}$.
\begin{enumerate}
\item \label{ba1}If $(k,l)\in \Z(q_{n},p_{n})$ one has $l/k=p_{n}/q_{n}$ and 
\be (40A^2)^{-1}\rho_{n}\leq  \frac{(2A)^{-1}}{2q_{n+1}q_{n}}\leq  |c_{p_{n}/q_{n}}^{(j)}|\leq \frac{(2A)}{q_{n+1}q_{n}}\leq \rho_{n}/5. \label{10.147}\ee
\item\label{ba2} If $(k,l)\notin \Z(q_{n},p_{n})$
\be    |c_{l/k}^{(j)}|\geq 4\rho_{n}. \label{a10.147}\ee
\end{enumerate}
\end{lemma}
\begin{proof}
Item \ref{ba1} comes from  (\ref{e11.228}) and    the twist condition (\ref{7.139twist}). 

To prove Item \ref{ba2} we observe that if $(k,l)\notin \Z(q_{n},p_{n})$ 
$$ |\omega_{0}-\frac{l}{k} | \geq |\frac{l}{k}-\frac{p_{n}}{q_{n}}|-|\omega_{0}-\frac{p_{n}}{q_{n}}|\geq \frac{1}{kq_{n}}-\frac{1}{q_{n}q_{n+1}}\geq \frac{99A^2}{q_{n}q_{n+1}}\geq 9A\rho_{n}$$
and from the twist condition (\ref{7.139twist}) we get $|c^{(j)}_{l/k}|\geq 4\rho_{n}$.
\end{proof}

 For $n\in\N^*$ define $i_{n}^-$ as the unique index $i$ such that 
 \be N_{i_{n}^- -1}\leq  q_{n}< N_{i_{n}^-}
 \ee
 and $i_{n}^+$ as the unique index (see the definition of the sequence $N_{i}$ in (\ref{requirements6.76})) such that 
 \be\frac{(3/4)q_{n+1}}{(10A)^2}\leq N_{i_{n}^+}< \frac{q_{n+1}}{(10A)^2}.
 \ee
We define 
$$c_{n}=c_{p_{n}/q_{n}}^{(i_{n}^-) },\qquad D_{n}:=\bD(c_{p_{n}/q_{n}}^{(i_{n}^-)},s_{i_{n}^-,i_{n}^+-1}K_{i_{n}^-}^{-1}),\qquad \hat D_{n}=\bD(c_{n},|c_{n}|/24)$$
\be U^{(n)}:=U_{i_{n}^+}\cap\bD(0,\rho_{n}).\label{defU(n)}\ee
Note that from (\ref{10.147})
\be  (40A^2)^{-1}\rho_{n}\leq  |c_{n}| \leq \rho_{n}/5.\label{10.147bis}
\ee
\begin{prop}\label{prop:11.2CC}For $n$ large enough,
\begin{enumerate}
\item\label{aab0} $ \bD(0,\rho_{n})\subset U_{i_{n}^-}$.
\item\label{aab1} One has 
$\cD_{\rho_{n}}:=\cD(U^{(n)})=\{D_{n}\}\label{e11.233}$
\item \label{aab2}One has the following inclusion $6\hat D_{n}\subset U_{i_{n}^-}$.
\item \label{aab3}One has $\bD(0,q_{n+1}^{-6})\subset U_{i_{n}^+}$.
\item\label{aab4} The triple $(U^{(n)},\hat D_{n}\setminus(1/10)\hat D_{n}, \bD(0,q_{n+1}^{-6} /2))$ is $1/(10|\ln \rho_{n}|)$-good. 
\end{enumerate}
\end{prop}  
\begin{proof}
{\it of Item \ref{aab0}.} If $j<i_{n}^-$ and $(k,l)\in E_{j}$ one has $0<k<N_{i_{n}^- -1}\leq q_{n}$ hence from (\ref{a10.147}) $|c_{l/k}^{(j)}|\geq 4\rho_{n}$ and from (\ref{e11.236}) $|c_{l/k}^{(j)}|-2K_{j}^{-1}\geq 3\rho_{n}$. The conclusion then follows from (\ref{abcd10.207}) applied with $i=i_{n}^-$.

\medskip\noindent{\it Proof of  Item \ref{aab1}. }  From  Item \ref{aab0}, equality  (\ref{abcd10.207})  can be written 
$$U_{i_{n}^+}\cap\bD(0,\rho_{n})=\bD(0,\rho_{n})\setminus\bigcup_{j=i_{n}^-}^{i_{n}^+-1}\bigcup_{(k,l)\in E_{j} }\bD(c_{l/k}^{(j)},s_{j,i-1}K_{j}^{-1}).$$
We observe that $(q_{n},p_{n})\in E_{i_{n}^-}$ and from (\ref{10.147}), (\ref{e11.236}) one sees that $D_{n}=\bD(c_{p_{n}/q_{n}}^{(i_{n}^-)},s_{i_{n}^-,i_{n}^+-1} K_{i_{n}^-}^{-1}) \subset \bD(0,\rho_{n})$.  More generally, if $(k,l)\in E_{j}$, $i_{n}^-\leq j\leq i_{n}^+-1$ and $(k,l)\notin \Z(q_{n},p_{n})$,  one has $N_{j}\leq q_{n+1}/(10A)^2$  and (\ref{a10.147}), (\ref{e11.236}) give that $\bD(0,c^{(j)}_{l/k},2K_{j}^{-1})\cap \bD(0,\rho_{n})=\emptyset$. Since the sets $\bD(0,c^{(j)}_{p_{n}/q_{n}},s_{q_{n},i_{+}^n-1}K_{j}^{-1})$, $i_{n}^-\leq j\leq i_{n}^+$, form a nested decreasing (for the inclusion) sequence of disks one gets 
$$U_{i_{n}^+}\cap \bD(0,\rho_{n})=\bD(0,\rho_{n})\setminus D_{n}.$$

\medskip\noindent{\it Proof of  Item \ref{aab2}. }  This comes from the fact that $|c_{n}|+6|c_{n}|/4\leq \rho_{n}$.

\medskip\noindent{\it Proof of  Item \ref{aab3}. } This comes from Item \ref{ba1} and the fact that $|c_{n}|-|c_{n}|/4\geq q_{n+1}^{-6}$ as is clear from the LHS inequality of (\ref{10.147}).

\medskip\noindent{\it Proof of  Item \ref{aab4}. } 
Notice that from (\ref{10.147}) $5\leq \rho_{n}/|c_{n}|\leq 40A^2$ and that $2K_{i_{n}^-}^{-1}\leq \rho_{n}^{1000A}$.
We use Corollary  \ref{cor:3.2};  we have to evaluate  
\begin{align*} I&=\frac{\ln (|c_{n}|/(4\rho_{n}))}{\ln (q_{n+1}^{-6}/ (2\rho_{n}))}-\frac{\ln(|c_{n}|/8\rho_{n})}{\ln (2K_{i_{n}^-}^{-1}/\rho_{n})}\\
&\geq\frac{\ln(20)}{7|\ln \rho_{n}|}-\frac{\ln(320A^{2})}{(1000A-1)|\ln \rho_{n}|}\\
&\geq \frac{1}{10|\ln\rho_{n}|}.
\end{align*}

\end{proof}
\subsection{Adapted Normal Forms}

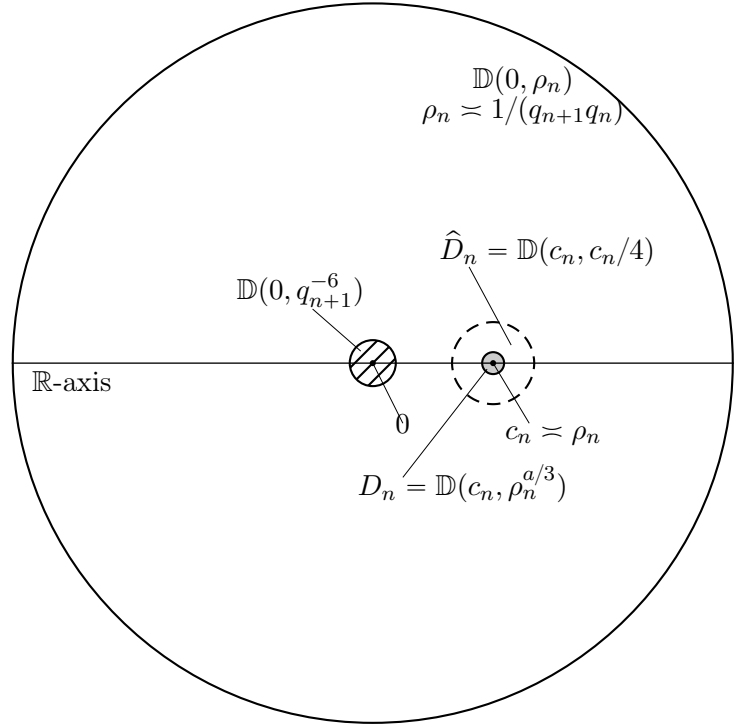
\begin{figure}
\begin{center}
\definecolor{mygray}{gray}{0.8}
\def\pictureLiouville{
\psset{unit=.8cm}
\begin{pspicture}(-6,-6)(6,6)

%
\pscircle(0,0){6}
\pscircle[fillstyle=hlines](0,0){0.4}

\pscircle[linestyle=dashed](2,0){0.7}
\pscircle[fillstyle=solid,fillcolor=mygray](2,0){0.2}
\psdot[dotsize=0.1](2,0)
\psdot[dotsize=0.1](0,0)
\rput(0.5,-1){$0$}
\rput(2.9,1.9){$\hat D_{n}=\bD(c_{n},c_{n}/4)$}
\rput(1.5,-2){$D_{n}=\bD(c_{n},\rho_{n}^{a/3})$}
\rput(2.5,4.7){$\bD(0,\rho_{n})$}
\rput(2.5,4.2){$\rho_{n}\asymp 1/(q_{n+1}q_{n})$}
\rput(-1.2,1.2){$\bD(0,q_{n+1}^{-6})$}
\rput(-5,-0.3){$\bR$-axis}
\rput(3,-1.2){$c_{n}\asymp \rho_{n}$}
\psline[linewidth=0.5pt](-6,0)(6,0)
\psline[linewidth=0.1pt](2,0)(2.6,-1)
\psline[linewidth=0.1pt](0.5,-1)(0,0)
\psline[linewidth=0.1pt](0.5,-1.9)(1.9,-0.1)
\psline[linewidth=0.1pt](1.6,1.6)(2.3,0.3)
\psline[linewidth=0.1pt](-1,0.9)(-0.2,0.2)


%
\end{pspicture}}
\psscalebox{1}\pictureLiouville
\end{center}
\caption{Adapted KAM Normal Forms (CC Case) in the complex $r$-plane. The triple $(U^{n},\hat U^{n}, \bD(0,q_{n+1}^{-6}))$ is $1/(10|\ln \rho_{n}|)$-good.}
\end{figure}

\begin{prop}\label{prop:adaptedCC}Let $\Omega\in\ti\cO_{\s}(U)$ and $F\in\cO_{\s}(W_{h,U})$ satisfy (\ref{eq:11.257}), (\ref{eq:11.257,5}), (\ref{eq:11.258}). Let  $0<\b\ll 1$ and $n\gg_{\b} 1$ such that  
\be q_{n+1}\geq q_{n}^{10}.\label{condqnqn+1}
\ee

\smallskip\noindent {\bf(KAM):}
Adapted KAM Normal Form ((Proposition \ref{prop:1.enonce})): One has $(\Omega_i, F_i, g_i\in\cN\cF_{ex,\s}(W_{h,U_{i}})$, $\Omega_{i}\in\cT\cC(2A,2B)$ and   
\begin{align}&   g_{1,i_{n}^\pm}^{-1}\circ \Phi_{\Omega}\circ f_{F}\circ g_{1,i_{n}^\pm}=\Phi_{\Omega_{i_{n}^\pm}}\circ f_{F_{i_{n}^\pm}}\qquad [W_{h,U_{i^\pm_{n}}}]\label{13.264ccante}\\
&  g_{i_{n}^-,i_{n}^+}^{-1}\circ  \Phi_{\Omega_{i_{n}^-}}\circ f_{F_{i_{n}^-}}\circ g_{i_{n}^-,i_{n}^+}=\Phi_{\Omega_{i_{n}^+}}\circ f_{F_{i_{n}^+}} \qquad  [W_{h,U^{(n)}}] \label{13.264cc}\\
&\|g_{1,i_{n}^+}-id\|_{C^1}, \|g_{i_{n}^-,i_{n}^+}-id\|_{C^1}\lesssim \bar \e^{1/2}\leq \rho_{n}^{m/3}\label{F255a}\\
&\|F_{i_{n}^+}\|_{W_{h,U^{(n)}}}\lesssim \exp(-q_{n+1}^{1-\b}).\label{F255}
\end{align}

\smallskip\noindent{\bf (HJ):} 
Hamilton-Jacobi Normal Form  (Proposition \ref{proppropHJ}). 

There exists $(\Omega_{{n}}^{HJ}, F_{{n}}^{HJ}, g_{{n}}^{HJ})\in\cN\cF_{\s}(W_{h/9,\hat D_{n}\setminus\check D_{n}})$  such that 
\begin{align}&  (g^{HJ}_{n})^{-1}\circ \Phi_{\Omega_{i_{n}^-}}\circ f_{F_{i_{n}^-}}\circ g^{HJ}_{n}=\Phi_{\Omega^{HJ}_{n}}\circ f_{F^{HJ}_{n}} \qquad  [W_{{h/9}, (\hat D_{n}\setminus \check D_{n})}] \label{13.265cc}\\
& \|g_{n}^{HJ}-id\|_{C^1}\lesssim \bar\e_{i_{n}^-}^{1/9}\leq \rho_{n}^{m/9}\label{F258}\\
&\Omega_{n}^{HJ}\in\cT\cC(2A,2A)\\
&\|F^{HJ}_{n}\|_{W_{{h/9}, (\hat D_{n}\setminus \check D_{n})}}\lesssim \exp(-q_{n+1}^{(1/4)-\b}).\label{F258a} \end{align}
The triple $(\Omega^{HJ}_{D},\hat D,\check D)$ satisfies the Extension Principle of Proposition \ref{proppropHJ}.

\smallskip\noindent{\bf (BNF):} Birkhoff Normal (Proposition \ref{nprop:9.2}):  

There exists $(\Omega_{q_{n+1}^{-1}}^{BNF}, F_{q_{n+1}^{-1}}^{BNF}, g_{q_{n+1}^{-1}}^{BNF})\in\cN\cF_{ex,\s}(W_{h,\bD(0,q_{n+1}^{-6})})$ such that 
\begin{align}& (g_{q_{n+1}^{-1}}^{BNF})^{-1}\circ \Phi_{\Omega}\circ f_{F}\circ g_{q_{n+1}^{-1}}^{BNF}= \Phi_{\Omega_{q_{n+1}^{-1}}^{BNF}}\circ f_{F_{q_{n+1}^{-1}}^{BNF}}   \qquad [W_{h,\bD(0,q_{n+1}^{-6})}]  \label{13.262cc}\\
&\|g_{q_{n+1}^{-1}}^{BNF}-id\|_{W_{h,\bD(0,q_{n+1}^{-6})}}\lesssim q_{n+1}^{-m}\label{F251a}\\
&\Omega_{q_{n+1}^{-1}}^{BNF}\in\cT\cC(2A,2B)\\
&\|F_{q_{n+1}^{-1}}^{BNF}\|_{W_{h,\bD(0,q_{n+1}^{-6})}}\leq \exp(-q_{n+1}^{1-\b}). \label{F251}
\end{align}
\end{prop}
\begin{proof}

\smallskip\noindent {\it KAM:} 
This is Proposition \ref{prop:1.enonce}. Inequality (\ref{F255}) comes from the corresponding (\ref{8.114})  $\bar \e_{i_{n}^+}\leq \exp(-N_{i_{n}^+}/(\ln N_{i_{n}^+})^3)$ and  the fact that $N_{i_{n}^+}\asymp q_{n+1}$.

\smallskip\noindent {\it HJ:} By Proposition \ref{prop:11.2CC}, Item \ref{aab2}, the disk $6\hat D_{n}=\bD(c_{n},|c_{n}|/4)$ is included in $U_{i_{n}^-}$. Since
$$ (6(|c_{n}|/24))^{1/8}<(Aq_{n})^{-1},\quad \textrm{and}\quad \|F_{i_{n}^-}\|_{h,6\hat D_{n}}\lesssim \bar \e_{i_{n}^-}<(|c_{n}|/4)^{\bar a_{3}}
 $$
(the first inequality is a consequence of (\ref{condqnqn+1}) and the second of  (\ref{e11.234}) and the fact that $|c_{n}|\asymp \rho_{n}$)
(\ref{8.179},  (\ref{8.177}) are satisfied and  we can apply  Proposition \ref{proppropHJ}  on Hamilton-Jacobi Normal Forms to $\Phi_{\Omega_{i_{n}^-}}\circ f_{F_{i_{n}^-}}$ on the domain $W_{h, \hat D_{n}}\subset W_{h, U_{i_{n}^-}}$ with 
$\hat \rho=|c_{n}|/24$:
there exists a disk $\check D_{n}\subset \hat D_{n}$ 
\be\check D_{n}:=\bD(c_{\check D_{n}},\rho_{\check D_{n}})\subset (1/10) \hat D_{n}\subset \hat D_{n}=\bD(c_{n},|c_{n}|/24)\subset U_{i_{n}^-}\label{disks}\ee
 and $(\Omega_{{n}}^{HJ}, F_{{n}}^{HJ}, g_{{n}}^{HJ})\in\cN\cF_{\s}(W_{h/9,\hat D_{n}\setminus\check D_{n}})$ such that  one has (\ref{13.265cc}) and 
 \begin{align}  
  & \|g_{\hat D_{n}}^{HJ}-id\|_{W_{h/9,(\hat D_{n}\setminus\check D_{n})}}   \lesssim q_{n}\bar\e_{i_{n}^-}^{1/8}\leq \bar \e_{i_{n}^-}^{1/9}\label{14.257b}\\
&\|F^{HJ}_{\hat D_{n}}\|_{W_{h/9, (\hat D_{n}\setminus \check D_{n})}}
\lesssim \exp(-1/(q_{n}|c_{n}|/24)^{1/4})\label{14.270}
\end{align}
and since $|c_{n}|\asymp (q_{n}q_{n+1})^{-1}$  ($n\gg_{\b}1$)
$$\|F^{HJ}_{\hat D_{n}}\|_{W_{h/40, (\hat D_{n}\setminus \check D_{n})}}\lesssim \exp(-q_{n+1}^{(1/4)-\b}).$$

\smallskip\noindent {\it BNF:} Since $\|F\|_{e^{1/10}W_{h,\bD(0,\rho_{n})}}\leq \rho_{n}^{\bar a_{1}}$ ({\it cf.} (\ref{eq:11.258bis})) we can apply Proposition \ref{nprop:9.2} on the existence of approximate BNF in the CC case (with $n+1$ in place of $n$): for $0<\b\ll 1$ and $n\gg_{\b} 1$:  there exists $(\Omega_{q_{n+1}^{-1}}^{BNF}, F_{q_{n+1}^{-1}}^{BNF}, g_{q_{n+1}^{-1}}^{BNF})\in\cN\cF_{ex,\s}(W_{h,\bD(0,q_{n+1}^{-6})})$ such that 
$$[W_{h,\bD(0,q_{n+1}^{-6})}]\qquad  (g_{q_{n+1}^{-1}}^{BNF})^{-1}\circ \Phi_{\Omega}\circ f_{F}\circ g_{q_{n+1}^{-1}}^{BNF}=\Phi_{\Omega^{BNF}_{q_{n+1}^{-1}}}\circ f_{F^{BNF}_{q_{n+1}^{-1}}} $$
with 
$$\|F^{BNF}_{q_{n+1}^{-1}}\|_{W_{h,q_{n+1}^{-6}}}\leq \exp({-q_{n+1}^{1-\b}}).$$

\end{proof}
\subsection{Comparision Principle}
These various Normal Forms match to some very good order of approximation.
\begin{lemma}[Comparing Adapted Normal Forms] One has for  any $\b\ll 1$, $n\gg_{\b}1$ 
\be \|\Omega_{i_{n}^+}-\Omega^{BNF}_{q_{n+1}^{-1}}\|_{(1/2)\bD(0,q_{n+1}^{-6} )}\lesssim \exp(-q_{n+1}^{1-\b}).\label{14.257cc}
\ee
and there exists $\g_{n} \lesssim q_{n+1}^{-m}\lesssim |c_{n}|^{m/2}$ such that 
\be\|\Omega_{i_{n}^+}-\Omega_{\hat D_{n}}^{HJ}(\cdot+\g_{n})\|_{(4/5)\hat D_{n}\setminus (1/5)\hat D_{n}}\lesssim \exp(-q_{n+1}^{(1/4)-\b})\label{14.258cc}
\ee
\end{lemma}
\begin{proof} Let us prove estimate (\ref{14.257cc}).  From (\ref{13.262cc})-(\ref{13.264ccante}), we see that on $ g_{1,i_{n}^+}(W_{h,\bD(0,q_{n+1}^{-6}})\cap g_{q_{n+1}^{-1}}^{BNF}(W_{h,\bD(0,q_{n+1}^{-6}})$ one has 
$$g_{q_{n+1}^{-1}}^{BNF}\circ  \Phi_{\Omega_{q_{n+1}^{-1}}^{BNF}}\circ f_{F_{q_{n+1}^{-1}}^{BNF}}\circ (g_{q_{n+1}^{-1}}^{BNF})^{-1} =   g_{1,i_{n}^+}\circ \Phi_{\Omega_{i_{n}^+}}\circ f_{F_{i_{n}^+}}\circ (g_{1,i_{n}^+})^{-1}.$$
We  then apply Proposition \ref{comparingBNFCC}
with $c=0$, $\rho_{1}=0$, $\rho_{2}=q_{n+1}^{-6}$, $\d=q_{n+1}^{-7}$, $\nu=\exp(-q_{n+1}^{1-\b})$ ({\it cf.} (\ref{F251}), (\ref{F255})), $\e=q_{n+1}^{-m}$ ({\it cf.} (\ref{F251a}), (\ref{F255a}))
 (estimates (\ref{estgj}) and (\ref{11308}) are satisfied since $ \bar Cq_{n+1}^{-m}	\leq q_{n+1}^{-7}/4\ll q_{n+1}^{-6}$ and $q_{n+1}^{7\hs \bar a_{5}}\exp(-q_{n+1}^{1/2})\ll 1$). 

\medskip
Estimate (\ref{14.258cc})  is a consequence of Proposition \ref{comparingBNFCC} applied to  (\ref{13.264cc}) and (\ref{13.265cc})
$$g_{i_{n}^-,i_{n}^+}\circ \Phi_{\Omega_{i_{n}^+}}\circ f_{F_{i_{n}^+}}\circ (g_{i_{n}^-,i_{n}^+})^{-1} =g^{HJ}_{n}\circ \Phi_{\Omega^{HJ}_{n}}\circ f_{F^{HJ}_{n}}\circ (g^{HJ}_{n})^{-1}
$$
 with $\bA(c;\rho_{1},\rho_{2})=\hat D_{n}\setminus (1/10)\hat D_{n}$, $c=c_{n}$, $\rho_{1}=|c_{n}|/40$, $\rho_{2}=|c_{n}|/4$, $\d_{n}=|c_{n}|/10$, $\nu=\exp(-q_{n+1}^{1/5})$ ({\it cf.} (\ref{F258a}), (\ref{F255}) ), $\e=\rho_{n}^{m/8}$ ({\it cf.} (\ref{F255a}), (\ref{F258})). Estimates  (\ref{estgj}) and (\ref{11308}) are satisfied since $\bar C \rho_{n}^{m/8}\leq |c_{n}|/40\ll |c_{n}|/5$ ({\it cf.} (\ref{10.147bis})) and $\bar C (|c_{n}|/10)^{-\bar a_{5}}\exp(-q_{n+1}^{(1/4)-})<1$ (recall that $|c_{n}|\asymp (q_{n}q_{n+1})^{-1}$).
\end{proof}

\section{Estimates on the measure of the set of KAM circles}  \label{sec:8curves}
We refer to Subsection \ref{sec:kamcircles} for the notations of this section. We observe
$$(W_{h,U})_{\R}:=W_{h,U}\cap M_{\R}=W_{U\cap \R}:=\{r\in U\cap\R\}\cap M_{\R}.$$
In particular in the (AA)-case $(W_{h,U})_{\R}=W_{U\cap \R}=\T\times (U\cap \R)$ and in the (CC*)-case $(W_{h,U})_{\R}=W_{U\cap \R}=\{(x,y)\in\R^2,\ (1/2)(x^2+y^2)\in U\cap \R_{+}\}$.

\subsection{Classical KAM estimates}
We first state a variant of the   classical KAM  theorem on abundance of invariant circles which is a consequence of Propositions \ref{prop:1.enonce},\ref{prop:1.enoncebis}, \ref{prop:7.5} and  Remark \ref{rem:7.1} on KAM Normal Forms.
\begin{theo}\label{measureestimates}  Let $U$ be holed domain with disjoint holes $D\in\cD(U)$ such that 
\be\sum_{D\in\cD(U)}|D\cap \R|^{1/2}\leq 1\label{cond12.265}
\ee
 and $\Omega\in \ti \cO_{\s}(U)\cap \cT\cC(A,B)$ ({\it cf.  (\ref{7.119})}) with $A,B$ satisfying (\ref{condUnew}), $F\in\cO_{\s}(W_{h,U})$
 $$\bar\e:=\|F\|_{W_{h,U}}\leq \ud(U)^{\hs \bar a_{2}}.
$$
Then, if $f=\Phi_{\Omega}\circ f_{F}$ one has
$${\rm Leb}_{M_{\R}}(W_{e^{-1/10}U\cap \R}\setminus \cL(f,(W_{U\cap \R} ))\lesssim (\|F\|_{W_{h,U}})^{1/(2(\bar a_{0}+3))}.
$$
\end{theo}
\begin{proof}
See the Appendix \ref{KAMestappendix}.
\end{proof}

We define for $\rho>0$, $\bD_{\R}(0,\rho)=\bD(0,\rho)\cap\R=]-\rho,\rho[$ and 
$$\ti m_{f}(\rho)={\Leb}_{M_{\R}}(W_{\bD_{\R}(0,\rho)}\setminus\cL(f,\bD_{\R}(0,e^{1/2}\rho))).$$

\subsection{Estimates on the measure of the set of invariant circles: $\omega_{0}$ diophantine (AA) or (CC) Case}\label{sec:12.2}
We assume that (both in the (AA) or (CC)-cases) (\ref{eq:11.257}),  (\ref{eq:11.257,5}) (\ref{eq:11.258}) hold.
\begin{theo}\label{prop:12.7}For any $\b>0$, $\rho\ll_{\b}1$
$$\textrm{(AA)-case}\qquad \ti m_{\Phi_{\Omega}\circ f_{F}}(\rho)\lesssim \exp(-(1/\rho)^{\frac{2}{1+\tau}-\b} )+\sum_{D\in\cD_{\rho}}|\check D\cap \R|.
$$
$$\textrm{(CC)\ or (CC*)-case}\qquad \ti m_{\Phi_{\Omega}\circ f_{F}}(\rho)\lesssim \exp(-(1/\rho)^{\frac{2}{1+\tau}-\b} )+\sum_{D\in\cD_{\rho}}|\check D\cap \R_{+}|.
$$
\end{theo}
\begin{proof}
If $S\subset \C$ we denote $S_{\R}=S\cap\R$ (if $c\in\R$, $\bD_{\R}(c,t)=\bD(c,t)\cap\R=]c-t,c+t[$).
 
Choose  ({\it cf.} Lemma \ref{lemma:12.5}) $\rho'\in [e^{1/4}\rho,e^{1/3}\rho]$ ($\rho\ll 1$)  such that 
$$U^{(\rho')}:=\bD(0,\rho')\cap U_{i_{+}(\rho)}=\bD(0,\rho')\setminus\bigcup_{\substack{D\in \cD(U_{i_{+}(\rho)})\\ D\subset \bD(0,\rho')  }}D$$
hence
\be e^{-1/10}\bD_{\R}(0,\rho')\subset e^{-1/10}U_{\R}^{(\rho')}\cup\bigcup_{D\in\cD_{\rho}}  (1/4)\hat D_{\R}.\label{a12356}\ee

Let us denote for short $f_{i}=\Phi_{\Omega_{i}}\circ f_{F_{i}}$, $f_{D}^{HJ}=\Phi_{\Omega_{D}^{HJ } }\circ f_{F_{D}^{HJ}}$ and 
$$\cL_{i_{+}(\rho)}=\cL(f_{i_{+}(\rho)},W_{U_{\R}^{(\rho')}} ),\quad \cL_{\hat D}=\cL(f_{D}^{HJ},W_{\hat D_{\R}\setminus\check D_{\R}}).
$$
We have from (\ref{12352bis}) (\ref{12356}) (\ref{12348})
\begin{align}&W_{h,U_{i_{D}}},\quad g_{i_{-}(\rho),i_{D}}^{-1}\circ f_{i_{-}(\rho)}\circ g_{i_{-}(\rho),i_{D}}=f_{i_{D}}\\
&W_{h/9,\hat D\setminus D},\quad (g^{HJ}_{\hat D})^{-1}\circ f_{i_{D}}\circ g^{HJ}_{\hat D}=f_{D}^{HJ}\\
&W_{h, U_{i_{+}(\rho)}},\quad g_{i_{-}(\rho),i_{+}(\rho)}^{-1}\circ f_{i_{-}(\rho)}\circ g_{i_{-}(\rho),i_{+}(\rho)}=f_{{i_{+}(\rho)}}\\
&W_{h, U_{i_{-}(\rho)}},\quad g_{1,i_{-}(\rho)}^{-1}\circ f\circ g_{1,i_{-}(\rho)}=f_{{i_{-}(\rho)}}\label{e12.271}.\end{align}
From Lemma \ref{lemma:12.4}, Remark \ref{rem:lemma3.1} and estimate (\ref{12.355}) on the one hand, and estimate (\ref{12358}) on the other hand, we see that we can apply Theorem (\ref{measureestimates}) to $f_{i_{+}(\rho)}$ and $f_{D}^{HJ}$ to get  the following decompositions
\be W_{e^{-1/10}U_{\R}^{(\rho')} } \setminus\cL_{i_{+}(\rho)}\subset B_{i_{+}(\rho)},\quad W_{e^{-1/10}\hat D_{\R}}\setminus\cL_{{\hat D}}\subset  B_{\hat D\setminus\check D}\cup  E_{\check D}\label{12.298e}
\ee
with $B_{\hat D\setminus\check D}\subset W_{\hat D_{\R}\setminus \check D_{\R}}$, $E_{\check D}=W_{e^{-1/10}\check D_{\R}}$
and 
$$\max\biggl({\rm Leb}_{M_{\R}}(B_{i_{+}(\rho)}), {\rm Leb}_{M_{\R}}(B_{{\hat D\setminus\check D}})\biggr)\lesssim  \exp(-(1/\rho)^{\frac{2}{1+\tau}-\b/2})
$$
\be {\rm Leb}_{M_{\R}}(E_{\check D})\lesssim {\rm Leb}_{M_{\R}}(W_{e^{-1/10}\check D_{\R}}).\label{12.275}
\ee
We now introduce
\be\ti \cL_{i_{+}(\rho)}:=g_{i_{-}(\rho),i_{+}(\rho)}(\cL_{i_{+}(\rho)}),\qquad \ti \cL_{\hat D}:=g_{i_{-}(\rho),i_{D}}\circ g_{D}^{HJ}(\cL_{\hat D})\label{12358a}\ee
$$\ti B_{i_{+}(\rho)}=g_{i_{-}(\rho),i_{+}(\rho)}(B_{i_{+}(\rho)}) ,\quad \ti B_{\hat D\setminus\check D}=g_{i_{-}(\rho),i_{D}}\circ g_{D}^{HJ}(B_{\hat D\setminus\check D}),$$
$$\ti E_{\check D}=g^{}_{i_{-}(\rho),i_{D}}\circ g_{D}^{HJ}(E_{\check D}).$$

\begin{lemma}\label{lemma:12.9}One has 
$$g_{i_{-}(\rho),i_{+}(\rho)}^{}(W_{ e^{-1/10} \bD_{\R}(0,\rho')})\setminus\ti \cL\subset \ti B$$
with 
$$\ti \cL= \ti\cL_{i_{+}(\rho)}\cup \bigcup_{D\in\cD_{\rho}}\ti \cL_{\hat D}\qquad \ti B= \ti B_{i_{+}(\rho)}\cup\bigcup_{D\in\cD_{\rho}} \ti B_{\hat D\setminus\check D}\cup \ti E_{\check D}.$$
\end{lemma}
\begin{proof}We observe that from (\ref{a12356}) one has 
\be W_{e^{-1/10}\bD_{\R}(0,\rho')}\subset W_{e^{-1/10}U_{\R}^{(\rho')}}\cup\bigcup_{D\in\cD_{\rho}}  W_{(1/4)\hat D_{\R}}\label{12.298d}
\ee
hence 
\begin{multline*}g_{i_{-}(\rho),i_{+}(\rho)}^{}(W_{e^{-1/10}\bD_{\R}(0,\rho')})\\
\subset g_{i_{-}(\rho),i_{+}(\rho)}^{}(W_{e^{-1/10}U_{\R}^{(\rho')}})\cup\bigcup_{D\in\cD_{\rho}}g_{i_{-}(\rho),i_{+}(\rho)}^{}(W_{(1/4)\hat D_{\R}}).
\end{multline*}
Note that by Proposition \ref{adaptedNFAA} one has  $\max(\|g_{i_{-}(\rho),i_{+}(\rho)}-id\|_{C^1}, \|g_{\hat D}^{HJ}-id\|_{C^1})\|\leq \bar \e_{i_{-}(\rho)}^{1/9}\ll \hat K_{i_{D}}^{-1}$ (since $N_{i_D}\leq N_{i_{-}(\rho)}^2$) hence
$$g_{i_{-}(\rho),i_{+}(\rho)}^{}(W_{ (1/4)\hat D_{\R}})\subset W_{ (1/2)\hat D_{\R}}\subset g^{}_{i_{-}(\rho),i_{D}}\circ g^{HJ}_{ \hat D}(W_{e^{-1/10}\hat D_{\R}})
$$
which yields
\begin{multline*}g_{i_{-}(\rho),i_{+}(\rho)}^{}(W_{e^{-1/10}\bD_{\R}(0,\rho')})\subset \\ g^{}_{i_{-}(\rho),i_{+}(\rho)}(W_{e^{-1/10}U_{\R}^{(\rho')}})\cup \bigcup_{D\in\cD_{\rho}} \biggl(g^{}_{i_{-}(\rho),i_{D}}\circ g^{HJ}_{ \hat D}(W_{e^{-1/10}\hat D_{\R} \setminus\check D_{\R}})\cup \ti E_{\check D}\biggr).\end{multline*}
We then conclude using (\ref{12.298e}) and the notations (\ref{12358a}).
\end{proof}
 \begin{lemma}\label{lemma:12.10}For some  
 $G\subset W_{e^{1/10}\bD_{\R}(0,\rho')}$ one has 
 $ \ti\cL=\cL(f_{i_{-}(\rho)},G)$ and 
 \be  {\rm Leb}_{M_{\R}}(\ti B)\lesssim \exp(-(1/\rho)^{\frac{2}{1+\tau}-\b })+ \sum_{D\in\cD_{\rho}}{\Leb}_{M_{\R}}(W_{e^{-1/10}\check D_{\R}}). \label{estlemma122}
\ee
 \end{lemma}
 \begin{proof}
We observe that from  (\ref{4.47ter})
\be\ti \cL_{i_{+}(\rho)}:=g_{i_{-}(\rho),i_{+}(\rho)}(\cL_{i_{+}(\rho)})=\cL(f_{i_{-}(\rho)},g_{i_{-}(\rho),i_{+}(\rho)}(W_{U_{\R}^{(\rho')}}))
\label{13.310}\ee
\be \ti \cL_{\hat D}:=g_{i_{-}(\rho),i_{D}}\circ g_{D}^{HJ}(\cL_{\hat D})=\cL(f_{i_{-}(\rho)},g_{i_{-}(\rho),i_{D}}\circ g_{D}^{HJ}(W_{\hat D_{\R}\setminus \check D_{\R}}))
\label{13.311}\ee
hence, 
$$\ti \cL=\cL(f_{i_{-}(\rho)},G)$$
with 
$$G=g_{i_{-}(\rho),i_{+}(\rho)}(W_{U_{\R}^{(\rho')}})\cup \bigcup_{D\in\cD_{\rho}}g_{i_{-}(\rho),i_{D}}\circ g_{D}^{HJ}(W_{\hat D_{\R}\setminus \check D_{\R}})$$
and clearly
$G\subset W_{e^{1/10}\bD_{\R}(0,\rho')}.
$

To get the estimate on the measure of $\ti B$ we use 
$${\rm Leb}_{M_{\R}}(\ti B_{i_{+}(\rho}))\lesssim \exp(-(1/\rho)^{\frac{2}{1+\tau}-\b/2 }),$$
$${\rm Leb}_{M_{\R}} \biggl(\bigcup_{D\in\cD_{\rho}}\ti B_{\hat D\setminus\check D}\biggr)\lesssim N_{i_{+}(\rho)}^2\exp(-(1/\rho)^{\frac{2}{1+\tau}-\b/2 }))
\lesssim \exp(-(1/\rho)^{\frac{2}{1+\tau}-\b })
$$
and (see (\ref{12.275}))
$${\rm Leb}_{M_{\R}} \biggl(\bigcup_{D\in\cD_{\rho}}\ti E_{\check D}\biggr)\lesssim \sum_{D\in\cD_{\rho}} {\Leb}_{M_{\R}}(W_{e^{-1/10}\check D_{\R}}).
$$
 \end{proof}

Lemmata \ref{lemma:12.9} and \ref{lemma:12.10} give 
$$g_{i_{-}(\rho),i_{+}(\rho)}^{}(W_{\bD_{\R}(0,\rho')})\setminus \cL(f_{i_{-}(\rho)},W_{e^{1/10}\bD_{\R}(0,\rho')})\subset \ti B$$
hence
\be g_{1,i_{-}(\rho)}^{}\circ g_{i_{-}(\rho),i_{+}(\rho)}^{}(W_{\bD_{\R}(0,\rho')})\setminus g_{1,i_{-}(\rho)}^{}( \cL(f_{i_{-}(\rho)},W_{e^{1/10}\bD_{\R}(0,\rho')}))\subset g_{1,i_{-}(\rho)}^{}(\ti B). \label{e12.280}\ee
Since the conjugation relation  $g_{1,i_{-}(\rho)}^{-1}\circ f\circ g_{1,i_{-}(\rho)}= f_{{i_{-}(\rho)}}$  holds on $W_{U_{i_{-}(\rho)}\cap \R}$ ({\it cf.} (\ref{e12.271})) and since  $W_{e^{1/10}\bD_{\R}(0,\rho')}\subset W_{U_{i_{-}(\rho)}\cap \R}$ (recall that by definition (\ref{eq:9.296}) $\bD(0,2\rho)\subset W_{h,U_{i_{-}(\rho)}}$ and that $\|g_{1,i_{-}(\rho)}-id\|\ll \rho$) one has by (\ref{4.47ter})
$$\cL(f,g_{1,i_{-}(\rho)}^{}(W_{e^{1/10}\bD_{\R}(0,\rho')}))=g_{1,i_{-}(\rho)}^{}\cL(f_{i_{-}(\rho)},W_{e^{1/10}\bD_{\R}(0,\rho')}).
$$
Equation (\ref{e12.280}) then implies that
$$ g_{1,i_{-}(\rho)}^{}\circ g_{i_{-}(\rho),i_{+}(\rho)}^{}(W_{\bD_{\R}(0,\rho')})\setminus\cL(f,g_{1,i_{-}(\rho)}^{}(W_{e^{1/10}\bD_{\R}(0,\rho')}))\subset g_{1,i_{-}(\rho)}^{}(\ti B).$$
Finally, inclusions $W_{\bD_{\R}(0,\rho)} \subset g_{1,i_{-}(\rho)}^{}\circ g_{i_{-}(\rho),i_{+}(\rho)}^{}(W_{\bD_{\R}(0,\rho')})$ and $g_{1,i_{-}(\rho)}^{}(W_{e^{1/10}\bD_{\R}(0,\rho')})\subset W_{e^{1/2}\bD(0,\rho)}$ yield
$${\Leb}_{M_{\R}}(W_{\bD_{\R}(0,\rho)} \setminus \cL(f,W_{e^{1/2}\bD(0,\rho)}))\lesssim {\Leb}_{M_{\R}}(\ti B).$$
We conclude by using the estimate (\ref{estlemma122}) and the fact that 
$${\Leb}_{M_{\R}}(W_{e^{-1/10}\check D_{\R}})\leq \begin{cases}&|\check D\cap\R |,\quad \textrm{(AA)-case}\\
&|\check D\cap\R_{+} |,\quad \textrm{(CC)\ or\ (CC*)-case}.
\end{cases}$$

\end{proof}

\subsection{Estimates on the measure of the set of invariant circles: $\omega_{0}$ Liouvillian, (CC)-Case}
We now assume that (\ref{eq:11.257bis}),  (\ref{eq:11.257,5bis}) (\ref{eq:11.258bis}) hold. 
\begin{theo}\label{prop:12.11}Let $\rho_{n}=(10A)/(q_{n}q_{n+1})$ and assume that $q_{n+1}\geq q_{n}^{10}$. Then, for all $\b\ll 1$ and $n\gg_{\b} 1$ one has  
$$\ti m_{\Phi_{\Omega}\circ f_{F}}(\rho_{n})\lesssim \exp(-q_{n+1}^{1/4-\b})+|(\check D_{n}\cap \R_{+})|.
$$
\end{theo}
\begin{proof}The principle of the proof is the same as that of Proposition \ref{prop:12.7} with the following modifications in the notations: we  set $f_{i_{n}^\pm}=f_{F_{i_{n}^\pm}}$, $f_{D_{n}}^{HJ}=f_{F_{D_{n}}^{HJ}}$  and  we  replace in the proof  the indices  $i_{\pm}(\rho)$ by $i_{n}^\pm$, $\hat D$, $D$, $\check D$ by  $\hat D_{n}$, $D_{n}$ $\check D_{n}$,   $i_{D}$ by $i_{n}^-$,  $\rho$ by $(4/3)\rho_{n}$, $\rho'$ by $2\rho_{n}$, $U^{\rho'}$ by $U^{(n)}$ and $\exp(-(1/\rho)^{\frac{2}{1+\tau} -\b})$ by $\exp(-q_{n+1}^{(1/4)-})$. Instead of using the conjugation relations of Proposition \ref{adaptedNFAA} (Adapted Normal Forms in the (CC) or (CC*)-case) we use those of Proposition \ref{prop:adaptedCC}.
\end{proof}
\begin{rem}\label{rem:12.1}Note that if the twist condition (\ref{eq:11.257bis}) is satisfied, then  any twist condition $\cT\cC(A',B)$ is satisfied with $A'\geq A$. We can thus replace in Theorem \ref{prop:12.11} $\rho_{n}=(10A)/(q_{n}q_{n+1})$ by  $\rho_{n}=(10A')/(q_{n}q_{n+1})$ for any fixed $A'\geq A$ (then $n$ has to be chosen larger).
\end{rem}

\section{Convergent  BNF implies small holes}\label{sec:smallholes13}

\subsection{Case where $\omega_{0}$ Diophantine in the  (AA) of (CC) setting}
We keep here the notations of Section \ref{sec:adaptedkam} in particular we assume  $\omega_{0}$ is $\tau$-diophantine and that (\ref{eq:11.257}), (\ref{eq:11.257,5}),  (\ref{eq:11.258}) hold. 
\begin{lemma}If $BNF(\Phi_{\Omega}\circ f_{F})$ converges and is equal to a holomorphic function $\Xi\in\cO(\bD(0,1))$ then  for all $\b>0$, $\rho\ll_{\b} 1$ and  for any $D\in\cD_{\rho}$
\be \|\Omega_{i_{+}(\rho)}-\Xi\|_{\hat D\setminus(1/10)\hat D}\lesssim\exp\biggl(-(1/\rho)^{(\frac{2}{1+\tau})-\b}\biggr).\label{14.261}
\ee
As a corollary,  for any $D\in\cD_{\rho}$ and $\g_{D}\leq \hat K_{i_{D}}^{-2}$
\be \|\Omega_{\hat D}^{HJ}-\Xi(\cdot-\gamma_{\hat D})\|_{(4/5)\hat D\setminus (1/5)\hat D}\lesssim \exp\biggl(-(1/\rho)^{(\frac{2}{1+\tau})-\b}\biggr).\label{14.262}
\ee
\end{lemma}
\begin{proof}Let us prove inequality  (\ref{14.261}). From (\ref{14.257})  and Proposition \ref{theo:compBNF} one gets
\begin{align*}\|\Omega_{i_{+}(\rho)}-\Xi\|_{(1/2)\bD(0,\rho^{b_{\tau}})}&\lesssim \exp(-(1/\rho)^{\frac{2}{1+\tau}-\b/2})+\exp(-(1/\rho)^{1-\b})\\
&\lesssim \exp(-(1/\rho)^{\frac{2}{1+\tau}-\b/2}).
\end{align*}
Since the function $U_{i_{+}(\rho)}-\Xi$ is holomorphic on $U^{(\rho)}$ and since    the triple
$(U_{i_{+}(\rho)},\hat D\setminus(1/10)\hat D,\bD(0,\rho^{b_{\tau}} /2))$ is $(10b_{\tau})^{-1}|\ln\rho|^{-1}$-good, {\it cf.} Proposition \ref{prop:10.4noscreening},  we have by definition
\begin{align*}\|\Omega_{i_{+}(\rho)}-\Xi\|_{\hat D\setminus(1/10)\hat D}&\lesssim\exp(-(10b_{\tau}|\ln \rho|)^{-1}(1/\rho)^{\frac{2}{1+\tau}-\b/2})\\
&\lesssim  \exp(-(1/\rho)^{\frac{2}{1+\tau}-\b}).
\end{align*}

The inequality (\ref{14.262}) is then a consequence of (\ref{14.261}) and (\ref{14.258}).
\end{proof}
\begin{cor}\label{cor:13.2}If $BNF(f)=\Xi$ then for all $\b>0$, $\rho\ll_{\b} 1$, and any $D\in\cD_{\rho}$ the radius $\rho_{\check D}$ of the disk $\check D$ satisfies
$$\rho_{\check D}\lesssim\exp\biggl(-(1/\rho)^{(\frac{2}{1+\tau })-\b}\biggr).
$$
\end{cor}
\begin{proof} This results from (\ref{14.262}) and the Extension Property in Proposition \ref{proppropHJ}.
\end{proof}

\begin{cor}\label{cor:13.3}If $BNF(\Phi_{\Omega}\circ f_{F})$ converges then for all $\b>0$, $\rho\ll_{\b} 1$
$$\ti m_{\Phi_{\Omega}\circ f_{F}}(\rho)\lesssim \exp\biggl(-(1/\rho)^{(\frac{2}{1+\tau })-\b}\biggr).
$$
\end{cor}
\begin{proof}This is a consequence of the previous Corollary \ref{cor:13.2} and of Proposition \ref{prop:12.7} since $\#\cD_{\rho}\lesssim \rho^{1-2(\mu/{\iota(\rho))}}\lesssim \rho^{-1}$ ({\it cf.} (\ref{numberofholes}), (\ref{eq:5.101}), (\ref{502})).
 \end{proof}
\subsection{Case $\omega_{0}$ is irrational in the (CC) setting}
The notations here are those of Section \ref{sec:10.2}. In particular we assume that (\ref{eq:11.257bis}) (\ref{eq:11.257,5bis}), {\ref{eq:11.258bis}) hold.

\begin{lemma}If $BNF(\Phi_{\Omega}\circ f_{F})$ converges and is equal to $\Xi\in\cO(\bD(0,1))$ then for all $\b\ll 1$, $n\gg_{\b}1$ such that $q_{n+1}\geq q_{n}^{10}$
\be \|\Omega_{i_{n}^+}-\Xi\|_{\hat D_{n}\setminus (1/10)\hat D_{n}}\lesssim \exp(-q_{n+1}^{1-\b}).\label{14.261cc}
\ee
As a corollary,  for $\g_{n} \lesssim q_{n+1}^{-m}\lesssim |c_{n}|^{m/2}$ 
\be \|\Omega_{\hat D_{n}}^{HJ}-\Xi(\cdot-\g_{n}))\|_{(4/5)\hat D_{n}\setminus (1/5)\hat D_{n}}\lesssim \exp(-q_{n+1}^{(1/4)-\b}).\label{14.262cc}
\ee
\end{lemma}
\begin{proof}
Let us prove (\ref{14.261cc}). From (\ref{14.257cc})  and Proposition \ref{theo:compBNF} one gets 
$$\|\Omega_{i_{n}^+}-\Xi\|_{\bD(0,q_{n+1}^{-6}/2)}
\lesssim \exp(-q_{n+1}^{1-\b/2}).
$$
Since the function $\Omega_{i_{n}^+}-\Xi$ is holomorphic on $U^{(n)}$ and since  the triple  $(U^{(n)},\hat D_{n}\setminus(1/10)\hat D_{n}, \bD(0,q_{n+1}^{-6} /2))$ is $1/(10|\ln \rho_{n}|)$-good (see Proposition \ref{prop:11.2CC}, Item (\ref{aab4})),  we have
\begin{align*}\|\Omega_{i_{n}^+}-\Xi\|_{\hat D_{n}\setminus(1/10)\hat D_{n}}&\lesssim\exp(-(10|\ln \rho_{n}|)^{-1}(q_{n+1})^{1-\b/2})\\
&\lesssim \exp(-q_{n+1}^{1-\b}).
\end{align*}

The inequality (\ref{14.262cc}) is then a consequence of (\ref{14.261cc}) and  (\ref{14.258cc}).
\end{proof}
\begin{cor}\label{cor:13.5}If $BNF(f)=\Xi$ then for any $\b>0$, $n\gg_{\b} 1$ such that $q_{n+1}\geq q_{n}^{10}$, the radius $\rho_{\check D_{n}}$ of the disk $\check D_{n}$ satisfies
$$\rho_{\check D_{n}}\lesssim \exp(-q_{n+1}^{(1/4)-\b}).
$$
\end{cor}
\begin{proof} This results from (\ref{14.262cc}) and  the Extension Property of Proposition \ref{proppropHJ}.
\end{proof}

\begin{cor}\label{cor:13.6}If $BNF(\Phi_{\Omega}\circ f_{F})$ converges, then for any $\b>0$, $A'\geq A$ and $n\gg_{\b,A'} 1$ such that $q_{n+1}\geq q_{n}^{10}$ one has  
$$\ti m_{\Phi_{\Omega}\circ f_{F}}(\rho_{n})\lesssim \exp(-q_{n+1}^{(1/4) -\b}),\qquad \rho_{n}=10A'/(q_{n+1}q_{n}).
$$
\end{cor}
\begin{proof}This follows from the previous Corollary \ref{cor:13.5} and Proposition \ref{prop:12.11} and Remark \ref{rem:12.1}.
\end{proof}

\section{Proof of Theorems \ref{theo:main1} and \ref{theo:mainprime1} }\label{sec:conclAA'}
\subsection{Proof of Theorem \ref{theo:main1}}
\subsubsection{(AA) Case}\label{proofThA-AA} Let $f(\th,r)=(\th+\omega_{0},r)+(O(r),O^2(r))$ be a real analytic symplectic diffeomorphism of the annulus $\T\times [-1,1]$ satisfying the twist condition (\ref{twistcond1.10}). We can perform some step of the classical Birkhoff Normal Form procedure, Proposition \ref{statementBNF}: for some $h>0, \rho_{0}>0$, there exists $\ti g=f_{\ti Z}=id+(O(r),O(r^2))$, $\ti \Omega\in\cO_{\s}(e^{10h}\bD(0,\rho_{0}))$, $Z, F\in\cO_{\s}(e^{10h}(\T_{h}\times \bD(0, \rho_{0})))\cap O(r^2)$, such that on $e^{10h}(\T_{h}\times \bD(0,\rho_{0}))$ one has 
\begin{align*}&\ti g^{-1}\circ f\circ \ti g=\Phi_{\ti \Omega}\circ f_{F},\\
&\forall 0\leq \rho\leq \rho_{0},\quad  \|F\|_{e^{10h}(\T_{h}\times \bD(0,\rho))}\leq \rho^{\hs m}\\
&(2\pi)^{-1}\ti\Omega(r)=\omega_{0}r+b_{2}(f)r^2+O(r^3)\\
&\ti Z(\th,r)=\sum_{j=2}^{9} \ti Z_{j}(\th)r^j+r^{10}\ti Z_{\geq 10}(\th,r)
\end{align*}
where $m$ is the constant appearing in (\ref{eq:11.258}).
Applying Lemma \ref{Whextension}   to $\ti\Omega(r)$ and Lemma \ref{lemma:2.3ee} to  $r^{10}\ti Z_{\geq 10}(\th,r)$  we can find, for some $0<\bar\rho\ll\rho_{0}$, $C^3$ Whitney extensions $\Omega\in\ti\cO_{\s}(e^{10h}\bD(0,\bar \rho))$  and  $Z\in\ti\cO_{\s}(e^{10h}W_{h,\bD(0,\bar \rho)})$   of  $(\ti \Omega,e^{10h}\bD(0, \bar \rho))$  and $(\ti Z,e^{h/10}W_{h,\bD(0,\bar\rho)})$ such that $g:=f_{Z}\in \ti{\rm Symp}_{ex,\s}(e^{h/10}W_{h,\bD(0,\bar\rho)})$, 
\begin{align}&\Omega\in\cT\cC(A,B), \quad A=3\min(b_{2}(f),b_{2}(f)^{-1}),\quad  B\geq 0 \label{14.353}    \\
&g(\{r=0\})=(\{r=0\}),\qquad   \|g-id\|_{C^1}\leq 1/100.\label{14.354}
\end{align}
We can then apply Corollary \ref{cor:13.3} to get 
$$\ti m_{\Phi_{\Omega}\circ f_{F}}(\rho)\lesssim \exp(-(1/\rho)^{\frac{2}{1+\tau}-\b}).
$$
Since 
$$ g^{-1}\circ f\circ g=\Phi_{\Omega}\circ f_{F} \qquad [e^{h/10}W_{h,\bD(0,\bar\rho)}]$$
one has from (\ref{4.47ter}), for any $\b>0$ and any  $0<\rho<\ll_{\b} 1$,
$$ \cL(f,g(W_{\bD(0,\rho)}))=g(\cL(\Phi_{\Omega}\circ f_{F}, W_{\bD(0,\rho)}))$$
hence, using the fact that $g(\{r=0\})=(\{ r=0\})$ and   $\|g-id\|_{C^1}\leq 1/100$, the inequality 
$$m_{f}(\rho)\lesssim \ti m_{\Phi_{\Omega}\circ f_{F}}(2\rho).$$
Now, if the BNF of $f$ converges, the same is true for $\Phi_{\Omega}\circ f_{F}$ and by Corollary \ref{cor:13.3}
$$m_{f}(\rho)\lesssim \exp\biggl(-(1/\rho)^{(\frac{2}{1+\tau })-\b}\biggr).$$
\hfill $\Box$
\subsubsection{(CC) Case}\label{sec:14.1.2} Let $f$ is a real analytic twist symplectic map of the real disk admitting the origin as an elliptic fixed point with Diophantine frequency $\omega_{0}$, $(x,y)\mapsto \Phi_{2\pi\omega_{0}r(x,y)}(x,y)+O^2(x,y)$, $r(x,y)=(1/2)(x^2+y^2)$ and satisfying the twist condition (\ref{twistcond1.10}). We first make the symplectic change of variables (\ref{changecoordxyzw} $(z,w)=\ph(x,y)$,
$$\begin{cases}&z=\frac{1}{\sqrt{2}}{(x+iy)}\\ &w=\frac{i}{\sqrt{2}}{(x-iy)}\end{cases}\iff \begin{cases}&x=\frac{1}{\sqrt{2}}(z-iw)\\ &y=\frac{-i}{\sqrt{2}}(z+iw)\end{cases}
$$
and we write the thus obtained symplectic map $(z,w)\mapsto \ti f(z,w)$, $\ti f=\ph\circ f\circ \ph^{-1}$ as 
$$\ti f=\Phi_{2\pi\omega_{0}r}\circ f_{F_{0}},\qquad r=-izw.
$$
We observe that 
\be \cL(f,W)=\cL(\ti f,\ph(W)).\label{14.320}\ee
Like in the (AA)-case (Subsection \ref{proofThA-AA})
we perform some steps of  Birkhoff Normal Form, Proposition \ref{statementBNFante} and make some Whitney extensions (Lemma \ref{lemma:2.3ee}) to obtain   for some  $h>0$, $\bar \rho>0$, maps $\Omega\in\ti\cO_{\s}(e^{10h}\bD(0,\bar \rho))$, $F\in\cO_{\s}(e^{10h}W_{h,\bD(0,\bar \rho)})$,  $g\in \ti{\rm Symp}_{ex,\s}(e^{h/10}W_{h,\bD(0,\bar\rho)})$ satisfying  
 \begin{align}&g^{-1}\circ \ti f\circ g=\Phi_{\Omega}\circ f_{F},\qquad [e^{10h}W_{h,\bD(0,\bar\rho)}]  \label{14.356}\\
 &g(\{r=0\})=(\{r=0\}),\qquad   \|g-id\|_{C^1}\leq 1/100.\label{14.354bis}\\
 &\Omega\in\cT\cC(A,B), \quad A=3\min(b_{2}(f),b_{2}(f)^{-1}),\quad  B\geq 0 \label{14.353bis}   \\
 &\forall\ \rho\leq \bar \rho,\quad \|F\|_{e^{10h}W_{h,\bD(0,\rho)}}\leq \rho^m\notag
\end{align}
where $m$ is the constant appearing in (\ref{eq:11.258}).

Applying  (\ref{14.320}), (\ref{14.356}), (\ref{4.47ter}) and Corollary \ref{cor:13.3} yields for any $\b>0$ and any $\rho\ll_{\b} 1$,
\be m_{f}(\rho)\leq m_{\ti f}(2\rho)\leq \ti m_{\Phi_{\Omega\circ f_{F}}}(4\rho)\lesssim \exp(-(1/\rho)^{\frac{2}{1+\tau}-\b}).\label{14.357}
\ee
\hfill $\Box$

\subsection{Proof of Theorem \ref{theo:mainprime1}} We proceed like in the previous Subsection \ref{sec:14.1.2} but we apply  Corollary \ref{cor:13.6} instead of Corollary \ref{cor:13.3}. 

Setting $\rho_{n}=10A/(q_{n}q_{n+1})$ with $A=3\min(b_{2}(f),b_{2}(f)^{-1})$ ({\it cf.} (\ref{14.353bis})) and assuming $q_{n+1}\geq q_{n}^{10}$ we now get 
$$\ti m_{\Phi_{\Omega}\circ f_{F}}(\rho_{n})\lesssim \exp(-q_{n+1}^{(1/4)-\b})
$$
hence if $t_{n}:=5\min(b_{2}(f),b_{2}(f)^{-1})/(q_{n}q_{n+1})\leq \rho_{n}/4$ one has  ({\it cf.} (\ref{14.357}))
$$m_{f}(t_{n})\leq m_{\ti f}(2t_{n})\leq \ti m_{\Phi_{\Omega\circ f_{F}}}(4t_{n})\lesssim \exp(-q_{n+1}^{1/5}).$$
\hfill $\Box$

\section{Opening hyperbolic eyes}\label{sec:openingtheeyes}
\bigskip
Let $\Omega\in \ti \cO_{\s}(\bD(0,\bar\rho))$ satisfy a twist condition,
\be \forall \ r\in\R,\  A^{-1}\leq (1/2\pi) \pa^2\Omega(r)\leq A,\quad \textrm{and}\ \ \|(1/2\pi)D^3\Omega\|_{\C}\leq B, \label{15.358}\ee
 $\bar a_{4}\in\N$, $\bar a_{4}\geq 10$,  be the constant appearing in Proposition \ref{prop:11.5} of the Appendix
and   $(p_{n}/q_{n})$   the sequence of convergents of $\omega_{0}=(2\pi)^{-1}\pa\Omega(0))$.  We introduce for $n\geq 1$, the sequence  $c_{n}$ defined by
\be(2\pi)^{-1}\pa\Omega(c_{n})=p_{n}/q_{n},\qquad \frac{(2A)^{-1}}{q_{n}q_{n+1} } \leq |c_{n}|\leq  \frac{A}{q_{n}q_{n+1}}. \label{e16.318}
\ee

\begin{prop}\label{prop:16.1}Let $h>0$, $n\in\N$  large enough and $F\in\cO_{\s}(\T_{h}\times \bD(c_{n},|c_{n}|^2))$ such that 
$$e^{-q_{n}h}<|c_{n}|^{10}$$
 $$\|F\|_{\T_{h}\times \bD(c_{n},|c_{n}|^2)}\lesssim  |c_{n}|^{\bar a_{4}}
 $$
and 
$$|\hat F(q_{n},c_{n})|\geq e^{-q_{n}h}|c_{n}|^{\bar a_{4}+1/2}.$$
Then,  
$$m_{\Phi_{\Omega}\circ f_{F}}(c_{n})\geq C_{h}^{-1}  |c_{n}|^{2\bar a_{4}+1}e^{-4q_{n}h}.$$
The constant $C_{h}$ can be chosen to be non increasing w.r.t. $h$.
\end{prop}
\begin{proof} The Proposition is a consequence of  the more precise statement given by Proposition \ref{prop:15.1} below. 

\medskip
Since  for $n$ large enough
$$e^{-q_{n}h}< \min(|c_{n}/4|,q_{n}^{-9})^{10}$$
$$|\hat F(q_{n},c_{n})|\geq e^{-q_{n}h}|c_{n}|^{1/2}\|F\|_{\T_{h}\times \bD(c_{n},|c_{n}|^2)}$$
$$\lim_{n\to\infty}|c_{n}|^{-1/2}q_{n}\min(|c_{n}/4|,q_{n}^{-9})=0$$
 we can apply 
Proposition \ref{prop:15.1} with $q=q_{n}$, $\e_{p/q}=\|F\|_{\T_{h}\times \bD(c_{n},|c_{n}|^2)}$,
$\nu_{q}=|c_{n}|^{1/2}$, $c_{p/q}=c_{n}$, $\rho_{p/q}=\min(|c_{n}/4|,q_{n}^{-9})$. We then get
$$m_{\Phi_{\Omega}\circ f_{F}}(c_{n})\gtrsim (|c_{n}|^{1/2} e^{-q_{n}h}\|F\|_{\T_{h}\times \bD(c_{n},|c_{n}|^2})^2.$$
But since $((2\pi)^{-1}\int_{0}^{2\pi}|F(\th,c_{n})|^2d\th)^{1/2}\geq |\hat F(q_{n},c_{n})|$ one has $\|F\|_{\T_{h}\times \bD(c_{n},|c_{n}|^2}\geq |\hat F(q_{n},c_{n})|$ hence
$$m_{\Phi_{\Omega}\circ f_{F}}(c_{n})\geq C_{h}^{-1}|c_{n}|^{2\bar a_{4}+1}e^{-4q_{n}h}.$$
\end{proof}
We now state Proposition \ref{prop:15.1}.

\medskip
For   $p\in\Z$, $q\in\N^*$, $p\wedge q=1$, $p/q$ small enough there exists a unique $c_{p/q}\in\bD(0,\bar\rho)$ such that 
 $$\omega(c_{p/q}):=(2\pi)^{-1}\pa\Omega(c_{p/q})=p/q.
 $$ 
 We define
 \be \rho_{p/q}=\min(|c_{p/q}/4|,q^{-9})\label{defrhop/q}
 \ee
 and assume that
 \be \e_{{p/q}}:=\|F\|_{\bD(c_{p/q},\rho_{p/q})}\leq |c_{p/q}|^{\bar a_{4}}.\label{16.320}
 \ee
The $\pm q$-th Fourier coefficients of  $F(\cdot,r)$, $\hat F(\pm q,r)=(2\pi)^{-1}\int_{0}^{2\pi}F(\th,r)e^{\mp iq\th}d\th$ satisfy 
$$|\hat F(\pm q,r)|\lesssim e^{-qh}\e_{{p/q}}.
$$
and since $F$ is $\s$-symmetric, for every $r\in\bD(0,\bar\rho)$, $\overline{\hat F( q,r)}=\hat F(-q,r)$. 
\begin{prop}\label{prop:15.1}Assume (\ref{16.320}) is satisfied and 
\be e^{-qh}<\rho_{p/q}^{10}\label{condexpqh}\ee
\be |\hat F(\pm q,c_{p/q})|=\nu_{q}e^{-qh} \e_{p/q}.
\label{nu16.321}
\ee
\be  
 \nu_{q}^{-1}q\rho_{p/q}\leq 1/q.\label{15.255a}
\ee
Then, there exists  in a neighborhood of $\T\times \{c_{p/q}\}\subset\T\times\R$  an open set of area $\geq C_{h}^{-1}(\nu_{q}\e_{p/q} e^{-qh})^{3/2}$, $C_{h}>0$, that has an empty intersection with any possible (horizontal) invariant circle of  the  symplectic diffeomorphism $\Phi_{\Omega}\circ f_{F}$.
\end{prop}
\begin{rem}One can choose the constant $C_{h}$ to be non increasing  with respect to  $h$. 
\end{rem}
The proof of this Proposition will occupy the next subsections.
\subsection{Putting the system into $q$-resonant Normal Form}
Conditions (\ref{defrhop/q}) and (\ref{16.320}) show that we can apply  Proposition \ref{prop:11.5}: it  provides us with the following $q$-resonant Normal Form
\be  g_{RNF}^{-1}\circ \Phi_{\Omega}\circ  f_{F}\circ g_{RNF}= \Phi_{2\pi(p/q)r}\circ \Phi_{\bar\Omega}\circ  f_{\bar F^{res}}\circ f_{ F^{cor}}\label{15.246}\ee
$$\begin{cases}&\bar \Omega=\Omega-2\pi(p/q)r+\cM_{0}(F^{res})\\
&\bar F^{res}=F^{res}-\cM_{0}(F^{res})\end{cases}$$
 $F^{res}\in\cO_{\s}(\T_{h-1/q}\times \bD(c_{p/q},e^{-1/q}\rho_{p/q}))$, $\bar F^{res}=F^{res}-\cM_{0}(F^{res})$ which are $1/q$-periodic (in the $\th$-variable) and such that 
$$ \|F^{res}\|_{\T_{h-1/q}\times \bD(c_{p/q},e^{-1/q}\rho_{p/q})}\lesssim \e_{p/q}\label{15.296-}
$$
$$ F^{res}=T_{N}^{res}(F+O(q\rho_{p/q}\|F\|_{W_{h,\bD(c_{p/q}, \rho_{p/q})}}))
$$
\begin{lemma}
On $\T_{1/q}\times \bD(c_{p/q},e^{-1/q}\rho_{p/q}/2)$ one has 
\be F^{res}(\th,r)=u_{0}^{res}(r)+\sum_{\pm}u_{1,\pm}^{res}(r)e^{\pm iq\th}+u^{res}_{\geq 2}(\th,r)\label{15.255}\ee
where on $\bD(c_{p/q},e^{-1/q}\rho_{p/q}/2)$ one has
$$u_{0}^{res}(r)=\cM_{0}(F^{res})=\hat F(0,r)+O(q\rho_{p/q} \e_{p/q})
$$
\be  u_{1,\pm}^{res}(r)=\hat F(\pm q,r)+O(e^{-qh} q\rho_{p/q} \e_{p/q})=O(e^{-qh}\e_{p/q})\label{f1pm}\ee
and
$$\|u_{\geq 2}^{res}\|_{\bT_{1/q}\times \bD(c_{p/q},e^{-1/q}\rho_{p/q}/2)}\lesssim e^{-2qh}\e_{p/q} 
$$
\end{lemma}
\begin{proof}
We recall that from (\ref{Fres})
$$ F^{res}=T_{N}^{q-res}(F+G)$$
where 
$$\|G\|_{h-1/q,e^{-1/q}\rho_{p/q}/2}=O(q\rho_{p/q}\|F\|_{h,\bD(c_{p/q},\rho_{p/q}/2)}).
$$
Hence
\be|\hat G(0,r)|\lesssim q\rho_{p.q}\e_{p/q},\qquad |\hat G(\pm q,r)|\lesssim e^{-q(h-2/q)}q\rho_{p.q}\e_{p/q}\lesssim e^{-qh}q\rho_{p.q}\e_{p/q}.\label{15.256}
\ee
On the other hand since $e^{-2q(h-3/q)}\lesssim e^{-2qh}$
$$\|T_{N}^{q-res}F-\hat F(0,r)-\sum_{\pm}\hat F(\pm q,r)e^{\pm iq \th}\|_{1/q,\rho_{p/q}/2}\lesssim e^{-2qh}\e_{p/q}
$$
and 
$$\|T_{N}^{q-res}G-\hat G(0,r)-\sum_{\pm}\hat G(\pm q,r)e^{\pm iq\th}\|_{1/q,e^{-1/q}\rho_{p/q}/2}\lesssim q \rho_{p/q}\e_{p/q} e^{-2qh}.
$$
Summing these two inequalities and using (\ref{15.256}) gives (\ref{15.255})
\end{proof}
With these notations
$$\begin{cases}&\bar \Omega=\Omega-2\pi(p/q)r+u_{0}^{res}(r)\\
&\bar F^{res}=F^{res}-u_{0}^{res}(r).\end{cases}
$$
We denote by $\bar c\in\R$ the point  where 
$$\pa\bar\Omega(\bar c)=0;$$
since $\|u_{0}^{res}\|_{\bD(c_{p/q},\rho_{p/q})}\lesssim \e_{p/q}$ and $\Omega$ satisfies the twist condition (\ref{15.358}) one has
$$ \bar c=c+O(\e_{p/q})\in \bD(c,(3/4)\rho_{p/q}),\qquad \bar \Omega(r)={\rm cst}+(\varpi/2) (r-c)^2+O((r-c)^3)
$$
for some $\varpi\gtrsim A^{-1}$.
Since  $F^{res}$ is $\s$-symmetric we can write
$$\sum_{\pm}u_{1,\pm}^{res}(r)e^{\pm iq\th}=a(r)\cos(q\th)+b(r)\sin(q\th)
$$
and from 
(\ref{nu16.321}), (\ref{f1pm}), (\ref{15.255a})
we can assume, shifting the variable $\th\in\T_{h}$ by  translation $\th\mapsto \th+\a_{c}$ ($\a_{c}\in\T$) if necessary, that 
$$ b(\bar c)=0,\quad a(\bar c)=\bar \nu_{q}e^{-qh}\e_{p/q},\quad \bar\nu_{q}=\nu_{q}-O(q\rho_{p/q})=\nu_{q}(1+o_{q^{-1}}(1))\label{16.328}
$$

with 
$$ \max(\|a\|_{\bD(c,\rho_{p/q})},\|b\|_{\bD(c,\rho_{p/q})})\lesssim e^{-qh}\e_{p/q}.\label{16.329}
$$
Thus,
$$\begin{cases}&\bar \Omega(r)=\Omega(r)-2\pi(p/q)r+u_{0}^{res}(r)={\rm cst}+(\varpi/2) (r-\bar c)^2+O((r-\bar c)^3)\\
&\bar F^{res}(\th,r)=a(r)\cos(q\th)+b(r)\sin(q\th)+u^{res}_{\geq 2}(\th,r).\end{cases}$$

\subsection{Coverings}Like in  subsection \ref{sec:8.2} ({\it cf.} (\ref{10.163})) we define 
$$\ti\Omega^{res}\in\cO_{\s}(\bD(0,qe^{-2/q}\rho_{p/q}/2))),\qquad \ti F^{res}\in \cO_{\s}(\T_{qh-2}\times \bD(0,qe^{-2/q}\rho_{p/q}/2))$$
\be \begin{cases}&\ti \Omega^{res}(r)=q^2\bar\Omega(\bar c+r/q)\\
& \ti F^{res}(\th,r)=q^2\bar F^{res}([\th/q]_{\textrm{mod\ }(2\pi/q)\Z},\bar c+r/q)
\end{cases}\label{15.298a}
\ee
hence
$$\begin{cases}&\ti\Omega^{res}(r)={\rm cst}+\varpi r^2/2+O(r^3)=\varpi r^2/2+\ti\omega(r)\\
&\ti F^{res}(\th,r)=\ti a(r)\cos(q\th)+\ti b(r)\sin(q\th)+\ti u^{res}_{\geq 2}(\th,r)\end{cases}$$
with 
$$\ti a(r)=q^2a(\bar c+r/q),\qquad \ti b(r)=q^2b(\bar c+r/q), \qquad  \ti u^{res}_{\geq 2}(\th,r)=q^2u^{res}_{\geq 2}(\th/q,\bar c+r/q).\label{22.23}
$$
Let us define
\begin{align}\ti H^{res}(\th,r):=&\ti\Omega^{res}(r)+\ti F^{res}(\th,r)\notag\\
=&{\rm cst}+(1/2)\varpi r^2 +\ti a(r)\cos\th+\ti b(r)\sin\th+\ti\omega(r)+\ti u^{res}_{\geq 2}(\th,r).\label{16.331}
\end{align}
Making explicit the linear plus quadratic part  $H_{Q}(\th,r)$ of $(1/2)\varpi r^2 +\ti a(r)\cos\th+\ti b(r)\sin\th$ at $(\th,r)=(0,0)\in\R^2$  we can write 
\be  \ti H^{res}(\th,r)=H_{Q}(\th,r)+\ti\omega(r)+g(\th,r)\label{16.336}\ee
with
\be\begin{cases} &H_{Q}(\th,r)=\frac{1}{2}\biggl\<Q\bm\th \\ r\em,  \bm \th\\ r \em\biggr\>+\biggl\<\bm 0\\ \pa_{r}\ti a(0)\em,\bm\th\\ r \em\biggr\> \\
&  Q=\bm -\ti a(0)& \pa_{r}\ti b(0)\\ \pa_{r}\ti b(0)&\varpi+\pa_{r}^2\ti a(0)\em
\end{cases}
\ee
and
\be g(\th,r)=g_{0}(\th,r)+\ti u^{res}_{\geq 2}(\th,r),\qquad g_{0}(\th,r)=O_{3}(\th,r).
\label{16.331}
\ee
For further records we mention the following estimates.
\be  \max(\|\ti a\|_{\bD(0,q\rho_{p/q})},\|\ti b\|_{\bD(0,q\rho_{p/q})})\lesssim q^2e^{-qh}\e_{p/q}\label{16.329}\ee
\be  \ti b(0)=0,\qquad \ti a(0)=q^2\bar \nu_{q}e^{-qh}\e_{p/q},\qquad \bar\nu_{q}=\nu_{q}(1+o_{1/q}(1))\label{16.329bis}
\ee
\be\|\pa_{\th}^{l_{1}}\pa_{r}^{l_{2}} g_{0}(\th,r)\|_{\bD(0,t)\times \bD(0,t)}\lesssim q^2t^{3-k_{2}} e^{-qh}\e_{p/q}\label{16.338}
\ee
\be \|\ti u_{\geq 2}^{res}\|_{\bT_{1/q}\times \bD(0,e^{-1/q}q\rho_{p/q}/2)}\lesssim q^2e^{-2qh} \e_{p/q} \label{16.334} 
\ee
\be \|\pa_{\th}^{k_{1}}\pa_{r}^{k_{2}} g(\th,r)\|_{\bD(0,t)\times \bD(0,t)}\lesssim [q^2t^{3-k_{2}} e^{-qh}\e_{p/q}+q^{k_{1}}q^2e^{-2qh} \e_{p/q} ].\label{16.338bis}
\ee
\subsection{Existence of a hyperbolic point for $f_{H_{Q}+\ti \omega}$}

We refer to the Appendix \ref{sec:K} for the definition of the notion of a $(\kappa,\d)$-hyperbolic fixed point. 
\begin{lemma}\label{lemma:15.4}The affine symplectic map $f_{H_{Q}+\ti\omega(r)}$ has a $(\kappa,\delta)$-hyperbolic fixed point $(\th_{0},r_{0})\in \bD(0,\rho_{p/q}^5)^2\cap\R^2$  with 
$$ \delta=\kappa=q(\varpi \nu_{q}^{}\e_{p/q}e^{-qh})^{1/2}(1+o_{1/q}(1))$$
with  stable and unstable directions at this point  of the form $\bm 1\\ m_{\pm}\em$
 where 
$$m_{\pm} =\pm q(  \nu_{q}e^{-qh}\e_{p/q}/\varpi)^{1/2}(1+o(1)).$$

\end{lemma}
\begin{proof} See the Appendix \ref{sec:K2}
\end{proof}

\subsection{Stable and unstable manifolds of $f_{\ti H^{res}}$}
\begin{lemma}\label{lem:15.3} The symplectic diffeomorphism $f_{\ti H^{res}}$ has a $(\kappa,\d)$-hyperbolic fixed point $(\th_{1},r_{1})\in \bD(0,\rho_{p/q}^4)^2\cap\R^2$
with
$$\kappa=\delta=q(\varpi  \nu_{q}e^{-qh}\e_{p/q})^{1/2}(1+o_{1/q}(1)).$$
The stable and unstable directions at this point are of the form $\bm 1\\ m_{\pm}\em$
 where 
$$m_{\pm} =\pm q(   \nu_{q}e^{-qh}\e_{p/q}/\varpi)^{1/2}(1+o(1)).$$
 \end{lemma}
\begin{proof}
From (\ref{16.336}),
\begin{align*}f_{\ti H^{res}}&=f_{H_{Q}+\ti\omega+g}\\
&=f_{H_{Q}+\ti\omega}\circ f_{\ti g},\qquad \ti g=\fO_{1}(g)
\end{align*}
and from (\ref{16.338bis}) we get 
\begin{align*}\|Df_{\ti g}-id\|_{\bD(0,10\th_{0})\times \bD(0,10r_{0})}&\lesssim  [q^2r_{0} e^{-qh}\e_{p/q}+q^4e^{-2qh} \e_{p/q} ]\\
&\lesssim (q^2\rho_{p/q}^5+q^{4}e^{-qh})\e_{p/q}e^{-qh}.
\end{align*}

The Stable Manifold Theorem  \ref{theo:stableunstabletheorem} of the Appendix shows that the conclusion of the Lemma is true provided for some constant $C>0$ ({\it cf.} (\ref{F407}))
$$\|f_{\ti g}-id\|_{C^1(\bD(0,10\th_{0})\times \bD(0,10r_{0})})\leq C^{-1}\kappa\d\rho_{p/q}
$$
a condition that is implied by (recall that from (\ref{15.255a}) one has $\nu_{q}\gg q\rho_{p/q}$)
$$(q^2\rho_{p/q}^5+q^4e^{-qh})< q^2\rho_{p/q}^2\qquad (<C^{-1}q^2\rho_{p/q}\nu_{q})
$$
But (\ref{defrhop/q}), (\ref{condexpqh}) show that this last inequality is satisfied if $q\gg1$.

\end{proof}

\subsection{Stable and unstable manifolds of $\Phi_{\Omega}\circ f_{F}$}
\begin{lemma} The diffeomorphism $\Phi_{\Omega}\circ f_{F}$ has a hyperbolic $q$-periodic point $(\bar \th,\bar r)$  the local stable and unstable manifolds of  which  are graphs of $C^1$-functions $w_{-},w_{+}:]\bar \th-\rho,\bar \th+\rho)[\to\R$ such that
$$\begin{cases}&(m/2)\kappa|\th-\bar\th|\leq |w_{+}(\th)-w_{-}(\th)|\leq 2m|\th-\bar\th|\\ 
& m=q(  \nu_{q}e^{-qh}\e_{p/q}/\varpi)^{1/2}(1+o(1))\\
&\rho=C^{-1}\nu_{q}e^{-qh}\e_{p/q}.
\end{cases}
$$
\end{lemma}
\begin{proof}

Recall, {\it cf.}  (\ref{15.246}), that 
\be g_{RNF}^{-1}\circ \Phi_{\Omega}\circ  f_{F}\circ g_{RNF}= \Phi_{2\pi(p/q)r}\circ \Phi_{\bar\Omega}\circ  f_{\bar F^{res}}\circ f_{ F^{cor}}.
\ee
From (\ref{15.298a}), the pre-image of $(\th_{1},r_{1})$, by $(\th,r)\mapsto ([(\th-\a_c)/q]_{\textrm{mod\ }(2\pi/q)\Z},\bar c+r/q)$ is a $q$-periodic orbit  $O_{q}\subset\T\times]c_{p/q}-\rho_{p/q}^4,c_{p/q}+\rho_{p/q}^4[\subset\T\times]c_{p/q}-\rho_{p/q}/3,c_{p/q}+\rho_{p/q}/3[$ of $\Phi_{\bar \Omega}\circ f_{\bar F^{res}}$ as well as of $\Phi_{2\pi(p/q)r}\circ \Phi_{\bar \Omega}\circ f_{\bar F^{res}}$ ($\bar F^{res}$ is $2\pi/q$-periodic); Lemma \ref{lem:15.3} tells us that this periodic orbit is hyperbolic. Let $u_{0}\in\T\times ]c_{p/q}-\rho_{p/q}^4,c_{p/q}+\rho_{p/q}^4[$ be a point of $O_{q}$ and denote   $\ph=\Phi_{2\pi(p/q)r}\circ \Phi_{\bar \Omega}\circ f_{\bar F^{res}}$. One has $\ph^q(u_{0})=u_{0}$ and we want to find a hyperbolic  fixed  point for $(\ph\circ f_{F^{cor}})^q$ (the $q$-th iterate of  $\ph\circ f_{F^{cor}}$) close to $u_{0}$. 

 We can write 
$$(\ph\circ  f_{ F^{cor}})^q=\ph^q\circ j$$
where 
$$j=(\ph^{-(q-1)}\circ f_{F^{cor}}\circ \ph^{q-1})\circ \cdots\circ (\ph^{-1}\circ  f_{F^{cor}}\circ \ph)\circ f_{F^{cor}}.$$
Since $\|(\Phi_{2\pi(p/q)r}\circ \Phi_{\bar\Omega})^n\|_{C^2(\T\times]c_{p/q}-\rho_{p/q}/3,c_{p/q}+\rho_{p/q}/3[) }\lesssim 1$ uniformly in $n$ and 
$$\| \bar F^{res}\|_{C^2(\T\times]c_{p/q}-\rho_{p/q}/3,c_{p/q}+\rho_{p/q}/3[)}\lesssim \e_{p/q} \rho_{p/q}^{-2}\lesssim 1$$ one has for $n\leq  \rho_{p/q}^2/\e_{p/q}$, 
\be\|\ph^n\|_{C^2(\T\times]c_{p/q}-\rho_{p/q}/3,c_{p/q}+\rho_{p/q}/3[)}\lesssim 1\label{15.13}
\ee
and consequently ($q\ll \bar\e^{-1}$)
\begin{align*}\|j-id\|_{C^1(\T\times]c_{p/q}-\rho_{p/q}/3,c_{p/q}+\rho_{p/q}/3[)}&\lesssim q\|F^{cor}\|_{C^1(\T\times]c_{p/q}-\rho_{p/q}/3,c_{p/q}+\rho_{p/q}/3[)}\\
&\lesssim \exp(-\rho_{p/q}^{\ -1/3}).
\end{align*}
Replacing $\ph^q$ and $j$ by $T\circ \ph^q\circ T^{-1}$ and  $T\circ j\circ T^{-1}$  where $T:u\mapsto u-u_{0}$ we can assume that $u_{0}=0\in\bD_{\R}(\ti c,\rho_{p/q}^4)^2\subset \bD_{\R}(\ti c,\rho_{p/q}/3)^2\subset \T\times ]\ti c-|c_{p/q}/3|,\ti c+\rho_{p/q}/3[$. We then have $\ph^q(0)=0$ and the matrix $D\ph^q(0)$ is $(\kappa,\delta)$-hyperbolic with 
$$\d \kappa=q^2\nu_{q}e^{-qh}\e_{p/q}(1+o(1)).
$$
 Write $\ph^q(u)=D\ph^q(0)\xi(u)$ with $\xi(0)=0$, $D\xi(0)=id$ so that 
$$\ph^q\circ j=D\ph^q(0)\circ \xi\circ j.
$$
Observe that for $0<\rho<\rho_{p/q}/4$ and $k=0,1$
\begin{align*}\|D^k(\xi\circ j-id)\|_{C^0(\bD_{\R}(0,\rho))^2}&\lesssim \|D^k(\xi-id)\|_{C^0(\bD_{\R}(0,\rho))^2}+\|j-id\|_{C^1(\bD_{\R}(0,\rho))^2}\\
&\lesssim  \rho^{2-k}+ q\exp(-\rho_{p/q}^{\ -1/3}).
\end{align*}
Let us choose $$\rho=C^{-1}\nu_{q}e^{-qh}\e_{p/q}
$$
with $C$ large enough.
The Stable Manifold Theorem ({\it cf.} Appendix, Theorem \ref{theo:stableunstabletheorem}) shows that the diffeomorphism 
$\ph^q\circ j$ has a hyperbolic fixed point the stable and unstable manifolds of which are graphs of $C^1$ functions of the form  $\ti w_{-},\ti w_{+}:]-\rho/2,\rho/2[\to\R$, $\ti w_{-}<0<\ti w_{+}$, such that
\begin{multline*}(3/2)m_{-}\th\leq \ti w_{-}(\th)\leq (2/3)m_{-}\th\leq 0\leq  (2/3)m_{+}\th<\ti w_{+}(\th)\leq (3/2)m_{+}\th.\end{multline*}
To conclude the proof of the Lemma we set $w_{\pm}=\ti w_{\pm}\circ g_{RNF}^{-1}$ and  note that 
$$ g_{RNF}^{-1}\circ \Phi_{\Omega}\circ  f_{F}\circ g_{RNF}=\ph^q\circ j$$
with $\|g_{RNF}^{-1}-id\|_{C^1}\leq 1/10$.

\end{proof}

\subsection{End of the proof of Proposition \ref{prop:15.1} }
Let $V$ be the set
$$V=\{(\th,r),\ \th\in[\bar\th,\bar\th+\rho/2], w_{-}(\th)\leq r\leq w_{+}(\th)\}$$
the boundary of which is made by  two pieces of stable and unstable manifolds and the vertical segment $L:=\{\bar\th+\rho/2\}\times [w_{-}(\bar\th+\rho/2),w_{+}(\bar\th+\rho/2]$. By a theorem of Birkhoff  \cite{Bi2} (cf. also \cite{He}), any invariant curve of the  twist diffeomorphism $\Phi_{\Omega}\circ f_{F}$ is the graph of a Lipschitz function $\gamma:\T\to [-1,1]$;  if this curve intersects the stable or  unstable manifold of $(\bar\th,\bar r)$ it  must be included in the union of these stable and unstable manifolds which is impossible. So if  this  invariant curve intersects the interior of  $V$ it has to enter in $V$ by first entering  the vertical segment $L$ by the right. But this is clearly impossible also (see  Figure \ref{Pict:3}).

Now the domain $V$ has an area which is 
\begin{align*}\area(V)&\gtrsim \rho\times  (m_{+}-m_{-}) \\
&\gtrsim ( \nu_{q}e^{-qh}\e_{p/q})^{3/2}.
\end{align*}
This concludes the proof of Proposition \ref{prop:15.1} if we notice that the  dependence on $h$ of the implicit constant in the symbol $\gtrsim$ appears only when we apply  Proposition \ref{prop:11.5} ({\it cf.} Remark \ref{remE1}).

\begin{figure}
\begin{center}
\definecolor{mygray}{gray}{0.8}
\begin{pspicture}(-3,-2)(3,3)

%
\psdot[dotsize=0.1](0,0)

\rput(0,-0.5){$(\bar\th,\bar r)$}
\rput(2,0){$w_{-}(\cdot)$}
\rput(1,1){$w_{+}(\cdot)$}
\rput(3,-2.2){$\bar\th+\rho/2$}
\rput(-0.5,2.5){$\gamma(\cdot)$}

\rput(1,-1){$V$}
\psline[linewidth=0.5pt](1,-0.8)(1.3,0.3)

\pscustom[fillstyle=solid,fillcolor=mygray]{%
\pscurve[linewidth=1pt,liftpen=0](3,1)(1.5,0.25)(0,0)
\pscurve[linewidth=1pt,liftpen=0](0,0)(1.5,0.8)(3,2)}

\psline[linewidth=1pt](3,1)(3,2)

\pscurve[linewidth=0.5pt](-3,2.5)(-1,2)(1,2)(2,2.5)(3,2.5)(4,1.5)(3,1.5)(2.5,1)

\rput(1,-1){$V$}
\psline[linewidth=0.5pt](1,-0.8)(1.3,0.3)

\psline[linewidth=0.5pt,linestyle=dashed](3,-1.8)(3,4)

\psline[linewidth=0.5pt](-3,-1.8)(3.3,-1.81)
\rput(-2.8,-1.5){$\T$}

\psline[linewidth=0.5pt](-3.5,-1.8)(-3.5,3)
\rput(-3.8,1.5){$\R$}

\rput(3.5,0.8){$L$}
\psline[linewidth=0.5pt](3,1.2)(3.4,0.9)

\end{pspicture}
\end{center}
\caption{Invariant graphs cannot intersect the interior of $V$.}\label{Pict:3}
\end{figure}
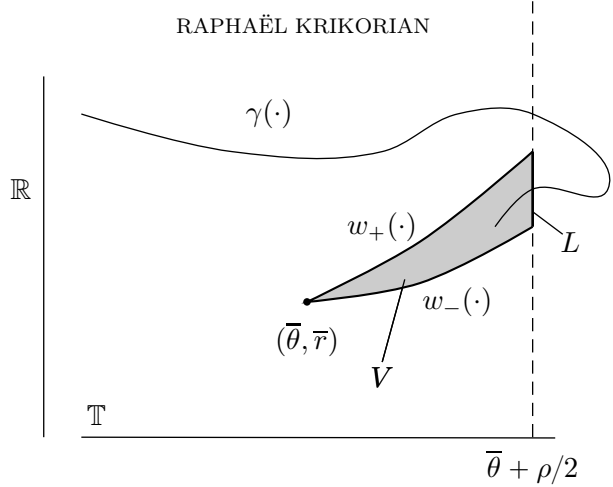

\section{Divergent BNF: proof of Theorems \ref{theo:MainD},  \ref{theo:main2} and \ref{theo:main2prime}}\label{sec:16}
We now use the result of the previous Section to construct examples of  real analytic symplectic diffeomorphisms of the disk and the annulus  with divergent BNF.

\subsection{Proof of Theorems \ref{theo:MainD} and \ref{theo:main2}: the (AA) Case}\label{sec:17.1}

Let $f=\Phi_{2\pi\omega_{0}r}\circ f_{O(r^2)}$ be a a real-analytic symplectic  twist map of the annulus of the form  (\ref{newn1.1bis})  and  satisfying the twist condition (\ref{twistcond1.10}). 
We perform a Birkhoff Normal Form, {\it cf.} Proposition \ref{statementBNF}, on $f$ up to order $\bar a_{4}$, where $\bar a_{4}$ is the integer of Proposition \ref{prop:16.1}:
 there exist $\bar \rho>0$,  $g\in{\rm Symp}_{ex,\s}(\T_{h}\times \bD(0,\bar\rho))$ exact symplectic, $\Omega\in \cO_{\s}(\bD(0,\bar\rho))$, $F\in \cO_{\s}(\T_{h}\times \bD(0,\bar\rho))$  such that 
$$g^{-1}\circ f \circ g=\Phi_{\Omega}\circ f_{F}$$
where ($b_{2}\ne 0$)
$$(2\pi)^{-1}\Omega(r)=\omega_{0}r+b_{2}r^2+O(r^3),\quad F(\th,r)=O(r^{\bar a_{4}}),\quad g-id=O(r^2).$$
Note that for $\bar \rho$ small enough $\Omega$ satisfies a $((5/2)\min(b_{2},b_{2}^{-1}) ,B)$-twist condition on $\bD(0,\bar\rho)$.
In particular, if $(p_{n}/q_{n})_{n\geq 1}$ are the convergents of $\omega_{0}$ and  $c_{n}\in\R$ ($n$ large enough) is  the point where 
\be (2\pi)^{-1}\pa\Omega(c_{n})=p_{n}/q_{n},\qquad |c_{n}|\leq \frac{(5/2)\max(b_{2}(f),b_{2}(f)^{-1})}{q_{n}q_{n+1}}\label{16.381}\ee
({\it cf.} (\ref{e16.318}) and the twist condition satisfied by $\Omega$)  one has 
$$\|F\|_{\T_{h}\times \bD(c_{n},|c_{n}|^2)}\lesssim |c_{n}|^{\bar a_{4}}.
$$
For $(\zeta_{1,k})_{k\geq 1}, (\zeta_{2,k})_{k\geq 1}\in [-1,1]^{\N^*}$, let $G_{\zeta}\in \T_{h}\times \bD(0,1)$ defined by 
$$G_{\zeta}(\th,r)=r^{\bar a_{4}}\sum_{k\geq 1}\zeta_{1,k} e^{-q_{k}h}\cos(q_{k}\th)+\zeta_{2,k} e^{-q_{k}h}\sin(q_{k}\th).$$
We now define $f_{\zeta}\in {\rm Symp}^\cO_{\s}(\T\times \bD(0,1))$, $F_{\zeta}\in \cO_{\s}(\T_{h}\times \bD(0,\bar\rho))$ by 
\begin{align*} &f_{\zeta}=f\circ f_{G_{\zeta}},\\
&\Phi_{\Omega}\circ f_{F_{\zeta}}:=g^{-1}\circ f_{\zeta}\circ g=\Phi_{\Omega}\circ f_{F}\circ g^{-1}\circ  f_{G_{\zeta}}\circ g.
\end{align*}

\begin{lemma}\label{lem:15.1bis}For $n$, $c_{n}$ as above, there exists a set $J_{n}(F)\subset [-1,1]^2$ of 2-dimensional Lebesgue measure $\lesssim |c_{n}|^2$ such that for any $\zeta\in ([-1,1]^2)^{\N^*}$, such that 
$\zeta_{n}\in [-1,1]^2\setminus J_{n}(F)$ one has
$$m_{\Phi_{\Omega}\circ f_{F_{\zeta}}}(c_{n})\gtrsim  |c_{n}|^{2\bar a_{4}+1}e^{-4q_{n}h}.
$$
\end{lemma}
\begin{proof}
Let $\a_{n}\in\T$  and $\nu_{q_{n}}\geq 0$ be such that 
$$\hat F(q_{n},c_{n})e^{iq_{n}\th}+\hat F(-q_{n},c_{n})e^{-iq_{n}\th}=|c_{n}|^{\bar a_{4}}\nu_{q_{n}}e^{-q_{n}h}\cos(q_{n}\th+\a_{q_{n}}).$$
Since $\Phi_{\Omega}\circ f_{F_{\zeta}}=\Phi_{\Omega}\circ f_{F}\circ g^{-1}\circ  f_{G_{\zeta}}\circ g$, $F=O(r^{\bar a_{4}})$ 
and $g-id=O(r^2)$ we see that 
\be F_{\zeta}=F+G_{\zeta}+O(r^{\bar a_{4}+1}).\label{17.349}\ee
We can write 
$$G_{\zeta}(\th,r)=r^{\bar a_{4}}\sum_{k\geq 1}\ti \zeta_{1,k} e^{-q_{k}h}\cos(q_{k}\th+\a_{q_{k}})+\ti \zeta_{2,k} e^{-q_{k}h}\sin(q_{k}\th+\a_{q_{k}})$$
with $\ti \zeta_{1,k}-i\ti\zeta_{2,k}=e^{-i\a_{q_{k}}}(\zeta_{1,k}-i\zeta_{2,k})$
and  from (\ref{17.349}) we see that 
\begin{multline*}\hat F_{\zeta}(q_{n},c_{n})e^{iq_{n}\th}+\hat F_{\zeta}(-q_{n},c_{n})e^{-iq_{n}\th}=\\ |c_{n}|^{\bar a_{4}}e^{-q_{n}h} \biggl((\nu_{q_{n}}+\ti \zeta_{1,n}+u_{n}(\zeta))\cos(q_{n}\th+\a_{q_{n}})+(\ti \zeta_{2,n}+v_{n}(\zeta))\sin(q_{n}\th+\a_{q_{n}})\biggr)
\end{multline*}
where 
$$\sup_{\zeta\in [-1,1]^{2\N^*}}(|u_{n}(\zeta),|v_{n}(\zeta))| \lesssim |c_{n}|.
$$ 
We can thus write for $\zeta_{1},\zeta_{2}\in ]-1,1[^{\N^*}$
$$2|\hat F_{\zeta}(q_{n},c_{n})|=\nu_{n}(\zeta)|c_{n}|^{\bar a_{4}}e^{-q_{n}h}$$
with
\begin{align*}\nu_{n}(\zeta)^2&=(\nu_{q_{n}}+\ti \zeta_{1,n}+u_{n}(\zeta))^2+(\ti \zeta_{2,n}+v_{n}(\zeta))^2\\
&\geq  |\ti \zeta_{2,n}-O(c_{n})|^2.
\end{align*}
Since $\ti \zeta_{1,k}-i\ti\zeta_{2,k}=e^{-i\a_{q_{k}}}(\zeta_{1,k}-i\zeta_{2,k})$, one can hence find a set $J_{n}(F)\subset [-1,1]^2$ of 2-dimensional Lebesgue measure 
$$|J_{n}(F)|\lesssim |c_{n}|^{1/2}
$$
such that 
$$(\zeta_{1,n},\zeta_{2,n})\in [-1,1]^2\setminus J_{n}(F)\implies |\nu_{n}(\zeta)| \gtrsim |c_{n}|^{1/2}.$$
By Proposition \ref{prop:16.1} we thus have
$$m_{\Phi_{\Omega}\circ f_{F_{\zeta}}}(c_{n})\gtrsim  |c_{n}|^{2\bar a_{4}+1}e^{-4q_{n}h}.
$$

\end{proof}
\begin{lemma}\label{lemma:17.1}Let  $\cN\subset\N$ be infinite. Then, for almost every $\zeta\in ([-1,1]^2)^{\N^{*}}$ for the product measure  $\mu_{\infty}=  ({\rm Leb}_{[-1,1]^2})^{\otimes \N^*}$, there exists an infinite subset $\ti\cN\subset\cN$ such that  for all $n\in\ti\cN$
$$m_{\Phi_{\Omega}\circ f_{F_{\zeta}}}(c_{n})\gtrsim   |c_{n}|^{2\bar a_{4}+1}e^{-4q_{n}h}.
$$
\end{lemma}
\begin{proof}
Since the random variables $\zeta_{n}$, $n\in\cN$ are independent, for any $m\in\N$, the  event $\{\zeta_{n}\in J_{n}(F),\ \forall\ n\geq m\}$  has zero $\mu_{\infty}$-probability as well as their union. Hence for $\mu_{\infty}$-almost every $\zeta\in\cX$, one has for infinitely many $n\in\cN$, $\zeta_{n}\notin J_{n}(F)$ and we conclude by Lemma (\ref{lem:15.1bis}).
\end{proof}

\subsubsection{Proof of Theorem  \ref{theo:MainD} }\label{sec:16.1.1}

We now observe that if  $\omega_{0}$ is diophantine with exponent $\tau$  
$$\tau=\limsup\frac{\ln q_{n+1}}{\ln q_{n}}$$
and for any $\b>0$ there exists a infinite set $\cN_{\b}$ such that for all $n\in\cN_{\b}$
$$q_{n+1}\geq q_{n}^{\tau-\b/4}.$$
On the other hand
$$|c_{n}|\asymp |\omega_{0}-\frac{p_{n}}{q_{n}}|\asymp \frac{1}{q_{n}q_{n+1}}\lesssim \frac{1}{q_{n}^{1+\tau-\b/4}}$$
hence
$$q_{n}\lesssim (1/|c_{n}|)^{(1/(1+\tau))+\b/4}
$$
and consequently, from Lemma \ref{lemma:17.1}, for an infinite number of $n\in\cN_{\b}$
\be m_{\Phi_{\Omega}\circ f_{F_{\zeta}}}(c_{n})\gtrsim  |c_{n}|^{2\bar a_{4}+1}e^{-4q_{n}h}\gtrsim \exp\biggl(-\biggl(\frac{1}{|c_{n}|}\biggr)^{(\frac{1}{1+\tau})+\b/2}\biggr).\label{16.354}
\ee
\ \hfill $\Box$
\subsubsection{Proof of Theorem \ref{theo:main2}}\label{sec:16.1.2}
We observe that since $t_{n}\geq 2|c_{n}|$ ({\it cf.} (\ref{ee1.20}) and (\ref{16.381}))   one has 
$$m_{f_{\zeta}}(t_{n})\gtrsim m_{\Phi_{\Omega}\circ f_{F_{\zeta}}}(c_{n})\gtrsim \exp\biggl(-\biggl(\frac{1}{|t_{n}|}\biggr)^{(\frac{1}{1+\tau})+\b}\biggr).\
$$
\ \hfill $\Box$

\subsection{Proof of  Theorems  \ref{theo:main2} and \ref{theo:main2prime}: (CC) Case}\label{sec:16.2}
Let $f$ be a real-analytic symplectic   diffeomorphism of the disk admitting the origin 0 as an elliptic equilibrium with irrational  frequency $\omega_{0}$   and satisfying the twist condition (\ref{twistcond1.10}); we {\it assume} that it is of the form 
$$f=\Phi_{\Omega((1/2)(x^2+y^2))}+O((x^2+y^2)^{\bar a_{4}})
$$
with $\Omega\in \cO_{\s}(\bD(0,1))$.  Passing to  the $(z,w)$-variables ({\it cf.} (\ref{changecoordxyzw})) we can write
$$\ph\circ f\circ \ph^{-1} =\Phi_{\Omega}\circ f_{F}$$
where  $F\in \cO_{\s}(\bD(0,1)^2)$
$$F(z,w)=O((zw)^{\bar a_{4}})\ (\textrm{and \ not\ only }O^{\bar a_{4}}(z,w)).$$
Let as before  $(p_{n}/q_{n})_{n\geq 1}$ be the convergents of $\omega_{0}$ and  $c_{n}\in\R$  the point where 
$(2\pi)^{-1}\pa\Omega(c_{n})=p_{n}/q_{n}$ ({\it cf.} (\ref{e16.318})).

For $(\zeta_{n})_{n\in\N^*}\in ([-1,1]^2)^{\N^*}$, let $G_{\zeta}\in\cO_{\s}(\bD(0,1)^2)$
\begin{multline*}G_{\zeta}(z,w)=(-izw)^{\bar a_{4}}\sum_{k=1}^\infty \frac{\zeta_{1,k}}{2}\times ((i^{-1/2}z)^{q_{k}}+(i^{-1/2}w)^{q_{k}})+\\ \frac{\zeta_{2,k}}{2i} ((i^{-1/2}z)^{q_{k}}-(i^{-1/2}w)^{q_{k}})
\end{multline*}

We now define
\begin{align*} &\Phi_{\Omega}\circ f_{F_{\zeta}}=\Phi_{\Omega}\circ f_{F}\circ \Phi_{G_{\zeta}}\\
&f_{\zeta}=\ph^{-1}\circ (\Phi_{\Omega}\circ f_{F_{\zeta}}) \circ \ph=f\circ \Phi_{G_{\zeta}\circ \ph}. \qquad ({\it cf.}\ (\ref{eq:4.4}))
\end{align*}
\begin{lemma}
Assume that for some $n$ large enough, $c_{n}$ {\it is positive}. Then, there exists $J_{n}(F)\subset [-1,1]^2$ of Lebesgue measure $\lesssim c_{n}^{1/2}$ such that if $\zeta_{n}=(\zeta_{1,n},\zeta_{2,n})\notin J_{n}(F)$ one has
$$m_{f_{\zeta}}(2c_{n})\gtrsim m_{\Phi_{\Omega}\circ f_{F_{\zeta}}}(c_{n})\gtrsim c_{n}^{2\bar a_{4}+1}e^{-4q_{n}h}.
$$
\end{lemma}
\begin{proof}
We define
\begin{align*}h_{n}&=-(1/2)\ln (c_{n}+c_{n}^2),
\end{align*}
and since 
$|\omega_{0}-\frac{p_{n}}{q_{n}}|\asymp \frac{1}{q_{n}q_{n+1}}\asymp c_{n}$
one has \begin{align*}
h_{n}&=(-1/2)\ln c_{n}+O(c_{n})\\
&=(-1/2)\ln c_{n}-O(1/q_{n}^{2})
\end{align*}
hence
\be e^{-q_{n}h_{n}}=c_{n}^{q_{n}/2}e^{O(1/q_{n})}<c_{n}^{10}.\label{cnqn/2}
\ee

Let $W_{n}^{CC}=W^{CC}_{h_{n},\bD(c_{n},c_{n}^2)}=\{(z,w)\in\C^2,\ \max(|z|,|w|)\leq e^{h_{n}}(c_{n}+c_{n}^2)^{1/2},\ -izw\in \bD(c_{n},c_{n}^2)\}$.
One has 
\be \|F\|_{W_{n}^{CC}}\lesssim |c_{n}|^{\bar a_{4}},\qquad \|G_{\zeta}\|_{W^{CC}_{n}}\lesssim c_{n}^{\bar a_{4}}. \label{estGzeta}\ee
Using Lemma \ref{lemma:4.5} we can pass to  (AA)-coordinates: if $\psi_{-}$ is the diffeomorphism defined in (\ref{defPsi})
$$\psi_{-}^{-1}(W^{CC}_{h_{n},\bD(c_{n},c_{n}^2)})   \supset W^{AA}_{h_{n},\bD(c_{n},c_{n}^2)}  =\T_{h_{n}}\times \bD(c_{n},c_{n}^2)
$$
and we can introduce $F^{AA}, F^{AA}_{\zeta}\in \cO_{\s}(\T_{h_{n}}\times \bD(c_{n},c_{n}^2))$ ({\it cf.}\ (\ref{eq:4.4}))
$$\psi_{-}^{-1}\circ f_{F}\circ \psi =f_{ F^{AA}},\qquad \psi_{-}^{-1}\circ f_{F_{\zeta}}\circ \psi =f_{ F_{\zeta}^{AA}}=f_{F^{AA}}\circ \Phi_{G_{\zeta}\circ \psi_{-}}.
$$
Since $F^{AA}=F\circ \psi_{-}+\fO_{2}(F)$ and  $F_{\zeta}^{AA}=F^{AA}+G_{\zeta}\circ \psi_{-}+\fO_{2}(F^{AA},G_{\zeta}\circ \psi_{-})$  ({\it cf.} (\ref{n4.64}), (\ref{l465})) one has on $W_{n}^{AA}$
\be\|F_{\zeta}^{AA}\|_{W_{n}^{AA}}\lesssim c_{n}^{\bar a_{4}},\qquad F_{\zeta}^{AA} =F^{AA}+G_{\zeta}\circ \psi_{-}+O(c_{n}^{(3/2)\bar a_{4}}). \label{estFAA}\ee
(we assume that $\bar a_{4}$ is large enough). If we define  $\nu_{n}$ and $\a_{n}\in\T$ by
$$\hat F^{AA}(q_{n},c_{n})e^{iq_{n}\th}+\hat F^{AA}(-q_{n},c_{n})e^{-iq_{n}\th}=|c_{n}|^{\bar a_{4}}\nu_{q_{n}}e^{-q_{n}h_{n}}\cos(q_{n}\th+\a_{q_{n}})$$
we see that 
on $\T_{h_{n}-1}\times \bD(c_{n},c_{n}^2/2)$ ({\it cf.} (\ref{estFAA}))
$$F^{AA}_{\zeta}=F^{AA}(\th,r)+r^{\bar a_{4}}\sum_{k=1}^\infty r^{q_{k}/2}( \zeta_{1,k}\cos(q_{k}\th)+\zeta_{2,k} \sin(q_{k}\th))+O(c_{n}^{(3/2) \bar a_{4}}).
$$
Hence
\begin{multline*}\hat F_{\zeta}^{AA}(q_{n},c_{n})e^{iq_{n}\th}+\hat F_{\zeta}(-q_{n},c_{n})e^{-iq_{n}\th}=\\ c_{n}^{\bar a_{4}} \biggl(\biggl((\nu_{q_{n}}+O(c_{n}^{(1/2)\bar a_{4}}))e^{-q_{n}h_{n}} +c_{n}^{q_{n}/2}\ti \zeta_{1,n}\biggr)\cos(q_{n}\th+\a_{q_{n}})\\+\biggl(c_{n}^{q_{n}/2}\ti \zeta_{2,n}+O(c_{n}^{(1/2)\bar a_{4}}) e^{-q_{n}h_{n}}\biggr)\sin(q_{n}\th+\a_{q_{n}})\biggr)
\end{multline*}
with $\ti \zeta_{1,k}-i\ti\zeta_{2,k}=e^{-i\a_{q_{k}}}(\zeta_{1,k}-i\zeta_{2,k})$.
We thus have ({\it cf.} (\ref{cnqn/2}))
\begin{align*}2|\hat F_{\zeta}^{AA}(q_{n},c_{n})|&\geq c_{n}^{\bar a_{4}}|c_{n}^{q_{n}/2}\ti \zeta_{2,n}+O(c_{n}^{(1/2)\bar a_{4}}) e^{-q_{n}h_{n}}|\\
&\geq c_{n}^{\bar a_{4}} |e^{-q_{n}h_{n}}e^{O(1/q_{n})}\ti \zeta_{2,n}+O(c_{n}^{(1/2)\bar a_{4}}) e^{-q_{n}h_{n}}|\\
&\gtrsim c_{n}^{\bar a_{4}}e^{-q_{n}h_{n}} |\ti \zeta_{2,n}+O(c_{n}^{(1/2)\bar a_{4}})|
\end{align*}
and we see that if 
$$|\ti\zeta_{2,n}|\geq c_{n}^{1/2}
$$
one can apply Proposition \ref{prop:16.1} ({\it cf.} (\ref{cnqn/2})): 
$$m_{\Phi_{\Omega}\circ f_{F_{\zeta}^{AA}}}(c_{n})\gtrsim C_{h_{n}}^{-1}  c_{n}^{2\bar a_{4}+1}e^{-4q_{n}h}\gtrsim c_{n}^{2\bar a_{4}+1}e^{-4q_{n}h}.$$
Now, since {\it $c_{n}$ is positive} and $m_{f_{\zeta}}(2c_{n})\gtrsim m_{\Phi_{\Omega}\circ F_{\zeta}}(c_{n})$ this provides
$$ m_{f_{\zeta}}(2c_{n})\gtrsim m_{\Phi_{\Omega}\circ f_{F}\circ \Phi_{G_{\zeta}}}(c_{n})\gtrsim c_{n}^{2\bar a_{4}+1}e^{-4q_{n}h}.
$$
\end{proof}
We can deduce the analogue of Lemma \ref{lemma:17.1}
\begin{lemma}\label{lemma:17.2} Let $\cN$ be an infinite  set of $n\in\N$ for which $c_{n}>0$. Then, for almost every $\zeta\in ([-1,1]^2)^{\N^*}$, there exists an infinite subset $\ti\cN\subset\cN$ such that  for all $n\in\ti\cN$
$$m_{f_{\zeta}}(2c_{n})\gtrsim m_{\Phi_{\Omega}\circ f_{F_{\zeta}}}(c_{n})\gtrsim c_{n}^{2\bar a_{4}+1}e^{-4q_{n}h}.
$$
\end{lemma}

\subsubsection{Proof of Theorem \ref{theo:main2} (CC) Case, $\omega_{0}$ diophantine} \label{sec:16.2.1}
We want to apply the previous Lemma \ref{lemma:17.2} to an infinite  set $\cN$ such that for all $n\in\cN$ one has both
\be c_{n}>0 \quad \textrm{and}\quad q_{n+1}\geq q_{n}^{\tau-}.\label{16.355}
\ee
Such a set may not exist for arbitrary choices of $\omega_{0}$ (Diophantine) and $\Omega$. On the other hand, if one chooses the sign of $\pa^2\Omega(0)$  depending on $\omega_{0}$ (or more precisely its sequence of convergents) this is possible. 

Assume that $\omega(r):=(2\pi)^{-1}\pa\Omega(r)$ is of the form 
$$\omega(r)=\omega_{0}+b_{2}r+O(r^2),\qquad b_{2}>0$$
(the case $b_{2}<0$ is treated similarly)
and let $\beta>0$ and 
$$ \cN_{\beta }=\{n\in\N,\ q_{n+1}\geq q_{n}^{\tau-\beta/2}\},\qquad \cQ_{\b}=\{p_{n}/q_{n},\ n\in \cN_{\b}\}.$$
Since $\cN_{\beta}$ is infinite, one of the two sets $\cQ_{\b}^\pm=\cQ_{\b}\cap (\pm ]\omega_{0},\infty[)$ is infinite. Assume it is $\cQ_{\b}^+$ (the other case is treated in a similar manner) and define
$$\cN_{\beta}^+=\{n\in\cN_{\b},\ p_{n}/q_{n}>\omega_{0}\}, \qquad \cC_{\b}=\omega^{-1}(\cQ_{\b}^+).
$$
Since $b_{2}>0$,  $\cC_{\b}$ is infinite, $\subset ]0,\infty[$ and its points $c_{n}$, $n\in\cN_{\beta}^+$ accumulate zero.

We then choose (we just have to produce an example)
$$\Omega(r)=2\pi\omega_{0}r+(1/2)r^2,\qquad F=0$$ and applying Lemma  \ref{lemma:17.2}  we see that for almost every $\zeta$ and infinitely many $n\in\cN_{\beta}^+$ 
$$m_{f_{\zeta}}(2c_{n})\gtrsim c_{n}^{2\bar a_{4}+1}e^{-4q_{n}h}$$
and, arguing like in Subsections (\ref{sec:16.1.1}), (\ref{sec:16.1.2}), we see that setting $t_{n}=2c_{n}$ we have for infinitely many $n$
$$m_{f_{\zeta}}(t_{n})\gtrsim \exp\biggl(-\biggl(\frac{1}{|t_{n}|}\biggr)^{(\frac{1}{1+\tau})+\b}\biggr).$$
\ \hfill $\Box$

\subsubsection{Proof of Theorem  \ref{theo:main2prime}: (CC) Case, $\omega_{0}$ Liouvillian} 
Since $\omega_{0}$ is Liouvillian, there exists an infinite set $\cN\subset \N$ such that 
$$\lim_{n\in\cN}\frac{\ln q_{n+1}}{\ln q_{n}}=\infty.$$
We define 
$$\cQ^\pm=\{p_{n}/q_{n},\ n\in \cN,\  p_{n}/q_{n}\in\pm ]\omega_{0},\infty[\}$$
and we observe that one of the two sets $\cQ^+$, $\cQ^-$ is infinite; we assume it is $\cQ^+$ (the other case is treated in a similar manner) and we set $\cN^+=\{n\in\cN,\ p_{n}/q_{n}>\omega_{0}\}$. Let us choose again $b_{2}>0$, for example
$$\Omega(r)=2\pi\omega_{0}r+(1/2)r^2,\qquad F=0$$
and observe that  $\cC=\omega^{-1}(\cQ^+)$ is infinite, contained in $]0,\infty[$ and its points $c_{n}$, $n\in\cN^+$ accumulate 0.

We now apply Lemma \ref{lemma:17.2}: for almost every $\zeta$ there exists  an infinite subset $\ti\cN^+_{\zeta}\subset \cN^+$ such that one has for $n\in\ti\cN_{\zeta}^+$
$$m_{f_{\zeta}}(2c_{n})\gtrsim c_{n}^{2\bar a_{4}+1}e^{-4q_{n}h}.$$
Since for $n\in\ti\cN_{\zeta}^+$
$$c_{n}\asymp |\omega_{0}-\frac{p_{n}}{q_{n}}|\asymp \frac{1}{q_{n}q_{n+1}}$$
one has for $\e>0$ and  $n\in\ti\cN_{\zeta}^+$ large enough
$$c_{n}^{2\bar a_{4}+1}e^{-4q_{n}h}\gtrsim (\frac{1}{q_{n+1}})^{2(2\bar a_{4}+1)}\exp(-q_{n+1}^{\e/2})\gtrsim \exp(-q_{n+1}^\e).$$
If we set $t_{n}=2c_{n}$ we have for $n$ large enough ({\it cf.} (\ref{16.381}))
$$2c_{n}\leq t_{n}:=\frac{5(b_{2}+b_{2}^{-1})}{q_{n}q_{n+1}}$$ 
and 
$$m_{f_{\zeta}}(t_{n})\geq m_{f_{\zeta}}(2c_{n})\gtrsim \exp(-q_{n+1}^\e)$$
for infinitely many $n$.
\ \hfill $\Box$

\appendix

\section{Estimates on composition and inversion}

\subsection{Proof of Lemma \ref{domaindef}}\label{sec:A2eante}
We shall do the proof in the (AA)-case; the proof in the  (CC)-case follows the same lines.

By (\ref{eq:cauchy:derivativesbis}),   there exists a numerical constant $C>0$ such that if  $\d>0$ satisfies
\be C\|F\|_{W_{h,U}}\d^{-2}\ud(W_{h,U})^{-1}<1,\label{e:4.45}
\ee
then  for any fixed $\th\in \T_{h-2\d}$ and any fixed $r\in e^{-\d}U$, the map $e^{-\d}U\to e^{-\d}U$, $R\mapsto r-\pa_{\th}F(\th,R)$ is contracting and by the Contraction Mapping Principle there thus  exists a unique $R\in e^{-\d}U$ depending holomorphically on $(\th,r)\in\T_{h-\d/2}\times e^{-\d}U$ such that 
$$r=R+\pa_{\th}F(\th,R).$$

Under assumption (\ref{e:4.45}) we see that if $C$ is large enough
$$|\pa_{R}F(\th,R)|\lesssim \d^{-1}\times \|F\|_{h,U}<(1/2)\d \times \times \ud(U),$$
and by Item \ref{i3.L.2.1} of Lemma \ref{lemma2.1}
$$\ph:=\th+\pa_{R}F(\th,R)\in U.$$
We can thus define a holomorphic map 
$$f_{F}:\T_{h-\d/2}\times e^{-\d}U\to \T_{h-\d}\times U$$ by 
\be f_{F}(\th,r)=(\ph,R) \ \iff\  \begin{cases}&r=R+\pa_{\th}F(\th,R)\\ &\ph=\th+\pa_{R}F(\th,R)\label{4.43b}\end{cases}
\ee
Notice that the maps $(\th,r)\mapsto \ph(\th,r)-\th$, $(\th,r)\mapsto R(\th,r)-r$  such defined are Lipschitz with Lipschitz constant $\lesssim \d^{-2}\ud(U)^{-2}\|F\|_{h,U}$. Thus, if for some numerical constant large enough
\be C \d^{-2}\ud(U)^{-2}\|F\|_{h,U}<1\label{e:4.47}\ee
the map $f_{F}$ is a holomorphic diffeomorphism from $\T_{h-\d/2}\times e^{-\d}U$ onto its image.

Conversely, if (\ref{e:4.45})  is satisfied, given $(\ph,R)\in\T_{h-\d}\times e^{-2\d}U$ there exists a unique $(\th,r)\in \T_{h-\d/2}\times e^{-\d}U$ such that $f(\th,r)=(\ph,R)$. We thus have if (\ref{e:4.47}) is satisfied  
$$\bT_{h-\d}\times e^{-2\d}U\subset f_{F}(\T_{h-\d/2}\times e^{-\d}U).$$

Finally, we observe that the  diffeomorphism $f_{F}$ is {\it exact symplectic} which means that the  differential form $Rd\ph-rd\th$ is exact; in particular, it is symplectic. Indeed 
\begin{align} Rd\ph-rd\th&=-\ph dR+d(\ph R)-rd\th\\
&=-(\th+\pa_{R}F(\th,R))dR-(R+\pa_{\th}F(\th,R))d\th+d(\ph R)\\
&=-dF+d(\ph R)-d(\th R)\\
&=d(-F+(\ph-\th)R)
\end{align}
(observe that the function $-F+(\ph-\th)R=-F(\th,R)+\pa_{R}F(\th,R)R$ is well defined on $\T_{h}\times U$).
We have thus proven that
there exists a numerical constant $\bar C>0$ such that if 
\be \bar C\d^{-2}\ud(U)^{-2}\|F\|_{h,U}<1\label{e:4.47bis}\ee
the diffeomorphism  $f_{F}$ previously defined is  exact symplectic and 
\be e^{-2\d}W_{h,U}\subset  f_{F}(e^{-\d}W_{h,U})\subset W_{h,U}.\label{f-dom}\ee

Estimate  (\ref{EE427}) comes from (\ref{4.43b}) and 
$$\max_{i=1,2}|\pa_{i}F(\th,R)-\pa_{i}F(\ph,r)||\leq 2\|D^2F\|\|DF\|.
$$

\hfill $\Box$

\subsection{Proof of Lemma \ref{lemma:6.2}}\label{sec:A2e}

We illustrate the proof in the (AA)-Case (it is the same in the (CC)-case).

Since $f$ is close to the identity, the map $\ti f:(\th,R)\mapsto (\ph,r)\ \iff \  f(\th,r)=(\ph,R)$ defines a diffeomorphism such that 
$\ti f-id=\fO(f-id)
$
and since $f$ is exact symplectic we know ({\it cf.} Subsection \ref{sec:A2eante}) that  $\ph dR+rd\th=dF$ for some holomorphic function $F:(\th,R)\mapsto F(\th,R)$.  Since  $F(\th,R)=\int_{\gamma_{\th,R}} (\ph dR+rd\th)$ where $\gamma_{\th,R}$ is the path joining $(0,R_{0})\in \{\th\in\C,\ |\Im\th|<h\}\times U$ to $(\th,R)$, the function $F$, which is unique  up to the addition of a constant,  thus satisfies  $F=\fO(f-id)$.

The  estimate (\ref{l465})   is a consequence of (\ref{4.43}) and the fact that 
\begin{align*}&|\pa_{\th}F(\th,R)-\pa_{\th}F(\th,r)|\leq  \|D\pa_{\th}F\||R-r|\ \leq \|D\pa_{\th}F\|\|\pa_{\th }F\|\\
&|\pa_{R}F(\th,R)-\pa_{R}F(\th,r)|\leq  \|D\pa_{R}F\||\ph-\th|\ \leq \|D\pa_{R}F\|\|\pa_{R }F\|.
\end{align*}

\hfill $\Box$

\subsection{Proof of Lemma  \ref{lemma:6.2.bis} }\label{sec:A3}

\

\medskip
\noindent
{\it 1) Proof of  (\ref{eq:4.55}).}
One has 
$$ f_{F}(\th,r)=(\ph,R) \ \iff\  \begin{cases}&r=R+\pa_{\th}F(\th,R)\\ &\ph=\th+\pa_{R}F(\th,R)\end{cases}
$$
$$ f_{\Omega}(\ph,R)=(\psi,Q) \ \iff\  \begin{cases}&R=Q\\ &\psi=\ph+\pa_{Q}\Omega(Q) \end{cases}
$$
hence
 $Q=R$ and 
\begin{align*}&\psi=\ph+\pa_{R}\Omega(R)\\
&=\th+\pa_{R}F(\th,R)+\pa_{R}\Omega(R)\\
&=\th+\pa_{Q}(\Omega+F)(\th,Q)
\end{align*}
thus, since $r=R+\pa_{\th}F(\th,R)=Q+\pa_{\th}F(\th,Q)$ and $\Omega$ does not depend on the $\th$-variable
$$\begin{cases} &r=Q+\pa_{\th}(\Omega+F)(\th,Q)\\
&\psi=\th+\pa_{Q}(\Omega+F)(\th,Q)
\end{cases}
$$
which is equivalent to 
$$ f_{\Omega+F}(\th,r)=(\psi,Q) =f_{\Omega}\circ f_{F}(\th,r).
$$

\noindent
{\it  2) Proof of (\ref{eq:4.53})} Assume that  $f_{F}(\th,r)=(\ph,R)$ and $f_{G}(\ph,R)=(\psi,Q)$. Then
\be f_{F}(\th,r)=(\ph,R),\qquad \begin{cases}&r=R+\pa_{\th}F(\th,R)\\ &\ph=\th+\pa_{R}F(\th,R)\end{cases}\label{eqsymp1}
\ee
\be f_{G}(\ph,R)=(\psi,Q),\qquad \begin{cases}&R=Q+\pa_{\ph}G(\ph,Q)\\ &\psi=\ph+\pa_{Q}G(\ph,Q)\end{cases}\label{eqsymp2}
\ee

\begin{align*} Qd\psi-rd\th&=Qd\psi-Rd\ph+Rd\ph-rd\th\\
&=d(-F-G+(\ph-\th)R+(\psi-\ph)Q).
\end{align*}
If $f_{G}\circ f_{F}=f_{H}$ then one has   $Qd\psi-rd\th=d(-H+Q(\psi-\th))$ 
and then
\begin{align*} 0&=d(-H+F+G+Q(\psi-\th)-R(\ph-\th)-Q(\psi-\ph))\\
&=d(-H+F+G-(Q-R)(\ph-\th))
\end{align*}
and so
$$ H(\th,Q)={\rm cst}+F(\th,R)+G(\ph,Q)-(Q-R)(\ph-\th).
$$
Let us write
$H(\th,Q)=F(\th,Q)+G(\th,Q)+A(\th,Q)$ where 
\begin{align*} -A&=F(\th,Q)-F(\th,R)+G(\th,Q)-G(\ph,Q)+(Q-R)(\ph-\th)\\
&=F(\th,Q)-F(\th,R)+G(\th,Q)-G(\ph,Q)-\pa_{\ph}G(\ph,Q)\pa_{R}F(\th,R)\\
\end{align*}

We can now estimate
\begin{align*}\|A\|_{h-\delta,U_{\d}}&\leq \|\pa_{R}F\|_{h,U}\|Q-R\|_{h,U}+\|\pa_{\ph}G\|_{h,U}\|\ph-\th\|_{h,U}\\
&  +\|\pa_{\ph}G(\ph,Q)\|_{h,U}\|\pa_{R}F(\th,R)\|_{h,U}\\
&\leq  \|\pa_{R}F\|_{h,U}\|\pa_{\ph}G\|_{h,U}+\|\pa_{\ph}G\|_{h,U}\|\pa_{R}F\|\\&\ +\|\pa_{\ph}G\|_{h,U}\|\pa_{R}F\|_{h,U}
\end{align*}
and deduce (\ref{eq:4.53}).

\medskip
\noindent
{\it 3) Proof of (\ref{n4.64})}
We just write 
\begin{align*}f_{F+G}\circ f_{G}^{-1}&=f_{F+G}\circ f_{-G+O(|D^2G|DG|)}\\
&=f_{F+\|DF\|\cO_{1}(G)}\qquad(\textrm{using}\ (\ref{eq:4.53}))
\end{align*}
and a similar expression for $f_{F}^{-1}\circ f_{F+G}=f_{-F+O(|D^2F||DF|)}$.

The proof of (\ref{eq:4.53}) and  (\ref{n4.64})  is the same in the (CC)-case.
\hfill $\Box$

\subsection{Proof of Proposition \ref{lemma:6.4} }\label{sec:proofprop4.4}
We first state two lemmata. 
\begin{lemma}\label{item4}Let  $W$ be an open subset of $M=\C^2$ or $\T_{\infty}\times \C$, $v\in\cO(W)$ and $g-id\in \cO(W)$   such that  $\|g-id\|_{W}\lesssim 1$. Then if $\|v\|_{W}$ is small enough
\be (id+v)\circ g\circ (id+v)^{-1}=g\circ (id+[g]\cdot v+\fO_{2}(v))
\ee
where 
\be [g]\cdot v=-v+(Dg^{-1}\cdot v)\circ g.
\ee
\end{lemma}
\begin{proof}
One has 
\begin{align*} (id+v)\circ g\circ (id+v)^{-1}&=g\circ (id-v+\fO_{2}(v))+v\circ g\circ(id-v+\fO_{2}(v))\\
&=g-Dg\cdot v+v\circ g+\fO_{2}(v)
\end{align*}
hence
\begin{align*} g^{-1}\circ (id+v)\circ g\circ (id+v)^{-1}&=g^{-1}\circ \biggl(g-Dg\cdot v+v\circ g+\fO_{2}(v)\biggr)\\
&=id-Dg^{-1}\circ g \cdot Dg\cdot v+Dg^{-1}\circ g\cdot v\circ g+\fO_{2}(v)\\
&=id-v+(Dg^{-1}\cdot v)\circ g+\fO_{2}(v).
\end{align*}
\end{proof}

\begin{lemma}\label{lemma:11.9} If $\Omega\in \cO(U)$, $Y\in \cO(W_{h,U}\cup \Phi_{\Omega}(W_{h,U})))$ 
then 
\be f_{Y}\circ \Phi_{\Omega}\circ f_{Y}^{-1}=\Phi_{\Omega}\circ f_{[\Omega]\cdot Y+\fO_{2}(Y)}
\ee
where
$$[\Omega]\cdot Y=Y\circ \Phi_{\Omega}-Y.
$$
\end{lemma}
\begin{proof}
 From Lemma \ref{item4},  and  (\ref{l465}) we have 
\begin{align} f_{Y}\circ \Phi_{\Omega}\circ f_{Y}^{-1}&=\Phi_{\Omega}\circ (id+[\Phi_{\Omega}]\cdot (J\nabla Y)+\fO_{2}(Y))\notag\\
&=\Phi_{\Omega}\circ \biggl(id-J\nabla Y+(D\Phi_{\Omega}^{-1}\cdot(J\nabla Y))\circ \Phi_{\Omega}+\fO_{2}(Y)\biggr).\label{eqa3}
\end{align}
On the other hand
\be J\nabla(Y\circ \Phi_{\Omega})=J\ ({}^tD\Phi_{\Omega})\cdot(\nabla Y\circ \Phi_{\Omega}).\label{eqa4}\ee
Since $\Phi_{\Omega}$ is symplectic,
$J\ {}^t(D\Phi_{\Omega})=(D\Phi_{\Omega})^{-1}J$, we deduce from (\ref{eqa3}) and (\ref{eqa4}) that 
\begin{align*} f_{Y}\circ \Phi_{\Omega}\circ f_{Y}^{-1}&=\Phi_{\Omega}\circ (id+J\nabla Y-J\nabla(Y\circ \Phi_{\Omega})+\fO_{2}(Y))\\
&=\Phi_{\Omega}\circ f_{[\Omega]\cdot Y+\fO_{2}(Y)}
\end{align*}
where 
$$[\Omega]\cdot Y=Y\circ \Phi_{\Omega}-Y.
$$
\end{proof}
From (\ref{l465}), (\ref{eq:4.53}) (we use the fact that $D(O(\|DF\|G))=\|DF\|\fO_{1}(G)$)
\be f_{Y}\circ f_{F}\circ f_{Y}^{-1}=f_{F+\|DF\|\fO_{1}(Y)}
\ee
and on the other hand, from Lemma \ref{lemma:11.9}
\be f_{Y}\circ \Phi_{\Omega}\circ f_{Y}^{-1}=\Phi_{\Omega}\circ f_{[\Omega]\cdot Y+\fO_{2}(Y)}.
\ee
Hence (using (\ref{eq:4.53}) in the last line)
\begin{align}  f_{Y}\circ  \Phi_{\Omega} \circ f_{F}\circ f_{Y}^{-1}&=  f_{Y}\circ \Phi_{\Omega}\circ f_{Y}^{-1}\circ  f_{Y}\circ f_{F}\circ f_{Y}^{-1}\\
&=\Phi_{\Omega}\circ  f_{[\Omega]\cdot Y+\fO_{2}(Y)}\circ  f_{F+	\|DF\|\fO_{1}(Y)}\\
&= \Phi_{\Omega}\circ f_{F+[\Omega]\cdot Y+\|DF\| \fO_{1}(Y)}. 
\end{align}

\hfill $\Box$

\section{Whitney type extensions}\label{sec:appendixBante}
\subsection{Proof of Lemma \ref{lemma:2.3ee} }   \label{sec:appendixBbis}

 Let $\chi_{\d}:\R\to [0,1]$ be a smooth function with support in $[-1,1]$ and equal to 1 on $[-e^{-\d/2},e^{-\d/2}]$ such that 
\be \sup_{\R}|\pa^j\chi_{\d}|\lesssim \d^{-j}.\label{ee2.54}\ee
 We define for $r\in\C$,  $\eta(r)=\chi_{\d} ( (e^{\d/2}|r|)^2/\rho^2)$ and for $i\in J_{U}$, $\eta_{i}(r)=(1-\chi_{\d}((e^{-\d} |r-c_{i}|)^2/\rho_{i}^2))$.  Note that  $\eta$ is equal to 1 on $e^{-\d}\bD(0,\rho)$ and 0 on $\C\setminus e^{-\d/2}\bD(0,\rho)$ and  $\eta_{i}$ is equal to 1 on $\C\setminus e^{\d}\bD(c_{i},\rho_{i})$ and 0 on $e^{\d/2}\bD(c_{i},\rho_{i})$ hence $\zeta=\eta\prod_{i\in J_{U}}\eta_{i}$ is equal to 1 on $e^{-\d}U$ and 0 on $V:=(\C\setminus e^{-\d/2}\bD(0,\rho))\cup \bigcup_{i\in J_{U}} e^{\d/2}\bD(c_{i},\rho_{i})$. The union of the open sets $W_{h,e^{-\d/10}U}$ (resp. $e^{-1/10}U$) and $W_{h,V}$ (resp. $V$) is $W_{h,\C}$ (resp. $\C$) and on their intersection the functions $\zeta F$ and 0 coincide. As a consequence, one can extend $\zeta F$ by 0 on $W_{h,V}$ as a smooth function $F^{Wh}:W_{h,\C}\to W_{h,\C}$. Note that since $\zeta$ is $\s$-symmetric, the same is true for $F^{Wh}$ and that $F^{Wh}$ and $F$ coincide on $W_{h,e^{-\d}U}$ (which contains $e^{-\d}W_{h,U}$).
 
 To get the estimates on the derivatives of $F^{Wh}$ we observe from (\ref{ee2.54}) and the definitions of $\eta$, $\eta_{i}$ that  
$$\max_{i}(\max_{0\leq j\leq k}\sup_{\C}|D^j\eta|,\max_{0\leq j\leq k}\sup_{\C}|D^j\eta_{i}|)\lesssim \d^{-k}\max_{i}(\rho^{-2k},\rho_{i}^{-2k})
$$
and since $\max_{i}(\eta, \eta_{i})\leq 1$, one has by Leibniz formula
$$\max_{0\leq j\leq k}\sup_{\C}|D^j\zeta|\lesssim (\#J_{U}+1)^k\d^{-k}\max_{i}(\rho^{-2k},\rho_{i}^{-2k}).$$
 Hence $F^{Wh}:=\zeta F$ satisfies 
 $$\sup_{0\leq j\leq k}\|D^j F^{Wh} \|_{W_{h,\C}}\leq  C(1+\#J_{U})^k(\d\ud(U))^{-2k}\max_{0\leq j\leq k}\|D^jF\|_{W_{h,e^{-\d/10}U}}.$$
\hfill $\Box$

\subsection{Proof of Lemma \ref{Whextension}}\label{sec:appendixBbisbis}
Write $(2\pi)^{-1}\Omega(z)=\sum_{n=0}^\infty b_{n}z^n$ with  $|b_{n}|\leq  \rho_{0}^{ -n}$, $(2\pi)^{-1}\Omega_{2}(z)=b_{0}+b_{1}z+b_{2}z^2$, $(2\pi)^{-1}\Omega_{\geq 3}(z)=\sum_{n=3}^\infty b_{n}z^n$. For $0\leq j\leq 3$ and $\d>0$, there exists $C_{j}>0$ such that for any  $\rho\leq \rho_{0}/2$
\be \|D^j\Omega_{\geq 3}\|_{\bD(0,\rho)}\leq C_{j}\rho^{3-j}.\label{eB.416}
\ee
Let $\chi:\C\to[0,1]$ be a smooth function with support in $\bD(0,1)$ and equal to 1 on $\bar\bD(0,1/2)$.
We define
$$\Omega_{\rho}^{Wh}(z)=\Omega_{2}(z)+(1-\chi_{}(z/\rho))\Omega_{\geq 3}(z).
$$
For any $z\in\bD(0,\rho/2)$ one has $\Omega_{\rho}^{Wh}(z)=\Omega(z)$ and by (\ref{eB.416}) and Leibniz formula, for some constant $B$ depending only on $b_{0},b_{1},b_{2}$, $\|D^j\chi\|_{C^0}$, $\|D^j\Omega_{\geq 3}\|_{\bD(0, \rho_{0})}$, $0\leq j\leq 3$, one has 
$$\forall z\in\C,\quad |\frac{1}{2\pi}D^3\Omega_{\rho}^{Wh}(z)|\leq B.$$
On the other hand, for some constant $C$ depending only on, $\|\pa^j\chi\|_{C^0}$, $\|\pa^j\Omega\|_{\bD(0,\rho_{0})}$, $0\leq j\leq 2$
$$\forall\ t\in\R,\qquad |\frac{1}{2\pi}\pa^2\Omega_{\rho}^{Wh}(t)-2b_{2}|\leq C\rho
$$
and if $\rho=\bar\rho$ is chosen small enough so that $C\bar\rho<b_{2}$, one has  (we assume $b_{2}>0$)  $b_{2}\leq \frac{1}{2\pi}\pa^2\Omega_{\bar \rho}^{Wh}(t)\leq 3b_{2}$.

\ \hfill $\Box$

\subsection{Proof of Proposition \ref{proppramexclusion}}\label{sec:appendixB}
The proof will follow from the following two lemmas. 
 \begin{lemma}\label{lemma:nor6.2}Let $\b\in\R$, $\nu>0$; if for  $t+is\in U$ ($t,s\in\R$) one has
 $$|\omega(t+is)-\beta|<\nu,$$
 then 
 $$\begin{cases}&|\omega^{}(t)-\beta|\leq 5A \nu\\
 &|s|\leq 4A\nu.
 \end{cases}$$
 \end{lemma}
 \begin{proof}
 Since $\omega$ is holomorphic on $U$ one has for any $z\in U$, $\bar\pa\omega(z)=0$ (we use in this proof the usual  notations $\bar \pa=(1/2)(\pa_{t}+i\pa_{s})$ and $\pa=(1/2)(\pa_{t}-i\pa_{s})$). For any point  $z\in \bD(0,\rho)$ one has ({\it cf.} Lemma \ref{lemma2.1})
 $$\dist(z,U)\leq 2\ua(U)$$
 and from the fact that $\|D\bar \pa\omega\|\lesssim B$  we thus  get if $\bar C_{0}$ is large enough in  (\ref{condUnew})
 \begin{align}
 \|\bar \pa\omega^{}\|_{C^0(\bD(0,\rho))}&\lesssim  \ua(U)B\notag\\
& \leq (8A)^{-1}.\label{baromeganew}
 \end{align}

Now, for   $t+is\in \bD(0,\rho)$, $t,s\in\R$ we write
\begin{multline} \omega(t+is)-\b=\omega^{}(t)-\b+\pa\omega^{}(t)\cdot (is)+\bar \pa\omega^{}(t)\cdot(-is)+O(s^2)\label{imnullenew}\end{multline}
where 
\begin{align}|O(s^2)|&\leq \|D^2\omega^{}\|_{C^0(\bD(0,\rho))}\times s^2\notag\\&\leq B\times \rho\times s\notag \\
&\leq (8A)^{-1}\times s\label{6.86new}
\end{align} 
if $\bar C$ is large enough in (\ref{condUnew}). Note that since  $\omega^{}$ is real-symmetric, $\pa\omega^{Wh}(t)$ and $\bar \pa\omega^{}(t)$ are real when $t$ is real. Hence if $|\omega^{}(t+is)-\b|=|\omega(t+is)-\b|<\nu$ ($t+is\in U$) one gets by  taking the imaginary part in (\ref{imnullenew}), using (\ref{baromeganew}), $\pa\omega^{}(t)\in [A^{-1},A]$ ({\it cf.} (\ref{6.86new})) that
\be |s|\leq (4A) \nu.
\ee
This and (\ref{imnullenew}) show that 
\be |\omega^{}(t)-\b|\leq 5A\nu.
\ee
\end{proof}
Since $t\mapsto \omega^{}(t)$ is increasing with a derivative bounded by below (this is the twist condition) the set of $t\in ]-\rho,\rho[:=\bD(0,\rho)\cap \R$ such that $|\omega^{}(t)-\beta|\leq 5A \nu$ is a (possibly empty) interval  $I_{\b}$ of length $\leq 5A^2 \nu$.
\begin{lemma}If $I_{\b}$  is not empty  
 there exists a unique $c_{\b}\in ]-\rho-10A^2\b,\rho+10A^2\b[$ such that 
$$\omega^{}(c_{\b})=\b,\qquad \omega^{}(\bD(c_{\b},10A^2\nu))\supset \bD(\b,5A\nu).$$
\end{lemma}
\begin{proof} For the existence of $c_{\b}$ we just notice that if $I_{\b}\subset ]-\rho,\rho[$ there is nothing to prove (notice that $\omega^{}$ is increasing) and otherwise  $|\omega^{}(\pm \rho)-\b|\leq 5A\nu$. But then,  the fact that $A^{-1}\leq \pa\omega^{}(t) \leq A$ shows the existence of a unique $c_{\b}\in\bD(0,e^{1/10}\rho)$ ($\nu$ small enough) such that $\omega^{}(c_{\b})=\b$ with $|c_{\b}-(\pm \rho)|\leq 10A^2\nu$.
 We then observe that 
\begin{align*}\omega^{}(c_{\b}+w)&=\b+\pa\omega^{}(c_{\b})w+\bar\pa\omega^{}(c_{\b})\bar w+O(w^2)\\
&=\b+aw+b\bar w+g(w)\end{align*}
with $a=\pa \omega^{}(c_{\b})\geq A^{-1}$, $|b|=|\bar \pa \omega^{}(c_{\b})|\leq (7A)^{-1}$ and
where $w\mapsto g(w)$ is $B\Delta$-Lipschitz on $|w|\leq \Delta$.   For any $|u|\leq 5A\nu$,  the map $w\mapsto (1/a)u-(b/a)\bar w-(1/a)g(w)$ is $(1/2)$-contracting  on the disk $|w|\leq \bar C_{0}^{-1}$ provided $\bar C_{0}$ is large enough in (\ref{condUnew}). Picard Fixed Point Theorem shows that there exists (a unique) $w\in \bD(0,10A^2\nu)$ such that $\omega^{}(c_{\b}+w)=\b+u$ provided $\nu$ is small enough.
\end{proof}

\section{Proof of Lemma \ref{ncor:5.4}}\label{appendixC}
\begin{lemma}\label{lem:6.3}Let $U$ be a $\s$-symmetric  open connected set of $\bD$ and  $F\in\cO_{\s}(W_{h,U})$ such that 
\be \forall \ t\in\R,\ F\circ \phi^t_{J\nabla r}=F.\label{invF}\ee
Then, there exists $\ti F\in\cO_{\s}(U)$ such that on $W_{h,U}$ one has
$$F=\ti F\circ r.$$
\end{lemma}
\begin{proof}

The Lemma is clear when we are in the (AA) case since the identity  $\forall\ t\in\R \ F(\th+t,r)=F(\th,r)$ clearly implies that $F$ does not depend on $\th$. So we consider the (CC)-case.

We shall prove  that for every $(z,w)\in W_{h,U}$ there exists an open neighborhood $V_{z,w}$ of $(z,w)$ and a holomorphic function $f_{z,w}$ such that $F=f_{z,w}\circ r$ on $V_{z,w}$.

We consider three cases:

\noindent 1) If $(z,w)=(0,0)\in W_{h,U}$. One can write for $\mu$ small enough and $(z,w)\in \bD(0,\mu)^2$, 
$F(z,w)=\sum_{k,l\in\N}F_{k,l}z^kw^l$. The identity (\ref{invF}) implies $\cM_{n}(F)=0$ for $n\ne 0$ hence from (\ref{formDn}) one has 
$F(z,w)=\cM_{0}(F)(z,w)=\sum_{k\in\N}F_{k,k}(zw)^k$ and we can choose $f_{0,0}(r)=\sum_{k\in\N}(ir)^k$.

\noindent 2) If $zw=0$, with for example $w\ne 0$. Then, from (\ref{e5.94}), $t\mapsto F(0,e^{it}w)$ is holomorphic with respect to $t\in \R+i ]-\ln(e^h\rho^{1/2}/|w|), \infty[$
 and constant on the real axis; it is hence constant on
 $ \R+i ]-\ln(e^h\rho^{1/2}/|w|), \infty[$.  In particular taking $t=is$, $s\in\R_{+}$, gives $F(0,e^{-s}w)=F(0,w)$ and by making $s\to\infty$ we get $F(0,w)=F(0,0)$ (notice that $(0,0)\in W_{h,U}$ in that case). The same argument shows that for $(\ti z,\ti w)\in W_{h,U}$ the function $t\mapsto F(e^{-it}\ti z,e^{it}\ti w)$ is constant on   $t\in \R+i ]-\ln(e^h\rho^{1/2}/|\ti w|),\ln(e^h\rho^{1/2}/|\ti z|) [$. Now if $(\ti z,\ti w)$ is close enough to $(0,w)$, in particular if there exists $0<s<\ln(e^h\rho^{1/2}/|\ti z|)$ such that $|\ti w|/\mu<e^s<\mu/|\ti z|$, one has  with $t=is$, $(e^{-it}\ti z,e^{it}\ti w)=(e^{s}\ti z,e^{-s}\ti w)\in \bD(0,\mu)^2$. 
 By (\ref{invF}) and point 1), one gets $F(\ti z,\ti w)=F(e^{-it}\ti z,e^{it}\ti w)=f_{0,0}(-i\ti z\ti w)$.

\noindent 3)  Otherwise, we can assume that $zw\ne 0$.
As before, we can argue that the function $t\mapsto g_{z,w}(t):=F(e^{-it}z,e^{it}w)$ is constant  on the set $$ \R+i ]-\ln(e^h\rho^{1/2}/|w|),\ln(e^h\rho^{1/2}/|z|) [.$$
Any point $(\ti z,\ti w)\in W_{h,U}$ which is close enough to $(z,w)$ is of the form $\ti z=e^{-it}z$, $\ti w=e^{it}\l w$, $t$ close to 0 and $\l$ close to 1. We thus have
$$F(\ti z,\ti w)=F(e^{-it}z,\l e^{it}w)=F(z,\l w)=F(z,\ti z \ti w z^{-1})=f_{z}(\ti z \ti w)$$
where we have defined $f_{z}(r)=F(z,irz^{-1})$.

We have thus proven that for each $(z,w)\in W_{h,U}$ there exist a neighborhood $V_{z,w}$ and a holomorphic  function $f_{z,w}$ such that $F=f_{z,w}\circ r$ on $V_{z,w}$. Now if $f_{z,w}\circ r=f_{z',w'}\circ r$ on a nonempty open set, the function $f_{z,w}$ and $f_{z',w'}$ coincide on a  nonempty open set and thus there exists a holomorphic  extension of $f_{z,w,z',w'}$ of these two functions such that $f_{z,w,z',w'}\circ r=f_{z,w}\circ r=f_{z',w'}\circ r$ on $V_{z,w}\cap V_{z',w'}$.
We can now conclude by using the connectedness of $U$.

\end{proof}
 
 \section{(Formal) Birkhoff Normal Forms.} \label{formalD}
 Our aim in this Section is to recall the proof of the existence and uniqueness of the formal BNF, Propositions \ref{statementBNFante},  \ref{statementBNF}. This is of course  a standard topic but we tried to develop here a framework that is convenient for the proof of Lemma \ref{BNF:lemma6.3}. We mainly concentrate on the (AA)-case since the  formalism in the (CC)-case is very similar to the one developed by Pérez-Marco in \cite{PM}.

 \subsection{Formal preliminaries}\label{formalD1}
\subsubsection{Formal series}
Let $\mathbb{A}$ be a commutative ring   and   $\mathbb{A}[[X_{1},\ldots,X_{d'}]]$ ($d'\in\N^*$) the ring of formal power series $\sum_{n\in\N^{d'}}a_{n}X^n$, $a_{n}\in\mathbb{A}$, $X^n=X_{1}^{n_{1}}\cdots X_{d'}^{n_{d'}}$ (for short $X=(X_{1},\ldots,X_{d'})$). 
We denote by $v(A)=\min\{|n|,\ a_{n}\ne 0\}$ the valuation of an element $A=\sum_{n\in\N^{d'}}a_{n}X^{n}$ and if  $B=(B_{1},\ldots,B_{d'})\in(\mathbb{A}[[X]])^{d'}$ we define $v(B)=\min_{l} v(B_{l})$. 

For any $k\in\N$ we define $[A]_{k}=\sum_{|n|=k}a_{k}X^k$   the homogenous part of $A$ of degree $k$
and 
we set $[A]_{\leq k}=\sum_{l=0}^k[A]_{l}$ (resp. $[A]_{\geq k}=\sum_{l=k}^\infty[A]_{l}$). 

As usual the product of $A=\sum_{n\in\N^{d'}}a_{n}X^n$ and $B=\sum_{n\in\N^{d'}}b_{n}X^n$ is $AB=\sum_{n\in\N^{d'}}(\sum_{k+l=n}a_{k}b_{l})X^n$ and the derivative $\pa_{X_{l}}A=\sum_{n\in\N^{d'}}n_{l}a_{n}X^{n'}$ with $n'_{j}=n_{j}$ if $j\ne l$ and $n'_{l}=n_{l}-1$ is $n_{l}\geq 1$ (if $n_{l}=0$ the derivative of the corresponding monomial is zero). 
Note that if $A_{i}\in\mathbb{A}[[X]]$, $1\leq i\leq j$ one has  
\be [A_{1}\ldots A_{j}]_{k}=\sum_{k_{1}+\ldots k_{j}=k}[A]_{k_{1}}\cdots [A]_{k_{j}}.\label{D.420}\ee

When $A=\sum_{n\in\N^{d'}}a_{n}X^n\in\mathbb{A}[[X]]$ and $B\in(\mathbb{A}[[X]])^{d'}$, $v(B)\geq 1$ one can define
$$A\circ B=\sum_{n\in\N^{d'}}a_{n}B^n.$$

If moreover $\mathbb{A}$ is endowed with   {\it derivations}     $\d_{i}:\mathbb{A}\to \mathbb{A}$, $1\leq i\leq d'$,    ($\d_{i}(a+b)=\d_{i}a+\d_{i}b$, $\d_{i}(ab)=(\d_{i}a)b+a(\d_{i}b)$)  we    define (Taylor formula) \footnote{We use a multi-index notation, $k=(k_{1},\ldots,k_{d'})$, $B^k=B_{1}^{k_{1}}\cdots B_{d'}^{k_{d'}}$, $k!=k_{1}!\cdots k_{d'}!$, $\d^k=\d_{{1}}^{k_{1}} \cdots \d_{{d'}}^{k_{d'}}$.}
for each $a\in\mathbb{A}$ and $B\in(\mathbb{A}[[X]])^{d'}$, $v(B)\geq 1$ 
\be a\circ_{\d} B=\sum_{k\in\N^{d'}}(1/k!)(\d^{k}a) B^{k}\in \mathbb{A}[[X]].\label{defcircdelta}\ee
Similarly, if $A=\sum_{k\in\N^{d'}}a_{k} X^{k},B,C\in(\mathbb{A}[[X]])^{d'}$, $v(B)\geq 1$, $v(C)\geq 1$, we can also define 
\be A\circ_{\d} (B,C)=\sum_{n\in\N^{d'}}(a_{n}\circ_{\d} B)C^n. \label{defcircdeltabis}\ee
\begin{lemma}\label{lemma:D1ee}
For $k\in\N^*$, $v(A)\geq 1$, $v(B)\geq 1$, $v(C)\geq 2$, $[A\circ_{\d} (B,X+C) -A]_{k}$ is a polynomial 
in the coefficients of $[\d^{l}A]_{k_{1}}, [B]_{k_{2}},[C]_{k_{3}}$ for $k_{1}+k_{2}+k_{3}\leq k-1$, $|l|\leq k$ (this polynomial being with rational coefficients).
\end{lemma}
\begin{proof}
Since $(a_{n}\circ_{\d}B)(X+C)^n-a_{n}X^n=(a_{n}\circ_{\d}B)((X+C)^n-X^n)+X^n((a_{n}\circ_{\d}B)-a_{n})$
\begin{align*} &A\circ_{\d} (B,X+C)-A= (I)+(II) \\
&(I):=\sum_{\substack {|l|\geq 0 \\ |n|\geq 1\\  m\leq n,\ |m|\geq 1}} {n\choose m}\frac{\d^la_{n}}{l!} B^l C^mX^{n-m},\qquad  (II)=\sum_{\substack {|l|\geq 1 \\ |n|\geq 1}} \frac{\d^la_{n}}{l!} B^l X^n
\end{align*}
and one can conclude using (\ref{D.420}).
\end{proof}
Assume now that $(\mathbb{A},\d)$ is endowed with a {\it translation} by which we mean   an action $\tau$   of an abelian group  (we suppose it is $(\R^d,+)$) on $\mathbb{A}$ that commutes with the derivations $\d_{i}$.

\subsubsection{Formal diffeomorphisms}

A {\it formal diffeomorphism} of  $\mathbb{A}[[X]]$ is a triple  $(\a,A,B)$ (we denote it by $f_{\a,A,B}$) with   $A,B\in(\mathbb{A}[[X]])^d$ with $v(B)\geq 2$ and where $v(A)\geq 1$ and $\a\in\R^d$. We can define the composition of to such objects:
\begin{align*}&f_{\e,E,D}=f_{\gamma,C,D}\circ f_{\a,A,B}\quad \iff\\
&\e=\a+\gamma,\qquad E=A+(\tau_{-\a}C)\circ_{\d}(A,X+B),\quad F=B+(\tau_{-\a}D)\circ_{\d}(A,X+B)
\end{align*} 
with $v(E)\geq 1$ and $v(F)\geq 2$.
One can check that the usual algebraic rules for compositions are satisfied and that each such diffeomorphism has an  inverse for composition.

\begin{rem} \label{rem:D1}One of the  example we have in mind is the following. Take  $d'=d\in\N^*$, $\mathbb{A}=C^\omega(\T^d)$ the ring of real analytic functions on $\T^d$ (taking real values on the real axis) and the ring of formal power series is $\mathbb{A}[[r]]=\{\sum_{n\in\N^d}a_{n}(\th)r^n, \ a_{n}\in C^\omega(\T^d)\}$, $r=(r_{1},\ldots,r_{d})$. The derivations in this case are  
$\d_{i}a=\pa_{\th_{i}}a$ if $a:(\th_{1},\ldots,\th_{d})\to\R$ is in $C^\omega(\T^d)$, the translation is  $\tau_{\a}a=a(\cdot-\a)$ ($\a\in\R)$ and the formal map $f_{\a,A,B}$ can be written under the more suggestive form
$$f_{(\a,A,B)}(\th,r)=(\th+\a+A(\th,r),r+B(\th,r))
$$
as a formal diffeomorphism of $\T\times \R$.
\end{rem}
\subsubsection{Degree}
In case  we can assign a  {\it degree} $\deg(a)$ to each element $a$ of the ring $\mathbb{A}$ (it satisfies by definition $\deg(0)=-\infty$, for all $a,b\in \mathbb{A}$, $\deg(a+b)=\max(\deg(a),\deg(b))$ and $\deg(ab)=\deg(a)+\deg(b)$) we can associate to each {\it weight} $p:\N^{d'}\to\N$ the set 
$$\cC(p)=\biggl\{\sum_{n\in\N^{d'}} a_{n}X^{n}\in\mathbb{A}[[X]],\ \forall \ n\in\N^{d'},\ \deg(a_{n})\leq p(n)\biggr\}.$$
By extension if $B=(B_{1},\ldots,B_{d'})\in(\mathbb{A}[[X]])^{d'}$ we say that $B$ is in $\cC(p)$ if each $B_{l}\in\cC(p)$, $1\leq l\leq d'$.
If $p,q:\N^d\to\N$ we define
$$p*q(n)=\max_{(k,l)\in\N^{d'}, k+l=n}(p(k)+q(l)).$$   In particular if  
$$\bar p(n):=|n|=n_{1}+\ldots+n_{d'},\qquad n=(n_{1},\ldots,n_{d'})\in\N^{d'}$$ one has $\bar p*\bar p=\bar p$ and  $(\bar p-1)^{*m}=\bar p-m$. 

We say that the degree $\deg$ is {\it compatible} with the derivations $\d_{i}$ and the translation $\tau$ if for any $\a\in\R$, $1\leq i\leq d$, $\deg( \tau_{\a}\d_{i}a)\leq\deg(a)$.

\begin{rem}\label{rem:D2} The relevant example for our purpose (proof of Lemma \ref{BNF:lemma6.3}) will be the following. Take $d'=d$, $\mathbb{A}=C^\omega(\T^d)[t]$ the set of polynomials in $t$ with coefficients in $C^\omega(\T^d)$, $F^{(t)}(\th)=a_{0}(\th)+\cdots+a_{n}(\th)t^n$, $a_{j}\in C^\omega(\T^d)$, $0\leq j\leq n$, $n\in\N$ and $a_{n}\ne 0$.  The  derivations $\d_{i}$, $1\leq i\leq d$ are  defined by $\d_{i}F^{(t)}(\th)=(\pa_{\th_{i}}a_{0})(\th)+\cdots+(\pa_{\th_{i}}a_{n})(\th)t^n$, the translations  $\tau_{\a}F^{(t)}(\th)=a_{0}(\th-\a)+\cdots+a_{n}(\th-\a)t^n$ and the degree  $\deg F^{t}=n$ is compatible with both of them.
\end{rem}

\bigskip
The following facts are easily checked. Assume that $\mathbb{A}$ is a ring with derivations $\d_{i}$, $1\leq i\leq d'$ and a compatible degree $\deg$ and let $(p_{l})_{l\in\N}$ be weights.
\begin{enumerate}
\item\label{iapp1} If $A,B\in\mathbb{A}[[X]]$, $A\in\cC(p_{1})$, $B\in\cC(p_{2})$ one has $AB\in \cC(p_{1}*p_{2})$.
\item\label{iapp2} If $A_{l}\in\mathbb{A}[[X]]$, $A_{l}\in \cC(p_{l})$, $\lim_{l\to\infty}v(A_{l})=\infty$ one has $\sum_{l\in\N}A_{l}\in\cC(\max_{l}p_{l})$.
\end{enumerate}
Let $p$ be a weight such that $p*p\leq p$. Using (\ref{defcircdelta}), point  \ref{iapp1} and \ref{iapp2}  we have
\begin{enumerate}[resume]
\item\label{iapp3} If $a\in \mathbb{A}$, $B\in\mathbb{A}[[X]]$, $B\in \cC(p)$ one has $a\circ_{\d} B-a\in \cC(\deg(a)+ p)$.
\end{enumerate}
\begin{lemma}\label{iapp4}
If $A\in\mathbb{A}[[X]]$, $B,C\in(\mathbb{A}[[X]])^{d'}$, $v(B)\geq 1$, $v(C)\geq 1$ with $A\in\cC(\bar p-c_{A})$, $B\in\cC(\bar p-c_{B})$, $C\in\cC(\bar p-1)$, $\min(c_{A},c_{B})\geq 0$,  then $A\circ_{\d}(B,C)-A\circ C\in \cC(\bar p-c_{A}-c_{B})$ and $A\circ_{\d}(B, C)\in \cC(\bar p-c_{A})$.
\end{lemma}
\begin{proof}
Recall that $A\circ_{\d} (B,C)=\sum_{n\in\N^{d'}}(a_{n}\circ_{\d} B)C^n$. From point \ref{iapp3} $a_{n}\circ_{\d} B-a_{n}\in\cC(\deg(a_{n})+\bar p-c_{B})$ and from point \ref{iapp1} $(a_{n}\circ_{\d} B)C^n-a_{n}C^n\in\cC((\deg(a_{n})+\bar p-c_{B})*({\bar p}-1)^{*|n|})\subset \cC((\deg(a_{n})+\bar p-c_{B})*(\bar p-|n|))$ hence  $(a_{n}\circ_{\pa} B)C^n-a_{n}C^n\in\cC(\bar p-c_{A}-c_{B})$. By using  point \ref{iapp2} we have  $A\circ_{\d}(B,C)-A\circ C\in \cC(\bar p-c_{A}-c_{B})$. A similar argument shows that $A\circ C\in\cC(\bar p-c_{A})$ whence the conclusion.
\end{proof}
Before stating the next lemma we introduce the following definition: we say that a formal diffeomorphism $f_{\a,A,B}$ is in $\cD(\bar p-1)$ if $A\in\cC(\bar p)\cap O(r)$, $B\in \cC(\bar p-1)\cap O^2(r)$. 
\begin{lemma}\label{lemma:eD3} One has
\begin{enumerate}
\item \label{z1} Let  $H\in\cC(\bar p-c)$, $c=0,1$, and $f_{\a,A,B}\in\cD(\bar p-1)$. Then $H\circ f_{\a,A,B}\in\cC(\bar p-c)$.
\item\label{z2} The composition of two formal diffeomorphisms in $\cD(\bar p-1)$ is in $\cD(\bar p-1)$.
\item\label{z3}  The inverse for the composition of a diffeomorphism of $\cD(\bar p-1)$ is in $\cD(\bar p-1)$.
\item \label{z4}If $f_{\a,A,B}^{-1}=f_{\ti\a,\ti A,\ti B}$, then for any $k\geq 1$, $[\ti A]_{k}$,  $[\ti B]_{k}$, are  polynomials in the coefficients of $[\tau_{m_{1}\a}\d^{l_{1}}A]_{k_{1}},[\tau_{m_{2}\a}\d^{l_{2}}B]_{k_{2}}$, $k_{1},k_{2}\leq k$, $l_{1},l_{2}\leq k$, $|m_{1}|,|m_{2}|\leq k$.
\end{enumerate}
\end{lemma}
\begin{proof}Items \ref{z1} and \ref{z2} are consequences of Lemma \ref{iapp4}.

For  point  \ref{z3} we just have to prove the result when $\a=0$. Let us  denote by $\cU$ the operator $H\mapsto H\circ f_{0,A,B}$. Note that  $v((\cU-id)H)\geq v(H)+1$ hence the series $\ti H:=\sum_{k=0}^\infty (\cU-id)^kH$ converges in $\mathbb{A}[[r]]$ and from \ref{z1} and \ref{z2} one sees that if $H\in\cC(\bar p-c)$, $c=0,1$, the same is true for $ \ti H$. To conclude we observe that $-(f_{0,A,B}^{-1}-id)=(\cU-id)(f_{0,A,B}^{-1}-id)+(f_{0,A,B}-id)$ hence 
\be -(f_{0,A,B}^{-1}-id)=\sum_{k=0}^\infty(\cU-id)^k(f_{0,A,B}-id).\label{eeD.423}
\ee
Finally point \ref{z4} is a consequence of (\ref{eeD.423}) and Lemma \ref{lemma:D1ee}.
\end{proof}

\subsection{Formal Birkhoff Normal Forms}\label{sec:fBNF}
 From now on we work in the setting of Remark \ref{rem:D2}.\footnote{Note that to prove the existence and uniqueness of the formal BNF of Subsection \ref{sec:D.2.2} it would be enough to work with $\mathbb{A}=\C^\omega(\T^d)$. }
\subsubsection{Formal exact symplectic diffeomorphism}
If $F\in\mathbb{A}_{}[[r]]$,  $F=\<\a,r\>+O^2(r)$ with $\a\in\R^d$  we define the formal diffeomorphism $f_{F}(\th,r)=(\th+A(\th,r),r+B(\th,r))$ as suggested  by the implicit relation 
$$ \ph=\th+\pa_{R}F(\th,R),\qquad r=R+\pa_{\th}F(\th,R),\qquad f_{F}(\th,r)=(\ph,R)
$$
or more formally
$$ A(\th,r)=\pa_{r}F(\th,r+B(\th,r)),\qquad 0=B(\th,r)+\pa_{\th}F(\th,r+B(\th,r))$$
$$ A=\pa_{r}F\circ_{\d}(0,B),\qquad 0=B+\pa_{\th}F\circ_{\d}(0,B)\qquad (\textrm{cf.}\ \ref{defcircdeltabis}).)$$
In this situation we use the more intuitive notations $R(\th,r)=r+B(\th,r)$, $\ph(\th,r)-\th=A(\th,r)$.
We shall call such formal diffeomorphisms $f_{F}$ {\it formal exact symplectic diffeomorphisms}. The set of all such diffeomorphisms is a group under composition. Let us define
$$\cE(\bar p-1) :=\{f_{F},\ F=\<\a,r\>+\sum_{|n|\geq 2}F_{n}(\th)r^n\in\mathbb{A}[[r]],\ F\in\cC(\bar p-1) \}.$$
The following result is then a consequence of Lemma \ref{lemma:eD3}.
\begin{lemma}The set $\cE(\bar p-1)$ is  a subset of $\cD(\bar p-1)$ stable by composition and inversion.
\end{lemma}

\begin{rem} In the (CC)-case the relevant choice for $\mathbb{A}$ is $\C[[t]]$ and the set of  formal series is $\mathbb{A}[[z,w]]$. One can extend in this context the notion of $\s_{2}$-symmetry. If $F=\sum_{(n,m)\in\N^d\times \N^d}F_{n,m}z^mw^m\in\mathbb{A}[[z,w]]$ we say it is $\s_{2}$-symmetric ($\s_{2}(z,w)=(i\bar w,i\bar z)$) if $\bar F_{n,m}=(i)^{|n|+|m|}F_{m,n}$ for all $(n,m)\in\N^d\times\N^d$ ($\bar F_{n,m}$ is the complex conjugate of $F_{n,m}$ and $i=\sqrt{-1}$). Similarly one can define the notion of $\s_{2}$-symmetric formal diffeomorphism. If  $F$ is $\s_{2}$-symmetric then $f_{F}$ is $\s_{2}$-symmetric.
\end{rem}

\bigskip

\subsubsection{Existence and uniqueness of the formal BNF}\label{sec:D.2.2}

Given  $f_{2\pi\<\omega_{0},r\>+F}$ we wish to find a formal exact symplectic diffeomorphism $f_{Z} =id+O^2(r)$  and $B(r)=2\pi\<\omega_{0},r\>+O^2(r)\in \R[[r]]$ such that 
\be f_{Z}\circ f_{2\pi\<\omega_{0},r\>+F}(\th,r)=f_{B}\circ f_{Z}(\th,r).\label{eeD.425}
\ee
We furthermore impose that $Z=O^2(r)$ is {\it normalized} in the sense that 
\be \int_{\T^d}Z(\ph,Q)d\ph=0.\label{eeD.427}\ee
We use the notations $f_{2\pi\<\omega_{0},r\>+F}:(\th,r)\mapsto (\ph,R)$, $f_{Z}:(\ph,R)\mapsto (\psi,Q)$ so that 
\begin{align*}&\psi=\ph+\pa_{Q}Z(\ph,Q),\qquad R=Q+\pa_{\ph}Z(\ph,Q)\\
&\ph=\th+2\pi\omega_{0}+\pa_{R}F(\th,R),\qquad r=R+\pa_{\th}F(\th,R).
\end{align*} 

Using  the relation $R=Q+\pa_{\ph}Z(\th+2\pi\omega_{0}+\pa_{R}F(\th,R),Q)$ and the fact that $g:(\th,Q,R)\mapsto (\th,Q,R-Q-\pa_{\ph}Z(\th+2\pi\omega_{0}+\pa_{R}F(\th,R),Q))$  is  a   formal diffeomorphism\footnote{Defined on $\mathbb{A}[[Q,R]]$.} we can define $R(\th,Q)=Q+O^2(Q)$ by $(\th,Q,R(\th,Q))=g^{-1}(\th,Q,0)$. 
\begin{lemma}\label{lemma:D.10}
\
\begin{enumerate}
\item \label{w1}For any $k\geq 1$, the coefficients of $[R(\th,Q)]_{k}$ are polynomials in the coefficients of $[\tau_{2\pi m_{1}\omega_{0}}\d^{l_{1}}F]_{k_{1}}, [\tau_{2\pi m_{2}\omega_{0}}\d^{l_{2}}Z]_{k_{2}}$, $k_{1},k_{2},l_{1},l_{2},|m_{1}|,|m_{2}|\leq k$.
\item \label{w2}If $F,Z\in\cC(\bar p-1)$, the formal diffeomorphism  $(\th,Q)\mapsto (\th,R(\th,Q))$ is in $\cD(\bar p-1)$.
\end{enumerate}
\end{lemma}
\begin{proof}These are consequences of Lemma \ref{lemma:eD3}.
\end{proof}

Let $f_{Z}(\th,r)=(\th',r')$; from the formal conjugation relation (\ref{eeD.425}) we get  $f_{B}\circ f_{Z}(\th,r)=(\th'+\nabla B(r'),r')=(\psi,Q)$ hence $Q=r'$ and $\th'=\th+\pa_{Q}Z(\th,Q)$. We thus have
$$\th+\pa_{Q}Z(\th,Q)+\nabla B(Q)=\ph+\pa_{Q}Z(\ph,Q)$$
and using the relations between 
$\ph,\th$ yields
\begin{multline*}-\pa_{R}F(\th,R)=\pa_{Q}Z\biggl(\th+2\pi\omega_{0}+\pa_{R}F(\th,R),Q\biggr)-\pa_{Q}Z(\th,Q)\\-(\nabla B(Q)-2\pi\omega_{0})
\end{multline*}
that we can write
\be -\pa_{Q} \cF(F,Z)= \pa_{Q}Z(\th+2\pi\omega_{0},Q)-\pa_{Q}Z(\th,Q)-(\pa_{Q} B(Q)-2\pi\omega_{0})\label{D.423}
\ee
where $\cF(F,Z)=O^2(r)$ is uniquely defined (note that the RHS of  (\ref{D.423}) is $O(r)$) by
\begin{multline}\pa_{Q}\cF(F,Z)=\pa_{Q}F(\th,R(\th,Q))+\\\biggl( \pa_{Q}Z\biggl(\th+2\pi\omega_{0}+\pa_{Q}F(\th,R(\th,Q)),Q\biggr) -\pa_{Q}Z(\th+2\pi\omega_{0},Q) \biggr). \label{D.424} \end{multline}
We thus have
\be  - \cF(F,Z)= Z(\th+2\pi\omega_{0},Q)-Z(\th,Q)-(B(Q)-2\<\pi\omega_{0},Q\>).\label{D.423bis}
\ee
\begin{lemma}\label{lemma:eeD.11}
\
\begin{enumerate}
\item \label{w1}For any $k\geq 1$, the coefficients of $[\cF(F,Z)-F]_{k}$ are polynomials in the coefficients of $[\tau_{2\pi m_{1}\omega_{0}}\d^{l_{1}}F]_{k_{1}}, [\tau_{2\pi m_{2}\omega_{0}}\d^{l_{2}}Z]_{k_{2}}$, $k_{1},k_{2}\leq k-1$, $l_{1},l_{2}\leq k$,  $|m_{1}|,|m_{2}|\leq k$.
\item \label{w2}If $F,Z\in\cC(\bar p-1)$, one has $\cF(F,Z)\in\cC(\bar p-1)$.
\end{enumerate}
\end{lemma}
\begin{proof}This is a consequence of (\ref{D.424}), Lemma \ref{lemma:D.10} and  Lemma \ref{lemma:eD3}.
\end{proof}

Now from (\ref{D.423bis}) one has 
\be\begin{cases}k=2,\quad&-[F]_{2}(\th,Q)= [Z]_{2}(\th+2\pi\omega_{0},Q)-[Z]_{2}(\th,Q)- [B]_{2}(Q),\\
 \forall\ k\geq 3\quad &- [\cF(F,Z)]_{k}(\th,Q)= [Z]_{k}(\th+2\pi\omega_{0},Q)-[Z]_{k}(\th,Q)- [B]_{k}(Q)
\end{cases}\label{eeD.431}\ee
We now recall the classical result ({\it cf.} Proposition \ref{prop:n5.2}):
\begin{lemma}\label{eelemmaD.12}If $\omega_{0}$ is Diophantine, for any $G\in\mathbb{A}[[r]]$ there exists a unique pair $(Z,B)$ with $Z\in\mathbb{A}[[r]]$ normalized in the sense of  (\ref{eeD.427}) and $B=B(r)\in\R[t][[r]]$ such that 
$$G(\th,Q)=Z(\th+2\pi\omega_{0},Q)-Z(\th,Q)+B(Q).$$
Furthermore: (1)  $B(Q)=\int_{\T^d}G(\th,Q)d\th$ and  if $G=[G]_{k}$ one has $Z=[Z]_{k}$, $B=[B]_{k}$ and  the coefficients of  $[Z]_{k}$ are  $\R$-linear functions of the coefficients of $[G]_{k}$; (2) if $G\in\cC(\bar p-1)$ then $Z,B\in\cC(\bar p-1)$.
\end{lemma}

\subsubsection{Existence and uniqueness of the BNF}
We can now give a proof of the uniqueness and the existence of formal Birkhoff Normal Forms (\ref{eeD.425}).

\medskip\noindent{\it Uniqueness}:  Equation (\ref{eeD.431}), Lemma \ref{lemma:eeD.11} and Lemma \ref{eelemmaD.12} show inductively that $[Z]_{k}, [B]_{k}$ are uniquely defined by $[F]_{j}$, $2\leq j\leq k-1$. Hence, $Z$ and $B$ are unique.

\medskip\noindent{\it Existence}: Define $[Z]_{2},[B]_{2}$ by (\ref{eeD.431}) and then inductively for $k\geq 3$,  $[Z]_{k},[B]_{k}$.
by 
\be - [\cF(F,[Z]_{\leq k-1})]_{k}(\th,Q)= [Z]_{k}(\th+2\pi\omega_{0},Q)-[Z]_{k}(\th,Q)- [B]_{k}(Q)\label{eeD.432} \ee
where $Z_{\leq k-1}=\sum_{l=2}^{k-1}[Z]_{l}$. Setting $F=\sum_{l=2}^{\infty}[Z]_{l}$, $B=\sum_{l=2}^{\infty}[B]_{l}$ one can check that (\ref{D.423bis}) holds modulo $O^k(r)$ for any $k$ and hence in $\mathbb{A}[[r]]$. \hfill $\Box$

\subsection{Proof of Lemma \ref{BNF:lemma6.3} }\label{appendix:D4}
We define $F^{(t)}(\th,r)=tF(\th,r)+(1-t)G(\th,r)$ which is in $\mathbb{A}[[r]] \cap \cC(\bar p-1)$, $\mathbb{A}=C^\omega(\T^d)[t]$. Note that  for any $k\geq 2$, $[F^{(t)}]_{k}\in\cC(\bar p-1)$. In particular, as a consequence of Lemmata  \ref{lemma:eeD.11}, item  \ref{w2} and \ref{eelemmaD.12}, point (2),   the sequences $[Z^{(t)}]_{k},[B^{(t)}]_{k}$, inductively constructed in (\ref{eeD.432}), are in $\cC(\bar p-1)$. Hence $B^{(t)}(r):=\sum_{n\in\N^d}b_{n}(t)r^n$ is in $\cC(\bar p-1)$ which by definition means that de degree in $t$ of each $b_{n}(t)$ is $\leq |n|-1$. \hfill $\Box$

 \section{Approximate Birkhoff Normal Forms.} 

We give in this section the proofs of Propositions  \ref{nprop:9.2} and  \ref{BNFprop}. 
 
\subsection{A useful Lemma}
Let be given for each  $\a\in ]0,1/2[$,  a function $P_{\a}:\R_{+}\times \R_{+}\to\R_{+}$, $P_{\a}:(k,t)\mapsto P_{\a}(k,t)$ nondecreasing in each variable and   assume that $(\e_{\a,k})_{k\in I_{\a}}$, $I_{\a}\subset\N$ is a sequence of nonnegative  real numbers  depending on $\a\in ]0,1/2[$ and defined inductively as long as a condition of the form 
\be  P_{\a}(k,\e_{\a,k})<1 \label{E462}
\ee
is satisfied (we assume that $\e_{\a,0}$ satisfies (\ref{E462})). Let us call $J_{\a}=\llbracket 0,k_{\a}^*\rrbracket$, $k_{\a}^*\geq 1$  the maximal set of integers $k\in\N$ for which $\e_{\a,k}$ is defined: this means that if $k\in J_{\a}$ and $\e_{\a,k}$ satisfies (\ref{E462}) then $k+1\in J_{\a}$ (in particular $P_{\a}(k_{\a}^*,\e_{\a,k_{\a}^*})\geq 1$). Let $\th>0$, $a>0$ and $\bar k_{\th,a}\in\N^*$ be such that 
 \be\forall k\in \llbracket 0,\min(k_{\a}^*,\bar k_{\th,a})-1\rrbracket,\quad \e_{\a,k+1}\leq C_{\th,a}\a^\th\times \biggl(1+\a^{-a}\sum_{j=0}^k\e_{\a,j}\biggr )\e_{\a,k}.\label{E458}
\ee

 We have the following type of Gronwall Lemma:
 \begin{lemma}\label{lem:E3} Assume that 
 \be (2C_{\th,a}\a)^\th<1/2,\qquad \e_{\a,0}\leq \a^{a}/2,\qquad  P_{\a}(\bar k_{\th,a}+1,\e_{\a,0})<1.\label{condalphaepsilon}\ee
 Then, 
 \be  k_{\a}^*\geq \bar k_{\th,a}
 \ee
  and 
  \be\forall \ k\in [0,\bar k_{\th,a}]\cap\N,\quad  \e_{\a,k}\leq (2C_{\th,a}\a)^{\th k}\e_{\a,0}.\label{E464}
\ee

 \end{lemma}
 \begin{proof}
 1) Let $k^*_{\a,\th,a}=\min(k^*_{\a},\bar k_{\th,a})$. We first prove that the set 
  $$K_{\a,\th,a}=\{k\in\llbracket 0,k^*_{\a,\th,a}\rrbracket,\ \e_{\a,k}>(2C_{\th,a}\a)^{\th k}\e_{\a,0}\}$$
  is empty. If this were not the case we could define $k_{\a,\th,a}=\inf K_{\a,\th,a}$ and write
\be \forall \ k\in \llbracket 0,k_{\a,\th,a}-1\rrbracket,\quad \e_{\a,k}\leq (2C_{\th,a}\a)^{\th k}\e_{\a,0}
\ee
hence
  $$\sum_{j=0}^k\e_{\a,j}\leq \frac{\e_{\a,0}}{1-(2C_{\a,\th,a}\a)^{\th }}\leq 2\e_{\a,0}$$
and thus from \ref{E458} and (\ref{condalphaepsilon}), for all $0\leq k\leq k_{\a,\th,a}-1$
$$\e_{\a,k+1}\leq C_{\th,a}\times \a^{\th}(1+2\a^{-a}\e_{\a,0})\e_{a,k}\leq (2C_{\th,a}\a^{\th})\e_{\a,k}.$$
  This implies that for all $0\leq k\leq k_{\a,\th,a}$ one has 
$$\e_{\a,k}\leq (2C_{\th,a}\a)^{\th k} \e_{\a,0}$$ 
and in particular $\e_{\a,k_{\a,\th,a}}\leq (2C_{\th,a}\a)^{\th k_{\a,\th,a}}\e_{\a,0}$. This contradicts the definition of $k_{\a,\th,a}$ as $\inf K_{\a,\th,a}$.

2) Since $K_{\a,\th,a}$ is empty, one has 
\be \forall \ k\in \llbracket 0, k^*_{\a,\th,a}\rrbracket,\quad \e_{\a,k}\leq (2C_{\th,a}\a)^{k\th}\e_{\a,0}\leq \e_{\a,0}.
\ee
But  since $k^*_{\a,\th,a}+1\leq \bar k_{\th,a}+1$ and $(2C_{\th,a}\a)\leq 1$
\begin{align*}P_{\a}(k^*_{\a,\th,a}+1,\e_{\a,k^*_{\a,\th,a}})&\leq P_{\a}(\bar k_{\th,a}+1,\e_{a,0})\\
&<1
\end{align*}
This implies that $k^*_{\a,\th,a}+1\in J_{\a}$ hence $k_{\a,\th,a}^*<k_{\a}^*$ and from the definition of $k^*_{\a,\th,a}$, we get  $k_{\a}^*\geq \bar k_{\th,a}$ which in its turn implies $ k^*_{\a,\th,a}=\bar k_{\th,a}$. We have thus proven that for all $k\leq \bar k_{\th,a}$ one has
$$\e_{\a,k}\leq (2C_{\th,a}\a)^{k\th}\e_{\a,0}.$$
 \end{proof}

\subsection{Proof of Proposition \ref{nprop:9.2} (BNF, (CC)-Case)}\label{appendixnprop:9.2}
For $n$ large enough we define
\be \rho_{0}= \frac{1}{q_{n}^{6}},\qquad W_{0}=e^hW_{h,\bD(0,\rho_{0})}\label{e8.135}
\ee
 and for $k\geq 0$ we introduce the sequences
\be\d_{k}=C^{-1}\frac{h}{(k+1)(\ln(k+2))^2}, \qquad \sum_{l=0}^\infty \d_{k}\leq h/10\label{deltaketcante}\ee
\be \rho_{k}=\exp(-\sum_{l=0}^{k-1}\d_{l})\rho_{0},\quad W_{k}=\exp(-\sum_{l=0}^{k-1}\d_{l})W_{0}.\label{deltaketc}\ee

Recall that $m=2\bar a+1$ and that ({\it cf.} (\ref{ee6.96})
\be F_{0}(z,w)=O^{2m}(z,w),\qquad \e_{0}:=\|F_{0}\|_{W_{0}}\lesssim \rho_{0}^m.\label{F0incond}\ee

 We  shall construct  inductively  for $k\geq 0$, sequences  $Z_{k}\in\cO_{\s}(W_{k})$, $F_{k}\in\cO_{\s}(W_{k})\cap O^{k+2m}(z,w)$, $\Omega_{k}\in \cO_{\s}(\bD(0,\rho_{k}))$, $\Omega_{k}(r)=2\pi\omega_{0}r+O^2(r)$, such that  
$$\Omega_{0}=\Omega,\qquad F_{0}=F
$$
and for $k\geq 1$, 
\be g_{k}^{-1}\circ \Phi_{\Omega_{0}}\circ  f_{F_{0}}\circ g_{k}=\Phi_{\Omega_{k}}\circ f_{F_{k}}.\label{indk}\ee
To do this we proceed the following way: assuming (\ref{indk}) holds and $\d_{k}\leq q_{n}^{-1}$, we apply  Proposition \ref{prop:n5.5}  with $\tau=0$, $K/2=N=q_{n}$ {\it cf.} (\ref{BNF7.122}),  and define $Y_{k}\in\cO_{\s}(W_{k})\cap O^{k+2m}(z,w)$ (see Remark \ref{rem:5.1}) satisfying
\be -[Q_{0}]\cdot Y_{k}=T_{q_{n}}F_{k}-\cM_{0}(F_{k}),\qquad \|Y_{k}\|_{e^{-\d_{k}/2}W_{k}}\lesssim q_{n}^{2}\|F_{k}\|_{W_{k}} \label{cohomeqprop6.2}\ee
where we denote $$Q_{0}(r)=2\pi\omega_{0}r.$$
Using Lemma \ref{lemma:5.4} we get ({\it cf.} formula (\ref{5.117b}))
\begin{align*}f_{Y_{k}}\circ \Phi_{\Omega_{k}}\circ  f_{F_{k}}\circ f_{Y_{k}}^{-1}&=
\Phi_{\Omega_{k}+M(F_{k})}\circ f_{F_{k}-\cM_{0}(F_{k})+[\Omega_{k}+M(F_{k})]\cdot Y_{k}+\dot\fO_{2}^{(\bar a)}(Y_{k},F_{k})}\\
&=\Phi_{\Omega_{k}+M(F_{k})}\circ f_{R_{q_{n}} F_{k}+[\Omega_{k}+M(F_{k})]\cdot Y_{k}-[Q_{0}]\cdot Y_{k}+\dot \fO_{2}^{(\bar a)}(Y_{k},F_{k})}.
\end{align*}
Hence $f_{Y_{k}}\circ \Phi_{\Omega_{k}}\circ  f_{F_{k}}\circ f_{Y_{k}}^{-1}= \Phi_{\Omega_{k+1}}\circ f_{F_{k+1}}$ with 
\be \Omega_{k+1}-\Omega_{k}=M(F_{k})\label{Omegak+1}
\ee
and, using the fact that $[\Omega_{k}+M(F_{k})]\cdot Y_{k}-[Q_{0}]\cdot Y_{k}=O(|\nabla Y_{k}|\times |\pa(\Omega_{k}-Q_{0})\circ r|)+\dot\fO_{2}^{(2)}(Y_{k},F_{k})$, 
\be F_{k+1}=R_{q_{n}}F_{k}+\dot\fO_{2}^{(\bar a)}(Y_{k},F_{k})+O\biggl(|\nabla Y_{k}|\times |\pa(\Omega_{k}-Q_{0})\circ r|\biggr).\label{Fk+1}\ee
Notice that from (\ref{Omegak+1}) and the fact that for $0\leq j\leq k-1$, $\cM_{0}(F_{j})=O^{j+2m}(z,w)$ ({\it cf.} the remark at the end of Subsection \ref{sec:5.1.1}), hence $M(F)(r)=O^{m}(r)$; one thus  has
\be\forall\ 0\leq j\leq k,\qquad  \Omega_{k}(r)-Q_{0}(r)=O(r^2).
\ee
Since $F_{k},\ Y_{k}\in O^{k+2m}(z,w)$  we have (see Remarks \ref{rem:2.3}, \ref{rem:5.1})   $\dot\fO_{2}^{(\bar a)}(Y_{k},F_{k})=O^{2k+4m-2\bar a}(z,w)$; also, $O(|\nabla Y|\times |\pa(\Omega_{k}-Q_{0})\circ r|)=O^{k+2m+1}(z,w)$ and from (\ref{n5.81ccbis})  $R_{q_{n}}F_{k}=O^{q_{n}}(z,w)$ (if $q_{n}\geq m$). As a consequence, since $2k+4m-2\bar a\geq k+1+2m$
 we see that $F_{k+1}=O^{k+1+2m}(z,w)$. Furthermore since $\Omega_{0}(r)-Q_{0}(r)=O(r^2)$
\begin{align*}\|\pa(\Omega_{k}-Q_{0})\circ r\|_{W_{k}}&\leq \|\pa(\Omega_{k}-\Omega_{0})\circ r\|_{W_{k}}+\|\pa(\Omega_{0}-Q_{0})\circ r\|_{W_{k}}\\
&\lesssim  \|\pa(\Omega_{k}-\Omega_{0})\circ r\|_{W_{k}}+\sup |r({W_{k}})|\\
&\lesssim  \|\pa(\Omega_{k}-\Omega_{0})\circ r\|_{W_{k}}+\rho_{k}
\end{align*}
and using (\ref{eq:cauchy:derivativesquarto})
\begin{align*}\|\nabla Y_{k}\|_{e^{-\d_{k}}W_{k}}&\lesssim  \d_{k}^{-1}\rho_{k}^{-1/2}\|Y_{k}\|_{e^{-\d_{k}/2}W_{k}}\\
&\lesssim {q_{n}^2}\rho_{k}^{-1/2}\d_{k}^{-1}\|F_{k}\|_{W_{k}}
\end{align*}
hence 
\be|\nabla Y_{k}|\times |\pa(\Omega_{k}-Q_{0}|\lesssim {q_{n}^2}\rho_{k}^{-2}\d_{k}^{-2}\|\Omega_{k}-\Omega_{0}\|_{\bD(0,\rho_{k})} \|F_{k}\|_{W_{k}}+{q_{n}^2}\d_{k}^{-1}\rho_{k}^{1/2}\|F_{k}\|_{W_{k}}.\label{rho1/2}
\ee
From  (\ref{Fk+1}), (\ref{cohomeqprop6.2}), 
(\ref{rho1/2}),
and (\ref{n5.83})
we get  that, provided 
$$ \rho_{k}^{-\bar a}\d_{k}^{-\bar a}{q_{n}^2}\|F_{k}\|_{W_{k}}<1,$$
one has the inequalities
\begin{multline}\|F_{k+1}\|_{e^{-\d_{k}}W_{k}}\lesssim \d_{k}^{-1}e^{-q_{n}\d_{k}/2}\|F_{k}\|_{W_{k}}+ \rho_{k}^{-\bar a}\d_{k}^{-\bar a}{q_{n}^2}\|F_{k}\|_{W_{k}}^2\\+{q_{n}^2}\rho_{k}^{-2}\d_{k}^{-2}\|\Omega_{k}-\Omega_{0}\|_{\bD(0,\rho_{k})} \|F_{k}\|_{W_{k}}+{q_{n}^2}\d_{k}^{-1}\rho_{k}^{1/2}\|F_{k}\|_{W_{k}}\label{}
\end{multline}
and 
\be \|\Omega_{k+1}-\Omega_{0}\|_{\bD(0,{e^{-\d_{k}}\rho_{k})}}\lesssim \sum_{j=1}^k \|F_{j}\|_{W_{j}}.
\ee
Let us define
$$s_{k}=\|\Omega_{k}-\Omega_{0}\|_{\bD(0,\rho_{k})},\qquad \e_{k}=\|F_{k}\|_{W_{k}},$$
\be\begin{cases}&\e_{k+1}\lesssim \biggl(\d_{k}^{-1}e^{-q_{n}\d_{k}/2}+(\rho_{k}\d_{k})^{-\bar a}{q_{n}^2}\e_{k}+{q_{n}^2}(\rho_{k}\d_{k})^{-2}s_{k}+{q_{n}^2}\d_{k}^{-1}\rho_{k}^{1/2}\biggr)\e_{k}\\
&s_{k+1}\lesssim \sum_{j=0}^{k}\e_{j}
\end{cases}\label{e6138}
\ee
as long as
$$\rho_{k}^{-\bar a}\d_{k}^{-\bar a}{q_{n}^2}\e_{k}<1.$$
Let $k^*$ be the largest integer for which the preceding sequences are defined and satisfy the stronger condition
\be\forall\ k<k^*,\quad P_{q_{n}}(k,\e_{k}):= \rho_{0}^{-1/5}\rho_{k}^{-\bar a}\d_{k}^{-\bar a}q_{n}^2\e_{k}<1.
\ee
From  (\ref{deltaketcante}) for any  $\s>0$ one has  $\d_{k}\gtrsim_{\s} (k+1)^{-(1+\s)}$. For $\th\in ]0,1/6[$ define
$$\mu_{0}=\frac{1}{6}-\th,\qquad \mu=\mu_{0}\frac{1}{1+\s}.$$
Since $\rho_{k}\gtrsim \rho_{0}$ one has 
$$\forall\ k< \min(k^*,\rho_{0}^{-\mu}),\quad \begin{cases}
& q_{n}\d_{k}\gtrsim \rho_{0}^{-1/6+\mu_{0}}=\rho_{0}^{-\th}\\
&(\rho_{k}\d_{k})^{-1} \lesssim \rho_{0}^{-1-\mu_{0}}=\rho_{0}^{-7/6+\th}\\
 &{q_{n}^2} \d_{k}^{-1}\rho_{k}^{1/2}\lesssim \rho_{0}^{-1/3+1/2-\mu_{0}}=\rho_{0}^\th. \end{cases}$$
thus (for $k<\rho_{0}^{-\mu}$, $\rho_{0}\ll_{\th} 1$),  $\d_{k}^{-1}e^{-q_{n}\d_{k}/2} \leq \rho_{0}^{1/5}$ and consequently,  if $\rho_{0}\ll_{\th} 1$,
\begin{align*}\forall\ k<\min(k^*,\rho_{0}^{-\mu}),\quad &\e_{k+1}\lesssim (\rho_{0}^{1/5}+\rho_{0}^{1/5}+\rho_{0}^{-1/3-7/3+2\th}\sum_{j=0}^k\e_{k}+\rho_{0}^\th)\e_{k}\\
&\leq \bar C_{\th}\rho_{0}^\th(1+\rho_{0}^{-3}\sum_{j=0}^k\e_{k})\e_{k}
\end{align*}
Since from (\ref{F0incond})
$\e_{0}\leq  \rho_{0}^{2\bar a+1}<\rho_{0}^{1/5+(7/6)\bar a+1/3 }$ we see that condition (\ref{condalphaepsilon}) of Lemma \ref{lem:E3} is satisfied (with $\a=\rho_{0}$, $\bar k_{\th,\a}=\rho_{0}^{-\mu}$) hence
\be \forall \ k\in [0,\rho_{0}^{-\mu}]\cap\N,\quad \e_{k}\leq (2\bar C_{\th}\rho_{0})^{\th k}\e_{0}.\ee
Now for any $0<\b \ll 1$, one can choose $\th$ and $\s$ so that $\rho_{0}^{-\mu}=q_{n}^{1-\b}$
and in particular taking $k=\bar k_{\b}=q_{n}^{1-\b}$ and using (\ref{e6138}) one gets for $q_{n}\gg_{\b} 1$
\be \e_{\bar k}\leq e^{-q_{n}^{1-\beta}},\qquad s_{\bar k}\leq 2\e_{0}.\label{EE4.56}\ee
(we assume $\bar a>10$).

We now define 
$$F^{BNF}_{q_{n}^{-1}}=F_{\bar k},\qquad \Omega^{BNF}_{q_{n}^{-1}}=\Omega_{\bar k},$$
and 
 $$g_{q_{n}^{-1}}^{BNF}=f_{Y_{1}^{Wh}}^{-1}\circ \cdots\circ  f_{Y_{\bar k}^{Wh}}^{-1}$$
 where $Y_{j}^{Wh}$ is a $C^2$ Whitney extension of $(Y_{j},e^{-\d_{j}/2}W_{j})$ given by Lemma \ref{lemma:2.3ee};
 one has
 \be \|g_{q_{n}^{-1}}^{BNF}-id\|_{C^1}\lesssim q_{n}^2 \sum_{k=0}^{\bar k}(2\bar C_{\th}\rho_{0})^{\th k}\e_{0}\rho_{0}^{-4}h^{-4}k^8\lesssim_{\th,h} q_{n}^{-(m-26)}.\label{EE4.57}\ee
 Inequalities (\ref{EE4.56}) and (\ref{EE4.57}) and the fact that $F_{\bar k}\in O^{\bar k+2m}(r)$ give the conclusion of the Proposition \ref{nprop:9.2}. Note that  (\ref{ee6.128}) is a consequence of $F^{BNF}_{q_{n}^{-1}}\in O^{q_{n}^{1-\b}}(z,w)$ and Remark \ref{rem:6.2}. For the last statement of the Proposition, we can choose $\Omega_{\bar k}=\Omega+\sum_{j=0}^{\bar k-1}M(F_{j})^{Wh}$ where $F_{J}^{Wh}$ is an $C^3$ Whitney extension of $(F_{j},W_{j})$ given by Lemma \ref{lemma:2.3ee}.
\hfill $\Box$

\subsection{Proof of Proposition \ref{BNFprop} (BNF (AA) or (CC) Case, $\omega_{0}$ Diophantine)}\label{appendixBNFprop}

The proof, that we mainly illustrate in the (AA)-Case, as well as  the notations,   are essentially the same as the ones of the proof of Proposition \ref{nprop:9.2} (see Section \ref{appendixnprop:9.2}, in particular, we use  the definitions (\ref{deltaketcante}), (\ref{deltaketc}) for $\d_{k}$, $\rho_{k}$, $W_{k}$) with the following differences:
\begin{itemize} \item we replace (\ref{e8.135}) by
$$ \rho_{0}= \rho^{(\tau+2)/\g},\qquad W_{0}=W_{h,\bD(0,\rho_{0})}$$
where $\g=1$ in the (AA)-case and $\g=1/2$ in the (CC)-case.
\item  since $\omega_{0}$ is Diophantine, we can and do solve instead of (\ref{cohomeqprop6.2}) the equation without truncation (using  Proposition \ref{prop:n5.5} with $N=\infty$,  $K=\kappa^{-1}$, {\it cf.} (\ref{omega0dioph})):
$$ -[Q_{0}]\cdot Y_{k}=F_{k}-\cM_{0}(F_{k}),\qquad \|Y_{k}\|_{e^{-\d_{k}/2}W_{k}}\lesssim \d_{k}^{-(1+\tau)}\|F_{k}\|_{W_{k}}.$$
\end{itemize}
Notice that both in the (AA) and (CC) cases
\be F_{0}=O^m(r),\qquad \|F_{0}\|_{W_{0}}\leq \rho_{0}^m.\label{incondF0bis}\ee
In place of (\ref{Fk+1}) we get (in the (AA)-case)
\be F_{k+1}=\dot\fO_{2}^{(\bar a)}(Y_{k},F_{k})+O\biggl(|\pa_{\th}Y_{k}|\times |\pa(\Omega_{k}-Q_{0})\circ r|\biggr).\label{Fk+1bis}
\ee
Since $F_{k},\ Y_{k}\in O^{k+m}(r)(\th)$ (see Remarks \ref{rem:2.3}, \ref{rem:5.1})  we have    $O(|\pa_{\th} Y_{k}|\times |\pa(\Omega_{k}-Q_{0})\circ r|)=O^{k+m+1}(r)(\theta)$ and  $\dot\fO_{2}^{(\bar a)}(Y_{k},F_{k})=O^{2k+2m-\bar a}(r)(\theta)$ ($\bar a$ from Proposition \ref{prop:n5.5}). As a consequence, since $2m\geq \bar a$ we see that $F_{k+1}=O^{k+1+m}(r)(\th)$.

From (\ref{Fk+1bis}) and the fact that ({\it cf.}  (\ref{eq:cauchy:derivativester}))
$$\|\pa_{\th}Y_{k}\|_{e^{-\d_{k}}W_{k}}\lesssim \d_{k}^{-1}\|Y_{k}\|_{e^{-\d_{k}/2}W_{k}},\qquad \|\pa(\Omega_{0}-Q_{0})\circ r\|_{W_{k}}\leq \rho_{k}
$$
hence
$$|\pa_{\th}Y_{k}|\times |\pa(\Omega_{0}-Q_{0})\circ r|\lesssim \d_{k}^{-(2+\tau)}\rho_{k}^{\g}\|F\|_{W_{k}}
$$
where $\g=1$ in the (AA)-case. A similar computation  ({\it cf.} (\ref{rho1/2})) shows that one can take $\g=1/2$ in the (CC)-case.
With the notations  $s_{k}=\|\Omega_{k}-\Omega_{0}\|_{\bD(0,\rho_{k})}$, $\e_{k}=\|F_{k}\|_{W_{k}}$, 
we then get
\be\begin{cases}&\e_{k+1}\lesssim \biggl((\rho_{k}\d_{k})^{-(1+\tau+\bar a)}\e_{k}+(\rho_{k}\d_{k})^{-(3+\tau)}s_{k}+\d_{k}^{-(2+\tau)}\rho_{k}^{\gamma}\biggr)\e_{k}\\
&s_{k+1}\lesssim \sum_{j=0}^{k}\e_{j}
\end{cases}\label{e6138bis}
\ee
provided 
\be (\rho_{k}\d_{k})^{-(1+\tau+\bar a)}\e_{k}<1.\label{e4.19a}\ee
Let $k^*$ be the largest integer for which the preceding sequences  are defined and satisfies 
\be \forall k< k_{*},\quad P_{\rho_{0}}(k,\e_{k}):=\rho_{0}^{-\gamma}(\rho_{k}\d_{k})^{-(1+\tau+\bar a)}\e_{k}<1;\label{e419}\ee
the condition involved in  (\ref{e419}) implies (\ref{e4.19a}).
From  (\ref{deltaketcante}) for any  $\s>0$ one has  $\d_{k}\gtrsim_{\s} (k+1)^{-(1+\s)}$. Fix $\th\in ]0,\gamma[$ and define
$$\mu=\frac{\gamma-\th}{2+\tau}\frac{1}{1+\s}.$$
Since $\rho_{k}\gtrsim \rho_{0}$ one has 
$$\forall\ k< \min(k^*,\rho_{0}^{-\mu}),\quad \begin{cases}&(\rho_{k}\d_{k})^{-1} \lesssim \rho_{0}^{-1-\mu(1+\s)}\\
 & \d_{k}^{-(2+\tau)}\rho_{k}^\gamma\lesssim \rho_{0}^{\gamma-(2+\tau)\mu(1+\s)}=\rho_{0}^\th. \end{cases}$$
 If we set 
 $$a=\biggl(1+\frac{\g-\th}{2+\tau}\biggr)(3+\tau)+\th\leq (4/3)(3+\tau)$$ we then get  
 using (\ref{e419}) and (\ref{e6138bis}) 
 \begin{align*}\forall\ k< \min(k^*,\rho_{0}^{-\mu}),\quad\e_{k+1}&\leq \bar C_{\s} \biggl(\rho_{0}^\gamma+\rho_{0}^{-a+\th}\sum_{j=0}^k\e_{j}+\rho_{0}^{\th}\biggr)\e_{k}\\
&\leq \bar C_{\s} \rho_{0}^{\th}\times \biggl(1+\rho_{0}^{-a}\sum_{j=0}^k\e_{j}\biggr)\e_{k}. 
\end{align*}
We now apply   Lemma \ref{lem:E3} with 
$\a=\rho_{0}$:
since  ($\bar a>10$) $$ \max\biggl(a, (1+\frac{\gamma-\th}{2+\tau})(1+\tau+\bar a)+1\biggr)\leq 2(\tau+\bar a+1)=m$$
condition (\ref{incondF0bis}) shows that 
(\ref{e419}) is satisfied with $\bar k=\min(k^*,\rho_{0}^{-\mu})$ as well as conditions (\ref{condalphaepsilon})
for $\rho_{0}{\mathop{\ll}}_{\th,\s}1$.
We thus get
 if $\bar k:=[\rho_{0}^{-\mu}]$, 
\be \forall \ k\in [0,\bar k]\cap\N,\quad \e_{k}\leq \rho_{0}^{\th k}\e_{0}.\ee

Since $\th$ and $\s$ can be taken arbitrarily close to 1, for any $0<\b\ll1$  one has for $\rho\ll_{\b} 1$ 
$$\e_{\bar k}\leq e^{-\rho^{-(1-\b)}},\qquad s_{\bar k}\leq 2\bar \e_{0}.$$
We  conclude like in the proof of Proposition \ref{nprop:9.2} (Section  \ref{appendixnprop:9.2}) by defining
$$F^{BNF}_{\rho}=F_{\bar k},\qquad \Omega^{BNF}_{\rho}=\Omega_{\bar k},\qquad  g_{\rho}^{BNF}=f_{Y_{1}^{Wh}}^{-1}\circ \cdots\circ  f_{Y_{\bar k}^{Wh}}^{-1}.$$
Note that  since $F_{\bar k}\in O^{\bar k+m}(r)$ one has $F_{\rho}^{BNF}\in O^{(1/\rho)^{1-\b}}(r)$ and   (\ref{ee6.128bis}) is a consequence of  Remark \ref{rem:6.4}.

\hfill $\Box$

\section{Resonant Normal Forms}\label{appendixE:RNF}

In this section we shall only consider the (AA)-Case.

Let  $c\in \bD(0,1)$, $\bar \rho>0$, $h>0$ and 
\be \Omega\in\ti \cO_{\s}(e^{h}\bD(c,\bar \rho))\  \textrm{and} \  F\in\cO_{\s}(e^hW_{h,\bD(c,\bar \rho)})\label{11.166}\ee
where $\Omega$ satisfies the twist condition
$$\forall\ r\in \R,\ A^{-1}\leq (2\pi)^{-1}\pa^2 \Omega(r)\leq A \ \textrm{and}\ \|(2\pi)^{-1}\Omega\|_{C^3}\leq B.
$$ 
Our aim in this section is to give an approximate  Normal Form for $\Phi_{\Omega}\circ f_{F}$ in a neighborhood of a $q$-resonant circle by which we mean  that 
for some  $(p,q)\in\Z\times \N^*$, $p\wedge q=1$
$$(2\pi)^{-1}\pa\Omega(c)=\frac{p}{q}.
$$
This Normal Form  is quite similar in spirit  (and in its construction) to the approximate BNF. It  is used in the paper in Sections \ref{sec:HJBNF} (approximate Hamilton-Jacobi Normal Form)  and \ref{sec:openingtheeyes} (creating hyperbolic periodic points).

As usual we define $\omega(c):=\frac{1}{2\pi}\pa\Omega(c)$.

\begin{prop}[$q$-resonant Normal Form]\label{prop:11.5}There exists a universal constant $\bar a_{4}\geq 10$ (not depending on $q$) such that,  if one has
\be \begin{cases} &\bar \rho<(Aq)^{-8}\\
&\|F\|_{W_{h, \bD(c,\bar \rho)}}<\bar \rho^{\ \bar a_{4}} 
\end{cases}\label{condF}
\ee
then the following holds:
There exist $\bar\Omega\in\ti \cO_{\s}(\bD(c,e^{-1/q}\bar\rho)) \cap \cT\cC(2A,2B)$,
$$
 \bar F^{res},F^{cor}\in\cO_{\s}(e^{-1/q}W_{h,\bD(c,\bar \rho)}),\qquad   g_{RNF}\in\ti{\rm Symp}_{ex,\s}(e^{-1/q}W_{h,\bD(c,\bar\rho)})$$ such that 
\be\begin{cases}& g_{RNF}^{-1}\circ \Phi_{\Omega}\circ  f_{F}\circ g_{RNF}= \Phi_{2\pi(p/q)r}\circ \Phi_{\bar\Omega}\circ  f_{\bar F^{res}}\circ f_{ F^{cor}}\\
&\bar F^{res}\ \textrm{ is}\  2\pi/q-\textrm{periodic},\qquad  \cM_{0}(\bar F^{res})=0,
\end{cases}
\label{10.162ter}
\ee
where 
\be\begin{cases}&\|\bar \Omega-(\Omega-2\pi(p/q)r)\|_{\bD(c,e^{-1/q}\bar\rho)}\lesssim \|F\|_{W_{h,\bD(c,\bar\rho)}}\\
&\|\bar F^{res}\|_{e^{-1/q}W_{h,\bD(c,\bar\rho)}}\lesssim \|F\|_{W_{h,\bD(c,\bar\rho)}}\\
&\|g_{RNF}-id\|_{C^1}\lesssim (q\bar \rho^{-2})^2\|F\|_{h,\bD(c,\bar\rho)}\leq  \bar \rho^{\hs \bar a_{4}-5}
\end{cases}\label{barFres}
\ee
and 
 \be \|F^{cor}\|_{e^{-1/q}W_{h,\bD(c,\bar \rho)}}\lesssim  \exp(-\bar \rho^{\hs-1/4}) \|F\|_{W_{h,\bD(c,\bar\rho)}}.\label{estFcor}\ee
\end{prop}
We give the proof of this Proposition in the next subsections.
\begin{rem}\label{remE1}The implicit constants in the symbol $\lesssim$  of the   preceding estimates  depend on $h$; if $h_{0}>0$, they can be bounded above by a constant $C_{h_{0}}$ whenever $h\geq h_{0}$. 
\end{rem}
\subsection{A preliminary Lemma}
\begin{lemma}\label{nlem:10.1}1) For any $(k,l)\in\Z^*\times \Z$ one has 
\be\begin{cases}&\textrm{either}\ q|k\ \textrm{and}\ p|l\\
&\textrm{or}\ |k\times \displaystyle\frac{p}{q}-l|\geq \frac{1}{q}.
\end{cases}
\ee
2) Let 
$$N=(q\bar \rho A)^{-1}.$$ For any $r\in \bD(c,\bar \rho)$ and any $(k,l)\in\N^*\times \Z$, $1\leq k\leq N$, which is not in $(q,p)\Z$   one has 
\be |k\omega(r)-l|\geq 1/(2q).\label{E332}
\ee
\end{lemma}
\begin{proof} 1) Indeed $|k(p/q)-l|=|kp-lq|/q$ and if the integer $kp-lq$ is 0 then $q|k$ and $p|l$.

2) We just notice that 
\begin{align*}|k\omega(r)-l|&\geq |k\omega(c)-l|-k|\omega(r)-\omega(c)|\\
&\geq (1/q)-N\|\pa \omega\|_{\bD(c,\bar \rho)}\bar \rho\\
&\geq 1/(2q).
\end{align*}
\end{proof}

Define 
$$T_{N}^{q-res}F=\sum_{\substack{k\in\Z\\ |k|<N \\ q|k}}\cM_{k}(F), \qquad T_{N}^{q-nr}F=T_{N}F-T_{N}^{q-res}F.$$
We shall often use in what follows the shortcuts $T^{res}_{N}$ and $T^{nr}_{N}$ for $T^{q-res}_{N}$, $T_{N}^{q-nr}$.

From (\ref{eq:5.69}), (\ref{n5.81}) we see that
\be T_{N}^{res}F\circ \phi_{r}^{2\pi/q}=T_{N}^{res}F\label{10.149}
\ee 
and
\be \|T^{res}_{N}F\|_{e^{-\d}h,\bD(c,\bar\rho)}\lesssim \d^{-1}\|F\|_{h,\bD(c,\bar\rho)}.
\ee
\begin{cor} \label{cor:10.3}For any  $F\in\cO_{\s}(W_{h,\bD(c,\bar\rho)})$, there exists  $Y\in\cO_{\s}(W_{h,\bD(c,\bar\rho)}^\Omega)$ such that $\cM_{0}(Y)=0$ and 
\be T_{N}^{nr}F=[\Omega]\cdot Y.\label{ncohoeq}
\ee
This $Y$ satisfies for any $0<\d<h$ 
\be \|Y\|_{e^{-\d}W_{h,\bD(c,\bar\rho)}}\lesssim q\d^{-1}\|F\|_{W_{h,\bD(c,\bar\rho)}}.\label{ne5.86b}
\ee
\end{cor}
\begin{proof}This is a simple adaptation of the proof of Proposition \ref{prop:n5.2} (the non resonance condition is replaced by (\ref{E332})).
\end{proof}

\subsection{Elimination of non-resonant terms}
\begin{prop}\label{propresNF10.3}There exists a universal constant $\bar a_{4}$ (not depending on $q$) such that if  $N=(q\bar \rho A)^{-1}$,
\be \begin{cases} &\bar \rho^{ 1/8}<(qA)^{-1}\\
&\|F\|_{W_{h, \bD(c,\bar \rho)}}<\bar \rho^{\ \hs\bar a_{4}} 
\end{cases}\label{condF}
\ee
then there exist  $F^{res}, F^{nr}\in\cO_{\s}(e^{-1/q}W_{h,\bD(0,\bar \rho)})$, $g\in\ti{\rm Symp}_{ex,\s}( e^{-1/q}W_{h,\bD(0,\bar \rho)})$ such that 
$$[e^{-1/q}W_{h,\bD(0,\bar \rho)}] \qquad g^{-1}\circ \Phi_{\Omega}\circ f_{F}\circ g=\Phi_{\Omega}\circ f_{ F^{res}+F^{nr}}
$$
\be F^{res}=T_{N}^{res}(F+O(q\bar\rho\|F\|_{W_{h,\bD(c,\bar \rho)}}))\label{Fres}\ee $F^{res}$ being $2\pi/q$-periodic 
and 
\be \|g-id\|_{C^1}\lesssim (q\bar \rho^{\hs -2})^2\|F\|_{h,\bD(c,\bar\rho)}\leq \bar \rho^{\hs \bar a_{4}-5 }\label{E339C2}
\ee
\be \|F^{nr}\|_{e^{-1/q}W_{h,\bD(c,\bar \rho)}}\lesssim q\exp(-\bar \rho^{\hs-1/3}) \|F\|_{h, \bD(c,\bar \rho)}.\label{E339C1}\ee
\end{prop}
\begin{proof} 
Note that we can assume, using Lemma \ref{lemma:2.3ee} that $F\in\ti \cO_{\s}(W_{h,\bD(c,\bar\rho)})$.
and satisfies 
\be \e_{0}:= \|F\|_{e^{-1/(10q)}W_{h,\bD(c,\bar\rho)}}\leq \bar \rho^{\hs \bar a_{4}},\qquad \|F\|_{C^3}\lesssim \bar \rho^{\hs \bar a_{4}-7} .\label{FF.480}\ee

Let $N=(q\bar \rho A)^{-1}$.
We define $W_{0}=e^{-1/(10q)}W_{h,\bD(c,\rho_{0})}$, $\rho_{0}=\bar \rho$ and for $k\geq 1$
\be\d_{k}=\frac{(5q)^{-1}}{(k+1)^{4/3}},\quad \rho_{k}=\exp(-\sum_{l=0}^{k-1}\d_{l})\bar \rho,\quad W_{k}=\exp(-\sum_{l=0}^{k-1}\d_{l})W_{0}\ee
and 
we  construct sequences 
 $Y_{k},F_{k}, F^{nr}_{k},F^{res}_{k}\in \cO_{\s}(W_{k})$
such that $F_{k}=F_{k}^{nr}+F_{k}^{res}$
\be F^{res}_0=T^{res}_{N}F,\qquad F^{nr}_{0}=F-F^{res}_{0},\qquad T_{N}^{res}F_{0}^{nr}=0\label{E340}\ee and  for $k\geq 0$
\be f_{Y_{k}}\circ\Phi_{\Omega}\circ  f_{ F^{nr}_{k}+F^{res}_{k}}\circ  f_{Y_{k}}^{-1}=\Phi_{\Omega}\circ  f_{F^{nr}_{k+1}+F^{res}_{k+1}}
\label{10.187}\ee
where for any $k$, 
$$F^{res}_{k}\circ \phi^{2\pi/q}_{J\nabla r}=F^{res}_{k}.$$
By  
Corollary \ref{cor:10.3} there exists  
$Y_{k}\in\cO_{\s}(e^{-\d_{k}/2}W_{k}^\Omega)$
 such that 
\be [\Omega]\cdot Y_{k}=-T_{N}^{nr} F^{nr}_{k},\qquad \|Y_{k}\|_{e^{-\d_{k}/2}W_{k}^\Omega}\lesssim q\d_{k}^{-1}\|F_{k}^{nr}\|_{W_{k}}.\label{10.156}\ee
Let $F_{k}:=F_{k}^{res}+F_{k}^{nr}$ and  compute using Proposition  \ref{lemma:6.4} 
\begin{align*}  f_{Y_{k}}\circ \Phi_{\Omega}\circ f_{F^{nr}_{k}+F^{res}_{k}}\circ  f_{Y_{k}}^{-1}&=\Phi_{\Omega}\circ f_{F^{nr}_{k}+F^{res}_{k}+[\Omega]\cdot Y_{k}+\|DF_{k}\|_{e^{-\d_{k}/2}W_{k}}\fO_{1}(Y_{k})}\\
&=\Phi_{\Omega}\circ f_{R_{N}F_{k}^{nr}+T^{res}_{N}F^{nr}_{k}+F_{k}^{res}+q\d_{k}^{-1}\|DF_{k}\|_{e^{-\d_{k}/2}W_{k}}\fO_{1}(F^{nr}_{k})}\\
&=\Phi_{\Omega}\circ f_{F_{k+1}^{nr}+F_{k+1}^{res}}
\end{align*}
with 
\be\begin{cases}&F^{nr}_{k+1}=R_{N}F_{k}^{nr}+ q( \rho_{k}\d_{k}^{-1})^2\|F_{k}\|_{W_{k}}\fO_{1}(F^{nr}_{k})\\
&F^{res}_{k+1}=T_{N}^{res}F^{nr}_{k}+F^{res}_{k}.
\end{cases}\label{10.157}
\ee
In particular since $F_{k+1}=F_{k+1}^{nr}+F_{k+1}^{res}=F_{k+1}^{nr}+T_{N}^{res}F^{nr}_{k}+F^{res}_{k}$
\be F_{k+1}=F_{k}+F_{k+1}^{nr}+T_{N}^{res}F_{k}^{nr}-F_{k}^{nr}.\label{10.158}
\ee

If we define  $\e^{*}_{k}=\| F^*_{k}\|_{W_{k}}$, $*=nr,res$, $\e_{k}=\|F_{k}\|_{W_{k}}$ we get from (\ref{E340}) and Lemma \ref{lemma:n5.1}
\be \e_{0}^{nr}\lesssim \d_{0}^{-1}\e_{0}\lesssim q\e_{0}\label{E345}
\ee
and from  (\ref{10.157}), (\ref{10.158}) and Lemma \ref{lemma:n5.1}
 that  for some $a>0$
\be\begin{cases}& \e^{nr}_{k+1}\lesssim \d_{k}^{-1}e^{-\d_{k}N/2}\e_{k}^{nr}+q \e_{k} ( \rho_{k}\d_{k})^{- a}\e^{nr}_{k}\\
&\e_{k+1}^{res}\lesssim \e^{res}_{k}+ \delta_{k}^{-1}\e^{nr}_{k}\\
& \e_{k+1}=\e_{k}+O(\e_{k+1}^{nr}+\d_{k}^{-1}\e_{k}^{nr} )
\end{cases}\label{11.176}
\ee
provided for some $a>0$
\be (\rho_{k}\d_{k})^{-a}\e_{k}^{nr}<1.\label{9.205}
\ee
From now on we define $\bar a_{4}=2a+2$ and we  assume that 
\be \e_{0}\leq \rho_{0}^{\bar a_{4}}, \qquad\bar a_{4}=2a+2 ;\label{choiceepsilon0}\ee
notice that this implies (see (\ref{E345}) and (\ref{FF.480}))
\be \e_{0}^{nr}\leq \rho_{0}^{2a}.\label{choiceepsilon0bis}
\ee
Let $k^*$ be the largest integer for which the sequences $\e_{k}^{nr},\e_{k}^{res},\e_{k}$ are defined. We notice that for $k<\min(k^*, \rho_{0}^{-1/3})$ one has from (\ref{condF})
$$(\rho_{k}\d_{k})^{-1}\lesssim \rho_{0}^{-1}q\rho_{0}^{-4/9}\lesssim A^{-1}\rho_{0}^{-1}\rho_{0}^{-1/8}\rho_{0}^{-4/9}\leq \rho_{0}^{-2}.
$$
Since
$\e_{k}\ =\e_{0}+O(\d_{k}^{-1}\sum_{j=0}^k\e_{k}^{nr})
$
we get that for $k+1\leq \rho_{0}^{-1/3}$, 
 $\e_{k}=\e_{0}+O(\rho_{0}^{-4/9}  \sum_{j=0}^k\e_{k}^{nr})$  hence if $k<\min(k^*, \rho_{0}^{-1/3})$ (recall $q\leq \rho_{0}^{-1/8}$)
$$q \e_{k} ( \rho_{k}\d_{k})^{-a}\e^{nr}_{k}\leq C\biggl(\rho_{0}^{-(2a+1)}\e_{0}+\rho_{0}^{-(2 a+2)} \sum_{j=0}^k\e_{k}^{nr}\biggr)\e_{k}^{nr}.$$

On the other hand, from (\ref{condF}),   $q^{-1} N=q^{-2}\rho_{0}^{-1} A^{-1}\geq A^{-1}\rho_{0}^{-3/4}$ hence, if $k+1\leq \rho_{0}^{-1/3}$  one has $\d_{k}N=q^{-1} N/(k+1)^{4/3}\geq\rho_{0}^{-3/4}\rho_{0}^{4/9}=\rho_{0}^{-11/36}$ and thus $\d_{k}^{-1}e^{-\d_{k}N}\leq \rho_{0}$ if $\rho_{0}$ is small enough. 
The outcome of this is that  for $k+1\leq \rho_{0}^{-1/3}$ one has (we use condition (\ref{choiceepsilon0}))
$$\e_{k+1}^{nr}\leq C\rho_{0}\biggl(1+\rho_{0}^{-(2a+3)} \sum_{j=0}^k\e_{k}^{nr}\biggr)\e_{k}^{nr}.
$$
Since for $\rho_{0}\ll 1$ one has ({\it cf.} (\ref{choiceepsilon0bis}))
\be \e_{0}^{nr}\leq \rho_{0}^{2a}\leq  (\rho_{0}\d_{0})^{a}\label{condepsilon0A}
\ee
we can thus apply   Lemma \ref{lem:E3} with $\a=\rho_{0}$, to get 
\be k^*\geq \rho_{0}^{-1/3},\qquad  \forall0\leq  k\leq k_{*},\quad \e_{k}^{nr}\leq (2C\rho_{0})^{k}\e^{nr}_{0}\leq e^{-k}q\e_{0}\label{11.180}.
\ee
We now  set
$$F^{res}=F^{res}_{k^*}, \qquad F^{nr}=F_{k^*}^{nr}, \qquad g=f_{Y_{1}^{Wh}}^{-1}\circ\cdots\circ f_{Y_{k^*-1} ^{Wh}}$$
where $Y_{j}^{Wh}$ it a $C^2$ Whitney extension of $(Y_{j},e^{-\d_{j}/2}W_{j})$ given by Lemma \ref{lemma:2.3ee}
and observe that conclusion (\ref{E339C1}) is satisfied since 
 $$e^{-1/q}W_{h,\bD(c,\bar\rho)}\subset W_{k^*},\qquad \|F^{nr}\|_{e^{-1/q}W_{h,\bD(c,\bar\rho)}}\lesssim e^{-\rho_{0}^{-1/3}}q\e_{0}.$$
To check (\ref{E339C2}) we just notice that from (\ref{10.156})
 $$\|g-id\|_{C^1}\lesssim_{h} q^2\rho_{0}^{-4}\e_{0}.$$
Finally, since $F^{res}=T_{N}^{res}F+T_{N}^{res}(\sum_{k=0}^{k_{*}}F^{nr}_{k})$ and   $T_{N}^{res}F_{0}^{nr}=0$ ({\it cf.} (\ref{E340})) one has from the inequality $\e_{k}^{nr}\leq (2C\rho_{0})^{k}\e^{nr}_{0}\lesssim \rho_{0}e^{-(k-1)}\e_{0}$ ($\rho_{0}\ll 1$, $k\geq 1$)
$$F^{res}=T_{N}^{res}F+T_{N}^{res}(\sum_{k=1}^{k_{*}}F^{nr}_{k}),
\qquad  \|\sum_{k=1}^{k_{*}}F^{nr}_{k}\|_{W_{k^*}} \leq \sum_{k=1}^{k^*}\e_{k}^{nr}\lesssim q\rho_{0}\e_{0}
$$
which gives  conclusion (\ref{Fres}):
$$F^{res}=T_{N}^{res}\biggl(F+O(q\bar \rho\e_{0})\biggr).
$$

\end{proof}

\subsection{Proof of Proposition \ref{prop:11.5}}\label{sec:RNF}
We apply Proposition \ref{propresNF10.3} and we write using Lemma \ref{lemma:6.2.bis}
\begin{align*}\Phi_{\Omega}\circ f_{F^{nr}+F^{res}}&= \Phi_{2\pi(p/q)(r-c)}\circ  \Phi_{\Omega-2\pi(p/q)(r-c)}\circ f_{{F^{res}}}\circ  f_{F^{nr}+\|DF^{res}\|_{W_{h,U}}\dot\fO_{1}(F^{nr})}\\
&= \Phi_{2\pi(p/q)(r-c)}\circ  \Phi_{\Omega-2\pi(p/q)(r-c)}\circ f_{{F^{res}}}\circ  f_{F^{cor}}
\end{align*}
with 
$$\|F^{cor}\|_{e^{-1/q}W_{h,\bD(c,\bar \rho)}}\lesssim  \exp(-\bar \rho^{\hs -1/4}) \|F\|_{W_{h, \bD(c,\bar \rho)}}$$
provided for some $a>0$
$$\bar \rho^{\ -\bar a}\|F^{nr}\|_{W_{h,\bD(0,\bar \rho)}}<1.$$
The inequality (\ref{E339C1})
  shows that this last condition is satisfied if $\bar \rho\ll 1$.

We now observe that
$$f_{F^{res}}= \Phi_{\cM_{0}(F)}\circ f_{F^{res}-\cM_{0}(F)}$$
and that 
$$\Phi_{\Omega-2\pi(p/q)(r-c)}\circ \Phi_{\cM_{0}(F)}=\Phi_{\Omega-2\pi(p/q)(r-c)+\cM_{0}(F)}.
$$
Note that since we have assumed that $F\in \ti\cO_{\s}(e^{-1/(10q)}W_{h,\bD(c,\bar\rho)})$ satisfies (\ref{FF.480}) we have $\bar \Omega\in \ti\cO_{\s}(e^{-1/q}W_{h,\bD(c,\bar\rho)})\cap\cT\cC(2A,2B)$ and the first inequality of (\ref{barFres}) is satisfied.
\hfill $\Box$

\begin{rem}Notice that from the first inequality of (\ref{n5.81}) in  Lemma \ref{lemma:n5.1}
\be\|F^{res}\|_{W_{e^{-1/q}h,\bD(c,e^{-1/q}\bar \rho)}}\lesssim \|F\|_{W_{h,\bD(c,\bar \rho)}}.\label{11.181}
\ee
\end{rem}

\medskip

\section{Approximations by vector fields} \label{sec:appendixG}
The main result of this Section is the following proposition on the  approximation of an exact symplectic diffeomorphism close to an integrable one by a vector field.
\begin{prop}\label{lemma:vf2} There exists a constant $\bar C>0$ for which the following holds. Let $0<\rho<1$,  $F\in \cO_{\s}(\T_{h}\times D(0,\rho))$ and $\Omega\in \cO_{\s}(\bD(0,\rho))$, $\Omega(r)=O(r^2)$. If $\rho>0$ is small enough, $h\gtrsim \rho^{1/3}$ and
\be\bar C\times (\rho h)^{-9}\|F\|_{h,\rho}<1\label{condSG}
\ee
 then, there exist $\Pi\in\cO_{\s}(\bD(0,\rho/2))$, $,A_{3}(F)\in \cO_{\s}(\T_{h/2}\times D(0,\rho/2))$ such that 
\be \Phi_{\Omega}\circ f_{F}=\Phi_{\Pi}\circ f_{A_{3}(F)}
\ee
with 
\begin{align}\Pi&=\Omega+F\circ\Phi_{-\Omega/2}+O(\rho^{1/4}\| F\|_{h,\rho})\\
&=\Omega+F+O(\rho^{1/4} \|F\|_{h,\rho})
\end{align} and 
\be \|A_{3}(F)\|_{h/2,\rho/2}< \exp(-\rho^{-1/4})\|F\|_{h,\rho}.
\ee
\end{prop}
The proof of this proposition is given in Subsection \ref{sec:G2}.

\subsection{Auxiliary result}
\begin{prop}\label{lemma:approxvfbis} Let $\rho>0$, $\Omega(r)=O(r^2)$, $\Omega\in \cO_{\s}(\bD(0,\rho))$, $F,G\in \cO_{\s}(\T_{h}\times\bD(0,\rho))$ such that ($C$ some universal constant)
\be  C\times  (\rho\d)^{-4}(\|F\|_{h,\rho}+\|G\|_{h,\rho})<1.\label{smallassump}\ee Then for any $h/2>\d\gtrsim \rho^{1/3}$, there exists $A(F,G)\in\cO_{\s}(e^{-\d}(\T_{h}\times\bD(0,\rho)))$ such that 
\be \Phi_{\Omega+F+G}=\Phi_{\Omega+F}\circ\Phi_{G\circ\Phi_{\Omega/2}}\circ f_{A(F,G)}\label{Dconh}
\ee
with 
\be \|A(F,G)\|_{h-\delta/2,e^{-\d}\rho}\lesssim  \biggl((\rho\d)^{-4}(\|F\|_{h,\rho} +\|G\|_{h,\rho}) +\rho\d^{-3}\biggr)\|G\|_{h,\rho}\label{estAFG}
\ee
\end{prop}
\begin{proof}To simplify the notations we denote $W=W_{h,\bD(0,\rho)}$ and we assume that $\omega(r):=\nabla\Omega(r)$, $\omega(0)=0$ satisfies
$$\omega(r)=r+O(r^2).$$
 If 
$$(\d\rho)^{-2}\max(\|F\|_{h,\rho},\|G\|_{h,\rho})<1
$$
the images of the domain $e^{-2\d}W$ by the flows  $\Phi^t_{\Omega},\Phi^t_{\Omega+F}, \Phi^t_{\Omega+F+G}$, $0\leq t\leq 1$, are contained in $e^{-\d}W$.

Let us denote
$$\s:=\max(\|DF\|_{h,\d},\|DG\|_{h,\d})
$$
and  for $x=(\th,r)\in e^{-2\d}W$ and $t\in [-1,1]$
$$\Delta(t,x)=\Phi^t_{\Omega+F+G}(x)- \Phi^t_{\Omega+F}(x).
$$
By classical theorems on ODE's for $t\in [-1,1]$
$$\Delta(t,\cdot)=O(\s),\qquad \Phi^t_{\Omega+F}-\Phi^t_{\Omega}=O(\s).$$
On the other hand one has 
\begin{align}\frac{d}{dt}\Delta(t,x)&=J\nabla (\Omega+F+G)\circ \Phi^{t}_{\Omega+F+G}(x)-J\nabla (\Omega+F)\circ \Phi^t_{\Omega+F}(x)\label{D427}\\
&=(I)(t,x)+(II)(t,x)+(III)(t,x)\notag
\end{align}
with
$$(I)(t,x)=J\nabla\Omega\circ \Phi^t_{\Omega+F+G}(x)-J\nabla\Omega\circ \Phi^t_{\Omega+F}(x)
$$
$$(II)(t,x)=J\nabla F\circ \Phi^t_{\Omega+F+G}(x)-J\nabla F\circ \Phi^t_{\Omega+F}(x)
$$
$$(III)(t,x)=J\nabla G\circ \Phi^t_{\Omega+F+G}(x).
$$
Since $\Phi^t_{\Omega+F+G}(x)=\Phi^t_{\Omega+F}(x)+(\Delta_{\th}(t,x),\Delta_{r}(t,x))$ and $\Phi_{\Omega+F}-\Phi_{\Omega}=O(\s)$ one has (note that $r\circ\Phi^t_{\Omega+F}=r+O(\s)$)
\begin{align}(I)(t,x)&=\bm \omega(r\circ\Phi^t_{\Omega+F}(x)+\D_{r}(t,x))-\omega(r\circ\Phi_{\Omega+F}^t(x))\\ 0\em\notag\\
&=\bm\pa\omega(r)\Delta_{r}(t,x)+O(\s|\D_{r}(t,x)|)\\0\em.\label{eqI}
\end{align}
and 
\be |(II)(t,x)|=O(|D^2F||\D(t,x)|).\label{eqII}
\ee
We have using the fact that $\Phi^t_{\Omega+F+G}-\Phi_{\Omega}=O(\s)$ and $\omega(r)=O(r)$
\begin{align}(III)(t,x)&=J\nabla G(\th+t\omega(r),r)+O(\e\|D^2G\|) \notag\\
&=\bm 
\pa_{r}G(\th,r)+t\omega(r)\pa^2_{\th r}G(\th,r)+O( \rho^2\|\pa^2_{\th}\pa_{r}G\|)\\
-\pa_{\th}G(\th,r)-t\omega(r)\pa^2_{\th}G(\th,r)+O(\rho^2\|\pa^3_{\th}G\|) \em+O(\s\|D^2G\|).\label{eqIII}
\end{align}
Summing (\ref{eqI}), (\ref{eqII}), (\ref{eqIII}) and integrating (\ref{D427}) gives 
\begin{multline}\bm\Delta_{\th}(t,x)\\ \D_{r}(t,x)\em =\bm\pa\omega(r)\int_{0}^t\Delta_{r}(s,x)ds\\ 0\em+\bm 
t\pa_{r}G(\th,r)+(t^2/2)\omega(r)\pa^2_{\th r}G(\th,r)\\
-t\pa_{\th}G(\th,r)-(t^2/2)\omega(r)\pa^2_{\th}G(\th,r)\em \\ +O(\e+|D^2F|) \int_{0}^t|\D(s,x)|ds +A_{1} \label{D429b}
\end{multline}
with 
$$A_{1}=O(\s\|D^2G\|)+O(\rho^2\|D\pa^2_{\th}G\|).$$
\begin{lemma}One has
$$|\De(t,x)|\leq A_{2}:=O(\|DG\|+\rho\|D\pa_{\th}G\|)+O(\s\|D^2G\|)+O(\rho^2\|D\pa^2_{\th}G\|).$$
\end{lemma}
\begin{proof}
From (\ref{D429b}) and the fact that $\pa\omega(r)\asymp 1$
$$ |\Delta(t,x)|\leq C(1+\e+\|D^2F\|_{h,\rho})\int_{0}^t|\Delta(s,x)|ds+O(\|DG\|+\rho\|D\pa_{\th}G\|)+A_{1}
$$
and  we conclude by Gr\"onwall inequality.
\end{proof}
Looking at the second component of (\ref{D429b}) gives ($\omega(r)=O(r)$)
$$\Delta_{r}(t,x)=-t\pa_{\th}G(\th,r)+O(\rho\|\pa^2_{\th}G\|)+O((\e+\|D^2F\|)A_{2})+A_{1}
$$
hence (integrating again and putting the result in (\ref{D429b}))
\begin{multline}\bm\Delta_{\th}(t,x)\\ \D_{r}(t,x)\em =\bm -\pa\omega(r)(t^2/2)\pa_{\th}G(\th,r)+
t\pa_{r}G(\th,r)+(t^2/2)\omega(r)\pa^2_{\th r}G(\th,r)\\
-t\pa_{\th}G(\th,r)-(t^2/2)\omega(r)\pa^2_{\th}G(\th,r)\em \\ +O(A_{3}) \label{D429}
\end{multline}
with 
$$A_{3}=O(\rho\|\pa^2_{\th}G\|)+O((\e+\|D^2F\|)A_{2})+A_{1}.$$
Taking $t=1$ gives 
\begin{multline*}\Phi_{\Omega+F+G}(x)= \Phi_{\Omega+F}(x)+\bm -(\pa\omega(r)/2)\pa_{\th}G(\th,r)+
\pa_{r}G(\th,r)+(\omega(r)/2)\pa^2_{\th r}G(\th,r)\\
-\pa_{\th}G(\th,r)-(\omega(r)/2)\pa^2_{\th}G(\th,r)\em \\ +O(A_{3}).
\end{multline*}
On the other hand 
$$\bm \pa_{r}(G(\th-\omega(r)/2,r)) \\  -\pa_{\th}(G(\th-\omega(r)/2,r) )\em=\bm  -(\pa\omega(r)/2)\pa_{\th}G(\th-\omega(r)/2,r)+\pa_{r}G(\th-\omega(r)/2,r)\\ -\pa_{\th}G(\th-\omega(r)/2,r)\em$$
hence
$$J\nabla (G\circ \Phi_{-\Omega/2})\circ \Phi_{\Omega}=\bm  -(\pa\omega(r)/2)\pa_{\th}G(\th+\omega(r)/2,r)+\pa_{r}G(\th+\omega(r)/2,r)\\ -\pa_{\th}G(\th+\omega(r)/2,r)\em$$
and from Taylor Formula and the fact that $\omega(r)=O(r)$
$$\Phi_{\Omega+F+G}(x)- \Phi_{\Omega+F}(x)=J\nabla (G\circ \Phi_{-\Omega/2})\circ \Phi_{\Omega}+O(\rho\|\pa^2_{\th}G\|)+O(A_{3}).
$$
Since $\Phi_{\Omega+F}=\Phi_{\Omega}+O(\s)$, this means that 
$$\Phi_{\Omega+F+G}=(id+J\nabla (G\circ \Phi_{-\Omega/2})\circ \Phi_{\Omega+F}+O(\s \|D^2G\|)+O(A_{3})$$
thus 
$$\Phi_{\Omega+F+G}=\Phi_{G\circ \Phi_{-\Omega/2}}\circ \Phi_{\Omega+F}+O(A_{3})$$
or
$$\Phi_{\Omega+F+G}=f_{O(A_{3})}\circ \Phi_{G\circ \Phi_{-\Omega/2}}\circ \Phi_{\Omega+F}$$
with
\begin{align*}A_{3}&\lesssim ((\rho\d)^{-2}\s+\rho^2(\rho\d)^{-1}\d^{-2}+(\s+(\rho\d)^{-2}\s)((\rho\d)^{-1}+\rho(\rho\d)^{-1} \d^{-1})+\rho\d^{-2})\|G\|_{h,\rho}\\
&\lesssim ((\rho\d)^{-3}\s+\rho\d^{-3})\|G\|_{h,\rho}\\
&\lesssim \biggl((\rho\d)^{-4}(\|F\|_{h,\rho} +\|G\|_{h,\rho}) +\rho\d^{-3}\biggr)\|G\|_{h,\rho}
\end{align*}
provided 
$$(\rho\d)^{-4}(\|F\|_{h,\rho}+\|G\|_{h,\d})<1,\qquad \rho\d^{-3}<1.$$
To conclude we observe that if we apply the preceding formula with $-\Omega$ and $-F$  instead of $\Omega$, $F$
$$ \Phi_{-\Omega-F-G}= f_{O(A_{3})}\circ \Phi_{-G\circ \Phi_{\Omega/2}}\circ \Phi_{-\Omega-F}$$
and inverting
$$\Phi_{\Omega+F+G}=\Phi_{\Omega+F}\circ \Phi_{G\circ\Phi_{\Omega/2}}+f_{O(A_{3})}.$$
\end{proof}

\begin{cor}\label{lemma:approxvf}Under the same conditions of Proposition \ref{lemma:approxvf} one has
$$\Phi_{\Omega+F}\circ f_{G}=\Phi_{\Omega+F+G\circ\Phi_{-\Omega/2}}\circ f_{A_{2}(F,G)}$$ with 
\be \|A_{2}(F,G)\|_{h-\delta/2,e^{-\d}\rho}\lesssim  \biggl((\rho\d)^{-4}(\|F\|_{h,\rho} +\|G\|_{h,\rho}) +\rho\d^{-3}\biggr)\|G\|_{h,\rho}.\label{estAFGbis}
\ee
\end{cor}
\begin{proof}If we apply (\ref{Dconh}) with $G\circ\Phi_{-\Omega/2}$ instead of $G$ we get 
$$\Phi_{\Omega+F+G\circ \Phi_{-\Omega/2}}=\Phi_{\Omega+F}\circ\Phi_{G}\circ f_{A(F,G\circ \Phi_{-\Omega/2})}$$
hence
\begin{align*}\Phi_{\Omega+F}\circ f_{G}&=\Phi_{\Omega+F+G\circ\Phi_{-\Omega/2}}\circ  f_{A(F,G\circ \Phi_{-\Omega/2})}^{-1} \circ \Phi_{G}^{-1}\circ f_{G} \\
&=\Phi_{\Omega+F+G\circ\Phi_{-\Omega/2}}\circ f_{A_{2}(F,G)}
\end{align*}
where $A_{2}(F,G)=A(F,G\circ \Phi_{-\Omega/2})+O(|DG||D^2G|)$ satisfies (\ref{estAFGbis}) ({\it cf.} (\ref{estAFG}
)).
\end{proof}
\subsection{Proof of Proposition \ref{lemma:vf2}}\label{sec:G2}
Let $\d_{k}=c/(k+1)^{3/2}$, $h_{k}=h-\d_{k}/2$, $\rho_{0}=(3/4)\rho$, $\rho_{k}=e^{-\d_{k}}\rho$ and $c$ chosen such that $h_{k}\geq h/2$, $\rho_{k}\geq \rho/2$ for all $k\in\N$.
Using Corollary \ref{lemma:approxvf} we construct sequences $S_{k},G_{k}$ such that $S_{0}=0$, $G_{0}=F$
\be \Phi_{\Omega+S_{k}}\circ f_{G_{k}}= \Phi_{\Omega+S_{k+1}}\circ f_{G_{k+1}}
\ee
\be \begin{cases}&S_{k+1}=S_{k}+G_{k}\circ\Phi_{-\Omega/2}\\
& G_{k+1}=A_{2}(S_{k},G_{k})\end{cases}
\ee
with 
$$\|S_{k+1}\|_{h_{k+1},\rho_{k+1}}\leq \|S_{k}\|_{h_{k},\rho_{k}}+\|G_{k}\|_{h_{k},\rho_{k}}
$$
and
\be \|G_{k+1}\|_{h_{k+1},\rho_{k+1}}\lesssim  \rho_{k}\delta_{k}^{-3}\|G_{k}\|_{h_{k},\rho_{k}}+(\rho_{k}\d_{k})^{-4}(\|S_{k}\|_{h_{k},\rho_{k}}+\|G_{k}\|_{h_{k},\rho_{k}})\|G_{k}\|_{h_{k},\rho_{k}}
\ee
as long as
$$ (\rho_{k}\d_{k})^{-4}(\|S_{k}\|_{h_{k},\rho_{k}}+\|G_{k}\|_{h_{k},\rho_{k}})<1.
$$
With $\e_{n}=\|G_{n}\|_{h_{n},\rho_{n}}$  and $\s_{n}:=\|S_{n}\|_{h_{n},\rho_{n}}$ we have ($s_{0}=0$)
\begin{align} \e_{k+1}&\leq C (\rho_{k} \d_{k}^{-3}+(\rho_{k}\d_{k})^{-4}\sum_{j=0}^k\e_{j})\e_{k}\\
\s_{k+1}&\leq\s_{k}+O(\e_{k})
\end{align}
as long as $(\rho_{k}\d_{k})^{-4}(s_{k}+\s_{k})<1$.

Let $k^*$ be the largest integer for which these sequences are defined. 
We observe that for $k<\min(k^*,\rho^{1/4})$ one has $(\rho_{k}\d_{k})^{-1}\leq \rho^{-2} $ and $\rho_{k}\d_{k}^{-3}\leq \rho_{k}^{1-3/4}=\rho^{1/4}$; hence, if $\bar k=\min(k^*,\rho^{-1/4})$ one has
$$\forall \ k<\bar k,\ \quad \e_{k+1}\leq C\rho^{1/4}(1+\rho^{-9}\sum_{j=0}^k\e_{j})\e_{k}.$$
We are in position to apply Lemma \ref{lem:E3} with $\a=\rho$, $\th=1/4$, $a=9$: since condition (\ref{condalphaepsilon}) is satisfied ({\it cf.}(\ref{condSG})) one has
$$k^*\geq \rho^{1/4},\qquad \forall\ k\leq k^*,\ \e_{k}\leq (2C\rho)^{k/4}\e_{0}
$$
and also 
$$s_{k}-\e_{0}\leq \sum_{j=1}^k\e_{j}\lesssim \rho^{1/4}\e_{0}.$$
To conclude the proof we set 
$$\Pi=S_{\bar k},\qquad A_{3}(F)=F_{\bar k}.$$
\hfill $\Box$

\section{Adapted KAM domains: Lemmas}\label{sec:appendixholes}
\subsection{Proof of Lemma \ref{aalemma:10.1}} \label{sec:appendixholes1}
From (\ref{formUn}) and the definition of $i_{-}(\rho)$
 for every $(k,l)\in E_{i_{-}(\rho)-1}$, $0<k<N_{i_{-}(\rho) -1}$, $0\leq|l|\leq N_{i_{-}(\rho)-1 }$ one has 
$$\bD(c_{l/k}^{(i_{-}(\rho)-1)},K_{i_{-}(\rho)-{1}}^{-1})\cap \bD(0,2\rho)=\emptyset,$$ 
hence 
$|c_{l/k}^{(i_{-}(\rho)-1)}|> \rho$. Since $\omega_{i_{-}(\rho)-1}^{}(c_{l/k}^{(i_{-}(\rho) -1)})=l/k$ we deduce from the fact that $\Omega_{i}$ satisfies an $(2A,2B)$-twist condition (\ref{7.139twist})  that 
$|(l/k)-\omega_{0}|= |\omega(c_{l/k}^{(i_{-}(\rho)-1)})-\omega(0)|\geq (2A)^{-1}\rho$. 
By Dirichlet Approximation  Theorem for any $L$ there exist $k,l\in\Z$, $0<|k|\leq L$ such that $|\omega_{0}-(l/k)|\leq 1/L^2$.   In particular if one chooses $L=N_{i_{-}(\rho)-1}-1\asymp N_{i_{-}(\rho)}$ one gets
$$N_{i_{-}(\rho)}^{-2}\gtrsim \rho$$
which proves  the  inequality of the RHS of (\ref{e10.204}).

Let us prove the second inequality of (\ref{e10.204}). By definition of $i_{-}(\rho)$ there exists $(l,k)\in\Z^2$, $0<k<N_{i_{-}(\rho)}$, $|l|\leq N_{i_{-}(\rho)}$ such that ({\it cf.} (\ref{formUn}))
$$\bD(c_{l/k}^{(i_{-}(\rho))},2K_{i_{-}(\rho)}^{-1})\cap \bD(0,2\rho)\ne \emptyset.$$
 In particular ({\it cf.} (\ref{requirements6.76})) $|c_{l/k}^{(i_{-}(\rho))}|\leq 3\rho$.
Since $\omega_{0}\in DC(\kappa,\tau)$, $|\omega(0)-(k/l)|\geq \kappa/k^{1+\tau}$ and from (\ref{7.114bis}) $|\omega_{i_{-}(\rho)}^{}(0)-\omega_{i_{-}(\rho)}^{}(c_{l/k}^{(i_{-}(\rho))})|\geq \kappa/k^{1+\tau}-2\bar\e^{1/2}$; by the twist condition $6A\rho\geq 2A|c_{l/k}^{(i_{-}(\rho))}|\geq \kappa N_{i_{-}(\rho)}^{-(1+\tau)}-\rho^2$
hence $$\rho\gtrsim   N_{i_{-}(\rho)}^{-(1+\tau)}
$$
which shows that the LHS of (\ref{e10.204}) holds.

Estimates (\ref{eq:5.101}), (\ref{502}) are then immediate.
\ \hfill $\Box$

\subsection{Proof of Items \ref{ab1}, \ref{ab2}, \ref{ab4} of Proposition \ref{prop:a10.4}} \label{sec:appendixholes2}
Recall  that from (\ref{abc10.207})
\be\begin{cases} &U_{i}^{((3/2)\rho)}=\bD(0,(3/2)\rho)\setminus\bigcup_{j=1}^{i-1}\bigcup_{(k,l)\in E_{j} }\bD(c_{l/k}^{(j)},s_{j,i-1}K_{j}^{-1}),\label{e10.207}\\
&s_{j,i-1}=e^{\sum_{m=j}^{i-1}{\d_{m}}}\in [1,2]\end{cases}\ee
where $E_{j}\subset\{(k,l)\in\Z^2,\ 0<k<N_{j},\ 0<|l|\leq N_{j}\}$, $\omega_{j}^{}(c_{l/k}^{(j)}) =l/k$.
In particular, any  $D\in\cD(U_{i})$ is of the form $D=\bD(c_{l/k}^{(j)},s_{j,i-1}K_{j}^{-1})$, where $j\leq i-1$,   $(k,l)\in E_{j}$.
 \begin{lemma}\label{ji-rho} If $D\in\cD_{(3/2)\rho}(U_{i})$ then  $j\geq i_{-}(\rho)$.
\end{lemma}
\begin{proof} Since $\bD(0,2\rho)=\bD(0,2\rho)\cap U_{i_{-}(\rho)}$, from (\ref{e10.207}) for all $j\leq i_{-}(\rho)-1$, $(k,l)\in E_{j}$ one has $|c^{(j)}_{l/k}|\geq 2\rho+K_{j}^{-1}$.  On the other hand, if $D\in\cD(U_{i})$ is of the form $D=\bD(c_{l/k}^{(j)},s_{j,i-1}K_{j}^{-1})$, where $j\leq i-1$, $(k,l)\in E_{j}$ and intersects $\bD(0,(3/2)\rho)$ one has $|c_{l/k}^{(j)}|\leq (3/2)\rho+ 2K_{j}^{-1}<2\rho+K_{j}^{-1}$ hence  $j\geq i_{-}(\rho)$. 
\end{proof}
From (\ref{formUn}) and Lemma \ref{ji-rho} we can thus write
\be\begin{cases} &U_{i}^{((3/2)\rho)}=\bD(0,(3/2)\rho)\setminus\bigcup_{j=i_{-}(\rho)}^{i-1}\bigcup_{(k,l)\in E_{j} }\bD(c_{l/k}^{(j)},s_{j,i-1}K_{j}^{-1}),\label{ee10.207}\\
&s_{j,i-1}=e^{\sum_{m=j}^{i-1}{\d_{m}}}\in [1,2].\end{cases}\ee
We define 
$$Q_{i}=\bigcup_{j=i_{-}(\rho)}^{i-1}\{l/k,\ (k,l)\in E_{j}\}$$
and for $t\in Q_{i}$
\begin{align*}&j(t,i)=\min\{j:\ j\in \N\cap[i_{-}(\rho),i-1],\ (k,l)\in E_{j}\ \textrm{and}\ l/k=t \}\\ 
&c(t,i)=c^{ ( j ( t , i ) ) }_{t},\qquad s(t,i)=s_{ j ( t , i ) , i-1 }. \end{align*}
Define for $i_{-}(\rho)\leq j\leq i\leq i_{+}(\rho)$, 
$$\kappa_{j,i}=s_{j,i-1}K_{j}^{-1}.$$
We observe that from the inequality $N_{i_{+}(\rho)} \leq N_{i_{-}(\rho)}^2$ for any  $i_{-}(\rho)\leq j\leq i\leq i_{+}(\rho)$, $i_{-}(\rho)\leq j'\leq i'\leq i_{+}(\rho)$   one has 
\footnote{This is clear if $j\ne j'$; if $j=j'$ observe that if $i\ne i'$, $\bar \e^{\hs 1/2}_{i_{-}(\rho) }\ll |s_{j,i-1}-s_{j,i'-1}| K_{j}^{-1}$.
}
$$\kappa_{j,i}+\kappa_{j',i'}\ll N_{\max(j,j')}^{-2},\qquad \bar \e^{\hs 1/2}_{\min(j,j')}\ll |\kappa_{j,i}-\kappa_{j',i'}|
$$
hence, from  Lemma \ref{cor:6.5},  Item (\ref{ni2}) for $(k,l)\in E_{j}$, $(k',l')\in E_{j'}$ one has if $\kappa_{j,i}\leq \kappa_{j',i'}$
\be\begin{cases}&\textrm{either}\ l/k\ne l'/k'\quad \textrm{and}\quad  \bD(c_{l/k}^{(j)},\kappa_{j,i})\cap \bD(c_{l'/k'}^{(j')},\kappa_{j',i'})=\emptyset\\
&\textrm{or}\  l/k=l'/k' \quad \textrm{and}\quad \bD(c_{l/k}^{(j)},\kappa_{j,i})\subset \bD(c_{l'/k'}^{(j')},\kappa_{j',i'}). \end{cases}\label{eee10.207}
\ee
As a consequence,  for $j\in \N\cap[i_{-}(\rho),i-1]$, $(k,l)\in E_{j}$ one has 
the inclusion $\bD(c_{l/k}^{(j)},s_{j,i-1}K_{j}^{-1})\subset \bD(c(t,i),s_{j(t,i),i-1}K_{j(t,i)}^{-1})$,
and therefore ({\it cf.} (\ref{ee10.207}), (\ref{10.207ante}))
$$U_{i}^{((3/2)\rho)}=\bD(0,(3/2)\rho)\setminus\bigcup_{t\in Q_{i}}\bD(c(t,i),s_{j(t,i), i-1}K_{j(t,i)}^{-1}).$$
This implies that any $D\in\cD_{(3/2)\rho}(U_{i})$  is of the form 
\be  D=\bD(c(t,i),s_{j(t,i), i-1}K_{j(t,i)}^{-1}),\qquad t\in Q_{i},\quad j(t,i)\leq i-1.\label{b10.210}
\ee

\medskip\noindent{\it Proof of  item \ref{ab1} of Proposition \ref{prop:a10.4}.}
This is  a consequence of (\ref{b10.210}) and (\ref{eee10.207}).\ \hfill $\Box$

\medskip\noindent{\it Proof of  item \ref{ab2} of Proposition \ref{prop:a10.4}.}
 One can write for some $t\in Q_{i}, t'\in Q_{i'}$,  $D=\bD(c(t,i),s_{j(t,i),i-1}K_{j(t,i)}^{-1})$,  $D'=\bD(c(t',i'),s_{j(t',i'), i'-1}K_{j(t',i')}^{-1})$ and from Lemma \ref{cor:6.5}, Item (\ref{ni2}) if  $D\cap D'\ne\emptyset$ one has $t=t'$.
On the other hand since $t=t'\in Q_{i'}\subset Q_{i} $ one has $j(t,i')=j(t,i)$. We now use the fact that $s_{j(t,i'), i'-1}\leq s_{j(t,i),i-1}$.\ \hfill $\Box$

\medskip\noindent{\it Proof of  item \ref{ab4} of Proposition \ref{prop:a10.4}.}
 Let us prove that  $D\in \cD_{\rho}(U_{i_{+}(\rho)})$  is a subset of $U_{i_{D}}$. If this were not the case,  there would exist $D'\in\cD_{}(U_{i_{D}})$ such that $D'\cap D\ne \emptyset$; in particular $D'\in\cD_{(3/2)\rho}(U_{i_{D}})$ and from item  \ref{ab2} $D'\subset D$; but this contradicts the definition of $i_{D}$. Hence $D\subset U_{i_{D}}$. 
 
 This latter inclusion and  (\ref{ee10.207}) applied  with $i=i_{D}$ show  that one has $D\cap \bD(c_{l/k}^{(j)},s_{j,i-1}K_{j}^{-1})=\emptyset$ for all $i_{-}(\rho)\leq j\leq i_{D}-1$, $(k,l)\in E_{j}$. As a consequence  $D=\bD(c_{l/k}^{(j)},s_{j,i_{+}(\rho)-1}K_{j}^{-1})$ for some $j\geq i_{D}$, $(k,l)\in E_{j}$.

 On the other hand, by definition of $i_{D}$ there exists  $D'\in \cD_{\rho}(U_{i_{D}+1})$ of the form $D'=\bD(c', s'K_{j'}^{-1})$ with $j'\leq i_{D}$, $s'\in [1,2]$ ({\it cf.} (\ref{b10.210})) such that $D'\subset D$. One hence have  $s_{j,i_{+}(\rho)-1}K_{j}^{-1}\geq K_{i_{D}}^{-1}$  thus $j\geq i_{D}$. We conclude that $j=i_{D}$.
\ \hfill $\Box$

\section{Classical KAM measure estimates}\label{KAMestappendixante}
\subsection{A lemma}\label{sec:H1}
\begin{lemma}\label{lemma42appendix} Let $A=I\setminus\bigcup_{j\in J}I_{j}$, where $I$ is an interval and all the intervals  are disjoint. Then if $\sum_{j\in J}|I_{j}|^{1/2}\leq 1$ and if $g:M_{\R}\to M_{\R}$ is  a $C^1$-symplectic diffeomorphism such that $\|g-id\|_{C^1}\leq 1/10$, then  one has ${\rm Leb}(W_{A}\ \triangle\  W_{g(A)})\lesssim \|g-id\|_{C^0}^{1/2}$.
\end{lemma}
\begin{proof}We can assume that the intervals $I_{j}$ are contained in $I$. Recall that ${\bf 1}_{A\ \triangle \ B}=|{\bf 1}_{A}-{\bf 1}_{B}|$ and notice that since the intervals  $I_{j}$ are pairwise disjoint   one has ${\bf 1}_{W_{A}}={\bf 1}_{W_{I}}-\sum_{j\in J}{\bf 1}_{W_{I_{j}}}$ hence
$${\bf 1}_{W_{A}\ \triangle\ g(W_{A})}=\biggl|\chi-\sum_{j\in J}\chi_{{j}}\biggr|$$
where $\chi={\bf 1}_{W_{I}}-{\bf 1}_{g(W_{I} ) }$,  $\chi_{{j}}={\bf 1}_{W_{I_{j}}  }-{\bf 1}_{g( W_{I_{j}} ) }$.
This gives
\begin{align*}{\Leb}_{M_{\R}}(W_{A} \ \triangle\ g(W_{A}  )  )&=\| \chi-\sum_{j\in J}\chi_{{j}} \|_{L^1}\\
&\leq \|\chi\|_{L^1}+\sum_{j\in J}\|\chi_{{j}}\|_{L^1}\\
&\leq {\Leb}_{M_{\R}}(W_{I} \ \triangle\ g(W_{I}  )  )+\sum_{j\in J}{\Leb}_{M_{\R}}(W_{I_{j}} \ \triangle\ g(W_{I_{j}}  )  ).
\end{align*}
On the other hand if $I$ is an interval there exist intervals $\check I\subset I \subset \hat I$ such that $W_{\check I}\subset g(W_{I})\subset W_{\hat I}$ and  $\max(| I\ \triangle  \ \hat I |, |I\ \triangle \  \check I||)\leq 2\max(\|g-id\|_{C^0},\|g^{-1}-id\|_{C^0})\leq C \|g-id\|_{C^0}$, $C>0$ depending only on $M$ (recall that we have assumed $\|g-id\|_{C^1}$ is small enough). This is clear in the (AA)-case and in the (CC) or (CC*)-case  it follows from the (AA)-case using the symplectic changes of coordinates $\psi_{\pm}$ and $\ph$ (\ref{defPsi}), (\ref{changecoordxyzw}). Therefore since $g$ is symplectic,
$${\Leb}_{M_{\R}}(W_{I_{j}} \ \triangle\ g(W_{I_{j}}  )  )\leq C\min(\|g-id\|_{C^0},{\Leb}_{M_{\R}}(W_{I_{j}})).
$$
In particular ${\Leb}_{M_{\R}}(W_{I_{j}} \ \triangle\ g(W_{I_{j}}  )  )\leq C\|g-id\|_{C^0}^{1/2}{\Leb}_{M_{\R}}(W_{I_{j}})^{1/2}$ and since ${\Leb}_{M_{\R}}(W_{I_{j}})\leq |I_{j}|$ the conclusion follows.
\end{proof}

\subsection{Proof of Theorem \ref{measureestimates} }\label{KAMestappendix}

We use the notations of Section \ref{sec:5} and Propositions \ref{prop:1.enonce},\ref{prop:1.enoncebis}, \ref{prop:7.5} and  Remark \ref{rem:7.1}. 

We apply Proposition \ref{prop:7.5} - Remark \ref{rem:7.1} with $m=1$ and Proposition \ref{prop:prop4.1} with $A=e^{-2\d_{1}}U$, $L=L_{1,\textrm{Prop.}\ \small\ref{prop:7.5}}$, $\ti A=\overline{U_{1}}=\overline{U}$,
\begin{align*}{\rm Leb}_{M_{\R}}(W_{\R\cap e^{-2\d_{1}} U}\setminus\cL(f,W_{\R\cap\bar U} )) &\leq C\times ( {\rm Leb}_{\R}(\R\cap (e^{-2\d_{1}} U\setminus L))+\|g_{1,\infty}-id\|_{C^0}^{1/2})\\
&\lesssim  \bar\e^{\frac{1}{2(\bar a_{0}+3)}}+\bar\e^{\hs 1/8}\lesssim \bar\e^{\frac{1}{2(\bar a_{0}+3)}}.
\end{align*}
\hfill $\Box$

\section{From (CC) to (AA) coordinates}\label{sec:4.6}
We sometime need to reduce the (CC)-case to the (AA)-case, for example when defining the Hamilton-Jacobi Normal Form in Section \ref{sec:HJBNF} or in Section \ref{sec:16}.

\medskip
For $\a\in ]0,\pi[$ define  the angular sector $\Delta^+_{\a}(\rho)=\{r\in \bD(0,\rho),\ \arg(r)\notin [-\a,\a] \}$ and $\Delta^-_{\a}(\rho)=-\Delta^+_{\a}(\rho)$. Recall the definition of the maps $\psi_{\pm}$, {\it cf.} (\ref{defPsi}) of Subsection \ref{sec:2.2}.
\begin{lemma}\label{lemma:4.5}Let $c\in\R$,  $F^{CC}\in\cO_{\s}(W^{CC}_{h,\bD(c,2\rho)})$, $\bar\e=Ce^{h/2}\|D^2F^{CC}\|_{W^{CC}_{h,\bD(c,2\rho)}}$, 
\be C\d^{-2}\rho^{-2}\bar\e<1\label{4.81a}
\ee
and  $\a\in ]\d,\pi-\d[$.
\begin{enumerate}
\item \label{item1:4.80} if $c=0$ and $F^{CC}=O^3(z,w)$ there exists $F_{\pm}^{AA}\in \cO_{\s}(\T_{h-\d}\times \Delta^\pm_{\a+4\d}(\rho-4\d))$ such that on $\T_{h-4\d}\times \Delta^\pm_{\a+4\d}(\rho-4\d) $ one has 
\be f_{F_{\pm}^{AA}}=\psi_{\pm}^{-1}\circ f_{F^{CC}}\circ \psi_{\pm},\qquad F_{\pm}^{AA}=F^{CC}\circ\psi_{\pm}+\fO_{2}(F^{CC}).\label{4.80a}\ee
\item\label{item2:4.80} if $|c|>4\rho$, there exists $F_{\pm}^{AA}\in \cO_{\s}(\T_{h-\d}\times \bD(0,\rho))$ such that (\ref{4.80a}) holds.
\end{enumerate}
\end{lemma}
\begin{proof}We prove item (\ref{item1:4.80}), the proof of item (\ref{item2:4.80}) is done in a similar (and even simpler) way. 
From $\bar\e\leq \d$ and $f_{F}(0)=0$ we get that 
if $z,w$ satisfy $|z|,|w|<e^{h-\d}(\rho-3\d)^{1/2}$, $r=-izw\in  \Delta_{\a+3\d}^\pm(\rho-3\d)$ then $\ti z,\ti w$ defined as $(\ti z,\ti w)=f_{F^{CC}}(z,w)$ satisfy $|\ti z-z|\leq e^{-h/2}\bar\e |z|$, $|\ti w-w|\leq e^{-h/2}\bar\e |w|$  and thus $|\ti z|\leq e^{\bar\e} |z|\leq e^{h}\rho^{1/2}$ and $|\ti w|\leq e^{\bar \e} |w|\leq e^{h}\rho^{1/2}$; on the other hand  if $\ti r=-i\ti z \ti w$ one has 
\be |\ti r -r|\leq 3\bar\e|r|,\qquad |\arg(\ti r)-\arg(r)|\leq 3\bar\e,\qquad \biggl|\frac{\ti z/\ti w}{z/w}-1\biggr|\leq (5/2)\bar\e.\label{e4.83}\ee 
Since $\bar\e<\d$, 
$$f_{F^{CC}}\circ \psi_{\pm}(\T_{h-3\d}\times \Delta^\pm_{\a+3\d}(0,\rho-3\d))\subset \psi_{\pm}(\T_{h}\times \Delta^\pm_{\a}(0,\rho))$$
hence $f^{AA}:=\psi_{\pm}^{-1}\circ f_{F^{CC}}\circ \psi_{\pm}:\T_{h-3\d}\times \Delta^\pm_{\a+3\d}(0,\rho-3\d)\to \T_{h}\times \Delta^\pm_{\a}(0,\rho) $ is well defined. On the other hand if  $\psi_{\pm}^{-1}(z,w)=(\th,r)$, $f^{AA}_{\pm}(\th,r)=(\ti\th,\ti r)$, $\psi_{\pm}(\ti\th,\ti r)=(\ti z,\ti w)$ one has from  (\ref{e4.83}) and Lemma \ref{lem:D1}
$$\max(|\ti \th-\th|_{2\pi\Z}, |\ti r-r|)\leq 3\bar\e
$$
hence
$$\|f^{AA}-id\|_{\T_{h-3\d}\times \Delta^\pm_{\a+3\d}(0,\rho-3\d)}\leq 3\bar\e$$
and from Remark \ref{sec:rem4.1},  Lemmata  \ref{domaindef}, \ref{lemma:6.2} and condition (\ref{4.81a}) there exists $F_{\pm}^{AA}\in\cO(\T_{h-4\d}\times \Delta^\pm_{\a+4\d}(0,\rho-4\d))$ such that $f_{F^{AA}_{\pm}}=f^{AA}$ and 
\be f_{F^{AA}_{\pm}}=\phi^1_{J\nabla F^{AA}_{\pm}}\circ f_{\fO_{2}(F_{\pm}^{AA})}. \label{4.84b}\ee
To get the second estimate in (\ref{4.80a}) we notice that 
$$f_{F^{CC}}=\phi^1_{J\nabla F^{CC}}\circ f_{\fO_{2}(F^{CC})}
$$
hence
\begin{align*} f_{F_{\pm}^{AA}}&=\psi_{\pm}^{-1}\circ \phi^1_{J\nabla F^{CC}}\circ f_{\fO_{2}(F^{CC})}\circ \psi_{\pm}\\
&=\phi^1_{J\nabla (F^{CC}\circ \psi_{\pm})}\circ \psi_{\pm}^{-1} \circ f_{\fO_{2}(F^{CC})}\circ \psi_{\pm}\\
&=\phi^1_{J\nabla (F^{CC}\circ \psi_{\pm})}\circ f_{\fO_{2}(F^{CC})}
\end{align*}
and from (\ref{4.84b})
$$F^{AA}_{\pm}=F^{CC}\circ \psi_{\pm}+\fO_{2}(F^{CC}).
$$
\end{proof}

\begin{rem}\label{rem:4.2}If $f^{CC}=\Phi^{CC}_{\Omega}\circ f_{F^{CC}}$ we have ({\it cf.} Subsection \ref{sec:4.2}).
$$\psi_{\pm}^{-1}\circ f^{CC}\circ \psi_{\pm}=\Phi^{AA}_{\Omega}\circ f_{F^{AA}_{\pm}}.$$
\end{rem}

\section{Some Lemmas from Section \ref{sec:HJBNF}}

\subsection{Proof of Lemma \ref{lem:10.7}  }\label{lemma:tiPibarPi}

Since $\pa^2_{r}\ti \Omega(r)\asymp 1$ ({\it cf.} (\ref{ee8.183})), (\ref{e11.212}) and (\ref{11.193}) show that  there exists  $e_{0}:\T_{qh/3}\to \C$, $e_{0}\in \cO_{\s}(\T_{qh/3})$  such that 
\be\forall \ \th\in \T_{qh/3},\  \pa_{r}\ti \Pi(\th,e_{0}(\th))=0,\qquad \|e_{0}\|_{\T_{qh/3}}\lesssim  (q\bar \rho)^{-1} q^2\bar \e.\label{tayl1}\ee
We now make a Taylor expansion: using (\ref{tayl1}) we see that 
\be
\begin{split}\ti \Pi(\th,r)&=\ti \Pi(\th,e_{0}(\th)+(r-e_{0}(\th))\\
&\begin{multlined}=\ti \Pi(\th,e_{0}(\th))+(1/2)\pa^2_{r}\ti \Pi(\th,e_{0}(\th))(r-e_{0}(\th))^2\\
+(r-e_{0}(\th))^3\sum_{k=3}^\infty \frac{1}{k!}\pa_{r}^k\ti \Pi(\th,e_{0}(\th))(r-e_{0}(\th))^{k-3}\end{multlined}
\end{split}
\ee
and if we define
\be \varpi(\th)=(1/2)\pa^2_{r}\ti \Pi(\th,e_{0}(\th)),\qquad e_{1}(\th)=-\ti \Pi(\th,e_{0}(\th))/\varpi(\th)\label{K.525}\ee
one gets for some $p(\th,r)$
\begin{align*}
&\ti\Pi(\th,r)=\varpi(\th)\biggl( -e_{1}(\th)+(r-e_{0}(\th))^2+(r-e_{0}(\th))^3p(\th,r)\biggr)\\
&=\bar\Pi(\th,r-e_{0}(\th))
\end{align*}
with $\bar \Pi\in\cO(\T_{qh/3}\times \bD(0,e^{-2/q}q\bar\rho/2-Cq\bar \rho^{\ -1}\bar \e))\subset \cO(\T_{qh/3}\times \bD(0,\rho_{q}))$ 
$$\bar\Pi(\th,r)=\varpi(\th)\biggl( r^2-e_{1}(\th)+r^3p(\th,r+e_{0}(\th))\biggr)$$
which gives the desired form for $\bar \Pi(\th,r)$ if one sets  $f(\th,r)=p(\th,r+e_{0}(\th))$.

The estimates (\ref{deftiPi}) on $e_{0},e_{1},\varpi$ are then clear from (\ref{tayl1}), (\ref{K.525}). Let us check the one on $f$. From (\ref{e11.212bis}) and (\ref{10.208}) we have
$$\varpi(\th)(r^2-e_{1}(\th)+r^3f(\th,r))=:\varpi r^2+\sum_{i=0}^2f_{i}(\th)(r+e_{0}(\th))^{i}+r^3(b(r)+ \ti f(\th,r+e_{0}(\th)))
$$
hence from (\ref{eq:another}) and the first two inequalities of (\ref{deftiPi})
$$r^3\biggl(f(\th,r)-\varpi(\th)^{-1}\biggl(b(r)-\ti f(\th,r+e_{0}(\th)) \biggr) \biggr)\lesssim (q\bar \rho)^{-3}q^2\bar\e
$$
and by the maximum principle
$$\sup_{(\th,r)\in \T_{qh/3} \times \bD(0,\rho_{q}) } \biggl|f(\th,r)-\varpi(\th)^{-1}\biggl(b(r)-\ti f(\th,r+e_{0}(\th))) \biggr)\biggr|\lesssim  \bar \rho^{-3}(q\bar \rho)^{ \hs-3}q^2\bar\e\ll 1.
$$
We the conclude by (\ref{ee8.183}) and (\ref{eq:another}).
\hfill $\Box$

\subsection{Square roots}
\begin{lemma}\label{lemma:11.10} Let $a\in\C^*$. There exist a unique function $m_{a}(z)=z(1+a/z^2)^{1/2}$ univalent  on $\C\setminus \bar\bD(0,|a|^{1/2})$ such that 
\be m_{a}^2(z)=z^2+a,\qquad m_{a}(z)=z+O(z^{-1}).
\ee 
It satisfies for $z,z'\in E_{L}:=\{w\in\C,|w|>L|a|^{1/2}\}$ ($L>3$)
\be (2/\pi)e^{-2/L^2} \leq \biggl|\frac{m_{a}(z)-m_{a}(z')}{z-z'}\biggr|\leq (\pi/2)e^{1/L^2}\label{K.488}
\ee
\end{lemma}
\begin{proof} The existence and uniqueness  of $m_{a}(z)=z(1+(a/z^2))^{1/2}$ is clear.  

Note that the inverse for the composition  of $m_{a}$ is $m_{-a}$ and that if $L>2$ $m_{a}(E_{L})\subset E_{3L/4}$.  On the other hand the derivative of $m_{a}(z)$ is equal to $\pa_{z}m_{a}(z)=(1+a/z^2)^{-1/2}$ and  since for $t\in [0,1/2]$, $(1-t)^{-1/2}\leq 1+t$ one gets for $z\in E_{L}$ ($L>2$)  $|\pa_{z}m_{a}(z)|\leq e^{1/L^2}$. Now any two points $z,z'\in E_{L}$ can be joined by a path in $E_{L}$ the length of which is $\leq (\pi/2) |z-z'|$; thus for any $z,z'\in E_{L}$, $|m_{a}(z)-m_{a}(z')|\leq (\pi/2)e^{1/L^2}|z-z'|$ which is the right hand side inequality of (\ref{K.488}). To get the left hand side  we use the fact that $|m_{-a}(m_{a}(z))-m_{-a}(m_{a}(z'))|\leq (\pi/2)e^{1/(3L/4)^2}|m_{a}(z)-m_a(z')| $ if  $L>3$ ($3L/4>2$).

\end{proof}

\subsection{Proof of Lemma \ref{lemma:10.2}}\label{sec:G2}
 From Lemma \ref{lemma:11.10}  $z\mapsto (z^2+a)^{1/2}$ is well defined on $\C\setminus\{|z|>|a|^{1/2}\}$.

Let $0\leq s \leq h/3$. We are looking for   $g(\th,z)=\varpi(\th)^{-1/2}z(1+\mathring{g}(\th,z))$ such that 
\begin{multline*} z^2=\varpi(\th)\biggl(z^2\varpi(\th)^{-1}(1+\mathring{g}(\th,z))^2-e_{1}(\th)+z^3\varpi(\th)^{-3/2}(1+\mathring{g}(\th,z))^3f(\th,g(\th,z))\biggr)
\end{multline*}
which can be written as a Fixed Point problem
\begin{multline}\mathring{g}(\th,z)=\biggl(1+\varpi(\th)\frac{e_{1}(\th)}{z^2}-z\varpi(\th)^{-1/2}(1+\mathring{g}(\th,z))^{3}f(\th,g(\th,z))\biggr)^{1/2}-1\label{10.217}
\end{multline}
Using the estimate on $f$ given by (\ref{deftiPi}) one can see that  the map $\Psi:\mathring{g}\mapsto$ R.H.S. of  (\ref{10.217}) defines a $2\rho_{q}$-contracting map on the ball $B(0,CL^{-2})$ of center 0 and radius $CL^{-2}$ of the Banach space $(\cO(\T_{sq}\times \bA(\l_{s,L},\rho_{q})),\|\cdot\|_{\infty})$ provided $L^{-1}$  and $\rho_{q}$ are small enough. By the Contraction Mapping Theorem it has a unique fixed point $\mathring{g}$ in this ball.
In other words
\be \bar \Pi(\th,g(\th,z))=z^2.
\ee
The fact that $g\in \cO(\T_{qs}\times\bA(\l_{s,L},\rho_{q}))$ is uniquely defined  shows that the various $g$ found for different values of $s$ must agree. Hence $g$ is defined on $\bigcup_{0\leq s\leq 1}(\T_{qs}\times\bA(\l_{s,L},\rho_{q}))$.
\hfill $\Box$

\subsection{Proof of Lemma \ref{lemma:7.6}}\label{sec:G3}
We look for $H$ under the form $H(z)=\gamma^{-1}z(1+\mathring{H}(z))$. Equation (\ref{eq:8.197}) can be written as a Fixed Point problem:
\be \mathring{H}(z)=-\frac{\mathring{\Gamma}(\gamma^{-1}z(1+\mathring{H}(z))))}{(1+\mathring{\Gamma}(\gamma^{-1}z(1+\mathring{H}(z))))}.\label{10.224}
\ee
By Cauchy estimates for $z\in \bA(\l_{s},\rho_{q})$
$$|\pa\mathring{\Gamma}(z)|\leq \frac{1}{\dist(z,\pa\bA(\l_{s,L},\rho_{q}))}L^{-2}.$$
Hence  if $z\in \bA(2\l_{s,L},(1/2)\rho_{q})$ the map $u\mapsto \mathring{\Gamma}(\gamma^{-1}zu)$ is $4L^{-2}$-lipschitz on  $\{(3/4)\leq |u|\leq 4/3\}$ and the map $\Psi$ defined by the R.H.S. of (\ref{10.224}) is $4L^{-2}$contracting on the ball $\{\|\mathring{H}\|_{\bA(2\l_{s,L},(1/2)\rho_{q})}\leq 2 L^{-2}\}$. It admits thus a unique fixed point in this ball. 
\hfill $\Box$

\section{Some other lemmas}

\begin{lemma}\label{lem:D1}For $z\in\C$ 
$$|e^{iz}-1|\geq \frac{1}{2}\min(1,\min_{l\in\Z}|z-2\pi l|).
$$ 
\end{lemma}
\begin{proof} Let $\eta:=e^{iz}-1$. We can assume $|\eta|<1/2$.  We can thus define $iz_{0}:=\ln(1+\eta)=\sum_{k\in\N^*}(-1)^kz^k/k$ such that $e^{iz_{0}}=1+\eta=e^{iz}$. There thus exists $l\in\Z$ such that $z_{0}=z-2\pi l$. But $|z_{0}|=|\ln(1+\eta)|\leq 2|\eta|$.
\end{proof}

\begin{lemma}\label{L2C0}Let $f\in C^\omega_{h}(\T)$ be such that for some $\d\in ]0,1[$, $\mu>0$
\be \|f\|_{L^2(\T)}\leq \d \|f\|_{C^0(\T)}+\mu.
\ee
Then, for some $C>0$,
\be
\|f\|_{C^0(\T)}\leq \d^{-1}\mu+ \frac{C}{h}e^{-h/(12\d^2)}\|f\|_{h}.
\ee
\end{lemma}
\begin{proof}If 
\be f(\th)=\sum_{k\in\Z}\hat f(k) e^{i ik \th}
\ee
is the Fourier expansion of $f$, one has for some $C>0$ and any  $N\in\N^*$
\begin{align} \|f\|_{C^0(T)}&\leq \sum_{|k|\leq N}|\hat f(k)|+\frac{C}{h}e^{- hN}\|f\|_{h}\\
&\leq (2N+1)^{1/2}\|f\|_{L^2(\T)}+\frac{C}{h}e^{- hN}\|f\|_{h}\\
&\leq (3N)^{1/2}(\d\|f\|_{C^0(\T)}+\mu)+\frac{C}{h}e^{- hN}\|f\|_{h}.
\end{align}
If we choose $N=\d^{-2}/12$ we have $(3N)^{1/2}\d\leq 1/2$ and 
\be
\|f\|_{C^0(\T)}\leq \d^{-1}\mu+ \frac{C}{h}e^{- h/(12\d^{2})}\|f\|_{h}.
\ee 
\end{proof}

\section{Stable and unstable  Manifolds}\label{sec:K}
\subsection{The Stable Manifold Theorem}
Let $(E,\|\cdot \|)$ be a Banach space, $M:E\to E$ an invertible linear continuous map. Let $\kappa,\d>0$. We say that $M$ is $(\kappa,\d)$-{\it  hyperbolic} if there exist $\kappa>0$ and continuous projectors $P_{s}$, $P_{u}$ satisfying $id_{E}=P_{s}+P_{u}$, $P_{s}P_{u}=P_{u}P_{s}=0$, $P_{s}MP_{u}=P_{u}MP_{s}=0$ such that
$$\begin{cases}&\max (\|P_{s}MP_{s}\|,\|(P_{u}MP_{u})^{-1}\|)\leq e^{-\kappa}\\
& \max(\|P_{s}\|,\|P_{u}\|)\leq \d^{-1}.
\end{cases}
$$
The spaces $E_{*}:=P_{*}E$, $*=s,u$, are then $M$-invariant and are the stable and unstable spaces of the linear map  $M$. We shall use the notations $M_{*}=P_{*}MP_{*}$, $*=s,u$. 

Let $B(0,\rho)\subset E$ be the ball of center 0 and radius $\rho>0$.
\begin{theo}[Stable/Unstable Manifold Theorem]\label{theo:stableunstabletheorem}Assume that $M$ is $(\kappa,\d)$-hyperbolic as above and let $F:B(0,\rho)\to E$ be $C^1$. Assume that 
\be \|F(0)\|\leq C^{-1}\d \kappa\rho,\quad  \|DF\|_{C^1(B(0,\rho))}\leq C^{-1}\d \kappa.\label{F407}
\ee
Then, if $C$ is large enough (but universal)
\begin{enumerate}
\item The map $x\mapsto Mx+F(x)$ has a unique  hyperbolic fixed point $\bar x$ such that $\max(\|P_{s}\bar x\|,\|P_{u}\bar x\|)\leq (\rho/4)$ (in particular, it is  located in $B(0,\rho/2)$).
\item The local stable (resp. unstable) manifold $$W^s_{loc}(\bar x;M+F):=\{y\in B(\bar x,\rho/4),\ \forall\ n\geq 0,\ (M+F)^n(y)\in B(\bar x,\rho/2)\}$$ (resp.  $W^u_{loc}(\bar x;M+F):=\{y\in B(\bar x,\rho/4),\ \forall\ n\leq 0,\ (M+F)^n(y)\in B(\bar x,\rho/2)\}$) 
of the point $\bar x$ for  $M+F$ is of the form $\{x_{s}+\g_{s,F}(x_{s}), \ x_{s}\in E_{s}\cap B(0,\rho/2)\}$ (resp. $\{x_{u}+\g_{u,F}(x_{u}), \ x_{u}\in E_{u}\cap B(0,\rho/2)\}$)
where $\g_{s,F}:E_{s}\to E_{u}$ (resp. $\g_{u,F}:E_{u}\to E_{s}$)
is a map of class $C^1$ and $\|D\gamma_{s,F}\|\leq  C\|DF\|_{B(0,\rho)}(\d\kappa)^{-1}$ (resp. $\|D\gamma_{u,F}\|\leq  C\|DF\|_{B(0,\rho)}(\d\kappa)^{-1}$).
\item If $G$ satisfies also (\ref{F407}) then for $*=s,u$,  $\|D\gamma_{*,F}-D\gamma_{*,G}\|\leq  C\|D(F-G)\|_{B(0,\rho)}(\d\kappa)^{-1}$.
\item If $F(0)=0$, $DF(0)=0$ then $\bar x=0$ and  $T_{0}W_{loc}^*(0)=E_{*}$, $*=s,u$.
\end{enumerate} 
\end{theo}
Notice that the Theorem gives the same size for the domains of definition for $\g_{s,F}$, $\g_{u,F}$.

\subsection{Proof of Lemma \ref{lemma:15.4}}\label{sec:K2}

Using the definition of $f_{H_{Q}+\ti\omega}(\th,r)=(\ph,R)$ one can see that 
$f_{H_{Q}+\ti\omega}(\th,r)=(\th,r)$
if and only if 
\be \nabla H_{Q}(\th,r)+\nabla\ti\omega(r)=0\label{fp0}\ee or equivalently 
\begin{align*}& 0=\ti a(0)\th+\pa_{r}\ti b(0)r\\
&0=\pa_{r}\ti b(0)\th+(\varpi+\pa^2_{r}\ti a(0))r+\pa_{r}\ti a(0)+\pa_{r}\ti\omega(r)
\end{align*}
Solving the first equation and inserting it into the second yields
\begin{align}& \th=-\frac{\pa_{r}\ti b(0)}{\ti a(0)}r \label{fp1}\\
&r=-\frac {\pa_{r}\ti a(0)+\pa_{r}\ti\omega(r)} {\varpi+\pa^2_{r}\ti a(0)-\frac {(\pa_{r}\ti b(0))^2}{\ti a(0)}}.\label{fp2}
\end{align}
We observe that, {\it cf.}  (\ref{16.329}), 
$$\max(|\pa_{r}^2\ti a(0)|,|(\pa\ti b_{r}(0))^2/\ti a(0)|)\lesssim q^2\bar \nu_{q}^{\hs -1}\rho_{p/q}^{-2}e^{-qh}\e_{p/q}<\varpi/10
$$
$$|\pa_{r}\ti a(0)|\lesssim q^2\rho_{p/q}^{-1}e^{-qh}\e_{p/q},\qquad \pa_{r}\ti \omega(r)=O(r^2)$$
and deduce by a simple fixed point theorem (in dimension 1) that (\ref{fp2}) has a unique real  solution $r_{0}\asymp \pa_{r}\ti a(0)$; returning  to (\ref{fp1}) and using ({\it cf.} (\ref{16.329bis}), (\ref{16.329}))
\be \ti a(0)=q^2\bar \nu_{q}e^{-qh}\e_{p/q},\qquad |\pa_{r}\ti b(0)|\lesssim q^2\rho_{p/q}^{-1}e^{-qh}\e_{p/q}\label{ememb}
\ee
we conclude  that (\ref{fp0}) has also a unique solution $(\th_{0},r_{0})\in \bD(0,\rho_{p/q})^2$
$$|\th_{0}|\lesssim \bar \nu_{q}^{\hs -1}q^2\rho_{p/q}^{-2}\e_{p/q}e^{-qh},\qquad |r_{0}|\lesssim q^2\rho_{p/q}^{-1}\e_{p/q} e^{-qh}
$$
and in particular since $\rho_{p/q}^{-8}=\max((c_{p/q}/4)^{-8},q^{72})$ ({\it cf.} \ref{defrhop/q})), $q^2e^{-qh}=O(q^{-100})$, $\e_{p/q}\leq c_{p/q}^{\bar a_{4}}$ ({\it cf.} \ref{16.320})), $\nu_{q}\gtrsim q\rho_{p/q}$    ({\it cf.} \ref{15.255a})), $\bar a_{4}\geq 10$, one has 
\be(\th_{0},r_{0})\in (\bD(0,\rho_{p/q}^5)\times \bD(0,\rho_{p/q}^5))\cap\R^2.\label{16.344}
\ee
We now compute $Df_{H_{Q}+\ti\omega}(\th_{0},r_{0})$.  Since $\ti\omega$ depends only on the $r$-variable one has, {\it cf.} (\ref{eq:4.55}) of Lemma \ref{lemma:6.2.bis},
$$f_{H_{Q}+\ti\omega}=\Phi_{\ti\omega}\circ f_{H_{Q}}$$
hence
$$Df_{H_{Q}+\ti\omega}(\th_{0},r_{0})=\bm 1& \pa_{r}^2\ti\omega(r_{0})\\ 0 &1\em Df_{H_{Q}}.$$
A simple computation shows that 
the derivative of the  symplectic map $f_{H_{Q}}$ is equal to
$$Df_{H_{Q}}=\bm 1+\pa\ti b(0)+\frac{(\varpi+\pa^2_{r}\ti a(0))\ti a(0) }{1+\pa\ti b(0)}& \frac{\varpi+\pa^2_{r}\ti a(0)}{1+\pa\ti b(0)}\\ \frac{\ti a(0)}{1+\pa\ti b(0)} & \frac{1}{1+\pa\ti b(0)}  \em
$$
hence 
$${\rm tr}(Df_{H_{Q}+\ti \omega})=2+\varpi \ti a(0)(1+O(q^2\rho_{p/q}^{-2}\e_{p/q}e^{-qh}))+O((\pa_{r}\ti b(0))^2)+\frac{\pa^2\ti\omega(r_{0})\ti a(0)}{1+\pa\ti b(0)}.$$
The estimate (\ref{16.344}) on $r_{0}$, the fact that $\pa^2\ti\omega(r_{0})=O(r_{0})$  and (\ref{ememb}) show that
$${\rm tr}(Df_{H_{Q}+\ti \omega})=2+\varpi q^2\bar \nu_{q}\e_{p/q}e^{-qh}(1+O(\rho_{p/q}^5)).$$

Since $Df_{H_{Q}+\ti \omega}(\th_{0},r_{0})\in SL(2,\R)$ we deduce that it  is a $(\kappa,\delta)$-hyperbolic matrix with 
\be \delta=\kappa=q(\varpi  \nu_{q}\e_{p/q}e^{-qh})^{1/2}(1+o_{1/q}(1))
\ee
(we used that  $\bar\nu_{q}=\nu_{q}(1+o_{1/q}(1))$).

The statement on the eigendirections is then a simple computation.
$\Box$

\bibliographystyle{plain}

\begin{thebibliography}{99}

\bibitem{Ar}V. I. Arnold, Proof of a theorem of A. N. Kolmogorov on the preservation of conditionally periodic motions under a small perturbation of the Hamiltonian, Usp. Mat. Nauk. {\bf 18} (1963), 13-40.

\bibitem{AKN}V. I. Arnold, V. V.  Kozlov, A. I.  Neishtadt,  {\it Mathematical aspects of classical and celestial mechanics.}  Springer-Verlag, Berlin, 1997. 291 pp. 

\bibitem{AdSK}A. Avila, J.  De Simoi, V.  Kaloshin, 
An integrable deformation of an ellipse of small eccentricity is an ellipse. 
{\it Ann. of Math.} (2) {\bf 184} (2016), no. 2, 527--558. 

\bibitem{Bi3}G. D. Birkhoff, Proof of Poincaré's last geometric theorem, {\it Trans. Amer. Math. Soc.,} {\bf 14} (1913), 14-22.

\bibitem{Bi1} G. D. Birkhoff, Dynamical Systems, A. M. S., Providence, RI, 1927.

\bibitem{Bi2} G. D.  Birkhoff Surface transformations and their dynamical applications, {\it Acta Math.} {\bf 43} (1922), 1-119.

\bibitem {Br} A. D. Brjuno, Analytical form of differential equations I, II, {\it Trans. Mosc. Math. Soc.} {\bf 25} (1971), 119-262, {\bf 26} (1972), 199-239.


\bibitem{CMS}C. Carminati, S. Marmi, D. Sauzin, There is only one KAM curve. {\it Nonlinearity} {\bf  27} (2014), n° 9, 2035--2062.

\bibitem {Del}C. E. Delaunay, Théorie du mouvement de la lune, {\it Paris Mem. Prés.} {\bf 28} (1860), {\bf 29} (1867).

\bibitem{EcVa}J. Ecalle and B. Vallet, Correction and linearization of resonant vector fields and diffeomorphisms, Math. Z. {\bf 229} (1998), 249--318.


\bibitem{E1} L. H. Eliasson, Normal forms for Hamiltonian systems with Poisson commuting integrals-elliptic case, {\it Comment. Math. Helv.} {\bf 65} (1990), 4-35.

\bibitem{E2} L. H. Eliasson, Hamiltonian systems with linear form near an invariant torus, in Non-linear Dynamics (Bologna, 1988), World Sci. Publ., Teaneck, NJ, 1989, 11-29.

\bibitem{EFKbis}  L.H. Eliasson, B. Fayad, R. Krikorian, KAM-tori near an analytic elliptic fixed point, {\it Regular and Chaotic Dynamics}, {\bf 18}  no. 6, pp. 806-836 (2013)


\bibitem{EFK} L.H. Eliasson, B. Fayad, R. Krikorian. Around the stability of KAM tori.  {\it Duke Math. Journ.}, {\bf 164} (2015), no 9, 1733-1775

\bibitem{FaFé} G. Farré, B. Fayad, Instabilities for analytic quasi-periodic invariant tori, \url{https://arxiv.org/pdf/1912.01575.pdf}

\bibitem{F}  B. Fayad, Lyapunov unstable elliptic equilibria. \url{https://arxiv.org/pdf/1809.09059.pdf}


\bibitem{Gong} X. Gong, { Existence of divergent Birkhoff normal forms of hamiltonian functions}, {\it Illinois Jour. Math. } {\bf 56} no 1, 85-94 (2012)

\bibitem{GongStolo2}X. Gong, L. Stolovitch, Real submanifolds of maximum complex tangent space at a CR singular point, I. {\it Invent. Math.} {\bf 206} (2016), no. 2, 293--377.

\bibitem{GongStolo1}X. Gong, L. Stolovitch,  Real submanifolds of maximum complex tangent space at a CR singular point, II.  {\it J. Differential Geom.} {\bf 112} (2019), no. 1, 121--198.




\bibitem{He79}  M. Herman,  Sur la conjugaison différentiable des difféomorphismes du cercle à des rotations. {\it Inst. Hautes Études Sci. Publ. Math.} No. 49 (1979), 5-233. 

\bibitem{He} M. R. Herman, Sur les courbes invariantes par les difféomorphismes de l'anneau, vol. 1, {\it Astérisque}, {\bf 103-104}, Société Mathématique de France, Paris, 1983. i+221 pp.

\bibitem{HSY}B. Hunt, T.  Sauer, J.  Yorke, Prevalence: a translation-invariant ``almost every'' on
infinite-dimensional spaces. {\it Bull. Amer. Math. Soc.} (N.S.) {\bf 27} (1992), no. 2, 217-238.

\bibitem{It}H. Ito, Convergence of Birkhoff normal forms for integrable systems, {\it Comment. Math. Helv.} {\bf 64} (1989), 412-461.

\bibitem{Kap}T. Kappeler, Y. Kodama, and A. Némethi, On the Birkhoff normal form of a completely integrable Hamiltonian system near a fixed point with resonance, {\it Ann. Scuola Norm. Sup. Pisa} {\bf 26} (1998), 623-661.

\bibitem{Ko}A. N. Kolmogorov, Théorie générale des systèmes dynamiques et mécanique classique (Amsterdam, 1954), {\it Proc. Internat. Congress of Math.} {\bf 1} (1957), 315-333.

\bibitem{FaKri} B. Fayad, R. Krikorian, Some questions around quasi-periodic dynamics. {\it Proc. Internat. Congress of Math.}--Rio de Janeiro (2018). Vol. III., 1909--1932, World Sci. Publ., Hackensack, NJ, 2018. 


\bibitem{KaSo}V. Kaloshin, A. Sorrentino, On the local Birkhoff conjecture for convex billiards. {\it Ann. of Math.} (2) {\bf 188} (2018), no. 1, 315--380. 

\bibitem{Lind}A. Lindstedt, Beitrag zur Integration der Differentialgleichungen der St\"orungstheorie, {\it Abh. K. Akad. Wiss. St. Petersburg} {\bf 31} (1882).


\bibitem{Ma}J.N. Mather, Differentiability of the minimal average action as a function of the rotation number, {\it Bol. Soc. Bras. Mat., Nova Ser.} {\bf 21} (1990), 59--70.

\bibitem{MaFo}J.N. Mather, G. Forni, Action minimizing orbits in Hamiltonian systems, in: Graffi (ed.): {\it Transition to Chaos in Classical and Quantum Mechanics}, Springer LNM 1589 (1992), 92--186.

\bibitem{Mo} J. K. Moser, On invariant curves of area-preserving mappings of an annulus, {\it Nachr. Akad. Wiss. G\"ottingen Math.-Phys.} {\bf 1962} (1962), 1-20.

\bibitem{MoWeb}J. Moser, S.M. Webster, Normal forms for real surfaces in $\C^2$ near complex tangents and hyperbolic surface transformations. {\it Acta Math.} {\bf 150} (1983), no. 3--4, 255--296.

\bibitem{Ne} N. N. Nekhoroshev, The behavior of Hamiltonian systems that are close to integrable ones, {\it Funct. Anal. Appl.} {\bf 5} (1971), 82--83.

\bibitem{PM}   R. Pérez-Marco,  Convergence or generic divergence of the Birkhoff normal form. {\it Ann. of Math.} (2) {\bf 157} (2003), no. 2, 557-574. 

\bibitem{Po} H. Poincaré, Les	Méthodes	Nouvelles	de	la	Mécanique	Céleste, Tome I, Chap. 5,	Paris,	1892.

\bibitem{Pos} J. P\"oschel, Integrability of Hamiltonian systems on Cantor sets. {\it Comm. Pure Appl. Math.} {\bf 35} (1982), no. 5, 653--696. 

\bibitem{Ra} T. Ransford, Potential theory in the complex plane, {\it London Math. Soc. Student Texts} {\bf 28}, Cambridge Univ. Press, Cambridge, 1995.

\bibitem{Ru}H. R\"ussmann, \"Uber die Normalform analytischer Hamiltonscher Differentialgleichungen in der N\"ahe einer Gleichgewichtsl\"osung, Math. Ann. 169, 55-72, (1967).

\bibitem{Sib} K. F. Siburg, Symplectic invariants of elliptic fixed points, {\it Comment. Math. Helv.} {\bf 75} (2000) 681--700.

\bibitem{Si2}C. L. Siegel, On the integrals of canonical systems, {\it Ann. of Math.} {\bf 42} (1941), 806--822.

\bibitem{Si} C. L. Siegel, Uber die Existenz einer Normalform analytischer Hamiltonscher Differentialgleichungen
in der N\"ahe einer Gleichgewichtl\"osung, {\it Math. Ann.} 128 (1954),
144-170.

\bibitem{SiMo}C. L. Siegel and J. Moser, Lectures on Celestial Mechanics 187, Springer-Verlag, New York, 1971.


\bibitem{Stein} E.M. Stein. {\it Singular Integrals and Differentiability Properties of Functions}, Princeton Univ. Press, (1970)

\bibitem{Stolovitch} L. Stolovitch, Singular complete integrability. {\it Inst. Hautes \'Etudes Sci. Publ. Math.} {\bf 91} (2000), 133--210 (2001).



\bibitem{V1} J. Vey, Orbites périodiques d'un système hamiltonien du voisinage d'un point d'équilibre. {\it Ann. Scuola Norm. Sup. Pisa Cl. Sci.} (4) {\bf  5} (1978), no. 4, 757--787.

\bibitem{V2}J. Vey, Sur certains systèmes dynamiques séparables. {\it Amer. J. Math.} {\bf 100} (1978), 591--614.

\bibitem{Yin} W. Yin, Divergent Birkhoff normal forms of real analytic area preserving maps. {\it Math. Z.} (2015) {\bf 280} 1005--1014.

\bibitem{Whi} H. Whitney, Analytic extensions of differentiable functions defined in closed sets. 
{\it Trans. Amer. Math. Soc.} {\bf 36} (1934), no. 1, 63--89.
 
\bibitem{Zung} N. T. Zung,  Convergence versus integrability in Birkhoff normal form. {\it Ann. of Math.}, {\bf 161} (2005), 141--156

\end{thebibliography}

\end{document}